\documentclass[12pt,titlepage]{book}
\usepackage[ a4paper,lmargin=2.5cm, rmargin=2.5cm, tmargin=2.85cm, bmargin=2.85cm]{geometry}
\usepackage{fancyhdr}
\newcommand{\chapquote}[2]{\vspace{-12pt}\begin{flushright} \textit{#1} \\ \vspace{3pt}--- #2 \end{flushright} \noindent}

\newcommand\Chapter[2]{
  \chapter[#1]{#1\\[1.5ex]\large\normalfont\textsl{#2}{\vspace{-1.5em}\\}}
}

\usepackage{scalerel,stackengine}
\stackMath
\newcommand\reallywidehat[1]{%
\savestack{\tmpbox}{\stretchto{%
  \scaleto{%
    \scalerel*[\widthof{\ensuremath{#1}}]{\kern-.6pt\bigwedge\kern-.6pt}%
    {\rule[-\textheight/2]{1ex}{\textheight}}
  }{\textheight}%
}{0.5ex}}%
\stackon[1pt]{#1}{\tmpbox}%
}

\usepackage[split]{splitidx}
\makeindex
\newindex[Index]{gen} 
\newindex[List of Symbols]{sym} %

\pdfoutput=1 

\usepackage{amsmath,amsthm,amssymb,enumerate, bbm ,graphicx,color,caption,upgreek, float, tikz, subcaption,booktabs,longtable, appendix,graphics, pdfpages,rotating,mathtools, array, bm, blkarray,setspace,pgfplots,textcomp}
\pgfplotsset{compat=1.7}
\pgfplotsset{every tick label/.append style={font=\tiny}}

\makeatletter
\newlength{\negph@wd}
\DeclareRobustCommand{\negphantom}[1]{%
  \ifmmode
    \mathpalette\negph@math{#1}%
  \else
    \negph@do{#1}%
  \fi
}
\newcommand{\negph@math}[2]{\negph@do{$\m@th#1#2$}}
\newcommand{\negph@do}[1]{%
  \settowidth{\negph@wd}{#1}%
 \hspace*{-\negph@wd}
}
\makeatother

	\let\oldlangle\langle
	\def\langle{\protect\oldlangle}
	\let\oldrangle\rangle
	\def\rangle{\protect\oldrangle}
	
	\let\oldwidetilde\widetilde
	\def\widetilde{\protect\oldwidetilde}
	
		\let\oldwidehat\widehat
	\def\widehat{\protect\oldwidehat}

\newcommand{\nhphantom}[1]{\setbox0=\hbox{#1}\hspace{-\the\wd0}}

\newcommand{\uispace}[0]{\hspace{2.5cm}}

\newcommand{\symlistsort}[3]{\sindex[sym]{#1@#2\negphantom{#2}\uispace #3,\hfill}}

\newcommand{\indexadd}[1]{\sindex[gen]{#1}}

\newcommand{\ftextnumero}{{\fontfamily{txr}\selectfont \textnumero}}
\newcommand{\ftextcopyright}{{\fontfamily{txr}\selectfont \textcopyright}}

\DeclarePairedDelimiter\ceil{\lceil}{\rceil}
\DeclarePairedDelimiter\floor{\lfloor}{\rfloor}
\usepackage[pdftex]{hyperref} 
\usepackage{xcolor,colortbl, youngtab}

\usepackage[overload]{textcase}

\newtheorem*{theoremnn}{Theorem}
\newtheorem*{stellingnn}{Stelling}

\newcommand{\height}{\text{\rm height}}

\definecolor{blauw}{RGB}{61,158,255}
\definecolor{donkerblauw}{RGB}{0,0,255}
\definecolor{donkergroen}{RGB}{46,148,0}
\definecolor{donkerrood}{RGB}{204,0,0}
\definecolor{firebrick}{rgb}{0.7, 0.13, 0.13}
\newenvironment{speciaalenumerate}{
\begin{enumerate}[(i)]
  \setlength{\itemsep}{1pt}
  \setlength{\parskip}{0pt}
  \setlength{\parsep}{0pt}
}{\end{enumerate}}

\makeatletter 
\newcommand\mynobreakpar{\par\nobreak\@afterheading} 
\makeatother

\makeatletter
\let\@fnsymbol\@arabic
\makeatother
\newcommand{\norm}[1]{\left\lVert#1\right\rVert}

\newcommand{\N}{\mathbb{N}}
\newcommand{\Z}{\mathbb{Z}}
\newcommand{\C}{\mathbb{C}}
\newcommand{\R}{\mathbb{R}}
\newcommand{\Q}{\mathbb{Q}}
\newcommand{\F}{\mathbb{F}}

\newcommand{\T}{^{\sf T}}

\newcommand{\CC}{\mathcal{C}}

\usepackage[english, dutch]{babel}
\newtheorem{theorem}{Theorem}[section]

\newtheorem{lemma}[theorem]{Lemma}

\newtheorem{proposition}[theorem]{Proposition}
\newtheorem{corollary}[theorem]{Corollary}

\theoremstyle{definition}
\newtheorem{defn}[theorem]{Definition} 
\newtheorem{assumption}[theorem]{Assumption} 
\newtheorem{examp}[theorem]{Example} 
\newtheorem*{examp*}{Example}

\newtheorem{remark}[theorem]{Remark}

\renewcommand{\part}{\text{\rm part}}
\DeclareMathOperator*{\sgn}{sgn}

\DeclareMathOperator*{\ev}{ev}
\DeclareMathOperator*{\trace}{tr}
\DeclareMathOperator*{\tr}{tr}

\DeclareMathOperator*{\Sym}{Sym}

\DeclareMathOperator*{\wt}{wt}

\theoremstyle{plain}
\newfloat{Algorithm}{!hbt}{alg}

\newcommand{\printv}{\textbf{print }}
\newcommand{\ndv}{\textbf{end }}
\newcommand{\ifv}{\textbf{if }}

\newcommand{\foreachv}{\textbf{foreach }}

\newcounter{thm}[section]

\usetikzlibrary{positioning}

\makeatletter
\tikzset{my loop/.style =  {to path={
  \pgfextra{}
  [looseness=12,min distance=10mm]
  \tikz@to@curve@path},font=\sffamily\small
  }}  
\makeatletter

\makeatletter
\def\cleardoublepage{\clearpage\if@twoside \ifodd\c@page\else
\hbox{}
\vspace*{\fill}
\vspace{\fill}
\thispagestyle{empty}
\newpage
\if@twocolumn\hbox{}\newpage\fi\fi\fi}
\makeatother

\clearpage{\pagestyle{empty}\cleardoublepage}     

\renewcommand{\title}{\LARGE{\bf  New Methods in Coding Theory\\ {\large Error-Correcting Codes and the Shannon Capacity}}}
\newcommand{\authorA}{Sven Carel Polak}
\newcommand{\watisdit}{ACADEMISCH PROEFSCHRIFT}

\begin{document}
\selectlanguage{english}
\hyphenation{Schrij-ver}

\hyphenation{Gijs-wijt}
\frontmatter

\begin{titlepage}
\pagestyle{empty}

\begin{center}
\null
\quad \\
\renewcommand{\title}{\LARGE{\bf  New Methods in Coding Theory\\ {\large Error-Correcting Codes and the Shannon Capacity}}}
\title
\vfill
\LARGE{\bf Sven Carel Polak}
\end{center}

\newpage

\pagestyle{empty}
\quad 
\newpage

\null
\begin{center}
\normalfont {\title\par}
\vfil
\quad \\
\quad \\
\quad \\
{\Large \watisdit}
\\
\vfil
\quad \\
ter verkrijging van de graad van doctor \\
aan de Universiteit van Amsterdam \\
op gezag van de Rector Magnificus \\
prof.\ dr.\ ir.\ K.I.J. Maex \\
ten overstaan van een door het College voor Promoties \\
ingestelde commissie, \\
in het openbaar te verdedigen in de Agnietenkapel \\
op dinsdag 10 september 2019, te 14:00 uur
\\
\quad \\
\quad \\
door
\quad \\
\quad \\
\quad \\
{\Large \authorA}
\\
\quad \\
\quad \\
geboren te Amsterdam
\end{center}

\newpage

\quad \\
Promotiecommissie:

\quad \\
Promotor:

Prof.\ dr.\ A.\ Schrijver \hfill Universiteit van Amsterdam

\quad \\
Copromotor:

Prof.\ dr.\ A.E.\ Brouwer \hfill Technische Universiteit Eindhoven

\quad \\
Overige leden:

Prof.\ dr.\ N.\ Bansal \hfill Technische Universiteit Eindhoven

Prof.\ dr.\ G.B.M.\ van der Geer \hfill Universiteit van Amsterdam

Dr.\ D.C.\ Gijswijt \hfill Technische Universiteit Delft

Dr.\ G.\ Regts \hfill Universiteit van Amsterdam

Prof.\ dr.\ J.V.\ Stokman \hfill Universiteit van Amsterdam

\quad \\
Faculteit der Natuurwetenschappen, Wiskunde en Informatica

\vspace{5cm}
\begin{figure}[H]
\includegraphics[width=7cm]{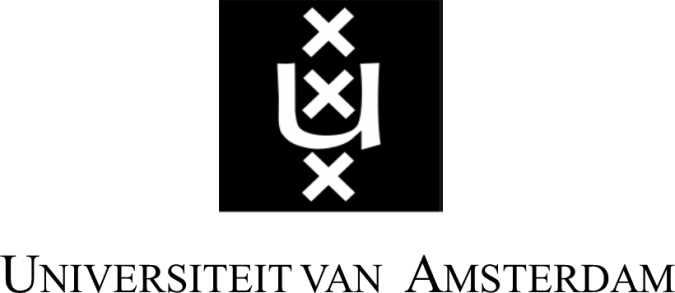}
\end{figure}
\noindent This research has been carried out at the Korteweg-de Vries Institute for Mathematics.
\vspace{12pt}

\noindent The research leading to these
results has received funding from the European Research Council under the European Union’s Seventh Framework Programme (FP7/2007-2013) / ERC grant agreement \ftextnumero\hspace{2pt}339109.
\vspace{12pt}

\noindent Cover design by Evelien Jagtman

\vspace{12pt}

\noindent Printed by Gildeprint

\vspace{12pt}

\noindent Copyright {\ftextcopyright} 2019 by Sven Polak

\vspace{12pt}

\noindent ISBN: 978-94-6323-747-5

\newpage
\pagestyle{empty}
\end{titlepage}

\tableofcontents

\chapter{Dankwoord} \markboth{Dankwoord}{Dankwoord}\selectlanguage{dutch}
Voor u ligt het proefschrift \emph{New Methods in Coding Theory: Error-Correcting Codes and the Shannon Capacity}, het resultaat van een vier jaar durend onderzoek aan het Korteweg-de Vries Instituut voor Wiskunde. Ik ben dankbaar dat ik de mogelijkheid daarvoor heb gekregen. Graag wil ik een aantal mensen persoonlijk bedanken.

De afgelopen vier jaar heeft Lex Schrijver, mijn promotor, mij deskundig en met  grote toewijding begeleid. Lex,  zonder jou was dit proefschrift er niet geweest.  Dank voor je uitstekende begeleiding, voor alle afspraken op het KdvI, voor de vele emails over en weer, voor het nauwkeurig lezen van talloze manuscripten,  voor je vertrouwen, en in het algemeen voor de fijne samenwerking. Ik heb heel veel van je geleerd en had me geen betere begeleider kunnen wensen. Bedankt Lex!  

Mijn dank gaat ook uit naar copromotor Andries Brouwer, samen met wie ik een artikel over uniciteit van codes heb  geschreven. De samenwerking verliep uitsluitend via `\emph{brieven}' (e-mail), maar was zeer nuttig. Bedankt Andries, ook voor het goede commentaar.  

Mijn medepromovendi en kamergenoten in F3.28 waren Bart Litjens en Bart Sevenster. Bart Litjens en ik hebben drie jaar lang dagelijks op het CWI getafeltennist. Bart is loyaal,  heeft me veel geholpen met praktische zaken, en zonder hem zou ik vast nog nooit carnaval hebben gevierd.  We hebben ook wiskundig goed  samengewerkt: zo heb ik met Bart en Lex mijn eerste artikel geschreven.    Ook met Bart Sevenster heb ik regelmatig  over wiskunde gepraat, getafeltennist en -- uiteraard -- koffie gedronken. Barts directe relativerende opmerkingen waren een verrijking voor de sfeer in kamer F3.28.  Bart en Bart, veel dank, het was een prachtige tijd. De \emph{kamerspelen} -- waarbij je bijvoorbeeld elastiekjes moest schieten op een bepaald doel, of zo snel mogelijk een artikel van Wikipedia moest overtypen -- 
zullen mij zeker bijblijven. Ik ben ervan overtuigd dat deze activiteiten aan onze wiskundige creativiteit hebben bijgedragen. 

Furthermore, I thank my other roommates and direct colleagues, to be specific: Guus,  Pjotr, Llu\'is, Viresh,   Jacob, Vaidy, Ewan, Ferenc, Fabian, for interesting discussions (both mathematical and non-mathematical), seminars,  conferences, dinners, drinks at the Polder and numerous other fun activities. I thank all colleagues of the KdvI for the lunch conversations and coffee breaks.  Evelien, Marieke en Tiny van het secretariaat, jullie zorgden voor uitstekende ondersteuning en ik liep altijd graag bij jullie binnen. 

Van het CWI dank ik om te beginnen Guido Sch\"afer, van wie ik veel geleerd heb in de fase voorafgaand aan het promotieonderzoek. Monique Laurent dank ik voor alle gesprekken en de cursus over semidefiniet programmeren. Verder dank ik al mijn CWI-vrienden voor de tafeltennis- en tafelvoetbalpartijen, vooral nadat Bart en Bart vertrokken waren. In het bijzonder wil ik hiervoor graag Sander, Pieter,  Math\'e en Ruben bedanken, ook voor de nuttige aanmerkingen op delen van dit proefschrift. 

Dion Gijswijt dank ik voor de cursus over geheeltallige programmering en de gesprekken in Toronto over het `trifference'-probleem, Jeroen Zuiddam voor de interessante discussies over de Shannoncapaciteit en voor zijn goede commentaar op Hoofdstuk~\ref{shannonchap} van dit proefschrift, en Carla Groenland voor (alweer) het tafeltennissen en de leuke seminars.

Naast het onderzoek heeft het geven van onderwijs mij de afgelopen vier jaar veel voldoening gegeven. Chris Zaal verdient lof en dank voor het in goede banen leiden van de bacheloropleiding, ook toen ik zelf bachelorstudent was. Verder dank ik de vaste medewerkers die ik mocht assisteren bij het onderwijs: Gerard, Guus, Han, Hessel, Jan, Rob, bedankt voor de fijne samenwerking! 

Ik dank de studenten die de afgelopen vier jaar mijn werkcolleges hebben gevolgd.   
Simon, het was me een genoegen je bachelorproject te begeleiden. Boudewijn, Edward, Helena, Mike, Ruben, Suzanne, Thijs en Wouter, bedankt voor jullie enthousiasme en voor het geven van die fantastische kaart, dat heeft een blijvende indruk op mij gemaakt. Edward en Suzanne dank ik in het bijzonder voor onze sportieve activiteiten. Ook alle andere werkcollegestudenten, bedankt! 

Nu richt ik mij tot een aantal andere vrienden die me de afgelopen periode gesteund hebben. Wadim, Jochem, Rob en Hanneke, veel dank voor de avonturen die we tijdens en na onze studie beleefd hebben. Rizky, bedankt voor al het tafeltennissen na werktijd, de Indonesische maaltijden en de reizen om ons aller Ajax aan het werk te zien op vreemde bodem. Ook Jos\'e voor dit laatste bedankt.  Vasily, Wadim and Yana, thank you for our wonderful journeys in Russia and `in Europe'. Carlo, Felix, Maarten en Rik, bedankt voor de leuke spelletjesmiddagen.  Mijn vrienden van zaalvoetbal, van de teams \emph{FC Barcorona}, \emph{De Echte Mannen} en \emph{Der Angstgegner}, dank ik ook. In het bijzonder dank ik Arne en Joost van \emph{Der Angstgegner} voor twee aanmerkingen op het voorblad van een eerdere versie van dit proefschrift. 

Speciale dank gaat uit naar Boudien, mijn voormalige oppas die (in haar eigen woorden) wel oud is, maar er altijd voor me is. 

Tot slot dank ik mijn familie. Mijn moeder, die me altijd helpt, me goed begrijpt en die ik zelfs 's nachts kan bellen.  Mijn vader, die me eveneens begrijpt, me praktische en wijze adviezen geeft, en die -- zelf gepromoveerd en professor in Leiden -- een voorbeeld voor me is bij mijn eigen promotie, maar ook Ajax met me toejuicht in de Arena. Heleen, bij wie ik altijd terecht kan. Mijn zusje Lara, die tijdens mijn PhD aan het Science Park studeerde, met wie ik talloze keren heb geluncht en met wie ik alles kan bespreken. Mijn zusje Noor, die al bijna halverwege het Vossius zit en vroeg of ze wel in het dankwoord kwam: reken maar, Noor!  Ook dank ik al mijn andere familieleden en vrienden die hier niet bij naam genoemd zijn,  maar wier aanwezigheid ik niettemin zeer waardeer. Bedankt dat jullie er zijn!

\vspace{24pt}

 \hfill Sven Polak 

\hfill Amsterdam, juli 2019 

\selectlanguage{english}

\mainmatter
\chapter{Introduction}	\label{introduction}

\chapquote{The art of doing mathematics consists in finding that\\ special case which contains all the germs of generality.}{David Hilbert (1862--1943)}\vspace{-6pt}

\noindent Between 1979 and 1981, the \emph{Voyager} spacecrafts  from NASA were travelling close to Jupiter and Saturn. They transmitted detailed images of both planets and their moons back to Earth. When a spacecraft sends a message over such a considerable distance, information gets lost or damaged easily.  In other words, the  information channel which is used is noisy.  \emph{Coding theory} is the branch of mathematics that deals with reliable transmission of information over noisy channels.\indexadd{coding theory} 

Claude Shannon initiated the study of coding theory in his seminal paper \emph{A Mathematical Theory of Communication} from 1948~\cite{shannonseminal}.  Shortly thereafter, Golay~\cite{golay} and Hamming~\cite{hamming} discovered important codes with good error-correction properties. Moreover, Hamming~\cite{hamming} and Lee~\cite{lee} introduced two distance functions, now central to coding theory and known as the Hamming and Lee distances. One of Golay's codes, the binary extended Golay code, was later
  used in the Voyager missions to facilitate the correction of transmission errors. 
 
This thesis contributes to the field of coding theory, focussing in particular on upper bounds on the cardinality of \emph{error-correcting codes} with certain parameters, on symmetry reductions of semidefinite programs in coding theory using representation theory, on uniqueness of certain codes ---mostly related to the binary Golay code--- and on the \emph{Shannon capacity}, a graph parameter introduced by and named after the founding father of coding theory. 
  
In the following sections we introduce the main topics studied in this thesis and we give historical background and motivation. We conclude this introduction with  a description of the organization of the thesis into chapters, which includes a summary of our contributions per chapter.

 \section{Error-correcting codes}\label{errorint}

Suppose that~$Q$ is a finite set of~$q \geq 2$ elements and fix a positive integer~$n$.  The set~$Q$ is our \emph{alphabet}.\symlistsort{Q}{$Q$}{alphabet}\indexadd{alphabet} A \emph{word} is an element of~$Q^n$ 
 and a  \emph{($q$-ary) code} is a subset~$C $ of $ Q^n$.\indexadd{word}\indexadd{code}\indexadd{code!$q$-ary}
 For two words~$u,v \in Q^n$,  their \emph{Hamming distance} $d_H(u,v)$\symlistsort{dH}{$d_H(u,v)$}{Hamming distance of~$u$ and~$v$} is the number of~$i$ with~$u_i \neq v_i$.\indexadd{Hamming distance}\indexadd{distance!Hamming}\indexadd{distance}
 For a code~$C \subseteq Q^n$, its \emph{minimum distance}~$d_{\text{min}}(C)$\symlistsort{dmin}{$d_{\text{min}}(C)$}{minimum Hamming distance of~$C$} is the minimum of~$d_H(u,v)$ over all distinct~$u,v \in C$. If~$|C| \leq 1$, we set~$d_{\text{min}}(C) = \infty$.\indexadd{minimum distance}\indexadd{Hamming distance!minimum} 
 Define, for any integer~$d$,
 \begin{align}\label{aqndintr}
 A_q(n,d) := \max \{ |C| \,\, | \,\, C \subseteq Q^n, \,\, d_{\text{min}}(C) \geq d \}.
 \end{align}
 Without loss of generality we usually take~$Q:=[q]$, where~$[q]$ denotes the set~$\{0,\ldots,q-1\}$.\symlistsort{Aq(n,d)}{$A_q(n,d)$}{maximum size of a code $C \subseteq [q]^n$ with~$d_{\text{min}}(C) \geq d$} Moreover, if~$q=2$ we often write~$A(n,d)$ instead of~$A_q(n,d)$.\symlistsort{A(n,d)}{$A(n,d)$}{maximum size of a code $C \subseteq \F_2^n$ with~$d_{\text{min}}(C) \geq d$}
 
J.H.\ van Lint characterized the study of the numbers~$A_q(n,d)$ as the \emph{central problem in (combinatorial) coding theory}~\cite{lint}.\indexadd{central problem in coding theory}  
The main motivation for studying~$A_q(n,d)$ is \emph{error-correction}.\indexadd{error-correction} 
Assume that  transmitter~$T$ aims to send a message to  receiver~$R$, but their communication channel is noisy. 
Every time~$T$ sends a word to~$R$ at most~$e$ symbols change, for some positive integer~$e$. 
Suppose that~$T$ only transmits words from a code~$C$ with~$d_{\text{min}}(C) \geq 2e+1$ to~$R$, where~$C$ is known to both~$T$ and~$R$. 
Then~$R$ can recover any transmitted word, since it is the word from~$C$ that is closest to the received word in Hamming distance. 
So in order to maximize the number of distinct words that can be sent over this channel, we should maximize~$|C|$ provided that~$d_{\text{min}}(C) \geq 2e+1$.  So we should find~$A_q(n,2e+1)$.
 
 The numbers~$A_q(n,d)$ are hard to compute in general. For many triples~$(q,n,d)$, only upper and lower bounds are known. Note that~$A_q(n,d)$ can be interpreted as the independent set number of a graph, as follows. (Here the \emph{independent set number}~$\alpha(G)$\symlistsort{alpha(G)}{$\alpha(G)$}{independent set number of graph~$G$}\indexadd{independent set number} of a graph~$G$ is the maximum cardinality of a set of vertices of~$G$, no two of which are adjacent.)\indexadd{independent set}   Consider 
 the graph~$G(q,n,d)$ with vertex set~$Q^n$, the set of all words, and edges between distinct words if their Hamming distance is strictly less than~$d$. Then~$A_q(n,d)$ is the independent set number~$\alpha(G(q,n,d))$ of this graph.\symlistsort{Gqnd}{$G(q,n,d)$}{graph with independent set number~$A_q(n,d)$}  
 
\hspace{-1.1pt}Explicit codes yield lower bounds on~$A_q(n,d)$ and they can be used for error-correction as explained. In this thesis we will study the problem of finding upper bounds on~$A_q(n,d)$, which has received considerable research attention --- see for example~\cite{ table3, 4ary, 5ary, delsarte, plotkinoriginal}.  A classical upper bound on~$A_q(n,d)$ is the \emph{Delsarte bound (in the Hamming scheme)}~\cite{delsarte}. The  Delsarte bound is equal to a special case of the following general  upper bound  on the independent set number~$\alpha(G)$ of a graph~$G=(V,E)$ in the spirit of Lov\'asz~\cite{lovasz}, introduced  by McEliece, Rodemich and Rumsey~\cite{thetaprime2} and~Schrijver~\cite{thetaprime}:\symlistsort{theta'(G)}{$\vartheta'(G)$}{upper bound on~$\alpha(G)$}
\begin{align}
\label{thetaprimeintro}
\vartheta'(G) := \max\left\{ \mbox{$\sum_{u,v \in V} X_{u,v}$} \,  \big| \,  X \in \R^{V \times V}_{\geq 0},  \, \trace(X)=1,  \text{ $X_{u,v}=0$ if~$uv \in E$},\, X\succeq 0 \right\}.
\end{align}
Here~$X\succeq 0$ denotes the condition that~$X$ is \emph{positive semidefinite}, i.e., symmetric with all eigenvalues nonnegative.\symlistsort{X}{$X \succeq 0$}{matrix $X$ is positive semidefinite} This optimization problem is an example of a semidefinite programming (SDP) problem.\indexadd{SDP} 
 
The upper bound in~\eqref{thetaprimeintro} can be computed in time polynomial in the number of vertices of the graph. For some graphs 
(graphs whose edge set is a union of classes of a symmetric association scheme)
it can be reformulated as a linear program, using symmetry reductions. An example of such a graph is the graph~$G(q,n,d)$ defined above. In this case,~$D_q(n,d):=\vartheta'(G(q,n,d))$ is called the \emph{Delsarte bound (in the Hamming scheme)}, and the linear program can be formulated as follows. Let~$K_t(x)$ be the~$t$-th \emph{Krawtchouk polynomial}:\indexadd{Krawtchouk polynomial}\symlistsort{Kt(x)}{$K_t(x)$}{Krawtchouk polynomial}
\begin{align} \label{kraw}
K_t(x) := \sum_{j=0}^t (-1)^{j} \binom{x}{j} \binom{n-x}{t-j} (q-1)^{t-j}, \,\,\,\,\, \text{ for $0 \leq t \leq n$}.
\end{align} 
Then~$A_q(n,d) \leq D_q(n,d)$, where\indexadd{Delsarte linear programming bound}\indexadd{Delsarte linear programming bound!in the Hamming scheme}\symlistsort{Dqnd}{$D_q(n,d)$}{Delsarte bound in the Hamming scheme} 
\begin{align} \label{delsintro}
D_q(n,d) = \max\big\{ \mbox{$\sum_{i=0}^n a_i$} \,\, \big| \,\, & a_0=1,\,\, a_1 =\ldots=a_{d-1}=0, \,\,\, a_i \geq 0 \text{ if $d \leq i \leq n$}, \,\,\,\,  \notag  
\\  & \mbox{$ \sum_{i=0}^{n} K_t(i) a_{i} \geq 0$} \text{ for all $ 0 \leq t \leq n$}\big\}.
\end{align}
The optimization problem~$\eqref{delsintro}$ contains only~$n+1$ variables, which number is linear in~$n$. By contrast, the number of variables in~$\eqref{thetaprimeintro}$, which is the number of edges of~$G(q,n,d)$,   is exponential in~$n$. So~$\eqref{delsintro}$ gives a considerable reduction. It can be used to compute upper bounds on~$A_q(n,d)$ for several triples~$(q,n,d)$, cf.~\cite{ table3, 4ary, 5ary, delsarte, vaessens}.

The Delsarte bound is an SDP bound based on pairs of codewords (edges in the underlying graph), as can be seen from~\eqref{thetaprimeintro}. For binary codes, the Delsarte bound was generalized to an SDP bound based on triples of codewords by Schrijver~\cite{schrijver}, 
and later to a quadruple bound by Gijswijt, Mittelmann and Schrijver~\cite{semidef}. Also, a bound for nonbinary codes based on triples of codewords has been studied by Gijswijt, Schrijver and Tanaka~\cite{tanaka} --- see also Gijswijt's thesis~\cite{gijswijtthesis}.   

In this thesis, we  consider the following SDP bound based on quadruples of codewords. For any~$k$, let~$\mathcal{C}_k$ be the collection of codes of cardinality at most~$k$.\symlistsort{Ck}{$\mathcal{C}_k$}{collection of codes of cardinality at most~$k$} For each~$x \, : \, \mathcal{C}_4 \to \R$ define the~$\mathcal{C}_2 \times \mathcal{C}_2$ matrix~$M(x)$ by\symlistsort{M(x)}{$M(x)$}{variable matrix} 
$$
M(x)_{C,C'} := x(C \cup C'). 
$$
Then~$A_q(n,d) \leq B_q(n,d)$, with\symlistsort{Bq(n,d)}{$B_q(n,d)$}{upper bound on~$A_q(n,d)$}
\begin{align}\label{introsemq}
B_q(n,d) :=  \max \big\{ \mbox{$\sum_{v \in [q]^n} x(\{v\})$}\,\, |\,\,&x:\mathcal{C}_4 \to \R_{\geq 0}, \,\, x(\emptyset )=1,  \,\,\,\,x(S)=0 \text{ if~$d_{\text{min}}(S)<d$},  \notag\\
& M(x)\succeq 0 \big\}. 
\end{align} 
Indeed, let $C$ be a code with $d_{\text{min}}(C) \geq d$ and $|C|=A_q(n,d)$.
Define $x:\CC_4\to\R$ by $x(S)=1$ if $S\subseteq C$ and $x(S)=0$ otherwise, for~$S \in \CC_4$.
Then $x$ satisfies the conditions in~\eqref{introsemq}: the condition~$M(x) \succeq 0$ follows from the fact that for this $x$ one has $M(x)_{S,S'}=x(S)x(S')$
for all $S,S'\in\CC_2$. Moreover, we have $\sum_{v\in[q]^n}x(\{v\})=|C|=A_q(n,d)$, yielding~$A_q(n,d) \leq B_q(n,d)$.

We will apply representation theory to reduce the size of the above optimization problem from exponential in~$n$ to polynomial in~$n$, with entries (i.e., coefficients) being polynomials in~$q$. We calculate the bound for some values of~$q,n,d$, yielding new upper bounds for five instances of~$A_q(n,d)$. In the reduction, we use a general method used in all symmetry reductions throughout this thesis. We outline this method in Section~\ref{symint}.

We also explore other methods of finding upper bounds on~$A_q(n,d)$, based on combinatorial divisibility arguments.
Our most prominent result in this direction gives in certain cases a strengthening of a bound implied by the Plotkin bound (cf.~\cite{plotkin,plotkinoriginalthesis, plotkinoriginal}).  We prove the following.
\begin{theoremnn}[Theorem~\ref{importantth}]
Suppose that~$q,n,d,m$ are positive integers with $q\geq 2$, such that~$d=m(qd-(q-1)(n-1))$, and such that~$n-d$ does not divide~$m(n-1)$. If~$r \in \{1,\ldots,q-1\}$ satisfies~
\begin{align*}
n(n-1-d)(r-1)r <  (q-r+1)(qm(q+r-2)-2r),
\end{align*}
 then~$A_q(n,d) < q^2m -r$. 
\end{theoremnn}
The theorem yields new bounds for the cases~$A_5(8,6)$ and~$A_4(11,8)$.  We also find a number of other new bounds on~$A_q(n,d)$ using divisibility arguments.

Before we explain the general method of symmetry reductions to reduce semidefinite programs as in~$\eqref{introsemq}$ in more detail, we first introduce two other types of error-correcting codes: binary \emph{constant weight codes} and $q$-ary \emph{Lee codes}. For these codes we will formulate  SDP bounds in this thesis as well, and show that they can be computed in time polynomially bounded in~$n$ (for fixed~$q$ in the case of Lee codes). We calculate the bound for some values of~$q,n,d$, yielding explicit new upper bounds for several instances of~$A_q^L(n,d)$.

\paragraph{Constant weight codes.}
Let~$Q:=\F_2=\{0,1\}$ be the binary alphabet, where~$\F_2$ denotes the field of two elements.\symlistsort{F2}{$\F_2$}{field $\{0,1\}$ of two elements} For~$v \in \F_2^n$, the \emph{weight} of~$v$ is~$\text{wt}(v):= d_H(v,\mathbf{0})$, where~$\mathbf{0}:=0\ldots0$ denotes the all-zeros word.\symlistsort{0}{$\bm{0}$}{all-zeros word}\symlistsort{wt(v)}{$\text{wt}(v)$}{weight of word~$v$} So~$\text{wt}(v)$ is the number of~$1$'s in~$v$. Fix a positive integer~$w$.  Instead of considering all (unrestricted) words in~$Q^n$, we now restrict to words with a constant weight~$w$. A (binary) \emph{constant weight code} is a code~$C \subseteq \F_2^n$ in which all words have a fixed weight~$w$.\indexadd{constant weight code}\indexadd{code!constant weight} Define\symlistsort{A(n,d,w)}{$A(n,d,w)$}{maximum size of a constant weight~$w$ code~$C \subseteq \F_2^n$ with $d_{\text{min}}(C) \geq d$}
\begin{align} \label{andwintr}
A(n,d,w):= \max \{ |C| \,\, | \,\, C \subseteq \F_2^n, \,\, d_{\text{min}}(C) \geq d, \,\, \text{wt}(v) = w \,\, \forall \, v \in C \}.
\end{align}
Similar to~$A_q(n,d)$, the numbers~$A(n,d,w)$ are subject to a wide range of research (see for example~\cite{agrell, table3, delsarte,  schrijver}), and they are also hard to compute in general.  Note that~$A(n,d,w)$ is the maximum cardinality of a collection~$\mathcal{A}$ of subsets of size~$w$ of a set~$X$ of $n$ elements such that any two subsets in the collection have at most~$w-d/2 $ elements in common.  

Constant weight codes are used in GSM mobile telephone networks to generate `frequency hopping lists' for cell towers~\cite{largen}. Suppose that we  assign to each cell tower a list of~$w$ out of~$n$ possible frequencies. Each tower  hops between frequencies from its list according to a given rule. Less overlap between the lists from any two cell towers (i.e., larger Hamming distance between the corresponding constant weight~$w$ codewords), contributes to less interference, i.e., fewer collisions on one frequency. Larger constant weight codes with a specified minimum distance then permit larger spaces between two towers with the same list in the GSM network, which also leads to less interference. 

Delsarte~\cite{delsarte} formulated a linear programming upper bound on~$A(n,d,w)$, now known as the \emph{Delsarte bound in the Johnson scheme}, which can be interpreted as an SDP bound based on pairs of codewords (as in~\eqref{thetaprimeintro}), which again is a linear programming bound. The Delsarte bound  was generalized to an SDP  bound based on triples of codewords  by Schrijver~\cite{schrijver}. Schrijver's bound was strengthened by Kim and Toan with extra  linear inequalities~\cite{KT}. 
 In this thesis we consider two SDP upper bounds on~$A(n,d,w)$ based on quadruples of codewords, which admit definitions similar to but different from~\eqref{introsemq}.  Using symmetry reductions, we can compute both bounds in time polynomially bounded in~$n$, resulting in several new upper bounds on~$A(n,d,w)$.

The new upper bounds imply the exact values~$A(22,8,10)=616$ and $A(22,8,11)=672$. Lower bounds on these two instances are obtained from the \emph{shortened binary Golay code}, which is the unique  (unrestricted) binary optimal code achieving~$A_2(22,7)=2^{11}$  containing~$\mathbf{0}$,  up to a permutation of the coordinate positions~\cite{brouwer2}.  From these new values in combination with previously known values of~$A(n,d,w)$, it can be concluded that the shortened binary Golay code is a union of constant weight~$w$ codes of sizes~$A(n,d,w)$.

\paragraph{The Lee distance and Lee codes.}
The Hamming distance measures how many symbols are different, but not \emph{to what extent} they are different. To give more importance to symbols which are far apart in two words and less importance to symbols which differ only slightly, C.Y. Lee introduced a different distance function~\cite{lee}.

Let~$Q:=\Z_q$, the cyclic group of order~$q$, which has a natural distance function $\delta: \Z_q \times \Z_q \to \Z$ given by\symlistsort{delta}{$\delta(x,y)$}{distance function $\Z_q \times \Z_q \to \Z$}
\begin{align}\label{deltaintro}
\delta(x,y):=\min\{ |x-y|,\, q-|x-y|\},
\end{align}
 where we consider~$x$ and~$y$ as integers in~$\{0,\ldots,q-1\}$. So~$\delta(x,y)$ is the length of a shortest path from~$x$ to~$y$ in the circuit graph~$C_q$ on~$q$ points. Fix~$n \in \N$. The \emph{Lee distance} of two words~$u,v \in \Z_q^n$ is\symlistsort{dL}{$d_L(u,v)$}{Lee distance of~$u$ and~$v$}\indexadd{Lee distance}\indexadd{distance!Lee}
\begin{align}\label{leeintro}
d_L(u,v):= \mbox{$\sum_{i=1}^n \delta(u_i,v_i)$}. 
\end{align}
 The Lee distance is used  in so called `phase modulated systems' --- see \cite[Chapter 8]{berlekamp}. These are examples of noisy information channels in which a symbol is more likely to change upon transmission into a symbol which is close to it, than into a symbol which is far from it.

The \emph{minimum Lee distance} $d_{\text{min}}^L(C)$  of a code~$C\subseteq \Z_q^n$ is the minimum of~$d_L(u,v)$ taken over distinct~$u,v \in C$.\symlistsort{dLmin}{$d_{\text{min}}^L(C)$}{minimum Lee distance of~$C$}\indexadd{minimum Lee distance}\indexadd{Lee distance!minimum}\indexadd{code!Lee}   If~$|C| \leq 1$, we set~$d_{\text{min}}^L(C) = \infty$.  For any natural number~$d$, define analogously to~\eqref{aqndintr} and~\eqref{andwintr},\symlistsort{AqL(n,d)}{$A_q^L(n,d)$}{maximum size of a code $C \subseteq \Z_q^n$ with~$d^L_{\text{min}}(C) \geq d$}
\begin{align} \label{aleeqndintroduction}
A^L_q(n,d):= \max \{ |C| \, \, | \,\, C \subseteq \Z_q^n, \,\, d_{\text{min}}^L(C) \geq d \}.
\end{align}
 
As is the case with~$A_q(n,d)$ and~$A(n,d,w)$, it generally  is an interesting and nontrivial problem to determine the numbers~$A_q^L(n,d)$. The classical Delsarte linear programming bound (in the Lee association scheme) based on pairs of codewords provides upper bounds on~$A_q^L(n,d)$~\cite{astola1, astola2, delsarte}. 
  In~$\cite{invariant}$, the possibility of applying semidefinite programming to Lee codes is mentioned and it is stated that to the best knowledge of the authors, such  bounds for Lee codes using triples have not yet been studied. In this thesis we study an SDP bound based on triples of codewords and show that it can be reduced to size bounded by a polynomial in~$n$ (for fixed~$q$). The method yields several new upper bounds on~$A_q^L(n,d)$ for~$q \in \{5,6,7\}$. We only consider~$q \geq 5$, since for~$q=4$, it holds that~$A^L_4(n,d)=A_2(2n,d)$ --- this follows by applying the Gray map~\cite{gray}. Moreover,  if~$q=2$ or~$q=3$, the Lee distance coincides with the Hamming distance.

\section{Symmetry reductions with representation theory}\label{symint}

For the sake of exposition, we first sketch how to reduce the  optimization problem~\eqref{introsemq}.  After that, we describe the general method that is used in all symmetry reductions throughout this thesis. 

To explain the reduction of~\eqref{introsemq}, let $H$ be the wreath product $S_q^n\rtimes S_n$, where~$S_q$ and~$S_n$\symlistsort{Sn}{$S_n$}{symmetric group on~$n$ elements} denote the symmetric groups on~$q$ and~$n$ elements, respectively. 
For each $k$, the group $H$ acts naturally on~$Q^n$, hence on the collection $\CC_k$ of all codes~$C \subseteq Q^n$ with~$|C|\leq k$, maintaining minimum distances and cardinalities
of elements of $\CC_k$ (being codes).
Then we can assume that $x$ is invariant under the $H$-action on $\CC_4$.
That is, we can assume that $x(C)=x(D)$ whenever $C,D\in\CC_2$ and $D=g \cdot C$ for some $g\in H$.
Indeed, the conditions in \eqref{introsemq} are maintained under replacing $x$ by $g\cdot x$.
(Note that $M(g\cdot x)$ is obtained from $M(x)$ by simultaneously permuting rows and columns.)
Moreover, the objective function does not change by this action.
Hence the optimum $x$ can be replaced by the average of all $g\cdot x$ (over all $g\in  H$),
by the convexity of the set of positive semidefinite matrices.
This makes the optimum solution $H$-invariant.

Let $\Omega_4$ be the set of $H$-orbits on $\CC_4$.\symlistsort{Omegak}{$\Omega_k$}{set of $H$-orbits on $\mathcal{C}_k$}
Note that $\Omega_4$ is bounded by a polynomial in $n$ (independently of $q$).
As~\eqref{introsemq} has an $H$-invariant optimum solution, we can replace, for each $\omega\in\Omega_4$ and $C\in\omega$, each variable $x(C)$ by a variable $z(\omega)$.
In this way we obtain $M(z)$.  Note that~$M(z)$ depends only on~$|\Omega_4|$ variables~$z(\omega)$, which is polynomially bounded in~$n$. Furthermore, $M(z)$ is invariant under the simultaneous action of $H$ on its rows and columns. However, its size is still exponential in~$n$. To reduce the matrix~$M(z)$ further, we use a general method for symmetry reductions, which we will give in detail in Chapter~\ref{orbitgroupmon} and which we sketch now. 

\paragraph{The main symmetry reduction.}

 Let~$G$ be a finite group acting on a finite set~$Z$. Let~$(\C^{Z \times Z})^G$ denote the set of (complex) $Z \times Z$ matrices  invariant under the simultaneous action of~$G$ on its rows and columns. This means that~$(\C^{Z \times Z})^G$ can be identified with the \emph{centralizer algebra} of the action of~$G$ on~$\C^{Z}$, i.e., the  collection of~$G$-equivariant endomorphisms~$\C^{Z} \to \C^{Z}$.   Then it is a standard fact from representation theory that there exists a \emph{block diagonalization} of~$(\C^{Z \times Z})^G$. For our purposes, this is a linear bijection\indexadd{block diagonalization}
 $$
\Phi \, : \,  (\C^{Z \times Z})^G \to \bigoplus_{i=1}^k \C^{m_i \times m_i} \,\,\,\,\, \text{ }
 $$
(for some $k, m_1,\ldots, m_k \in \N$), such that $M$ is positive semidefinite if and only if~$\Phi(M)$ is positive semidefinite, for each~$M \in (\C^{Z \times Z})^G$.  The map~$\Phi$ can be given by $ M   \mapsto U^* M U$  for a matrix~$U$ depending on~$G$ but independent of~$M$. Here~$U^*$ denotes the conjugate transpose of the matrix~$U$.

 Now, let~$Z$ be a finite set and let~$n \in \N$. Then the wreath product~$H:=G^n \rtimes S_n$\symlistsort{GnSn}{$G^n \rtimes S_n$}{wreath product of group~$G$ and~$S_n$} acts on~$Z^n$, by  permuting the~$n$ coordinates and by acting with~$G$ on the elements of~$Z$ in each coordinate separately. Then we can apply the previous considerations to the finite group~$H$ acting on the finite set~$Z^n$. So there exists a block diagonalization~$ M \mapsto U^* M U$ of~$M \in (\C^{Z^n \times Z^n})^H$, for a matrix~$U$ depending on~$H$ but independent of~$M$. In this case the order of~$U^* M U$ is polynomial in~$n$.
 
The matrices~$M$ we consider in this thesis are real matrices. Moreover, it turns out that in our applications the matrices~$U$ can be taken to be real matrices; so~$U\T=U^*$. The issue which is crucial to us is that~$M$ is positive semidefinite if and only if each of the smaller matrix blocks in the image~$U\T M U$ is positive semidefinite. This will allow us to reduce our semidefinite programs to size polynomially bounded in~$n$.

In this thesis we show how to compute the matrix entries of the image~$U\T MU$ of the block diagonalization of~$(\C^{Z^n \times Z^n})^H$ explicitly, from a known block diagonalization of $(\C^{Z \times Z})^G$  (with corresponding matrix~$U$). This is the \emph{main symmetry reduction}.\indexadd{main symmetry reduction} An algorithm appeared in a manuscript of Gijswijt~\cite{gijswijt}, but our method is an adaptation of the ---more direct--- method of~\cite{onsartikel}.  In particular, we consider the following cases.
\begin{speciaalenumerate}
\item For~$q$-ary codes with the Hamming distance we take~$G:=S_q $ and~$Z:=[q]$ or~$Z:=[q]^2$. We use the set~$Z=[q]$ in the reductions for the Delsarte bound, to  illustrate the method (see Section~\ref{delsil}). The set~$Z=[q]^2$ is used in the reductions for our quadruple bound~$B_q(n,d)$ in Chapter~\ref{onsartchap}. 

Note that the matrices~$M(z)$ described above are~$\mathcal{C}_2 \times \mathcal{C}_2$ matrices, i.e., the rows and columns of the matrices are indexed by \emph{unordered} pairs of codewords. However, the method gives a reduction of~$Z^n \times Z^n = ([q]^2)^n \times ([q]^2)^n \cong ([q]^n)^2 \times ([q]^n)^2$ matrices, i.e., matrices whose rows and columns are indexed by \emph{ordered} pairs of codewords. This is not a problem; we could have defined~$M(z)$ to be indexed by ordered pairs of codewords in the definition of~$B_q(n,d)$. But in the present case we can further reduce the program by a factor~$2$, since there is an~$S_2$-action on the ordered pairs. See Chapter~\ref{onsartchap}.
\item For binary constant weight codes we take~$G:=\{1\}$, the trivial group and~$Z:=\F_2$ or~$Z:=\F_2^2$. So the group acts trivially on the alphabet, as the weight of each word must remain fixed. See Chapter~\ref{cw4chap}.
\item For~$q$-ary codes with the Lee distance we take~$Z:=\Z_q$ and~$G:=D_q$, the dihedral group of order $2q$, or~$G:=S_2$.\symlistsort{Dq}{$D_q$}{dihedral group of order~$2q$} We take the latter group in case the symbol~$0 \in \Z_q$ is fixed, which is the case in a matrix occurring in the semidefinite programming bound based on triples of codewords. Here the action of~$S_2$ on~$\Z_q$ is the \emph{reflection action}, i.e., we consider~$0,\ldots,q-1 \in \Z_q$ as vertices of a regular~$q$-gon, and the non-identity element of~$S_2$ switches the vertices~$i$ and~$q-i$ (for~$i=1,\ldots,\floor{(q-1)/{2}}$).\indexadd{reflection action} See Chapter~\ref{leechap}.
\end{speciaalenumerate}

We also give a generalization of the method to groups of the form $ (G_1^{j_1} \rtimes S_{j_1}) \times \ldots \times (G_s^{j_s} \rtimes S_{j_s}) $ acting on sets of the form $ Z_1^{j_1} \times \ldots\times Z_s^{j_s}$, where~$Z_1,\ldots,Z_s$ are finite sets and~$G_1,\ldots,G_s$ are finite groups. In coding theory, this can be used in semidefinite programs to reduce matrices depending on a fixed code~$D$. We use this in the reductions for constant weight codes and in certain reductions for Lee codes --- see Sections~\ref{D2} and~\ref{D2lee}.  Other applications of the generalization, which we do not discuss here, include reductions of matrices in semidefinite programs for codes with mixed alphabets~\cite{mixed}, ternary (or $q$-ary) constant weight codes~\cite{regts}, or doubly constant weight codes~\cite{agrell, doublycw}.

\section{Uniqueness of codes}
Until now we have discussed bounds on~$A_q(n,d)$, $A(n,d,w)$ and~$A_q^L(n,d)$. If we know the value of a certain case of one of these parameters, it is natural to ask if we can \emph{classify} the corresponding optimal codes up to equivalence (see below).   An analysis of a complete catalogue of certain objects often provides knowledge regarding the common structure of the objects, and may help in proving theorems.  See the book by Kaski and  \"Osterg\r{a}rd~\cite{bookclas}. 

Two $q$-ary codes~$C,D \subseteq [q]^n$ are \emph{equivalent} if~$D$ can be obtained from~$C$ by first permuting the~$n$ coordinates and by subsequently permuting the alphabet~$[q]$ in each coordinate separately, i.e., if there is a~$g \in S_q^n \rtimes S_n$ such that~$g \cdot C = D$.\indexadd{equivalent} Similarly, two binary constant weight codes~$C,D$ are \emph{equivalent} if~$D$ can be obtained from~$C$ by permuting the~$n$ coordinates. Finally, two codes~$C,D \subseteq \Z_q^n$ are \emph{Lee equivalent}  if there  is a $g \in D_q^n \rtimes S_n$ such that $g \cdot C = D$.\indexadd{Lee equivalent}\indexadd{equivalent!Lee}

As mentioned before, a famous  code is the binary Golay code, discovered by Marcel J.E. Golay~\cite{golay}. It is the unique binary code achieving~$A(23,7)=2^{12}$, up to the equivalence relation described above. The uniqueness was proved by Snover~\cite{snover} and later with a simpler proof by Delsarte and Goethals~\cite{delsartegolay}. The Golay code is \emph{perfect}, which means that the balls with radius~$3$ (in Hamming distance) around the codewords form a partition of $\F_2^{23}$.  It can be extended with  a~$24$th \emph{parity} bit, which equals the weight (mod 2) of the codeword, to obtain the \emph{extended binary Golay code}. This is the unique (up to equivalence) optimal code  achieving~$A(24,8)=2^{12}$  \cite{delsarte,golay}.  If  it contains~$\bm{0}$, the zero word, then the extended binary Golay code is \emph{linear}, i.e., it is a subspace of~$\F_2^{24}$.  A basis is given by the rows of the `generator' matrix in Figure~\ref{golaygenmatintro}.

\begin{figure}[ht]
\centering
\scriptsize
\begin{tabular}{rrrrrrrrrrrrrrrrrrrrrrrr}
  \cellcolor{donkergroen!25}1 &  \cellcolor{donkerrood!65}0 &  \cellcolor{donkerrood!65}0 &  \cellcolor{donkerrood!65}0 &  \cellcolor{donkerrood!65}0 &  \cellcolor{donkerrood!65}0 &  \cellcolor{donkerrood!65}0 &  \cellcolor{donkerrood!65}0 &  \cellcolor{donkerrood!65}0 &  \cellcolor{donkerrood!65}0 &  \cellcolor{donkerrood!65}0 &  \cellcolor{donkerrood!65}0 &
   \cellcolor{donkergroen!25}1 &  \cellcolor{donkergroen!25}1 &  \cellcolor{donkerrood!65}0 &  \cellcolor{donkergroen!25}1 &  \cellcolor{donkergroen!25}1 &  \cellcolor{donkergroen!25}1 &  \cellcolor{donkerrood!65}0 &  \cellcolor{donkerrood!65}0 &  \cellcolor{donkerrood!65}0 &  \cellcolor{donkergroen!25}1 &  \cellcolor{donkerrood!65}0 &  \cellcolor{donkergroen!25}1 \\
  \cellcolor{donkerrood!65}0 &  \cellcolor{donkergroen!25}1 &  \cellcolor{donkerrood!65}0 &  \cellcolor{donkerrood!65}0 &  \cellcolor{donkerrood!65}0 &  \cellcolor{donkerrood!65}0 &  \cellcolor{donkerrood!65}0 &  \cellcolor{donkerrood!65}0 &  \cellcolor{donkerrood!65}0 &  \cellcolor{donkerrood!65}0 &  \cellcolor{donkerrood!65}0 &  \cellcolor{donkerrood!65}0 &  \cellcolor{donkergroen!25}1 &  \cellcolor{donkerrood!65}0 &  \cellcolor{donkergroen!25}1 &  \cellcolor{donkergroen!25}1 &  \cellcolor{donkergroen!25}1 &  \cellcolor{donkerrood!65}0 &  \cellcolor{donkerrood!65}0 &  \cellcolor{donkerrood!65}0 &  \cellcolor{donkergroen!25}1 &  \cellcolor{donkerrood!65}0 &  \cellcolor{donkergroen!25}1 &  \cellcolor{donkergroen!25}1 \\
  \cellcolor{donkerrood!65}0 &  \cellcolor{donkerrood!65}0 &  \cellcolor{donkergroen!25}1 &  \cellcolor{donkerrood!65}0 &  \cellcolor{donkerrood!65}0 &  \cellcolor{donkerrood!65}0 &  \cellcolor{donkerrood!65}0 &  \cellcolor{donkerrood!65}0 &  \cellcolor{donkerrood!65}0 &  \cellcolor{donkerrood!65}0 &  \cellcolor{donkerrood!65}0 &  \cellcolor{donkerrood!65}0 &  \cellcolor{donkerrood!65}0 &  \cellcolor{donkergroen!25}1 &  \cellcolor{donkergroen!25}1 &  \cellcolor{donkergroen!25}1 &  \cellcolor{donkerrood!65}0 &  \cellcolor{donkerrood!65}0 &  \cellcolor{donkerrood!65}0 &  \cellcolor{donkergroen!25}1 &  \cellcolor{donkerrood!65}0 &  \cellcolor{donkergroen!25}1 &  \cellcolor{donkergroen!25}1 &  \cellcolor{donkergroen!25}1 \\
  \cellcolor{donkerrood!65}0 &  \cellcolor{donkerrood!65}0 &  \cellcolor{donkerrood!65}0 &  \cellcolor{donkergroen!25}1 &  \cellcolor{donkerrood!65}0 &  \cellcolor{donkerrood!65}0 &  \cellcolor{donkerrood!65}0 &  \cellcolor{donkerrood!65}0 &  \cellcolor{donkerrood!65}0 &  \cellcolor{donkerrood!65}0 &  \cellcolor{donkerrood!65}0 &  \cellcolor{donkerrood!65}0 &  \cellcolor{donkergroen!25}1 &  \cellcolor{donkergroen!25}1 &  \cellcolor{donkergroen!25}1 &  \cellcolor{donkerrood!65}0 &  \cellcolor{donkerrood!65}0 &  \cellcolor{donkerrood!65}0 &  \cellcolor{donkergroen!25}1 &  \cellcolor{donkerrood!65}0 &  \cellcolor{donkergroen!25}1 &  \cellcolor{donkergroen!25}1 &  \cellcolor{donkerrood!65}0 &  \cellcolor{donkergroen!25}1 \\
  \cellcolor{donkerrood!65}0 &  \cellcolor{donkerrood!65}0 &  \cellcolor{donkerrood!65}0 &  \cellcolor{donkerrood!65}0 &  \cellcolor{donkergroen!25}1 &  \cellcolor{donkerrood!65}0 &  \cellcolor{donkerrood!65}0 &  \cellcolor{donkerrood!65}0 &  \cellcolor{donkerrood!65}0 &  \cellcolor{donkerrood!65}0 &  \cellcolor{donkerrood!65}0 &  \cellcolor{donkerrood!65}0 &  \cellcolor{donkergroen!25}1 &  \cellcolor{donkergroen!25}1 &  \cellcolor{donkerrood!65}0 &  \cellcolor{donkerrood!65}0 &  \cellcolor{donkerrood!65}0 &  \cellcolor{donkergroen!25}1 &  \cellcolor{donkerrood!65}0 &  \cellcolor{donkergroen!25}1 &  \cellcolor{donkergroen!25}1 &  \cellcolor{donkerrood!65}0 &  \cellcolor{donkergroen!25}1 &  \cellcolor{donkergroen!25}1 \\
  \cellcolor{donkerrood!65}0 &  \cellcolor{donkerrood!65}0 &  \cellcolor{donkerrood!65}0 &  \cellcolor{donkerrood!65}0 &  \cellcolor{donkerrood!65}0 &  \cellcolor{donkergroen!25}1 &  \cellcolor{donkerrood!65}0 &  \cellcolor{donkerrood!65}0 &  \cellcolor{donkerrood!65}0 &  \cellcolor{donkerrood!65}0 &  \cellcolor{donkerrood!65}0 &  \cellcolor{donkerrood!65}0 &  \cellcolor{donkergroen!25}1 &  \cellcolor{donkerrood!65}0 &  \cellcolor{donkerrood!65}0 &  \cellcolor{donkerrood!65}0 &  \cellcolor{donkergroen!25}1 &  \cellcolor{donkerrood!65}0 &  \cellcolor{donkergroen!25}1 &  \cellcolor{donkergroen!25}1 &  \cellcolor{donkerrood!65}0 &  \cellcolor{donkergroen!25}1 &  \cellcolor{donkergroen!25}1 &  \cellcolor{donkergroen!25}1 \\
  \cellcolor{donkerrood!65}0 &  \cellcolor{donkerrood!65}0 &  \cellcolor{donkerrood!65}0 &  \cellcolor{donkerrood!65}0 &  \cellcolor{donkerrood!65}0 &  \cellcolor{donkerrood!65}0 &  \cellcolor{donkergroen!25}1 &  \cellcolor{donkerrood!65}0 &  \cellcolor{donkerrood!65}0 &  \cellcolor{donkerrood!65}0 &  \cellcolor{donkerrood!65}0 &  \cellcolor{donkerrood!65}0 &  \cellcolor{donkerrood!65}0 &  \cellcolor{donkerrood!65}0 &  \cellcolor{donkerrood!65}0 &  \cellcolor{donkergroen!25}1 &  \cellcolor{donkerrood!65}0 &  \cellcolor{donkergroen!25}1 &  \cellcolor{donkergroen!25}1 &  \cellcolor{donkerrood!65}0 &  \cellcolor{donkergroen!25}1 &  \cellcolor{donkergroen!25}1 &  \cellcolor{donkergroen!25}1 &  \cellcolor{donkergroen!25}1 \\
  \cellcolor{donkerrood!65}0 &  \cellcolor{donkerrood!65}0 &  \cellcolor{donkerrood!65}0 &  \cellcolor{donkerrood!65}0 &  \cellcolor{donkerrood!65}0 &  \cellcolor{donkerrood!65}0 &  \cellcolor{donkerrood!65}0 &  \cellcolor{donkergroen!25}1 &  \cellcolor{donkerrood!65}0 &  \cellcolor{donkerrood!65}0 &  \cellcolor{donkerrood!65}0 &  \cellcolor{donkerrood!65}0 &  \cellcolor{donkerrood!65}0 &  \cellcolor{donkerrood!65}0 &  \cellcolor{donkergroen!25}1 &  \cellcolor{donkerrood!65}0 &  \cellcolor{donkergroen!25}1 &  \cellcolor{donkergroen!25}1 &  \cellcolor{donkerrood!65}0 &  \cellcolor{donkergroen!25}1 &  \cellcolor{donkergroen!25}1 &  \cellcolor{donkergroen!25}1 &  \cellcolor{donkerrood!65}0 &  \cellcolor{donkergroen!25}1 \\
  \cellcolor{donkerrood!65}0 &  \cellcolor{donkerrood!65}0 &  \cellcolor{donkerrood!65}0 &  \cellcolor{donkerrood!65}0 &  \cellcolor{donkerrood!65}0 &  \cellcolor{donkerrood!65}0 &  \cellcolor{donkerrood!65}0 &  \cellcolor{donkerrood!65}0 &  \cellcolor{donkergroen!25}1 &  \cellcolor{donkerrood!65}0 &  \cellcolor{donkerrood!65}0 &  \cellcolor{donkerrood!65}0 &  \cellcolor{donkerrood!65}0 &  \cellcolor{donkergroen!25}1 &  \cellcolor{donkerrood!65}0 &  \cellcolor{donkergroen!25}1 &  \cellcolor{donkergroen!25}1 &  \cellcolor{donkerrood!65}0 &  \cellcolor{donkergroen!25}1 &  \cellcolor{donkergroen!25}1 &  \cellcolor{donkergroen!25}1 &  \cellcolor{donkerrood!65}0 &  \cellcolor{donkerrood!65}0 &  \cellcolor{donkergroen!25}1 \\
  \cellcolor{donkerrood!65}0 &  \cellcolor{donkerrood!65}0 &  \cellcolor{donkerrood!65}0 &  \cellcolor{donkerrood!65}0 &  \cellcolor{donkerrood!65}0 &  \cellcolor{donkerrood!65}0 &  \cellcolor{donkerrood!65}0 &  \cellcolor{donkerrood!65}0 &  \cellcolor{donkerrood!65}0 &  \cellcolor{donkergroen!25}1 &  \cellcolor{donkerrood!65}0 &  \cellcolor{donkerrood!65}0 &  \cellcolor{donkergroen!25}1 &  \cellcolor{donkerrood!65}0 &  \cellcolor{donkergroen!25}1 &  \cellcolor{donkergroen!25}1 &  \cellcolor{donkerrood!65}0 &  \cellcolor{donkergroen!25}1 &  \cellcolor{donkergroen!25}1 &  \cellcolor{donkergroen!25}1 &  \cellcolor{donkerrood!65}0 &  \cellcolor{donkerrood!65}0 &  \cellcolor{donkerrood!65}0 &  \cellcolor{donkergroen!25}1 \\
  \cellcolor{donkerrood!65}0 &  \cellcolor{donkerrood!65}0 &  \cellcolor{donkerrood!65}0 &  \cellcolor{donkerrood!65}0 &  \cellcolor{donkerrood!65}0 &  \cellcolor{donkerrood!65}0 &  \cellcolor{donkerrood!65}0 &  \cellcolor{donkerrood!65}0 &  \cellcolor{donkerrood!65}0 &  \cellcolor{donkerrood!65}0 &  \cellcolor{donkergroen!25}1 &  \cellcolor{donkerrood!65}0 &  \cellcolor{donkerrood!65}0 &  \cellcolor{donkergroen!25}1 &  \cellcolor{donkergroen!25}1 &  \cellcolor{donkerrood!65}0 &  \cellcolor{donkergroen!25}1 &  \cellcolor{donkergroen!25}1 &  \cellcolor{donkergroen!25}1 &  \cellcolor{donkerrood!65}0 &  \cellcolor{donkerrood!65}0 &  \cellcolor{donkerrood!65}0 &  \cellcolor{donkergroen!25}1 &  \cellcolor{donkergroen!25}1 \\
  \cellcolor{donkerrood!65}0 &  \cellcolor{donkerrood!65}0 &  \cellcolor{donkerrood!65}0 &  \cellcolor{donkerrood!65}0 &  \cellcolor{donkerrood!65}0 &  \cellcolor{donkerrood!65}0 &  \cellcolor{donkerrood!65}0 &  \cellcolor{donkerrood!65}0 &  \cellcolor{donkerrood!65}0 &  \cellcolor{donkerrood!65}0 &  \cellcolor{donkerrood!65}0 &  \cellcolor{donkergroen!25}1 &  \cellcolor{donkergroen!25}1 &  \cellcolor{donkergroen!25}1 &  \cellcolor{donkergroen!25}1 &  \cellcolor{donkergroen!25}1 &  \cellcolor{donkergroen!25}1 &  \cellcolor{donkergroen!25}1 &  \cellcolor{donkergroen!25}1 &  \cellcolor{donkergroen!25}1 &  \cellcolor{donkergroen!25}1 &  \cellcolor{donkergroen!25}1 &  \cellcolor{donkergroen!25}1 &  \cellcolor{donkerrood!65}0 \\
\end{tabular}
\caption{\small A generator matrix of the extended binary Golay code  (i.e., the $2^{12}$ codewords are sums mod~$2$ of the rows of this matrix). The extended binary Golay code was used in the \emph{Voyager} missions to Jupiter and Saturn~\cite{wicker}.} \label{golaygenmatintro}
\end{figure}

The binary Golay code, assuming it contains~$\bm{0}$, contains~$1288$ words of weight~$11$. So $A(23,8,11) \geq 1288$. With semidefinite programming based on triples of codewords, A.\ Schrijver proved an upper bound matching the lower bound~\cite{schrijver}. So $A(23,8,11) = 1288$.
 Of these~$1288$ words in the binary Golay code, exactly~$616$ words have a~$1$ at any fixed position, and exactly~$672$ words have a~$0$ at any position. So~$A(22,8,10) \geq 616$ and~$A(22,8,11) \geq 672 $. As mentioned in Section~\ref{errorint}, in this thesis we show SDP upper bounds matching these lower bounds. 
 
 Using information obtained from the SDP output, we  prove that the optimal codes achieving~$A(23,8,11)=1288$, $A(22,8,10) = 616$ and $A(22,8,11) = 672 $ are the \emph{unique} codes achieving these values, up to a permutation of the coordinate positions.  
 
 Uniqueness of codes related to the binary Golay code has been studied by Brouwer (cf.~\cite{brouwer2}). He explains that there exist unique codes up to equivalence achieving~$A(24-i,8)=2^{12-i}$ for~$i=0,\ldots,3$. The unique code for the case~$i=0$ is the extended binary Golay code, and the unique codes for the cases~$i=1,\ldots,3$ are obtained by taking subcodes and throwing away coordinates of this code (also called \emph{shortening}). The question whether the statement also holds for the case~$i=4$ was still open. In this thesis, we  prove that there exist several nonequivalent codes achieving~$A(20,8)=2^{10}$. So the quadruply shortened binary extended Golay code is not the unique optimal code attaining~$A(20,8)=2^{10}$. We show that there exist precisely 15  equivalence classes of such codes with all distances divisible by~$4$. We also show that there exist such codes with not all distances divisible by~$4$.

\paragraph{Kirkman systems and Lee codes.}
A classic combinatorial problem is \emph{Kirkman's schoolgirl problem}, posed by Thomas P. Kirkman in 1850 in the \emph{Lady's and Gentleman's Diary}.\indexadd{Kirkman's schoolgirl problem}

\begin{quotation}
\textsl{``Fifteen young ladies in a school walk out three abreast for seven days
in succession: it is required to arrange them daily, so that no two walk
twice abreast.''}
\end{quotation}
There exist~$7$ nonequivalent solutions to Kirkman's problem (up to a permutation of the girls and a permutation of the days of the week)~\cite{7sol}. It is known that there is a 1-1-correspondence between equivalence classes of solutions to Kirkman's problem and  equivalence classes of optimal codes achieving~$A_5(7,6)=15$~\cite{zinoviev} (see also Chapter~\ref{divchap}). In Chapter~\ref{divchap}, we exploit the classification of solutions to Kirkman's problem and the described 1-1-correspondence to prove that~$A_5(8,6) \leq 65$, using a divisibility argument (the best previously known upper bound was~$A_5(8,6)\leq 75$). 

We also explore a connection with Lee codes in Chapter~\ref{leechap}. Previously, the upper bound $A^L_5(7,9) \leq 15$ was known~\cite{astola2,quistorff}, but it was not known whether a code attaining this value exists. Any code achieving this upper bound is equivalent to an arrangement of~$15$ girls $7$ days in succession into triples, where the triples are placed at the corners of a pentagon, such that the total Lee distance between any two girls over the $7$ days is~$9$ --- see Section~\ref{579}. This problem connects to, but is different from, the above-mentioned Kirkman's schoolgirl problem.
 
Consider the following~$\Z_7 \times \Z_7$ matrix with entries in~$\Z_5$: \symlistsort{M tilde}{$\widetilde{M}$}{$\Z_7 \times \Z_7$ matrix with entries in~$\Z_5$}
 \begin{align}\label{fanocodeintro} \mbox{\small $
\widetilde{M}:=
 \begin{pmatrix}
0  & 1&1&3&1&3&3 \\
 1&1&3&1&3&3 &0 \\
1&3&1&3&3 &0  & 1\\
3&1&3&3&0  & 1&1 \\
1&3&3&0  & 1&1&3 \\
3&3&0  & 1&1&3&1 \\
3&0  & 1&1&3&1&3 \\
 \end{pmatrix}, \,\,\,\, \text{ so } \widetilde{M}_{i,j} = \begin{cases}
0, &  \text{if~$i+j \equiv 0 \pmod{7}$}, \\
1, & \text{if~$i+j$ is a nonzero  square mod~$7$}, \\
3, & \text{if~$i+j$ is not a square mod~$7$}. 
\end{cases}$}
 \end{align}
We show that  $A_5^L(7,9)=15$, and an optimal code proving this is $\widetilde{M} \cup  -\widetilde{M} \cup \{\bm{0} \}$, where the matrix~$\widetilde{M}$ is interpreted as a $5$-ary code of size~$7$ (i.e., the rows are the codewords). We also obtain, using the SDP output and computer experiments, that this is the \emph{unique} code achieving this value, up to Lee equivalence.  As the mentioned code has minimum Hamming distance~$5$, there is no arrangement of $15$ girls~$7$ days in succession into triples which is a solution simultaneously to both problems given above.

\section{The Shannon capacity\label{shanintro}}

Now we introduce the Shannon capacity, the final topic of study in this thesis.
 For any graph~$G=(V,E)$ and~$n \in \N$, let~$G^{\boxtimes n}$ denote the graph with vertex set~$V^n$ and edges between two distinct vertices $(u_1,\ldots,u_n)$ and~$(v_1,\ldots,v_n)$ if and only if for all~$i \in \{1,\ldots,n\}$ one has either $u_i=v_i$ or~$u_iv_i \in E$. The graph~$G^{\boxtimes n}$ is known as the~$n$-th \emph{strong product power} of~$G$.\symlistsort{Gboxn}{$G^{\boxtimes n}$}{$n$-th strong product power of graph~$G$} The \emph{Shannon capacity} of~$G$ is defined as\symlistsort{Theta(G)}{$\Theta(G)$}{Shannon capacity of graph $G$}\indexadd{Shannon capacity}
\begin{align}
    \Theta(G):= \sup_{n \in \N} \sqrt[n]{\alpha(G^{\boxtimes n})},
\end{align}
where~$\alpha(G^{\boxtimes n})$ denotes the maximum cardinality of an independent set in~$G^{\boxtimes n}$.
Since $\alpha(G^{\boxtimes(n_1+n_2)}) \geq \alpha(G^{\boxtimes n_1})\alpha(G^{\boxtimes n_2})$ for any two positive integers~$n_1$ and~$n_2$, Fekete's lemma (cf.~$\cite{fekete}$, see also \cite[Corollary 2.2a]{schrijverpolyhedra}) implies that $\Theta(G) = \lim_{n \to \infty} \sqrt[n]{\alpha(G^{\boxtimes n})}$. 

The Shannon capacity was introduced by Shannon~$\cite{shannon}$ and is an important and widely studied  parameter in information theory (see e.g.,~\cite{ alon,bohman, haemers,lovasz, zuiddam}). It is the effective size of an alphabet in an information channel represented by the graph~$G$. The input is a set of letters~$V(G)=Q:=\{0,\ldots,q-1\}$ and two letters are `confusable' when transmitted over the channel if and only if there is an edge between them in~$G$.\indexadd{confusable} Then~$\alpha(G)$ is the maximum size of a set of pairwise non-confusable single letters. Moreover,~$\alpha(G^{\boxtimes n})$ is the maximum size of a set of pairwise non-confusable~$n$-letter words. (Here two $n$-letter words $u,v$ are confusable if and only if~$u_i$ and~$v_i$ are confusable for all~$ 1 \leq i \leq n$.) Taking~$n$-th roots and letting~$n$ go to infinity, we find the effective size of the alphabet in the information channel:~$\Theta(G)$. 

A classical upper bound on~$\Theta$ is Lov\'asz's $\vartheta$-function~\cite{lovasz}.  For any graph~$G$, the number~$\vartheta(G)$ is equal to the maximum~$\eqref{thetaprimeintro}$ over all~$X \in \R^{V \times V}$ instead of over nonnegative~$X$.\symlistsort{theta(G)}{$\vartheta(G)$}{Lov\'asz's theta number of graph~$G$}\indexadd{Lov\'asz's theta number} The Shannon capacity of~$C_5$, the cycle on~$5$ vertices, was already discussed by Shannon in 1956~\cite{shannon}, who showed~$\sqrt{5} \leq \Theta(C_5) \leq 5/2$. The  easy lower bound of~$\sqrt{5}$ is obtained from the independent set $\{(0,0),(1,2),(2,4),(3,1),(4,3)\}$ in~$C_5^{\boxtimes 2}$. More than twenty years later, Lov\'asz~$\cite{lovasz}$ determined that~$\Theta(C_5)=\sqrt{5}$ by proving an upper bound matching this lower bound: he showed~$\Theta(C_5) \leq \vartheta(C_5)= \sqrt{5}$.  More generally, for odd~$q$,
\begin{align} \label{varthetainntro}
    \Theta(C_q) \leq \vartheta(C_q) = \frac{q\cos(\pi/q)}{1+\cos(\pi/q)}.
\end{align}
Here~$C_q$ denotes the cycle on~$q$ vertices.
For~$q$ even it is not hard to see that~$\Theta(C_q)=q/2$. 

The Shannon capacity of~$C_7$ is still unknown and its determination is a notorious open problem in extremal combinatorics~\cite{bohman, godsil}. Many lower bounds have been given by explicit independent sets in some fixed power of~$C_7$ \cite{baumert, matos, veszer}, while the best known upper bound is~$\Theta(C_7)\leq \vartheta(C_7) < 3.3177$. In this thesis we give an independent set of size~$367$ in~$C_7^{\boxtimes 5}$, which yields~$\Theta(C_7)\geq 367^{1/5} > 3.2578$. The best previously known lower bound on~$\Theta(C_7)$ is~$\Theta(C_7) \geq 350^{1/5} > 3.2271$, found by Mathew and \"Osterg{\aa}rd~$\cite{matos}$. They proved that~$\alpha(C_7^{\boxtimes 5}) \geq 350$ using stochastic search methods that utilize the symmetry of the problem. We obtained our new lower bound on~$\alpha(C_7^{\boxtimes 5})$ and on~$\Theta(C_7)$ by considering so-called circular graphs --- see Chapter~\ref{shannonchap} for details.
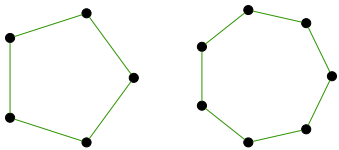
\begin{figure}[H]
\centering
\scalebox{0.4}{
    \begin{tikzpicture} 
    \begin{scope} [vertex style/.style={draw,
                                       circle,
                                       minimum size=3mm,
                                       inner sep=0pt,
                                       outer sep=0pt, fill}] 
      \path \foreach \i in {0,...,4}{%
       (72*\i:2.25) coordinate[vertex style] (a\i)}
       ; 
    \end{scope}

     \begin{scope} [edge style/.style={draw=black}]
       \foreach \i  in {0,...,4}{%
       \pgfmathtruncatemacro{\nexta}{mod(\i+1,5)}  
       \draw[edge style,donkergroen, thick] (a\i)--(a\nexta);
       }  
     \end{scope}

  \end{tikzpicture} \,\,\,\,\,\,\,\, \quad\,\,\,\,\,\,\,\,
    \begin{tikzpicture} 
    \begin{scope} [vertex style/.style={draw,
                                       circle,
                                       minimum size=3mm,
                                       inner sep=0pt,
                                       outer sep=0pt, fill}] 
      \path \foreach \i in {0,...,6}{%
       (51.429*\i:2.25) coordinate[vertex style] (a\i)}
       ; 
    \end{scope}

     \begin{scope} [edge style/.style={draw=black}]
       \foreach \i  in {0,...,6}{%
       \pgfmathtruncatemacro{\nexta}{mod(\i+1,7)}  
       \draw[edge style, thick,donkergroen] (a\i)--(a\nexta);
       }  
     \end{scope}

  \end{tikzpicture} }
        \caption{\small The graphs $C_5$ (left) and $C_7$ (right).} 
\end{figure}
 Next we will define circular graphs, explain the connection of independent sets in strong product powers of these graphs to Lee codes, and we explain which other results about circular graphs are proved in this thesis.
\paragraph{The Shannon capacity of circular graphs.}
 
For positive integers~$q,d$ with~$q \geq 2d$, the \emph{circular graph} $C_{d,q}$ is the graph with vertex set~$\Z_q$ (the cyclic group of order~$q$) in which two distinct vertices are adjacent if and only if their distance (mod~$q$) is strictly less than~$d$.\symlistsort{Cdq}{$C_{d,q}$}{circular graph}\indexadd{circular graph}\indexadd{graph!circular} So~$C_{2,q} = C_q$. A closed formula for~$\vartheta(C_{d,q})$, Lov\'asz's upper bound on $\Theta(C_{d,q})$, is given in~$\cite{bachoc}$, see~\eqref{varthetaclosedintro} below.

To study the Shannon capacity of circular graphs, we are interested in the independent set numbers~$\alpha(C_{d,q}^{\boxtimes n})$. Independent sets in~$C_{d,q}^{\boxtimes n}$ are strongly related to Lee codes, as we will now explain. 

If~$(X,d)$ is a metric space, then~$(X^n, d_p)$ is also a metric space. Here~$d_p$ denotes the distance function given by\symlistsort{dp}{$d_p$}{distance function}
\begin{align}\label{dpdef}
d_p((x_1,\ldots,x_n),(x_1',\ldots,x_n')):= \norm{(d(x_1,x_1'),\ldots, d(x_n,x_n'))}_p,
\end{align}
where~$p$ either is a real number at least~$1$, or~$p=\infty$. Now fix~$X:=\Z_q$ and $d:=\delta$ from~\eqref{deltaintro}.
If we take~$p=1$ in~\eqref{dpdef}, we obtain the Lee distance as in~\eqref{leeintro}. However, if we take~$p=\infty$ in~\eqref{dpdef} we obtain the so-called \emph{Lee\textsubscript{$\infty$} distance}. So for~$u,v \in \Z_q^n$, their Lee\textsubscript{$\infty$} distance is defined as\symlistsort{dLi}{$d_{L_{\infty}}(u,v)$}{Lee\textsubscript{$\infty$} distance of~$u$ and~$v$}\indexadd{Lee\textsubscript{$\infty$} distance}\indexadd{distance!Lee\textsubscript{$\infty$}}
\begin{align}\label{leeinftyintro}
d_{L_{\infty}}(u,v)= \max_{i \in \{1,\ldots,n\}} \delta(u_i,v_i). 
\end{align}
The \emph{minimum Lee\textsubscript{$\infty$} distance}~$d_{\text{min}}^{L_{\infty}}(C)$\symlistsort{dLmin(C)}{$d_{\text{min}}^{L_{\infty}}(C)$}{minimum Lee\textsubscript{$\infty$} distance of~$C$}\indexadd{Lee\textsubscript{$\infty$} distance!minimum}\indexadd{minimum Lee\textsubscript{$\infty$} distance} of a set~$C \subseteq \Z_q^n$ is the minimum Lee\textsubscript{$\infty$} distance  between any pair of distinct elements of~$C$. If~$|C|\leq1$, we set~$d_{\text{min}}^{L_{\infty}}(C)=\infty$. Then~$d_{\text{min}}^{L_{\infty}}(C) \geq d$ if and only if~$C$ is independent in~$C_{d,q}^{\boxtimes n}$. Define\symlistsort{AqLi(n,d)}{$A_q^{L_{\infty}}(n,d)$}{maximum size of a code~$C \subseteq \Z_q^n$ with $d_{\text{min}}^{L_{\infty}}(C) \geq d$, i.e., $\alpha(C_{d,q}^{\boxtimes n})$}
\begin{align} \label{strongalphaprodcdqintro}
A_q^{L_{\infty}}(n,d):=\alpha(C_{d,q}^{\boxtimes n})= \max \{ |C| \, \, | \,\, C \subseteq \Z_q^n, \,\, d_{\text{min}}^{L_{\infty}}(C) \geq d \}.
\end{align}
Compare with the definition of~$A_q^L(n,d)$ in~\eqref{aleeqndintroduction}. The SDP upper bound on~$A_q^L(n,d)$ based on triples of codewords is easily adapted to a new SDP upper bound on~$A_q^{L_{\infty}}(n,d)=\alpha(C_{d,q}^{\boxtimes n})$, sharpening~$\vartheta$ (cf.~\cite{lovasz}) and~$\vartheta'$ (cf.~\cite{thetaprime2, thetaprime}), using the same symmetry reductions as for Lee codes --- see Section~\ref{circbounds}.

The circular graphs have the property that~$\alpha(C_{d,q}^{\boxtimes n})$ (for fixed~$n$), $\vartheta(C_{d,q})$ and~$\Theta(C_{d,q})$ only depend on the fraction~$q/d$. Moreover, the three mentioned quantities are nondecreasing functions in~$q/d$ (see Chapter~\ref{shannonchap} for details). 
Let us consider the graph of the function
\begin{align} \label{varthetaclosedintro}
q/d \mapsto \vartheta(C_{d,q}) = \frac{q}{d}\sum_{i=0}^{d-1} \prod_{j=1}^{d-1} \frac{\cos\left(\frac{2i\pi}{d}\right)-\cos\left(\floor{\frac{qj}{d}}\frac{2\pi}{q}\right)}{1-\cos\left(\floor{\frac{qj}{d}}\frac{2\pi}{q}\right)},
\end{align}
where the closed form formula for~$\vartheta(C_{d,q})$ is proved by Bachoc,  P\^{e}cher and Thi\'ery~\cite{bachoc}. See Figure~\ref{thetafigintro}.

\begin{figure}[ht]
\centering
\begin{tikzpicture}[scale=1.37]
    \pgfmathsetlengthmacro\MajorTickLength{
      \pgfkeysvalueof{/pgfplots/major tick length} * .65
    }
    \begin{axis}[enlargelimits=false,axis on top,xlabel ={\footnotesize\color{red}$q/d$}, ylabel = {\footnotesize\color{red}$\vartheta(C_{d,q})$}, 
                 xtick={2,2.5,3,3.5,4,4.5,5},ytick={2,2.5,3,3.5,4,4.5,5},
  major tick length=\MajorTickLength,
        x label style={
        at={(axis description cs:0.5,-0.05)},
        anchor=north,
      },
      y label style={
        at={(axis description cs:-0.0775,.5)}, 
        anchor=south,
      }, 
                ]
                                
       \addplot graphics
       [xmin=2,xmax=5,ymin=2,ymax=5,
      includegraphics={trim=5cmm 10.5cm 4.5cm 10.105cm,clip}]{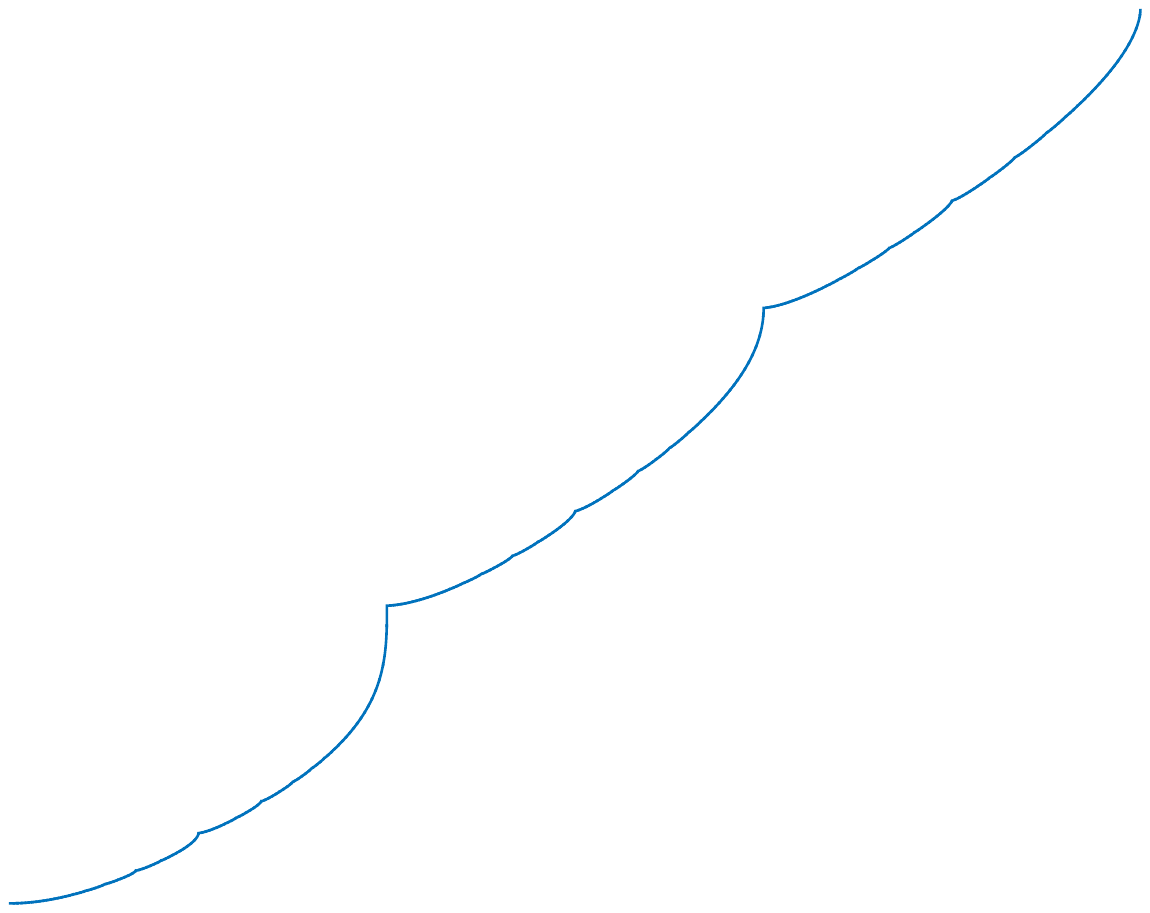};
    \end{axis}
\end{tikzpicture}
\caption{\small A graph of the function~$q/d \mapsto  \vartheta(C_{d,q})$.\label{thetafigintro}}
\end{figure}

 Lov\'asz's sandwich theorem~\cite{lovasz} (see also~\cite{sandwich}) implies that
\begin{align}\label{sandwichintro}
\Theta(C_{d,q}) \leq \vartheta(C_{d,q}) \leq q/d \quad \quad \text{for all } q,d \in \N \text{ with } q \geq 2d.
\end{align}
Here~$q/d$ equals the fractional clique cover number of~$C_{d,q}$, which is an upper bound on~$\vartheta(C_{d,q})$ cf.~\cite{lovasz}.  Currently, the only known values of~$\Theta(C_{d,q})$ for~$q\geq 2d$ are for~$q/d$ integer and~$q/d=5/2$, which corresponds with~$C_5$.  

For understanding the Shannon capacity of~$C_{d,q}$ it may be useful to determine whether $\vartheta(C_{d,q})$ and $\Theta(C_{d,q})$ are continuous functions of~$q/d$. In this thesis, we prove that this is indeed the case at~$q/d \geq 3$ integer. Fix an integer~$r \geq 3$.  Right-continuity of~$q/d \mapsto \vartheta(C_{d,q})$ and~$q/d \mapsto \Theta(C_{d,q})$ at~$q/d=r$ follows immediately from~\eqref{sandwichintro} in combination with  the fact that~$\Theta(C_{d,q})=r$ if~$q/d=r$ and the fact that~$q/d \mapsto \Theta(C_{d,q})$ is a nondecreasing function in~$q/d$.  To prove left-continuity, we prove the following  result (using an explicit construction).
\begin{theoremnn}[Theorem~\ref{shmaxth}]
For each~$r,n\in \N$ with~$r \geq 3$,
\begin{align} \label{intromax}
\max_{\frac{q}{d} < r } \alpha(C_{d,q}^{\boxtimes n}) = \frac{1+r^n(r-2)}{r-1}.
\end{align}
\end{theoremnn}
 This theorem can be shown to imply that $\lim_{\frac{q}{d} \uparrow r} \Theta(C_{d,q})  = r$. As~$\Theta(C_{d,q})=r$ if~$q/d=r$, it follows that~$q/d \mapsto \Theta(C_{d,q})$ is left-continuous at~$q/d=r$  (and hence~$q/d \mapsto \vartheta(C_{d,q})$, also by~\eqref{sandwichintro}).  Continuity of~$q/d \mapsto \Theta(C_{d,q})$ at \emph{non}-integer points~$q/d$ seems hard to prove and even continuity of~$q/d \mapsto \vartheta(C_{d,q})$ at these points is not yet known, although the latter function \emph{looks} continuous (cf.\ Figure~\ref{thetafigintro}).
 
We also prove that the independent set achieving~$\alpha(C_{5,14}^{\boxtimes 3})=14$, one of the independent sets used in our proof of Theorem~\ref{shmaxth}, is unique up to Lee equivalence. Finally, we give  the mentioned independent set of size~$367$ in~$C_7^{\boxtimes 5}$, demonstrating the new lower bound on~$\Theta(C_7)$. It is obtained by adapting an independent set of size~$382$ in~$C_{108,382}^{\boxtimes 5}$ which we found by computer, inspired by the construction used to prove~$``\geq"$ in~\eqref{intromax}.

\section{An overarching theme:\hspace{-.35pt} independent sets in graph products\label{indprod}}
We describe a framework which encompasses the main objects studied in this thesis. First we describe two graph products.
Let~$G=(V,E)$ and~$H=(V',E')$ be graphs. 

The \emph{Cartesian product} $G\> \square \> H$ is the graph with vertex set~$V \times V'$ in which two distinct vertices~$(u,u' )$ and~$(v,v')$ are adjacent if either~$uv \in E$ and~$u'=v'$, or~$u=v$ and~$u'v' \in E'$.\symlistsort{GsquareH}{$G \> \square \> H$}{Cartesian product of graphs~$G$ and~$H$}\indexadd{Cartesian product}  For~$n \in \N$,   write\symlistsort{Gsquaren}{$G^{\square n}$}{$n$-th Cartesian product power of graph~$G$}
$$
G^{\square n}:= \underbrace{G \> \square \> G \>\square \> \ldots \> \square \> G}_{\text{$n$ times}}. \,\,\,\,
$$
So two distinct vertices~$(u_1,\ldots,u_n)$ and~$(v_1,\ldots,v_n)$ of~$G^{\square n}$ are adjacent if and only if there is an~$ i \in \{1,\ldots,n\}$ such that~$u_i v_i \in E$, and~$u_j=v_j$ for all~$j \neq i$. 

The \emph{strong product} $G \boxtimes H$ is the  graph with vertex set~$V \times V'$ in which two distinct vertices~$(u,u' )$ and~$(v,v')$ are adjacent if ($u=v$ or~$uv \in E$) and ($u'=v'$ or~$u'v' \in E'$).\symlistsort{GboxH}{$G \boxtimes H$}{strong product of graphs~$G$ and~$H$}\indexadd{strong product} For~$n \in \N$,   write
$$
G^{\boxtimes n}:= \underbrace{G  \boxtimes  G \boxtimes \ldots \boxtimes G}_{\text{$n$ times}}.\,\,\,\,  
$$
Note that this definition coincides with the definition of~$G^{\boxtimes n}$ given in Section~\ref{shanintro}, since two distinct vertices~$(u_1,\ldots,u_n)$ and~$(v_1,\ldots,v_n)$ of~$G^{\boxtimes n}$ are adjacent if and only if for each~$i \in \{1,\ldots,n\}$ one has either~$u_i  = v_i$ or~$u_i v_i \in E$. 

For any connected graph~$G=(V,E)$ and~$u,v \in V$, let~$d_G(u,v)$ denote the smallest length (in edges) of a path between~$u$ and~$v$ in~$G$.\symlistsort{dG}{$d_G(u,v)$}{length of a shortest path (in edges) between~$u$ and~$v$ in graph~$G$} So~$d_G :  V \times V \to \R_{\geq 0}$ is a metric, and~$(V,d_G)$ is a metric space. Then, cf.~\eqref{dpdef},
\begin{align}
d_{G^{\square n}} = (d_G)_1 \,\,\,\, \text{ and } \,\,\,  d_{G^{\boxtimes n}} = (d_G)_{\infty}.
\end{align}
Define the number\symlistsort{alpha(G)d}{$\alpha_d(G)$}{$\max\{ \lvert U \rvert \, \mid  \,  U \subseteq V(G),\,  d_G(u,v) \geq d \text{ for all distinct } u,v\in U  \}$}
\begin{align}
\alpha_d(G) := \max \{| U| \,\, | \,\, U \subseteq V, \,  d_G(u,v) \geq d \text{ for all distinct } u,v\in U  \}. 
\end{align} 
So $\alpha_2(G)=\alpha(G)$. Let~$K_q$ denote the complete graph on~$q$ vertices, and let~$C_q$ denote the circuit on~$q$ vertices.\symlistsort{Kq}{$K_q$}{complete graph on~$q$ vertices}\symlistsort{Cq}{$C_q$}{circuit graph on~$q$ vertices} Then
\begin{align*}
A_q(n,d)&= \alpha_d(K_q^{\square n}), \phantom{veelletters}
\\A_q^L(n,d)& = \alpha_d(C_q^{\square n}),
\\ A_q^{L_{\infty}}(n,d) = \alpha(C_{d,q}^{\boxtimes n}) &= \alpha_d(C_q^{\boxtimes n}).
\end{align*}
So the main objects studied in this thesis are all of the form~$\alpha_d(G^n)$, where~$G \in \{C_q, K_q\}$, and where~$G^n$ denotes either~$G^{\boxtimes n}$ or~$G^{\square  n}$.  Only the parameter~$A(n,d,w)$ does not fit directly into this framework. We have~$A(n,d,w)=\alpha_d(H)$, where~$H$ is the subgraph of~$K_2^{\square n}$ induced by the vertices~$(u_1,\ldots,u_n)$ with~$u_i=1$ for exactly~$w$ indices~$i \in \{1,\ldots,n\}$  (where the vertices of~$K_2$ are labeled with~$0$ and~$1$).

\section{Publications}
This thesis is based on six published papers.

\begin{itemize}
\item[\cite{uniqueness}] A.E.\ Brouwer, S.C.\ Polak, Uniqueness of codes using semidefinite programming, \textsl{Designs, Codes and Cryptography}  {\bf 87} (2019), 1881--1895.
\item[\cite{onsartikel}] B.M.\ Litjens, S.C.\ Polak, A.\ Schrijver, Semidefinite bounds for nonbinary codes based on quadruples, \textsl{Designs, Codes and Cryptography} {\bf 84} (2017), 87--100.
\item[\cite{divisibility}] S.C.\ Polak, New nonbinary code bounds based on divisibility arguments,  \textsl{Designs, Codes and Cryptography} {\bf 86} (2018), 861--874.
\item[\cite{cw4}] S.C.\ Polak, Semidefinite programming bounds for constant weight codes, \textsl{IEEE Transactions on Information Theory} {\bf 65} (2019), 28--38.
\item[\cite{leeartikel}] S.C.\ Polak, Semidefinite programming bounds for Lee codes,  \textsl{Discrete Mathematics}   {\bf 342} (2019), 2579--2589.
\item[\cite{shannonc7}] S.C.\ Polak, A.\ Schrijver, New lower bound on the Shannon capacity of~$C_7$, \textsl{Information Processing Letters}  {\bf 143} (2019), 37--40.
\end{itemize}
To each of the papers mentioned above, the contribution of each of the authors is equivalent.

\section{Organization of the thesis}

Here we give a description of the organization of the thesis into chapters, which includes a summary of our contributions per chapter. 

\paragraph{Chapter~\ref{prem}. Preliminaries.}
We give   basic definitions and notation used throughout the thesis. We will also introduce semidefinite programming, one of our main tools. The last section is devoted to the definitions and results from representation theory used in the symmetry reductions  in this thesis. 

\paragraph{Chapter~\ref{orbitgroupmon}. The main symmetry reduction.}
We give the main symmetry reduction used in the semidefinite programs throughout the thesis.
Suppose that~$G$ is a finite group acting on a finite set~$Z$ and let~$n\in \N$. We consider the natural action of~$H:=G^n \rtimes S_n$ on~$Z^n$. 
We show how to obtain a reduction~$\Phi(M)$ of matrices~$M \in  (\C^{Z^n \times Z^n})^{H}$ to size polynomially bounded in~$n$.  
The reduction can be derived from a manuscript of Gijswijt~\cite{gijswijt}, but we give a more direct method. 
Chapter~\ref{orbitgroupmon} is a generalization of the method of~$\cite{onsartikel}$, which is joint work with Bart Litjens and Lex Schrijver.

\paragraph{Chapter~\ref{onsartchap}. Semidefinite programming bounds for unrestricted codes.}
For nonnegative integers~$q,n,d$, let~$A_q(n,d)$ denote the maximum size of a code~$C \subseteq [q]^n$ with minimum (Hamming) distance at least~$d$. In this chapter we give an SDP  upper bound on~$A_q(n,d)$ based on quadruples of codewords and we give a reduction of the size of the matrices involved with the method of Chapter~\ref{orbitgroupmon}, so that the upper bound can be computed in time bounded by a polynomial in~$n$.   We provide new upper bounds on~$A_q(n,d)$ for five triples~$(q,n,d)$, that were obtained with this SDP. Chapter~\ref{onsartchap} is based on joint work with Bart Litjens and Lex Schrijver~\cite{onsartikel}.

\paragraph{Chapter~\ref{divchap}. New nonbinary code bounds based on divisibility arguments.}
We give a divisibility argument that results in new upper bounds on~$A_q(n,d)$: we find that~$A_5(8,6) \leq 65$, $A_4(11,8)\leq 60$ and~$A_3(16,11) \leq 29$. The divisibility argument builds upon the work of Plotkin~\cite{plotkinoriginal}. Our main result in this chapter is the following. Suppose that~$q,n,d,m $ are positive integers with $q\geq 2$, such that~$d=m(qd-(q-1)(n-1))$, and such that~$n-d$ does not divide~$m(n-1)$. If~$r \in \{1,\ldots,q-1\}$ satisfies~
\begin{align*}
n(n-1-d)(r-1)r <  (q-r+1)(qm(q+r-2)-2r),
\end{align*}
 then~$A_q(n,d) < q^2m -r$. 

In the last part of Chapter~\ref{divchap},  we prove that for~$\mu,q \in \N$,  there is a 1-1-correspondence between symmetric~$(\mu,q)$-nets (which are certain designs) and codes~$C \subseteq [q]^{\mu q}$ of size~$\mu q^2$ with minimum distance at least~$\mu q - \mu$. We derive the new upper bounds~$A_4(9,6) \leq 120$ and~$A_4(10,6) \leq 480$ from these `symmetric net' codes. 

Chapter~\ref{divchap} is self-contained and does not use SDP or representation theory, but entirely relies on basic combinatorics.  Chapter~\ref{divchap} is based on~\cite{divisibility}.

\paragraph{Chapter~\ref{cw4chap}. Semidefinite programming bounds for constant weight codes.}
For nonnegative integers~$n,d,w$, let~$A(n,d,w)$ be the maximum size of a code~$C \subseteq \{0,1\}^n$ with constant weight~$w$ and minimum distance at least~$d$. We consider two SDP upper bounds based on quadruples of code words that yield several new upper bounds on~$A(n,d,w)$ and we show that they can be computed in polynomial time by applying symmetry reductions with the method of Chapter~\ref{orbitgroupmon}. The new upper bounds imply the exact values~$A(22,8,10)=616$ and~$A(22,8,11)=672$. Lower bounds on~$A(22,8,10)$ and~$A(22,8,11)$ are obtained from the~$(n,d)=(22,7)$ shortened binary Golay code of size~$2048$. It can be concluded that the shortened binary Golay code is a union of constant weight~$w$ codes of sizes~$A(22,8,w)$. Chapter~\ref{cw4chap} is based on~\cite{cw4}.

\paragraph{Chapter~\ref{cu17chap}. Uniqueness of codes using semidefinite programming.}
\hspace{-3.629pt}With semidefinite programming, some upper bounds on~$A(n,d,w)$ have recently been obtained that are equal to the best known lower bounds: it has been established that $A(23,8,11)=1288$ (see~\cite{schrijver}), and that $A(22,8,11)=672$ and $A(22,8,10)=616$ (see Chapter~$\ref{cw4chap}$). We show using the output of the corresponding semidefinite programs that the corresponding codes of maximum size are unique up to coordinate permutations for these~$n,d,w$. 

For unrestricted (non-constant weight) binary codes, the bound~$A(n,d)=A(20,8)\leq256$ was obtained by Gijswijt, Schrijver and Mittelmann in~\cite{semidef}, implying that the quadruply shortened extended binary Golay code of size~$256$ is optimal.  Up to equivalence the optimal binary codes attaining $A(24-i,8)=2^{12-i}$ for $i=0,1,2,3$ are unique, namely they are the~$i$ times shortened extended binary Golay codes \cite{brouwer2}. We show that there exist several nonequivalent optimal codes achieving~$A(20,8)=256$. We classify such codes under the additional condition that all distances are divisible by~$4$, and find~$15$ such codes. Chapter~\ref{cu17chap} is based on joint work with Andries Brouwer~\cite{uniqueness}.

\paragraph{Chapter~\ref{leechap}. Semidefinite programming bounds for Lee codes.}
For~$q,n,d \in \N$, let~$A_q^L(n,d)$ denote the maximum cardinality of a code~$C \subseteq \Z_q^n$ with minimum Lee distance at least~$d$, where~$\Z_q$ denotes the cyclic group of order~$q$. We consider an SDP upper bound based on triples of codewords, which bound can be computed efficiently using symmetry reductions with the method of Chapter~\ref{orbitgroupmon}, resulting in several new upper bounds on~$A_q^L(n,d)$.  We also give constructions for obtaining lower bounds and derive uniqueness of two instances: $A_5^L(7,9)=15$ and $A_6^L(4,6)=18$, using the SDP output. The code achieving~$A_5^L(7,9)=15$ is obtained by arranging 15 girls  7 days in succession into triples in a special way, which is different from but connects to Kirkman's schoolgirl problem.  Chapter~\ref{leechap} is based on~\cite{leeartikel}, except for the constructions and uniqueness of~$A_5^L(7,9)=15$ and~$A_6^L(4,6)=18$, which are new. 

\paragraph{Chapter~\ref{shannonchap}. The Shannon capacity of circular graphs.}

We consider the Shannon capacity~$\Theta(C_{d,q})$ of  circular graphs~$C_{d,q}$. The circular graph $C_{d,q}$ is the graph with vertex set~$\Z_q$ (the cyclic group of order~$q$) in which two distinct vertices are adjacent if and only if their distance (mod~$q$) is strictly less than~$d$. The value of~$\Theta(C_{d,q})$ can be seen to only depend on the quotient~$q/d$.  We show that the function~$q/d \mapsto \Theta(C_{d,q})$ is continuous at \emph{integer} points~$q/d \geq 3$. We also prove that the independent set achieving~$\alpha(C_{5,14}^{\boxtimes 3})=14$, one of the independent sets used in our proof, is unique up to Lee equivalence. 
Furthermore, we adapt the SDP bound of Chapter~\ref{leechap} to an upper bound on~$\alpha(C_{d,q}^{\boxtimes n})$ (see~\eqref{strongalphaprodcdqintro}) based on triples of codewords. The SDP bound does not seem to improve significantly over the bound obtained from Lov\'asz's theta-function, except for very small~$n$. Finally, we give a new lower bound of~$367^{1/5}$ on the Shannon capacity of the~$7$-cycle. Chapter~\ref{shannonchap} is based on joint work with Lex Schrijver, part of which can be found in~$\cite{shannonc7}$.

\chapter{Preliminaries}\label{prem}
\vspace{-16pt}
\chapquote{Mathematics is the supreme judge;\\from its decisions there is no appeal.}{Tobias Dantzig (1884--1956)} \vspace{-9pt}

\noindent In this chapter we give the preliminaries on coding theory, semidefinite programming and representation theory which are used in the other chapters. We start by introducing notation and giving some basic definitions.

\section{Notation and basic definitions}

\paragraph{Sets and graphs.} We denote the sets of nonnegative integers and nonnegative real numbers by~$\Z_{\geq 0}$ and~$\R_{\geq 0}$, respectively.\symlistsort{Zgeq 0}{$\mathbb{Z}_{\geq 0}$}{set of nonnegative integers}\symlistsort{Rgeq 0}{$\mathbb{R}_{\geq 0}$}{set of nonnegative real numbers} Moreover, we use the notation~$\N := \Z_{\geq 0} \setminus \{0\}$.\symlistsort{N atural}{$\mathbb{N}$}{set of positive integers} For~$m \in \Z_{\geq 0}$ we define $[m]:=\{1,\ldots,m\}$.\symlistsort{m}{$[m]$}{set~$\{1,\ldots,m\}$} So $[0] = \emptyset$.  If~$q \in \N$ and~$Q:=[q]$ denotes an alphabet in coding theory, we follow the convention that $[q]:=\{0,1\ldots,q-1\}$  in this context.\symlistsort{q}{$[q]$}{alphabet~$\{0,\ldots,q-1\}$} We write~$\Z_q$ for the group of integers modulo~$q$.\symlistsort{Zq}{$\Z_q$}{group of integers modulo~$q$} Furthermore, we write~$|Z|$ for the cardinality of a set~$Z$.

 A \emph{graph}\indexadd{graph} is a pair~$(V,E)$, where~$V$ is a finite set and~$E$ is a collection of unordered pairs from~$V$. The elements of~$V$ are called \emph{vertices} and the elements of~$E$ are called \emph{edges}. Sometimes we will denote an edge~$\{u,v\}$ by~$uv$. If~$G$ is a graph, we sometimes write~$V(G)$ and~$E(G)$ for the vertex and edge sets of~$G$, respectively.\symlistsort{V(G)}{$V(G)$}{vertex set of graph~$G$}\symlistsort{E(G)}{$E(G)$}{edge set of graph~$G$}  For~$q \in \N$, the complete graph on~$q$ vertices is denoted by~$K_q$, and the circuit on~$q$ vertices is denoted by~$C_q$. Finally, for~$q,d \in \N$ with~$q \geq 2d$, we define the \emph{circular graph} $C_{d,q}$ to be the graph with vertex set~$V(C_q)=\Z_q$ in which two distinct vertices are adjacent if and only if their distance (mod~$q$) is strictly less than~$d$.  We refer to Section~\ref{indprod} for definitions of two often used graph products. Finally, we recall that~$U \subseteq V$ is called an \emph{independent set} in a graph~$G=(V,E)$ if for all~$e \in E$, we have~$e \nsubseteq U$. The maximum cardinality of an independent set in~$G$ is the \emph{independent set number of~$G$}, denoted by~$\alpha(G)$.\symlistsort{alpha(G)}{$\alpha(G)$}{independent set number of graph~$G$}

\paragraph{Free vector space and group algebra.} If~$Z$ is a finite set and~$K$ is a field, we denote by~$KZ$ the \emph{free vector space} of~$Z$ over~$K$.\indexadd{free vector space}\symlistsort{KZ}{$KZ$}{free vector space of~$Z$ over~$K$} It consists of all~$K$-linear combinations of elements in~$Z$, i.e.,  of all formal sums~$\sum_{z \in Z} \lambda_z z$ with~$\lambda_z \in K$ for all~$z \in Z$. If~$\sum_{z \in Z} \lambda_z z$ and~$\sum_{z \in Z} \gamma_z z $ are elements of~$KZ$ and~$\alpha \in K$, then
\begin{align*}
\mbox{$\left(\sum_{z \in Z} \lambda_z z\right) + \left(\sum_{z \in Z} \gamma_z z\right) = \sum_{z \in Z}(\lambda_z+ \gamma_z) z $} \,\,\, \text{ and } \,\,\, \mbox{$\alpha \left(\sum_{z \in Z} \lambda_z z\right)  = \sum_{z \in Z} (\alpha \lambda_z) z$.} 
\end{align*} 
If~$V$ is a vector space, we denote by~$V^*$ its dual vector space.\symlistsort{Vstar}{$V^*$}{dual vector space of~$V$}  If~$KZ$ is the free vector space of~$Z$ over~$K$, the dual space~$(KZ)^*=K^Z$ is the vector space consisting of all functions~$Z \to K$.\symlistsort{K Z}{$K^Z$}{vector space consisting of all functions~$Z \to K$}  
If~$G$ is a finite (multiplicative) group, we write~$KG$ for the \emph{group algebra} of~$G$ over~$K$.\symlistsort{KG}{$KG$}{group algebra of~$G$ over~$K$}\indexadd{group algebra} It is the free vector space~$KG$  equipped with the product from~$G$, extended bilinearly.

\paragraph{Matrices.}  For two finite sets~$Y,Z$ and a set~$R$, we write~$R^{Y \times Z}$ for the set of matrices with entries in~$R$ whose rows and columns are indexed by the elements of~$Y$ and~$Z$ respectively.\symlistsort{RYtimesZ}{$R^{Y \times Z}$}{set of~$Y \times Z$ matrices with entries in~$R$} We call an element of~$R^{Y \times Z}$ an~$Y \times Z$ \emph{matrix}.\indexadd{matrix} It is a function~$Y \times Z \to R$, and if~$R=K$ is a field, it is an element of the function space~$K^{Y\times Z}$ defined above. Moreover,~$R^{m \times n} = R^{[m] \times [n]}$ for nonnegative integers~$m,n$, and  we write~$R^Z := R^{Z \times 1} = R^{Z \times [1]}$. 

For a matrix~$A \in \C^{Y \times Z}$, the \emph{conjugate transpose} of~$A$ is the~$Z \times Y$ matrix~$A^*$ with~$A^*_{i,j}=\overline{A_{j,i}}$ for~$i \in Z$ and~$j \in Y$, where for any complex number~$z$ its complex conjugate is denoted by~$\overline{z}$.\symlistsort{A star}{$A^*$}{conjugate transpose of matrix~$A$}\indexadd{conjugate transpose}\symlistsort{zbar}{$\overline{z}$}{complex conjugate of~$z \in \C$} If~$A \in \C^{Z \times Z}$ with~$A=A^*$, then~$A$ is called \emph{Hermitian}.\indexadd{matrix!Hermitian}
 For two matrices~$A,B \in \C^{Y \times Z}$, their \emph{(trace) inner product}~$\langle \,,{}\rangle$ is defined as\symlistsort{A,B}{$\langle A, B \rangle$}{trace inner product of matrices~$A$ and~$B$}\indexadd{inner product}
 $$
 \langle A , B \rangle = \tr(B^*A) = \sum_{i \in Y, j \in Z} \overline{B_{i,j}} {A_{i,j}}.
 $$  
 If~$A \in R^{Y \times Z}$, the \emph{transpose} of~$A$ is the $Z \times Y$ matrix~$A\T$ with~$A\T_{i,j} = A_{j,i}$ for~$i \in Z$ and~$j \in Y$.\symlistsort{AT}{$A\T$}{transpose of matrix~$A$}\indexadd{transpose} Finally, for~$n \in \N$ we denote by~$I_n$ and~$J_n$ the~$n \times n$ identity and all-ones matrices, respectively.\symlistsort{In}{$I_n$}{$n \times n$ identity matrix}\symlistsort{Jn}{$J_n$}{$n \times n$ all-ones matrix}

 \paragraph{Polynomials on a vector space.}
 For a $\C$-vector space~$V$, we write~$\mathcal{O}(V)$ for the \emph{coordinate ring} of~$V$.\symlistsort{O (V)}{$\mathcal{O}(V)$}{coordinate ring of~$V$} It is the algebra of functions~$p \, : \,V \to \C$ generated by elements of the dual vector space~$V^*$ of~$V$, i.e., the algebra consisting of all~$\C$-linear combinations of products of elements from~$V^*$.\indexadd{coordinate ring} An element of~$\mathcal{O}(V)$ is called a \emph{polynomial} on~$V$. If~$p\in \mathcal{O}(V)$ is a~$\C$-linear combination in which each term is a constant times a product of~$n$ elements of~$V^*$ (for a fixed integer~$n \in \Z_{\geq 0}$), we call~$p$ \emph{homogeneous} of degree~$n$. Moreover we write~$\mathcal{O}_n(V) := \{p \in \mathcal{O}(v)\,\,  | \,\, \text{ $p$ homogeneous of degree~$n$} \}$.\symlistsort{O n(V)}{$\mathcal{O}_n(V)$}{set of homogeneous polynomials of degree~$n$ on~$V$}

\section{Coding theory}\label{codingprem}

In this section we give the main definitions from coding theory used throughout the thesis. Many definitions have already passed in Chapter~\ref{introduction}, but for convenience of the reader we restate them here.

\paragraph{The parameters~$A_q(n,d)$, $A(n,d,w)$ and~$A_q^L(n,d)$.}
Suppose that~$Q$ is a finite set of~$q \geq 2$ elements and fix a positive integer~$n$.  A \emph{word} is an element of~$Q^n$ 
 and a \emph{($q$-ary) code} is a subset~$C $ of $ Q^n$. So~$Q$ is our \emph{alphabet}. We call elements of~$Q$ \emph{symbols} or \emph{letters}.\indexadd{symbol}\indexadd{letter}
 For two words~$u,v \in Q^n$,  their \emph{(Hamming) distance} $d_H(u,v)$ is the number of~$i$ with~$u_i \neq v_i$. 
 For a code~$C \subseteq Q^n$, its \emph{minimum distance}~$d_{\text{min}}(C)$ is  the minimum of~$d_H(u,v)$ over all~$u,v \in C$ with~$u \neq v$. If~$|C| \leq 1$, we set~$d_{\text{min}}(C) = \infty$.  Without loss of generality we usually take~$Q:=[q]$, where~$[q]$ denotes the set~$\{0,\ldots,q-1\}$ (in this context).
 Define, for any integer~$d$,
 \begin{align}\label{aqndprem}
 A_q(n,d) := \max \{ |C| \,\, | \,\, C \subseteq [q]^n, \,\, d_{\text{min}}(C) \geq d \}.
 \end{align}
    If~$q=2$ we often write~$A(n,d)$ instead of~$A_q(n,d)$. If~$C \subseteq [q]^n$ with~$d_{\text{min}}(C) \geq d$, we call~$C$ an~\emph{$(n,d)_q$-code}, and if~$d=2$ simply an \emph{$(n,d)$-code}.\indexadd{code!$(n,d)_q$-}\indexadd{code!$(n,d)$-} The weight of a word~$ v \in [q]^n$ is~$\wt(v):=d_H(v,\mathbf{0})$, where~$\bm{0}=  0\ldots 0$ denotes the all-zeros word.\indexadd{weight}\symlistsort{wt(v)}{$\text{wt}(v)$}{weight of word~$v$}
   A classical upper bound on~$A_q(n,d)$ is the \emph{Delsarte bound in the Hamming scheme}, which is given in~\eqref{delsintro}.

Now, let~$Q:= \F_2 = [2]=\{0,1\}$ denote the binary alphabet. Here~$\F_2$ denotes the field of two elements.    A (binary) \emph{constant weight code} is a code~$C \subseteq \F_2^n$ in which all words have a fixed weight~$w$.  Define, for~$n,d,w \in \N$,
\begin{align}\label{andw}
    A(n,d,w) :=  \max \{ | C| \,\, | \,\, C \subseteq \F_2^n \,\,\, d_{\text{min}}(C) \geq d, \,\,\, \text{wt}(v) = w   \text{ for all } v \in C \}.
\end{align}
  If~$C \subseteq \F_2^n$ is a constant weight $w$ code with $d_{\text{min}}(C) \geq d$, we call~$C$ an \emph{$(n,d,w)$-code}.\indexadd{code!$(n,d,w)$-} A classical upper bound on~$A(n,d,w)$ is the \emph{Delsarte bound in the Johnson scheme}~$D(n,d,w)$. It is defined as follows. We assume without loss of generality that~$d$ is even (as any two weight~$w$ words have even distance) and that~$w \leq n/2$ (as~$A(n,d,w) = A(n,d,n-w)$).  Define~$E_i$ to be the~$i$-th \emph{Eberlein polynomial}:\indexadd{Eberlein polynomial}\indexadd{Delsarte linear programming bound!in the Johnson scheme}\symlistsort{Ei(x)}{$E_i(x)$}{Eberlein polynomial}
\begin{align}
E_i(x) := \sum_{j=0}^i (-1)^{j} \binom{x}{j} \binom{w-x}{i-j} \binom{n-w-x}{i-j}, \,\,\,\,\, \text{ for $0 \leq i \leq w$}.
\end{align}
Then~$A(n,d,w) \leq D(n,d,w)$, where\symlistsort{Dn,d,w}{$D(n,d,w)$}{Delsarte bound in the Johnson scheme} 
\begin{align}\label{delsjohn}
D(n,d,w) := \max\big\{ \mbox{$\sum_{i=0}^{w} a_{2i}$} \,\, \big| \,\, & a_0=1,\,\, a_{2}=a_4 = \ldots = a_{d-2}=0, \,\,\, a_{2i} \geq 0 \text{ if $d/2 \leq i \leq w$},   \notag  
\\
& \mbox{$\sum_{i=0}^{w}  \binom{w}{i}^{-1} \binom{n-w}{i}^{-1}E_i(k) a_{2i} \geq 0$} \text{ for all $ 0 \leq k \leq w$}\big\}.
\end{align}
Finally, let~$Q:=\Z_q$.  For two words~$u,v \in \Z_q^n$, their \emph{Lee distance} is $d_L(u,v):= \sum_{i=1}^n \min \{ |u_i-v_i|, q-|u_i-v_i|  \}$, where we consider~$u_i$ and~$v_i$ as integers in~$\{0,\ldots,q-1\}$.      The \emph{minimum Lee distance} $d_{\text{min}}^L(C)$   of a code~$C\subseteq \Z_q^n$ is the minimum of~$d_L(u,v)$  taken over all distinct~$u,v \in C$. If~$|C|\leq 1$, we set~$d_{\text{min}}^L(C)= \infty$. Define, for~$q,n,d \in \N$,
\begin{align} \label{aleeqndprem}
A^L_q(n,d):= \max \{ |C| \, \, | \,\, C \subseteq \Z_q^n, \,\, d_{\text{min}}^L(C) \geq d \}.
\end{align}
If~$C \subseteq \Z_q^n$ with~$d_{\text{min}}^L(C) \geq d$, we call~$C$ an~\emph{$(n,d)_q^L$-code}.\indexadd{code!$(n,d)_q^L$-}  The Delsarte bound can be adapted to the Lee scheme to obtain a linear programming upper bound on~$A^L_q(n,d)$~$\cite{astola1,delsarte}$. 

\paragraph{Linear codes.} 
Assume that~$Q$ is a field (often,~$Q=\F_2$). A code~$C \subseteq Q^n$ is called \emph{linear} if it is a linear subspace of~$Q^n$.\indexadd{code!linear} So~$|C|=|Q|^k$, with~$k$ the dimension of~$C$ as a~$Q$-vector space. Any~$k \times n$ matrix with a basis of~$C$ as rows is called a \emph{generator} matrix of~$C$.\indexadd{generator matrix}\indexadd{matrix!generator}  

If~$C \subseteq \F_2^n$ is a linear code, the minimum distance of~$C$ is the minimum of the nonzero weights of the words in~$C$: if~$u,v \in C$ also~$u-v \in C$ and~$d_H(u,v)=\text{wt}(u-v)$ for all~$u,v \in C$. 

\paragraph{Distance distribution.} 
 The \emph{distance distribution}~$(a_i)_{i=0}^n$ of a nonempty code~$C \subseteq Q^n$ is the sequence~$(a_i)_{i=0}^n$ given by\indexadd{distance distribution}\symlistsort{ai i0n}{$(a_i)_{i=0}^n$}{distance distribution}
 $$
 a_i :=|C|^{-1} \cdot |\{(u,v) \in C \times C \,\, | \, \, d_H(u,v)=i\}|, \,\,\,\,\,\, \text{ for~$i=0,\ldots,n$}.
$$
The inequalities on the~$a_i$ in the Delsarte bound~\eqref{delsintro} are satisfied by the distance distribution~$(a_i)_{i=0}^n$ of any code~$C \subseteq Q^n$ with~$ d_{\text{min}}(C) \geq d$. Similarly, if~$C \subseteq \F_2^n$ is a constant weight~$w$ code with~$ d_{\text{min}}(C) \geq d$, its distance distribution~$(a_{2i})_{i=0}^w$ satisfies the inequalities in the Delsarte bound in the Johnson scheme~\eqref{delsjohn}. 

\paragraph{Equivalence of codes.} Two $q$-ary codes~$C,D \subseteq [q]^n$ are \emph{equivalent} if~$D$ can be obtained from~$C$ by first permuting the~$n$ coordinates and by subsequently permuting the alphabet~$[q]$ in each coordinate separately, i.e., if there is a~$g \in S_q^n \rtimes S_n$ such that~$g \cdot C = D$. Similarly, two binary constant weight codes~$C,D$ are \emph{equivalent} if~$D$ can be obtained from~$C$ by permuting the~$n$ coordinates.  Two codes~$C,D \subseteq \Z_q^n$ are \emph{Lee equivalent}  if there  is a $g \in D_q^n \rtimes S_n$ such that $g \cdot C = D$. Here~$D_q$ denotes the dihedral group of order~$2q$ (acting on~$\Z_q$).

\paragraph{The binary (extended) Golay code.}
In Chapters~\ref{cw4chap} and~\ref{cu17chap}, we will encounter  the (extended) binary Golay code.\indexadd{binary Golay code} There exists a unique optimal code attaining~$A(23,7)=2^{12}$ up to equivalence~\cite{delsartegolay, snover}. This code is called the \emph{binary Golay code}, discovered by Marcel J.E.\ Golay~\cite{golay}. It is \emph{perfect}, i.e., the balls~$B(u):=\{ v \in \F_2^{23} \,\, | \,\, d_H(u,v) \leq 3)\}$ around the codewords~$u\in C$ partition the vector space~$\F_2^{23}$. (This follows by counting, as for any~$u \in C$ we have~$|B(u)|=1+\binom{23}{1} + \binom{23}{2} + \binom{23}{3} = 2^{11}$. Since~$|C|=2^{12}$ and~$B(u) \cap B(v) = \emptyset$ if~$u,v \in C$ with~$u \neq v$, we conclude
$ \sum_{u \in C} |B(u)| = 2^{12}\cdot 2^{11} = |\F_2^{23}|$.) \symlistsort{Bu}{$B(u)$}{ball around codeword~$u$}

We will also need the \emph{extended binary Golay code}.\indexadd{binary Golay code!extended} This is the code obtained from the binary Golay code by adding a~$24$th \emph{parity check} bit  to every codeword, which is equal to the sum (mod~$2$) of the original~$23$ bits of the respective codeword. So each codeword of the extended binary Golay code has even weight. As two binary words with even weight have even Hamming distance, and the binary Golay code has minimum distance~$7$ (which is odd), the extended binary Golay code has minimum distance~$8$. The extended binary Golay code is the unique binary code achieving~$A(24,8)=2^{12}$, up to equivalence. 

Usually, we   assume that the (extended) binary Golay code has been chosen such as to contain~$\bm{0}=0\ldots0$, the all-zeros word. Then both codes are linear. A generator matrix of the extended binary Golay code is given in Figure~\ref{golaygenmatintro} on page~\pageref{golaygenmatintro}. 

Also, we will encounter the~$i$-times shortened (extended) binary Golay code (for~$i=1,\ldots,4$). Shortening of a code~$C$ means fixing a coordinate position (in~$\{1,\ldots,n\}$) of~$C$, taking the subcode of~$C$ consisting of all codewords having symbol~$0$ in this coordinate position, and then deleting this fixed coordinate position of this subcode, resulting in a code~$D$. Note that~$D$ has word length~$n-1$ and that~$d_{\text{min}}(D) \geq d_{\text{min}}(C)$. The code~$D$ is obtained by \emph{shortening}~$C$.\indexadd{binary Golay code!shortened (extended)}\indexadd{shortening} 

 Since the automorphism groups of the binary Golay and extended binary Golay codes, the \emph{Mathieu groups}~$M_{23}$ and~$M_{24}$, act~$4$- and~$5$-transitively on the coordinate positions respectively (see~\cite{sloane}), the~\emph{$i$-times shortened (extended) binary Golay code} is well-defined for~$i=1,\ldots,4$ (i.e., these codes do not depend on which coordinate position we chose as  `fixed' coordinate position in each shortening step).\indexadd{Mathieu groups}\symlistsort{M23}{$M_{23}$}{Mathieu group}\symlistsort{M24}{$M_{24}$}{Mathieu group}

\section{Semidefinite programming}
Recall that a matrix~$A \in \C^{n \times n}$ is called \emph{Hermitian} if~$A^*=A$. The eigenvalues of an Hermitian matrix are real. Let~$A \in \C^{n \times n}$. Then~$A$ is called \emph{positive semidefinite}\indexadd{positive semidefinite}  if~$A$ is Hermitian and all eigenvalues of~$A$ are nonnegative.\indexadd{matrix!positive semidefinite} We use the notation $A \succeq 0$ to denote that~$A$ is positive semidefinite. It is known that the following are equivalent:
 \begin{speciaalenumerate}
 \item \label{psd1} $A \succeq 0$,
 \item \label{psd2} $A^*=A$ and $v^*Av \geq 0$ for all~$v \in \C^n$,
 \item \label{psd3} $A=L^*L$ for some~$L \in \C^{n\times n}$.
 \end{speciaalenumerate}
 In case~$A$ is real, we may restrict~$v$ to real vectors in~(\ref{psd2}) and~$L$ to real matrices in~(\ref{psd3}). Note that if~$A=L_1^*L_1$ and~$B=L_2^*L_2$ are positive semidefinite matrices, then
 \begin{align}\label{psdnn}
 \langle A,B \rangle = \tr(B^*A)= \tr(L_2^* L_2L_1^*L_1) = \tr(L_1L_2^* L_2L_1^*)=\tr((L_2L_1^*)^*(L_2L_1^*)) \geq 0,
 \end{align}
  as~$(L_2L_1^*)^*(L_2L_1^*)$ is positive semidefinite by~(\ref{psd3}). 

We now describe semidefinite programming. Let~$m,n \in \Z_{\geq 0}$,~$b_1,\ldots,b_m \in \R$ and let $C,F_1,\ldots,F_m \in \R^{n \times n}$ be symmetric real matrices. A \emph{semidefinite program} is an optimization problem of the form\indexadd{semidefinite program}
 \begin{align}\label{primalpr}
 \text{Maximize: }& b_1 z_1 + \ldots + b_m z_m \notag  \\
 \text{subject to: }&M:= C- (F_1 z_1 + \ldots + F_m z_m) \succeq 0.
 \end{align}  
Here~$z_1,\ldots,z_m$ are called the \emph{variables}. Any~$(M,z)$ satisfying the constraints in~$\eqref{primalpr}$ is called a \emph{feasible solution} to~$\eqref{primalpr}$,  with \emph{objective value}~$b_1z_1 + \ldots + b_m z_m$.\indexadd{feasible solution}\indexadd{objective value} The program~$\eqref{primalpr}$ has an associated dual program which reads   
   \begin{align}\label{dualpr}
 \text{Minimize: }&  \langle C, X \rangle  \notag  \\
 \text{subject to: }&  \langle F_{i}, X\rangle = b_{i}  \text{ for all~$i=1,\ldots,m $},  \notag \\
 & X\succeq 0.
 \end{align}  
 Any~$X \in \R^{n \times n}$ satisfying the constraints in~$\eqref{dualpr}$ is called a \emph{feasible solution} to~$\eqref{dualpr}$ with \emph{objective value}~$\langle C, X \rangle$.
 Note that if~$(M,y)$ and~$X$ are feasible solutions to~$\eqref{primalpr}$ and~$\eqref{dualpr}$ respectively, then
 \begin{align}
 \langle C , X \rangle &= \left\langle M + \mbox{$ \sum_{i=1}^m F_iz_i$}, X \right\rangle = \langle M, X\rangle + \mbox{$ \sum_{i=1}^m\langle F_i, X\rangle z_i$} = \langle M,X \rangle +\mbox{$\sum_{i=1}^m b_iz_i$} \notag
 \\& \geq \mbox{$ \sum_{i=1}^m b_iz_i$},   
 \end{align} 
 where the inequality follows from~$\eqref{psdnn}$. So the maximum in~\eqref{primalpr} is at most the minimum in~\eqref{dualpr}. This is called \emph{weak duality}.\indexadd{weak duality} Under a certain condition (called \emph{Slater's condition}), which is satisfied in the cases when we apply SDP, even \emph{strong duality} holds, that is, the optimal values in~$\eqref{dualpr}$ and~$\eqref{primalpr}$ are equal.  To exhibit upper bounds in coding theory using semidefinite programs, we only need weak duality, i.e., the fact that a given feasible dual solution has an objective value greater than or equal to the objective value of \emph{any} primal feasible solution.
  
Semidefinite programs can be solved approximately ---up to any fixed precision--- in polynomial time by the ellipsoid method~\cite{ellipsoid}.  In practice more recent interior point methods \cite{interiorpoint} are preferred, which also run in polynomial time. See~\cite{todd} for more information about semidefinite programming.

 \section{Preliminaries on representation theory\label{prelimrep}}
In this section we give the definitions and  notation from representation theory  used throughout the thesis.   For more information, the reader can consult Sagan~$\cite{sagan}$.

A \emph{group action} of a group~$G$ on a set~$Z$ is a group homomorphism~$\phi: G \to S_Z$, where~$S_Z$ is the group of bijections of~$Z$ to itself.\symlistsort{SZ}{$S_Z$}{group of bijections of~$Z$ to itself}\indexadd{group action} If~$\phi$ is a group action of a group~$G$ on a set~$Z$, we say that~$G$ \emph{acts} on~$Z$, and we write~$g \cdot x:= \phi(g)(x)$ for all~$g \in G$ and~$x \in Z$. Moreover, we write~$Z^G$ for the set of elements of~$Z$ invariant under the action of~$G$.\symlistsort{ZG}{$Z^G$}{subset of~$Z$ of $G$-invariant elements}
 If~$Z$ is a linear space, we restrict~$S_Z$ to the linear bijections~$ Z \to Z$. The action of~$G$ on a finite set~$Z$   induces an action of~$G$ on the linear space~$\C Z$ by linear extension. It also induces an action on the linear space~$\mathbb{C}^Z$, by~$(g \cdot f)(z):= f(g^{-1} \cdot z)$, for~$g\in G$,~$f \in \mathbb{C}^Z$ and~$z \in Z$.

If~$G$ is a group acting on a finite dimensional complex vector space~$V$, then~$V$ is called a~$G$\emph{-module}.\indexadd{G-module@$G$-module} If~$V$ is a $G$-module and~$U \subseteq V$ is a $G$-invariant subspace, the space~$U$ is called a \emph{submodule} of~$V$.\indexadd{submodule}  If~$V$ and~$W$ are~$G$-modules, then a $G$\emph{-homomorphism} (or: $G$\emph{-equivariant map}) $\psi: V \to W$ is a linear map such that~$g \cdot \psi(v)=\psi(g \cdot v)$ for all~$g \in G$,~$v \in V$. Two $G$-modules~$V$ and $W$ are  ($G$-)\emph{isomorphic} or \emph{equivalent} if there exists a bijective $G$-homomorphism from~$V$ to~$W$, which is also called a $G$\emph{-isomorphism}.\indexadd{G-isomorphism@$G$-isomorphism}\indexadd{G-homomorphism@$G$-homomorphism}\indexadd{G-equivariant map@$G$-equivariant map}

A $G$-module~$V$ is called \emph{irreducible} if~$V \neq \{0\}$ and the only submodules of~$V$ are~$\{0\}$ and~$V$ itself.\indexadd{irreducible}  We write $\text{End}_G(V)$ for the  \emph{centralizer algebra} of the action of~$G$ on~$V$, i.e., the algebra of~$G$-equivariant maps~$V \to V$. We recall the well-known Schur's Lemma,  which   characterizes all $G$-equivariant maps between irreducible $G$-modules.\symlistsort{EndGV}{$\text{End}_G(V)$}{centralizer algebra for the action of~$G$ on $V$}\indexadd{centralizer algebra}\indexadd{Schur's lemma}
\begin{lemma}[Schur's Lemma] Let~$V $ and~$W$ be irreducible~$G$-modules, and let~$f : V \to W$ be a $G$-equivariant map. \vspace{-2pt}
\begin{speciaalenumerate}
\item If $V$ and $W$ are nonisomorphic, then~$f=0$.
\item If there exists a $G$-isomorphism~$\phi \, : \, V \to W$, then~$f=\lambda \phi$ for some~$\lambda \in \C$.  
\end{speciaalenumerate}
\end{lemma}

  Let~$G$ be a finite group acting on a finite dimensional complex vector space~$V$. Then~$V$ can be  decomposed as $V=V_1 \oplus \ldots \oplus V_k$ such that each~$V_i$ is a direct sum~$V_{i,1} \oplus \ldots \oplus V_{i,m_i}$ of irreducible~$G$-modules with the property that~$V_{i,j}$ and~$V_{i',j'}$ are isomorphic if and only if~$i=i'$.  The number~$k$ and the~$V_i$ are unique up to permuting indices.  The $G$-modules $V_1,\ldots,V_k$ are called the~$G$\emph{-isotypical components}.\indexadd{G-isotypical component@$G$-isotypical component}

Now choose, for each~$i \leq k$ and~$j \leq m_i$,    a nonzero vector~$u_{i,j} \in V_{i,j}$ so that for each~$i$ and all~$j,j'\leq m_i$ there exists a~$G$-isomorphism~$V_{i,j} \to V_{i,j'}$ mapping~$u_{i,j}$ to~$u_{i,j'}$. For each~$i \leq k$,  define~$U_i$ to be the $m_i$-tuple~$(u_{i,1},\ldots ,u_{i,m_i})$ of elements~$u_{i,j}$ ($j=1,\ldots,m_i$). (We order the elements of~$U_i$ for later purposes.)
\begin{defn}[Representative set]
Any set~$\{U_1,\ldots,U_k\}$ obtained in this way is called a \emph{representative set} for the action of~$G$ on~$V$.\indexadd{representative set}\indexadd{representative set!for the action of~$G$ on~$V$}
\end{defn}

It is convenient to note that, since~$V_{i,j}$ is the linear space spanned by~$G \cdot u_{i,j}$ (for each~$i,j$), we have
\begin{align} \label{Rm}
V = \bigoplus_{i=1}^k\bigoplus_{j=1}^{m_i} \C G \cdot u_{i,j},
\end{align}
where~$\C G$ denotes the complex group algebra of~$G$. Moreover, note that 
\begin{align}
\dim\text{End}_G(V) = \dim\text{End}_G\left(\bigoplus_{i=1}^k\bigoplus_{j=1}^{m_i}V_{i,j}\right) = \sum_{i=1}^k m_i^2,
\end{align}
by Schur's lemma.

\begin{proposition}
Let~$G$ be a finite group acting on a finite dimensional complex vector space~$V$. Let $k,$ $m_1,\ldots,m_k \in \N$, and $u_{i,j} \in V$ for~$i=1,\ldots,k$,~$j=1,\ldots,m_i$. Then the set
\begin{align}\label{lemmaset}
\left\{ \,\left(u_{i,1},\ldots ,u_{i,m_i}\right)\, \, | \,\, i=1,\ldots,k \right\}
\end{align}
is representative for the action of~$G$ on~$V$ if and only if:
\begin{enumerate}[(i)]
\item \label{1eig} $ V=  \bigoplus_{i=1}^k\bigoplus_{j=1}^{m_i} \C G \cdot u_{i,j}$,
\item \label{2eig} $\forall \,\, i=1,\ldots, k$ and  $j,j' =1,\ldots, m_i$, there exists a $G$-isomorphism~$\C G\cdot  u_{i,j} \to \C G \cdot u_{i,j'}$ mapping~$u_{i,j}$ to~$u_{i,j' }$,
\item \label{3eig} $\sum_{i=1}^k m_i^2  \geq \dim (\text{\normalfont End}_G(V))$.
\end{enumerate}
\end{proposition}
\proof
The implication ``$\Longrightarrow$'' follows from the definition and the above observations. To prove~``$\Longleftarrow$'', suppose that $k,m_1,\ldots,m_k \in \N$, and that $u_{i,j} \in V$ for~$i=1,\ldots,k$,~$j=1,\ldots,m_i$, satisfy the three   conditions from the lemma. Set~$V_{i,j}:=\C G \cdot u_{i,j}$, for $i=1,\ldots,k$, $j=1,\ldots,m_i$. 
Then by~(\ref{1eig}) and~(\ref{2eig}) the $V_{i,j}$ form a {\em decomposition} of $V$
into  $G$-modules and  $V_{i,j}$ and $V_{i',j'}$ are equivalent representations
if $i=i'$. In fact, by~(\ref{3eig}) the $V_{i,j}$ form a {decomposition} of $V$
into {\em irreducible} representations and $V_{i,j}$ and $V_{i',j'}$ are equivalent representations
if and {\em only if} $i=i'$, as any further representation,
or decomposition, or equivalence would yield that the sum of the squares of the multiplicities of the
irreducible representations is strictly larger than~$\sum_{i=1}^k m_i^2 \geq \dim (\text{End}_G(V))$, which is not possible.

 So the set in~$(\ref{lemmaset})$ is representative for the action of~$G$ on~$V$.
\endproof

Let~$G$ be a finite group acting on a finite dimensional complex vector space~$V$. Let~$\langle \, , \rangle$ be a $G$-invariant inner product on~$V$. (Note that any inner product $\langle \, ,  \rangle$  on~$V$ gives rise to a $G$-invariant inner product~$\langle \, , \rangle_G$ on~$V$ via the rule~$\langle x ,y \rangle_G := \sum_{g \in G} \langle g \cdot x , g \cdot y \rangle$.)\indexadd{inner product}  If~$\{U_1,\ldots,U_k\}$ is a representative set for the action of~$G$ on~$V$, define the map\symlistsort{Fi(A)}{$\Phi$}{block-diagonalizing map}
\begin{align} \label{PhiC}
    \Phi : \text{End}_G(V) \to \bigoplus_{i=1}^k \C^{m_i \times m_i} \,\, \text{ with } \,\, A \mapsto \bigoplus_{i=1}^k \left( \langle A u_{i,j'} , u_{i,j}  \rangle \right)_{j,j' =1}^{m_i}.
\end{align}
Note that~$\Phi$ is linear. An element~$A \in \text{End}(V)$ is called \emph{positive semidefinite} if~$\langle Av,w \rangle = \langle v, A w\rangle$ for all~$v,w \in V$ and~$\langle Av, v \rangle \geq 0$ for all~$v \in V$. 
\begin{proposition}
The map~$\Phi$ is bijective. Moreover, for any $A \in \text{\normalfont End}_G(V)$:
\begin{align}
 \text{$A$ is positive semidefinite} \,\,\, \Longleftrightarrow  \,\,\, \text{$\Phi(A)$ is positive semidefinite}.
 \end{align}
\end{proposition}
\proof 
It is a standard fact from representation theory that the $G$-isotypical components $V_{1},\ldots,V_k$ are orthogonal with respect to $\langle \, , \rangle$. For each~$i =1,\ldots,k$, the decomposition~$V_i = V_{i,1} \oplus \ldots \oplus V_{i,m_i}$ of~$V_i$ into irreducible and mutually isomorphic $G$-modules is not necessarily orthogonal, but there exists an \emph{orthogonal} decomposition~$V_i = W_{i,1} \oplus \ldots \oplus W_{i,m_i}$ of~$V_i$ into irreducible and mutually isomorphic $G$-modules (and~$V_{i,j}$ is isomorphic to~$W_{i,j}$ for each~$i\in [k]$, $j\in[m_i]$). 

By Schur's lemma, $\text{span}\{u_{i,1},\ldots,u_{i,m_i}\} = \{ Bu_{i,1} \,\,| \,\, B \in \text{End}_G(V)\}$. For each~$i\in [k]$, $j \in [m_i]$, take a~$w_{i,j} \in  (  \{ Bu_{i,1} \,\,| \,\, B \in \text{End}_G(V)\} \cap W_{i,j})$ such that~$\langle w_{i,j}, w_{i,j} \rangle = 1$. Then there exists a $G$-invariant map~$W_{i,j} \to W_{i,j' }$ mapping~$w_{i,j}$ to~$w_{i,j'}$ (for each~$i \in [k]$, $j,j' \in [m_i]$). Moreover, for each~$i\in [k]$, the set  $\{w_{i,1},\ldots,w_{i,m_i}\}$ is an orthonormal basis for~$\text{span}\{u_{i,1},\ldots,u_{i,m_i}\}$. If~$P \in \C^{m_i \times m_i}$ is the invertible change-of-basis matrix from~$\{u_{i,1},\ldots,u_{i,m_i}\}$ to~$\{w_{i,1},\ldots,w_{i,m_i}\}$, then  
$$
 (\langle A {u_{i,j'}}, u_{i,j} \rangle)_{j,j'=1}^{m_i} =  P^* \,\, (\langle A {w_{i,j'}}, w_{i,j} \rangle)_{j,j'=1}^{m_i}\,\, P. 
$$
Hence, it suffices to prove the statements of the theorem for the linear map
\begin{align*} 
    \Phi': \text{End}_G(V) \to \bigoplus_{i=1}^k \C^{m_i \times m_i} \,\, \text{ with } \,\, A \mapsto \Phi'(A):= \bigoplus_{i=1}^k \left( \langle A w_{i,j'} , w_{i,j}  \rangle \right)_{j,j' =1}^{m_i}
\end{align*}
instead of~$\Phi$.  

For each~$i \in [k]$, we write~$d_i:=\dim W_{i,1}$. Furthermore, for each~$i \in [k]$, $j \in[m_i]$,  let~$\{w_{i,j,l}\,\, | \,\, l \in [d_i]\}$ be an orthonormal basis of~$W_{i,j}$ such that for each~$j, j'$ there exists a~$G$-equivariant map~$W_{i,j} \to W_{i,j'}$ mapping~$w_{i,j,l}$ to~$w_{i,j',l}$ for each~$l \in [d_i]$, and such that~$w_{i,j,1}=w_{i,j}$. (The bases can be chosen in this way since any irreducible $G$-module has a \emph{unique} $G$-invariant inner product, up to scalar multiplication. This, in turn, is a consequence of Schur's Lemma.) 

Then a direct computation (with Schur's lemma) gives that, for $i,i'\in[k]$, $j\in[m_i]$, $j'\in[m_{i'}]$, $l\in[d_i]$, $l'\in[d_{i'}]$:
\begin{align}\label{zeroinnerprod}
 \langle A {w_{i',j',l'}}, w_{i,j,l} \rangle = \begin{cases} \langle A {w_{i,j'}}, w_{i,j} \rangle &\mbox{if $i = i'$ and $l=l' $,} \\ 
 0  & \mbox{otherwise}.  \end{cases}
  \end{align}
Since~$\{w_{i,j,l} \, | \, i \in [k], j \in [m_i], l \in [d_i]\}$ is a basis of~$V$,  the linear map~$A$ is  positive semidefinite if and only if the matrix
$$
 M:=( \langle A {w_{i',j',l'}}, w_{i,j,l} \rangle)_{i\in [k],  l \in [d_i],j \in [m_i], i'\in [k], l'\in [d_{i'}], j'\in [m_{i'}]}
$$
is positive semidefinite, and~$A=0$ if and only if~$M=0$. 
By~$\eqref{zeroinnerprod}$, the matrix~$M$ is the direct sum
$$
M= \bigoplus_{i=1}^k \bigoplus_{l=1}^{d_i}  (\langle A {w_{i,j'}}, w_{i,j} \rangle)_{j,j'=1}^{m_i},
$$ 
and~$M$ contains~$d_i$ identical blocks (for~$i=1,\ldots,k$). So~$A$ is positive semidefinite if and only if $\Phi'(A) $ 
is positive semidefinite. Moreover,~$A=0$ if and only if~$\Phi'(A)=0$, so~$\Phi' $ is injective. Bijectivity of~$\Phi'$ now follows from the fact that
$
\dim \text{End}_G(V)=  \sum_{i=1}^k m_i^2 = \dim (\oplus_{i=1}^k \C^{m_i \times m_i}). 
$
\endproof 
This result is basic in our symmetry reductions. Firstly,~$  \dim (\text{End}_G(V)) =\sum_{i=1}^k m_i^2$, which can be considerably smaller than~$\dim V$.  Secondly, any~$A \in \text{End}_G(V)$ is positive semidefinite if and only if the image~$\Phi(A)$ is positive semidefinite, i.e., if and only if each of the matrices~$ \left( \langle A u_{i,j'} , u_{i,j}  \rangle \right)_{j,j' =1}^{m_i}$ is positive semidefinite.

We will often apply the previous theory  to the following situation. Suppose that~$G$ is a finite group acting on a finite set~$Z$, hence on~$V:=\C^Z$. Consider the standard inner product~$\langle x, y \rangle :=y^*x$ for~$x,y \in \C^Z$, where~$y^*$ denotes the conjugate transpose of~$y \in \C^Z$.\indexadd{inner product} This inner product is~$G$-invariant, since for each~$g\in G$ the $Z \times Z$ matrix~$L_g$ with~$g \cdot x = L_gx$ for all~$x \in V$ is unitary, since in this case~$L_g$ is a permutation matrix.
\begin{defn}[Representative set for the action of~$G$ on~$Z$]
We call any representative set~$\{U_1,\ldots,U_k\}$ for the action of~$G$ on~$V=\C^Z$ induced by the action of~$G$ on~$Z$ a \emph{representative set for the action of~$G$ on~$Z$}.\indexadd{representative set!for the action of~$G$ on~$Z$}
\end{defn}
Let~$\{U_1,\ldots,U_k\}$ be a representative set for the action of~$G$ on~$Z$. Then each~$U_i$ is an ordered set of~$m_i$ elements from~$\C^Z$ (for~$i=1,\ldots,k$), hence it is a~$Z \times m_i$ matrix. We will often call~$\{U_1,\ldots,U_k\}$ a \emph{representative matrix set} for the action of~$G$ on~$Z$ (or on~$\C^Z$).\indexadd{representative set!representative matrix set}
Moreover, we will regularly write~$U_i=[u_{i,1},\ldots,u_{i,m_i}]$ instead of~$U_i= (  u_{i,1},\ldots,u_{i,m_i})$, to emphasize the fact that~$U_i$ is a matrix (here $i=1,\ldots,k$).

The action of~$G$ on~$Z$ induces an action of~$G$ on~$Z\times Z$, hence on~$\C^{Z\times Z}$. So $(\C^{Z \times Z})^G$ is the algebra of complex~$Z \times Z$-matrices invariant under the action of~$G$ on~$Z\times Z$. Since for each~$g\in G$ there is a permutation $Z \times Z$ matrix~$L_g$ such that~$g \cdot x = L_gx$ for all~$x \in \C^Z$,  the algebra $(\C^{Z \times Z})^G$ can be identified with  the algebra $\text{End}_G(\C^Z)$ of~$G$-equivariant maps~$\C^Z \to \C^Z$. Since~$\{U_1,\ldots,U_k\}$ is a representative matrix set for the action of~$G$ on~$\C^Z$,~\eqref{PhiC} implies that
\begin{align} \label{PhiC2}
    \Phi :(\C^{Z \times Z})^G \to \bigoplus_{i=1}^k \C^{m_i \times m_i} \,\, \text{ with } \,\, A \mapsto \bigoplus_{i=1}^k U_i^* A U_i
\end{align}
is bijective. 

It turns out that all representative sets determined in this thesis consist of real matrices. Then
\begin{align} \label{PhiR} 
\Phi(A) = \bigoplus_{i=1}^k U_i\T A U_i \text{ for }A \in (\R^{Z \times Z})^G, \,\,\,\,\text{ and }\Phi\left((\R^{Z \times Z})^G\right) = \bigoplus_{i=1}^k \R^{m_i \times m_i}.
\end{align}
Also,~$A \in (\R^{ Z \times Z})^G$ is positive semidefinite if and only if each of the matrices~$U_i\T A U_i$ is positive semidefinite ($i=1,\ldots,k$). As mentioned, this is very useful for checking whether~$A$ is positive semidefinite.

Finally, we note that if~$G_1$ and~$G_2$ are finite groups and if $\{U_1^{(1)},\ldots,U_{k_1}^{(1)}\}$ and $\{U_1^{(2)},\ldots,U_{k_2}^{(2)}\}$ are representative matrix sets for the actions of~$G_1$ and~$G_2$ on finite sets~$Z_1$ and~$Z_2$ respectively, then
\begin{align} \label{prodreprset}
\left\{U_i^{(1)} \otimes U_j^{(2)}  \,\, | \,\, i=1,\ldots,k_1, \,\, j=1,\ldots,k_2\right\}
\end{align} 
is representative for the action of~$G_1 \times G_2$ on~$Z_1 \times Z_2$. Moreover, if~$M_1 \in (\C^{ Z_1 \times Z_1})^{G_1}$, $M_2\in (\C^{Z_2 \times Z_2})^{G_2}$, and~$i \in [k_1]$, $j \in [k_2]$,  
then
\begin{align} \label{tensorcomp}
 & \left(  \left(U_i^{(1)} \otimes U_j^{(2)}\right) \T (M_1 \otimes M_2) \left( U_i^{(1)} \otimes U_j^{(2)} \right) \right)_{(a,b),(a',b') } \notag \\ &=  \left(U_i^{(1)}(a) \otimes U_j^{(2)}(b)\right) \T (M_1 \otimes M_2) \left( U_i^{(1)}(a') \otimes U_j^{(2)}(b') \right)  \notag \\ 
&= \left(U_i^{(1)}(a)\T M_1 U_i^{(1)}(a') \right) \left(U_j^{(2)}(b)\T M_2 U_j^{(2)}(b') \right).
\end{align}
So the entries in the matrix blocks of~$\Phi(M_1 \otimes M_2)$ are products of entries in the matrix blocks of~$\Phi(M_1)$ and~$\Phi(M_2)$. This will be useful in explicitly calculating the blocks.

\subsection{A representative set for the action of~\texorpdfstring{$S_n$}{Sn} on~\texorpdfstring{$V^{\otimes n}$}{Vtensorn}}\label{repr}
Fix~$n\in \N$ and a finite-dimensional vector space~$V$. We consider the natural action of the symmetric group~$S_n$ on~$V^{\otimes n}$ by permuting the indices. Based on classical representation theory of the symmetric group, we describe a representative set for the action of~$S_n$ on~$V^{\otimes n}$. It will be used repeatedly in the reductions throughout this thesis. 

A \emph{partition}~$\lambda $ of~$n$ is a sequence~$(\lambda_1,\ldots, \lambda_h)$ of natural numbers with~$\lambda_1 \geq \ldots \geq \lambda_h >0$ and~$\lambda_1 + \ldots + \lambda_h = n$.\indexadd{partition} The number~$h$ is called the \emph{height} of~$\lambda$.\indexadd{partition!height} We write~$\lambda \vdash n$ if~$\lambda$ is a partition of~$n$.\symlistsort{lambda vdash n}{$\lambda\vdash n$}{$\lambda$ is a partition of~$n$}  The~\emph{Young shape} (or \emph{Ferrers diagram})~$Y(\lambda)$ of~$\lambda$ is the set\symlistsort{Y(lambda)}{$Y(\lambda)$}{Young shape of~$\lambda$}\indexadd{Young shape}
\begin{align}
    Y(\lambda) := \{(i,j) \in \N^2 \, | \, 1 \leq j \leq h, \, 1 \leq i \leq \lambda_j\}.  
\end{align}
(Here French notation is used, cf.~\cite{procesi}.)
Fixing an index~$j_0 \leq h$, the set of elements~$(i,j_0)$ (for~$1 \leq i \leq \lambda_{j_0}$) in~$Y(\lambda)$ is called the~$j_0$\emph{-th row} of~$Y(\lambda)$. Similarly, fixing an element~$i_0 \leq \lambda_1$, the set of elements~$(i_0,j)$ (where~$j$ varies) in~$Y(\lambda)$ is called the~$i_0$\emph{-th column} of~$Y(\lambda)$.   Then the \emph{row stabilizer}~$R_{\lambda}$ of~$\lambda$ is the group of permutations~$P$ of~$Y(\lambda)$ with~$P(L)=L$ for each row~$L$ of~$Y(\lambda)$.\symlistsort{Rlambda}{$R_{\lambda}$}{row stabilizer of~$\lambda$} Similarly, the \emph{column stabilizer}~$C_{\lambda}$ of~$\lambda$ is the group of permutations~$P$ of~$Y(\lambda)$ with~$P(L)=L$ for each column~$L$ of~$Y(\lambda)$.\symlistsort{Clambda}{$C_{\lambda}$}{column stabilizer of~$\lambda$}\indexadd{row stabilizer}\indexadd{column stabilizer} 

A \emph{Young tableau with shape~$\lambda$} (also called a~$\lambda$\emph{-tableau}) is a function~$\tau : Y(\lambda) \to \N$.\indexadd{Young tableau} A Young tableau with shape~$\lambda$ is \emph{semistandard} if the entries are nondecreasing in each row and strictly increasing in each column.\indexadd{Young tableau!semistandard} Let~$T_{\lambda,m}$ be the collection of semistandard $\lambda$-tableaux with entries in~$[m]$.\symlistsort{Tlambda,m}{$T_{\lambda,m}$}{collection of semistandard $\lambda$-tableaux with entries in~$[m]$} Then~$T_{\lambda,m} \neq \emptyset$  if and only if~$m$ is at least the height of~$\lambda$. We write~$\tau \sim \tau'$ for~$\lambda$-tableaux~$\tau,\tau'$ if~$\tau'=\tau r$ for some~$r \in R_{\lambda}$.\symlistsort{tau sim tau}{$\tau \sim \tau'$}{$\tau'=\tau r$ for some~$r \in R_{\lambda}$}

Let~$B=(B(1),\ldots,B(m))$ be an ordered basis of~$V$. For any~$\tau \in T_{\lambda,m}$, define\symlistsort{utau,B}{$ u_{\tau,B}$}{element of ordered $\lvert T_{\lambda,m}\rvert$-tuple in representative set in~(\ref{utau})}
\begin{align} \label{utau}
    u_{\tau,B}:= \sum_{\tau'\sim \tau } \sum_{c \in C_{\lambda}} \text{sgn}(c) \bigotimes_{y \in Y(\lambda)} B\left(\tau '(c(y))\right)  \,\,\,\,\in V^{\otimes n}.
\end{align}
Here the Young shape~$Y(\lambda)$ is ordered by concatenating its rows. Then (cf.~$\cite{sagan}$ and~$\cite{onsartikel}$) the set 
\begin{align} \label{reprsetdef}
\left\{\, (u_{\tau,B} \,\, | \,\, \tau \in  T_{\lambda,m}) \,\, | \,\, \lambda \vdash n  \right\},
\end{align}
 is a representative set for the natural action of~$S_n$ on~$V^{\otimes n}$, for any ordering of the elements of~$T_{\lambda,m}$. So if~$Z$ is a finite set and~$V=\C^Z$, then~$\eqref{reprsetdef}$ gives a representative set for the natural action of~$S_n$ on~${Z^n}$ via the natural isomorphism~$\C ^{Z^n} \cong (\C^{Z})^{\otimes n}$.

\chapter{Main symmetry reduction}\label{orbitgroupmon}
\vspace{-6pt}
\chapquote{Beauty is bound up with symmetry.}{Hermann Weyl (1885--1955)}

\noindent  Suppose that~$G$ is a finite group acting on a finite set~$Z$ and let~$n\in \N$. We consider the natural action of~$H:=G^n \rtimes S_n$ on~$Z^n$. 
If~$U_1,\ldots,U_k$ form a representative set of matrices for the action of~$H$ on~$Z^n$, then with~$\eqref{PhiC}$ we obtain a reduction~$\Phi(M)$ of matrices~$M \in  (\C^{Z^n \times Z^n})^{H}$ to size polynomially bounded in~$n$. 

First, we give in Section~\ref{reprSnkrSq} a representative set for the action of~$H=G^n \rtimes S_n$ on~$V^{\otimes n}$ (for any finite dimensional complex vector space~$V$), with the help of the representative set for the natural action of~$S_n$ on~$V^{\otimes n}$  from Section \ref{repr}. Here we assume that a representative set~$\bm{B}=\{B_1,\ldots,B_{k}\}$  for the action of~$G$ on~$V$ is given.  If the given representative set is real, also the representative set for the action of~$H$ on~$V^{\otimes n}$ turns out to be real. So if~$V=\C^Z$, then our method yields a representative set for the action of~$H$ on~${Z^n}$ via the natural isomorphism~$\C ^{Z^n} \cong (\C^{Z})^{\otimes n}=V^{\otimes n}$. 

Subsequently, we show in Section~\ref{computationsection} how to compute~$\Phi(M)$  in polynomial time. In Section~$\ref{genmult}$ we give a generalization to groups of the form $ (G_1^{j_1} \rtimes S_{j_1}) \times \ldots \times (G_s^{j_s} \rtimes S_{j_s}) $ acting on sets of the form $ Z_1^{j_1} \times \ldots\times Z_s^{j_s}$. The chapter is concluded with two appendices. The first appendix gives two subroutines to compute a polynomial used in the computation of~$\Phi(M)$. In the second appendix we show as an illustration how to apply the symmetry reduction of this chapter to the matrix occurring in the analytical definition of the Delsarte bound for error-correcting codes with the Hamming distance.

 The reduction explained in this chapter can also be derived from a manuscript of Gijswijt~\cite{gijswijt}, but we give a more direct approach. Our method is a generalization of the method of~$\cite{onsartikel}$, which is joint work with Bart Litjens and Lex Schrijver.

\section{A representative set for the action of~\texorpdfstring{$G^n \rtimes S_n$}{G power n rtimes Sn} on~\texorpdfstring{$V^{\otimes n}$}{Vtensorn}}\label{reprmult}\label{reprSnkrSq}
Let~$G$ be a finite group acting on a finite dimensional vector space~$V$. Suppose that a representative set~$\bm{B}=\{B_1,\ldots,B_k\}$ for the action of~$G$ on~$V$ is given. Here each~$B_i$ is an ordered~$m_i$-tuple of elements from~$V$, for some integers~$m_1,\ldots,m_k$. \symlistsort{B bold}{$ \bm{B}$}{representative set~$\bm{B}=\{B_1,\ldots,B_k\}$ for the action of~$G$ on~$V$} 

Let $n \in \N$ and let ${\bm{N}}$ be the collection of all $k$-tuples $(n_1,\ldots,n_k)$ of nonnegative integers adding up
to $n$.\symlistsort{N bold}{$ \bm{N}$}{collection of $k$-tuples $(n_1,\ldots,n_k)$ of nonnegative integers with sum~$n$} 
For $\bm{n}=(n_1,\ldots,n_k)\in{\bm N}$, let
$\bm{\lambda\vdash n}$ mean that $\bm{\lambda}=(\lambda_1,\ldots,\lambda_k)$ with
$\lambda_i\vdash n_i$ for $i=1,\ldots,k$.\symlistsort{lambda vdash n bold}{$\bm{\lambda\vdash n}$}{$k$ partitions~$\lambda_i\vdash n_i$} 
(So each $\lambda_i$ is equal to a partition $(\lambda_{i,1},\ldots,\lambda_{i,h_i})$ of~$n_i$, for some $h_i$.) Define~$\bm{m}:=(m_1,\ldots,m_k)$.\symlistsort{m bold}{$ \bm{m}$}{$k$-tuple of integers $(m_1,\ldots,m_k)$} 

For $\bm{\lambda}\vdash\bm{n}$ define\symlistsort{Tlambda,mbold}{$\bm{T_{\lambda,m}}$}{product of~$T_{\lambda_i,m_i}$}
\begin{align}\label{wlambda} 
\bm{T_{\lambda,m}}:=T_{\lambda_1,m_1}\times\cdots\times T_{\lambda_k,m_k},
\end{align} 
and for $\bm{\tau}=(\tau_1,\ldots,\tau_k)\in  \bm{T_{\lambda,m}}$ define\symlistsort{utau,Bbold}{$\bm{u_{\tau,B}}$}{element of ordered tuple in representative set in~(\ref{matset})}
\begin{align}\label{vtau}
\bm{u_{\tau,B}}:=\bigotimes_{i=1}^ku_{\tau_i,B_i} = \bigotimes_{i=1}^k  \sum_{\tau_i'\sim \tau_i } \sum_{c_i \in C_{\lambda_i}} \text{sgn}(c_i) \bigotimes_{y \in Y(\lambda_i)} B_i\left(\tau_i '(c_i(y))\right) \,\,\,\,\,\,\,\,\, \in V^{\otimes n}.
\end{align}
\begin{proposition}\label{prop2}
The set
\begin{align}\label{matset}
\{~~
\left(\bm{u_{\tau,B}}
\mid
\bm{\tau}\in \bm{T_{\lambda,m}} \right)
~~\mid
\bm{n}\in\bm{N},\bm{\lambda\vdash n}
\}
\end{align}
is representative for the action of $H:=G^n\rtimes S_n$ on $V^{\otimes n}$ (for any ordering of the elements of~$\bm{T_{\lambda,m}}$).
\end{proposition}

\proof
For~$i=1,\ldots,k$, let $L_i$ denote the $\C$-linear space spanned by $B_i(1),\ldots,B_i(m_i)$.
Then
\begin{align}
V^{\otimes n}
&\stackrel{\text{\eqref{Rm}}}{=}
\left(\bigoplus_{i=1}^k\bigoplus_{j=1}^{m_i}\C G\cdot B_i(j)\right)^{\otimes n} 
=
\C S_n\cdot\bigoplus_{\bm{n}\in\bm{N}}\bigotimes_{i=1}^k
\left(\bigoplus_{j=1}^{m_i}\C G\cdot B_i(j)\right)^{\otimes n_i} \notag 
\\&=
\C S_n\cdot (\C G)^{\otimes n}\cdot
\bigoplus_{\bm{n}\in\bm{N}}\bigotimes_{i=1}^kL_i^{\otimes n_i}
 \stackrel{\text{\eqref{Rm}, \eqref{reprsetdef}}}{=}
\C H\cdot
\bigoplus_{\bm{n}\in\bm{N}}\bigotimes_{i=1}^k\bigoplus_{\lambda_i\vdash n_i}\bigoplus_{\tau_i\in T_{\lambda_i,m_i}}
\C S_{n_i}\cdot u_{\tau_i,B_i} \notag 
\\&=
\bigoplus_{\bm{n}\in\bm{N}}
\bigoplus_{\bm{\lambda\vdash n}}
\bigoplus_{\bm{\tau}\in \bm{T_{\lambda,m}}}
\C H\cdot \bm{u_{\tau,B}}.
\end{align}
Now for each ${\bm{n,\lambda}}$ and $\bm{\tau},\bm{\sigma}\in \bm{T_{\lambda,m}}$,
there is an $H$-isomorphism
$\C H\cdot \bm{u_{\tau,B}}\to\C H\cdot \bm{u_{\sigma,B}}$ bringing
$ \bm{u_{\tau,B}}$ to $\bm{u_{\sigma,B}}$,
since for each $i=1,\ldots,k$, setting $H_i:=G^{n_i}\rtimes S_{n_i}$,
there is an $H_i$-isomorphism $\C H_i\cdot u_{\tau_i,B_i}\to\C H_i\cdot u_{\sigma_i,B_i}$, since for each~$j,j' \in [m_i]$ there is a~$G$-isomorphism~$\C G \cdot B_i(j) \to \C G \cdot B_i(j')$ and by the results in Section~\ref{repr}, there is an~$S_{n_i}$-isomorphism~$\C S_{n_i}\cdot u_{\tau_i,B_i}\to\C S_{n_i}\cdot u_{\sigma_i,B_i}$.
Hence (where $\Sym_t(X):=(X^{\otimes t})^{S_t}$ for any $t\in\Z_{\geq 0}$ and any linear space $X$, with the natural
action of $S_t$ on $X^{\otimes t}$)\symlistsort{SymtX}{$\text{Sym}_t(X)$}{vector space $(X^{\otimes t})^{S_t}$}
\begin{align}
\dim\left((V^{\otimes n}\otimes V^{\otimes n})^H\right) \notag
&\geq
\sum_{\bm{n}\in\bm{N}}
\sum_{\bm{\lambda\vdash n}}|\bm{T_{\lambda,m}}|^2
=
\sum_{\bm{n}\in\bm{N}}
\sum_{\bm{\lambda\vdash n}}\prod_{i=1}^k|T_{\lambda_i,m_i}|^2 \notag 
\\& =
\sum_{\bm{n}\in\bm{N}}\prod_{i=1}^k\sum_{\lambda_i\vdash n_i}|T_{\lambda_i,m_i}|^2 
=
\sum_{\bm{n}\in\bm{N}}\prod_{i=1}^k\dim\text{Sym}_{n_i}(\C^{m_i}\otimes\C^{m_i}), \notag
\end{align}
where the last equality follows from the fact that the set in~\eqref{reprsetdef} is representative for the action of~$S_{n_i}$ on~$(\C^{m_i})^{\otimes n_i}$. So
\begin{align} \label{newlabelcomp}
\dim\left((V^{\otimes n}\otimes V^{\otimes n})^H\right) 
& \geq \sum_{\bm{n}\in\bm{N}}\prod_{i=1}^k\dim\text{Sym}_{n_i}(\C^{m_i}\otimes\C^{m_i}) \notag
\\&=\sum_{\bm{n}\in\bm{N}}\prod_{i=1}^k\binom{m_i^2+n_i-1}{n_i-1}
=
\binom{\sum_{i=1}^km_i^2+n-1}{n-1}\notag
\\& =
\dim\text{Sym}_n((V \otimes V)^{G}) \notag 
\\& =
\dim\left((V^{\otimes n} \otimes V^{\otimes n})^H\right),
\end{align}
as $\sum_{i=1}^km_i^2=\dim((V \otimes V)^{G})$.
So we have equality throughout in \eqref{newlabelcomp}, and hence each $\C H\cdot \bm{u_{\tau,B}}$ is irreducible, and if
$\bm{\lambda}\neq\bm{\lambda'}$,
then for each
${\bm{\tau}}\in \bm{T_{\lambda,m}}$
and
${\bm{\tau'}}\in \bm{T_{\lambda',m}}$, the spaces
$\C H\cdot \bm{u_{\tau,B}}$ 
and
$\C H\cdot\bm{u_{\tau',B}}$ 
are not $H$-isomorphic. 
\endproof 
 Note that the representative set in~\eqref{matset} is real if we start with a real representative set $\bm{B}=\{B_1,\ldots,B_k\}$. 
Moreover, if~$G$ is a finite group acting on a finite set~$Z$ and~$V=\C^Z$, then~$\eqref{matset}$ gives a representative set for the action of~$G^n \rtimes S_n$ on~${Z^n}$ via the natural isomorphism~$\C ^{Z^n} \cong (\C^{Z})^{\otimes n}=V^{\otimes n}$.

\section{Computation}\label{computationsection}

 If~$G$ is a finite group acting on a finite set~$Z$,  we often tacitly identify~$\C Z$ and~$\C^Z=(\C Z)^*$ via the standard inner product.\footnote{The standard inner product on the free vector space~$\C Z$ is the unique inner product (linear in the first coordinate, conjugate-linear in the second coordinate) on~$\C Z$ with respect to which~$Z$ is an orthonormal basis.}  It will turn out that we only need this identification for real vectors.

Let~$G$ be a finite group acting on a finite set~$Z$, hence on~$Z \times Z$. Set~$\Lambda:=(Z \times Z)/G$ and~$W:=(\C Z\otimes \C Z)^G$.\symlistsort{Lambda}{$\Lambda$}{$(Z \times Z)/G$}\symlistsort{W}{$W$}{$(\C Z\otimes \C Z)^G$} For each~$P \in \Lambda$, define\symlistsort{aP}{$a_P$}{basis element of~$W$}
\begin{align} \label{apdef}
a_P := \sum_{\substack{(x,y) \in P}} x \otimes y \,\,\,\in W.
\end{align}
Then the set~$A:= \{ a_P \, | \, P \in \Lambda\}$ is a basis of~$W$.\symlistsort{A}{$A$}{basis  $\{ a_P \, \mid \, P \in \Lambda\}$ of~$W$} 

Let~$n \in \N$ and set again~$H:=G^{n} \rtimes S_{n}$. There is a natural bijection~$\Lambda^n/S_n \to (Z \times Z)^n/H$ given by
\begin{align} \label{natbijection}
\omega  \mapsto \{ P_1 \times \cdots \times P_n \,\, | \,\,  (P_1,\ldots,P_n) \in \omega \},
\end{align}
for~$\omega \in \Lambda^n/S_n$. For each~$\omega \in \Lambda^n/S_n$, let~$K_{\omega}$ be the \emph{adjacency matrix} of~$\omega$, i.e., the~$Z^n \times Z^n$ matrix with\symlistsort{Komega}{$K_{\omega}$}{adjacency matrix of~$\omega$}
\begin{align}
    (K_{\omega})_{(\alpha,\beta)}:= \begin{cases} 1 &\mbox{if } (\alpha,\beta) \in \omega,  \\ 
0 & \mbox{otherwise,} \end{cases} 
\end{align}
for~$\alpha, \beta \in Z^n$. Using the natural identification of~$Z^n\times Z^n$ and~$(Z \times Z)^n$, we have
\begin{align} \label{komegaap}
  K_{\omega} = \sum_{(P_1,\ldots,P_n) \in \omega}
a_{P_1} \otimes \cdots \otimes a_{P_n} \,\,\,\,\,\in W^{\otimes n}. 
\end{align}
Then any  matrix in~$(\C^{Z^n \times Z^n})^H$  can be written as\symlistsort{M(z)}{$M(z)$}{matrix in  in~$(\C^{Z^n \times Z^n})^H$}
\begin{align*}
M(z):=\sum_{\omega \in \Lambda^n/S_n} z(\omega) K_{\omega}, \,\,\,\,\,\, \text{ for~$z :  \Lambda^n/S_n \to \mathbb{C}$}.
\end{align*}
A representative matrix set for the action of~$H$ on~$Z^n$ is given by~$\eqref{matset}$ (with~$V:= \C^Z$). Given~$\bm{n}\in\bm{N}$, for each~$\bm{\lambda}\vdash\bm{n}$ we write~$U_{\bm{\lambda}}$ for the matrix in~$(\ref{matset})$ that corresponds with~$\bm{\lambda}$.  For any~$z :  \Lambda^n/S_n \to \mathbb{C}$ we obtain with~$(\ref{PhiC2})$ that\symlistsort{Ulambda}{$U_{\bm{\lambda}}$}{matrix in representative set that corresponds with~$\bm{\lambda}$}
\begin{align} \label{blocks1ttbasicC}
    \Phi(M(z))= \Phi \left(\sum_{\omega \in  \Lambda^n/S_n}z(\omega) K_{\omega} \right ) = \bigoplus_{\bm{n}\in\bm{N}} \bigoplus_{\bm{\lambda}\vdash\bm{n}} \sum_{\omega \in \Lambda^n/S_n} z(\omega) U_{\bm{\lambda}}^* K_{\omega} U_{\bm{\lambda}}. 
\end{align}

The matrices~$M(z)$ we consider in this thesis are all real. Moreover, in all our applications the representative sets turn out to be real, i.e., all~$U_{\bm{\lambda}}$ are real matrices. So we can write~$U_{\bm{\lambda}}\T$ instead of~$U_{\bm{\lambda}}^*$. In the remainder of this chapter, we assume that $M(z)$ and all~$U_{\bm{\lambda}}$ are real matrices. We leave the adaptations for the case that~$U_{\bm{\lambda}}$ is not real to the reader. 

So in our case, for any~$z :  \Lambda^n/S_n \to \mathbb{R}$ we obtain with~$(\ref{PhiR})$ that\symlistsort{z(omega)}{$z(\omega)$}{variable of reduced program}
\begin{align} \label{blocks1ttbasic}
    \Phi(M(z))= \Phi \left(\sum_{\omega \in  \Lambda^n/S_n}z(\omega) K_{\omega} \right ) = \bigoplus_{\bm{n}\in\bm{N}} \bigoplus_{\bm{\lambda}\vdash\bm{n}} \sum_{\omega \in \Lambda^n/S_n} z(\omega) U_{\bm{\lambda}}\T K_{\omega} U_{\bm{\lambda}}. 
\end{align}
 
Fixing~$G$ and~$Z$, the number of~$\bm{n}\in\bm{N}$, $\bm{\lambda} \vdash \bm{n}$, and the numbers~$|\bm{T_{\lambda,m}}|$ and~$| \Lambda^n/S_n|$  are all bounded by a polynomial in~$n$. This implies that the number of blocks in~$(\ref{blocks1ttbasic})$, the size of each block and the number of variables occurring in all blocks are polynomially bounded in~$n$.

We still need to show how to compute the entries in the blocks $\sum_{\omega \in \Lambda^n/S_n} \hspace{-1.3715pt}z(\omega)U_{\bm{\lambda}}\T K_{\omega} U_{\bm{\lambda}}$ for~$\bm{n}\in {\bm{N}}$, $\bm{\lambda} \vdash \bm{n}$ in~$\eqref{blocks1ttbasic}$ in polynomial time, since the orders of $\bm{u_{\tau,B}}$, $\bm{u_{\sigma,B}}$ and~$K_{\omega}$ are exponential in~$n$. That is, we must compute~$\sum_{\omega \in  \Lambda^n/S_n} z(\omega)\bm{u_{\tau,B}}\T K_{\omega} \bm{u_{\sigma,B}}$, for~$\bm{\tau}, \bm{\sigma} \in \bm{T_{\lambda,m}}$,  $\bm{\lambda} \vdash \bm{n}$ and $\bm{n}\in {\bm{N}}$.

\subsubsection{How to compute~\texorpdfstring{$ \bm{u_{\tau,B}}\T K_{\omega}  \bm{u_{\sigma,B}}$}{boldutauBT Komega boldusigmaB}}\label{computecoef}

First we examine a relation between~$S_n$-orbits on~$\Lambda^n$ and monomials of degree~$n$ expressed in the dual basis~$A^*$ of~$A$, which will be used in the computations.\symlistsort{A star}{$A^*$}{basis  $\{ a_P^* \, \mid \, P \in \Lambda\}$ of~$W^*$, i.e., dual basis of~$A$} 

For each~$\omega \in \Lambda^n/S_n$, the monomial~$\mu(\omega)$ is defined as
\begin{align} \label{muomegadef}
    \mu(\omega):= a_{P_1}^* \cdots a_{P_n}^* \in \mathcal{O}_n(W),
\end{align}
where~$(P_1,\ldots,P_n)$ is an arbitrary element of~$\omega$ (this does not depend on the choice of $(P_1,\ldots,P_n)$). Note that this gives a bijection between~$\Lambda^n/S_n$ and the set of degree~$n$ monomials expressed in the basis~$A^*$.\symlistsort{mu(omega)}{$\mu(\omega)$}{monomial $ a_{P_1}^* \cdots a_{P_n}^* \in \mathcal{O}_n(W)$}

Let~$w \mapsto \widehat{w}$ be the linear function~$(W^*)^{\otimes n} \to \mathcal{O}_n(W)$ satisfying
$$
(w_1^* \otimes \cdots \otimes w_n^*) \mapsto w_1^* \cdots w_n^*
$$
for all~$w_1^*,\ldots,w_n^* \in W^*$.\symlistsort{w hat}{$\widehat{w}$}{element of~$ \mathcal{O}_n(W)$ associated to~$w=(w_1^* \otimes \cdots \otimes w_n^*)  \in (W^*)^{\otimes n}$}

Define for~$i=1,\ldots,k$,  the~$m_{i} \times m_i$ matrix $F_i \in (W^*)^{m_i \times m_i}$ by\symlistsort{Fi}{$F_i$}{matrix in~$ (W^*)^{m_i \times m_i}$} 
\begin{align} \label{fformula}
(F_i)_{j,h}&:= (B_i(j) \otimes B_i(h))|_W = \sum_{P \in \Lambda} (B_i(j) \otimes B_i(h))(a_P) a_P^*.
\end{align}
For any $n',m'\in\Z_{\geq 0}$, $\lambda\vdash n'$, and $\tau,\sigma\in T_{\lambda,m'}$, define
the polynomial $p_{\tau,\sigma}\in\C[x_{j,h}\mid j,h=1,\ldots,m']$ by\symlistsort{ptausigma}{$p_{\tau,\sigma}$}{polynomial corresponding to semistandard Young tableaux} 
\begin{align}\label{ptsdef}
p_{\tau,\sigma}(X):=
\sum_{\substack{\tau'\sim\tau\\ \sigma'\sim\sigma}}\sum_{c,c'\in C_{\lambda}}\sgn(cc')
\prod_{y\in Y(\lambda)}
x_{\tau'c(y),\sigma'c'(y)},
\end{align}
for $X=(x_{j,h})_{j,h=1}^{m'}\in\C^{m'\times m'}$.
 This polynomial can be computed (expressed as a linear combination of monomials in~$ x_{j,h}$) in polynomial time, as proven in~$\cite{gijswijt, 
 onsartikel}$ --- see Section~\ref{21de15b}.
 \begin{proposition}\label{ptsprop} We have \begin{align}
 \sum_{\omega \in \Lambda^n/S_n} \left( \bm{u_{\tau,B}}\T K_{\omega}\bm{u_{\sigma,B}}\right)\mu(\omega) = \prod_{i=1}^k p_{\tau_i,\sigma_i}(F_i).
\end{align}
\end{proposition}
\proof 
For each~$\omega  \in \Lambda^n/S_n$ we can write
$
 \bm{u_{\tau,B}}\T K_{\omega}\bm{u_{\sigma,B}} = ( \bm{u_{\tau,B}}\otimes  \bm{u_{\sigma,B}})( K_{\omega}),
$
using the fact that~$ \bm{u_{\tau,B}},  \bm{u_{\sigma,B}} \in  ((\C {Z})^{\otimes n} )^*$ and~$K_{\omega} \in  (\C {Z})^{\otimes n} \otimes   (\C {Z})^{\otimes n}$. Set
$$
f:= ( \bm{u_{\tau,B}}\otimes  \bm{u_{\sigma,B}})|_{W^{\otimes n}} = \bigotimes_{i=1}^k  \sum_{\substack{\tau_i'\sim \tau_i \\ \sigma_i'\sim \sigma_i }} \sum_{c_i,c_i' \in C_{\lambda_i}} \text{sgn}(c_ic_i') \bigotimes_{y \in Y(\lambda_i)} (F_i)_{\tau_i 'c_iy,\, \sigma_i'c_i' y}.
$$
Then, as~$K_{\omega} \in W^{\otimes n}$,
\begin{align*}\label{tensorproductsomtoshow}
  \sum_{\omega  \in \Lambda^n/S_n}  ( \bm{u_{\tau,B}}\otimes  \bm{u_{\sigma,B}})( K_{\omega}) \mu(\omega)&=     \sum_{\omega  \in \Lambda^n/S_n}   f( K_{\omega}) \mu(\omega)
   \\ &\hspace{-1.4em}\overset{\eqref{komegaap},\, \eqref{muomegadef}}{=}    \hspace{-0.5em} \sum_{(P_1,\ldots,P_n) \in \Lambda^n}  f(a_{P_1} \otimes \cdots \otimes a_{P_n})\, a_{P_1}^* \cdots a_{P_n}^* 
  \\&=    \sum_{(P_1,\ldots,P_n) \in \Lambda^n} f(a_{P_1} \otimes \cdots \otimes a_{P_n}) \, \reallywidehat{a_{P_1}^* \otimes \ldots \otimes a_{P_n}^*}
   =\widehat{f}
   \\&= \prod_{i=1}^k  \sum_{\substack{\tau_i'\sim \tau_i \\ \sigma_i'\sim \sigma_i }} \sum_{c_i,c_i' \in C_{\lambda_i}} \text{sgn}(c_ic_i') \prod_{y \in Y(\lambda_i)} (F_i)_{\tau_i 'c_iy,\, \sigma_i'c_i' y}
  \\&= \prod_{i=1}^k p_{\tau_i,\sigma_i}(F_i),
\end{align*}
which gives the desired equality.
\endproof 
So it remains to compute the entries of the matrices~$F_i$, i.e., to express each $(B_i(j)\otimes B_i(h))|_W$
as a linear function into the basis $A^*$. So we must
calculate the numbers $(B_i(j)\otimes B_i(h))(a_P)$ for all~$i=1,\ldots,k$ and
$j,h=1,\ldots,m_i$, and $P\in \Lambda$.

Now one computes the entry $\sum_{\omega \in \Lambda^n/S_n} z(\omega) \bm{u_{\tau,B}}\T  K_{\omega} \bm{u_{\sigma,B}}$ by replacing each monomial $\mu(\omega)$ in $ \prod_{i=1}^k p_{\tau_i,\sigma_i}(F_i)$ with the variable~$z(\omega)$.   

\section{Generalization: multiple \texorpdfstring{$Z_i$}{Zi} and \texorpdfstring{$G_i$}{Gi}}\label{genmult}

Let~$s \in \N$ be fixed. For~$i=1,\ldots,s$, let~$G_i$ be a finite group acting on a finite set~$Z_i$ and let~$j_i \in \N$ be such that~$j_1+\ldots+j_s=n$. Set
\begin{align}
H:= (G_1^{j_1} \rtimes S_{j_1}) \times \ldots \times (G_s^{j_s} \rtimes S_{j_s}) \,\,\, \text{ and } \,\,\, R:= Z_1^{j_1} \times \ldots\times Z_s^{j_s}.
\end{align}
\symlistsort{R}{$R$}{The set~$Z_1^{j_1} \times \ldots\times Z_s^{j_s}$ (in Sect.~\ref{genmult})}

For~$i=1,\ldots,s$, let~$\Lambda_i:=(Z_i \times Z_i)/G_i$ and~$W_i:= (\C {Z_i} \otimes \C {Z_i})^{G_i}$.\symlistsort{Lambda i}{$\Lambda_i$}{$(Z_i \times Z_i)/G_i$}\symlistsort{W_i}{$W_i$}{$(\C Z_i\otimes \C Z_i)^{G_i}$} By~\eqref{natbijection}, there is a natural bijection~$\Lambda_i^{j_i}/S_{j_i} \to (Z_i \times Z_i)^{j_i}/ (G_i^{j_i} \rtimes S_{j_i})$ for each~$i=1,\ldots,s$.  Also, there is a bijection
$$
 \prod_{i=1}^s \left( (Z_i \times Z_i)^{j_i}/ (G_i^{j_i} \rtimes S_{j_i}) \right) \to R^2/H,
$$
obtained by identifying~$ (Z_1 \times Z_1)^{j_1} \times \ldots \times (Z_s \times Z_s)^{j_s} $ and~$R^2$ in the natural way.
So this gives is a bijection between the set~$(\Lambda_1^{j_1}/S_{j_1}) \times \ldots  \times (\Lambda_{s}^{j_s}/S_{j_s} )$ and the set of~$H$-orbits on~$ R^2$. So we can write each~$\omega \in R^2/H$ as
\begin{align} \label{omegaschrijfwijze}
\omega= \omega_1 \times \ldots \times \omega_s, \,\,\,\text{ where~$\omega_i \in \Lambda_{i}^{j_i}/S_{j_i}$ for each~$1 \leq i \leq s$}.  
\end{align}
 For each~$P \in \Lambda_i$, define~$a_P := \sum_{ (x,y) \in P} x \otimes y \in W_i$. Then the set~$A_i:=\{a_P \,\, | \,\, P \in \Lambda_i\}$ is a basis of~$W_i$.\symlistsort{Ai}{$A_i$}{basis $\{a_P \, \mid \, P \in \Lambda_i\}$ of~$W_i$} As before, let for each orbit~$\omega_i \in \Lambda_i^{j_i}/S_{j_i}$ the matrix~$K_{\omega_i}^{(i)}$ be the \emph{adjacency matrix} of~$\omega_i$, i.e., the~$Z_i^{j_i} \times Z_i^{j_i}$ matrix with\symlistsort{Komegai}{$K_{\omega_i}^{(i)}$}{adjacency matrix of~$\omega_i$}
\begin{align}
(K^{(i)}_{\omega_i})_{\alpha,\beta} := \begin{cases} 1 &\mbox{if } (\alpha,\beta) \in \omega_i,  \\ 
0 & \mbox{otherwise,} \end{cases} 
\end{align}
for~$\alpha, \beta \in Z^{j_i}$. Using the natural identification of~$Z_i^{j_i}\times Z_i^{j_i}$ and~$(Z_i \times Z_i)^{j_i}$, we have
\begin{align}
  K_{\omega_i}^{(i)} = \sum_{(P_1,\ldots,P_{j_i}) \in \omega_i}
a_{P_1} \otimes \cdots \otimes a_{P_{j_i}} \,\,\,\,\,\in W_i^{\otimes j_i}. 
\end{align}
Next, we define for each~$\omega \in R^2/H$ the~$R \times  R$ adjacency matrix~$K_{\omega}$ of~$\omega$ by\symlistsort{Komega}{$K_{\omega}$}{adjacency matrix of~$\omega$}
\begin{align} \label{KomegaIH}
(K_{\omega})_{\alpha,\beta} := \begin{cases} 1 &\mbox{if } (\alpha,\beta) \in \omega,  \\ 
0 & \mbox{otherwise,} \end{cases} 
\end{align}
for~$\alpha, \beta \in R$.       So 
\begin{align}\label{Komegaprod}
K_{\omega} =   K_{\omega_1}^{(1)} \otimes \ldots \otimes K_{\omega_s}^{(s)},
\end{align}
if~$\omega=\omega_1\times \ldots \times \omega_s$.  The matrix~$M(z):=\sum_{\omega \in R^2/H } z(\omega) K_{\omega}$ (for any~$z : R^2/H \to \mathbb{R}$) can be reduced to size polynomially bounded in~$n$. We sketch the reduction.

With Proposition~$\ref{prop2}$, we obtain representative sets for the separate actions of~of $G_i^{j_i}\rtimes S_{j_i}$ on $Z_i^{j_i}$, for each~$i=1,\ldots,s$. Subsequently, with~$\eqref{prodreprset}$ we find a representative set for the action of $(G_1^{j_1} \rtimes S_{j_1}) \times \ldots \times (G_s^{j_s} \rtimes S_{j_s})$ on $(\C^{Z_1^{j_1}}) \otimes \ldots \otimes (\C^{Z_s^{j_s}})\cong \C^R$. The computation of~$\Phi(M(z))$ now entirely follows from~\eqref{tensorcomp} in combination with the results described in Section~\ref{computationsection}, as
\begin{align}
(u_1 \otimes \ldots \otimes u_s)\T \left(K_{\omega_1}^{(1)} \otimes \ldots \otimes K_{\omega_s}^{(s)}\right)(v_1 \otimes \ldots \otimes v_s) = \prod_{i=1}^s \left(u_i\T K_{\omega_i}^{(i)} v_i\right),
\end{align}
for all vectors~$u_i, v_i \in \C^{Z_i^{j_i}} \cong  (\C^{Z_i})^{\otimes j_i}$.

\section{Appendices}
\subsection{Appendix 1: Two algorithms to compute \texorpdfstring{$p_{\tau,\sigma}$}{p tau,sigma}\label{21de15b}}

For any $n,m\in\Z_{\geq 0}$, $\lambda\vdash n$, and $\tau,\sigma\in T_{\lambda,m}$, recall that
the polynomial $p_{\tau,\sigma}\in\C[x_{j,h}\mid j,h=1,\ldots,m]$ is defined by  
$$
p_{\tau,\sigma}(X):=
\sum_{\substack{\tau'\sim\tau\\ \sigma'\sim\sigma}}\sum_{c,c'\in C_{\lambda}}\sgn(cc')
\prod_{y\in Y(\lambda)}
x_{\tau'c(y),\sigma'c'(y)}  \,\,\,\,\, \text{ (as in~\eqref{ptsdef})},
$$
for $X=(x_{j,h})_{j,h=1}^m\in\C^{m\times m}$. Note that in this way~$p_{\tau,\sigma}(X)$ is described with an exponential number of terms. However:

\begin{proposition}\label{19se15d}
Expressing $p_{\tau,\sigma}(X)$ as a linear combination of monomials can be done in time polynomial in~$n$, for fixed $m$.
\end{proposition}
\proof
First observe that
\begin{align*}
p_{\tau,\sigma}(X)
&=
|C_{\lambda}|\sum_{\substack{\tau'\sim\tau\\\sigma'\sim\sigma}}\sum_{c\in C_{\lambda}}\sgn(c)
\prod_{y\in Y(\lambda)}x_{\tau'(y),\sigma'c(y)}
\\&=
|C_{\lambda}|
\sum_{\substack{\tau'\sim\tau\\\sigma'\sim\sigma}}\prod_{i=1}^{\lambda_1}\det((x_{\tau'(i,j),\sigma'(i,j')})_{j,j'=1}^{\lambda^*_i}).
\end{align*}
($\lambda^*$ is the dual partition of $\lambda$; that is, $\lambda^*_i$ is the height of column $i$.)\symlistsort{lambda star}{$\lambda^*$}{dual partition of $\lambda$}

For fixed $m$, when $n$ grows, there will be several columns of $Y(\lambda)$ that are the same both in $\tau'$
and in $\sigma'$.
More precisely, for given $\tau',\sigma'$ let the `count function' $\kappa$ be defined as follows:\symlistsort{kappa}{$\kappa$}{count function}
for $t\in\Z_{\geq 0}$ and $v,w\in[m]^t$, $\kappa(v,w)$
is the number of columns $i$ of height $t$ such that
$\tau'(i,j)=v_j$ and $\sigma'(i,j)=w_j$ for all $j=1,\ldots,t$.
Furthermore, write for each $j\leq h:=\height(\lambda)$ and each $s\in[m]$: \symlistsort{r(s,j)}{$r(s,j)$}{$\#$ symbols~$s$ in row~$j$ of~$\tau$}\symlistsort{u(s,j)}{$u(s,j)$}{$\#$ symbols~$s$ in row~$j$ of~$\sigma$}
\begin{align} \label{rsiusi}
r(s,j)&:= \text{$\#$ symbols~$s$ in row~$j$ of~$\tau$}  , \notag\\ 
u(s,j)&:= \text{$\#$ symbols~$s$ in row~$j$ of~$\sigma$}. 
\end{align} 
Then for each $j\leq h$ and each $s\in[m]$:
\begin{align}\label{5se15d}
\sum_{t=j}^{h}\sum_{\substack{v,w\in[m]^t\\ v_j=s}}\kappa(v,w) &=r(s,j), \text{ and}  \notag \\
\sum_{t=j}^{h}\sum_{\substack{v,w\in[m]^t\\ w_j=s}}\kappa(v,w) &= u(s,j).
\end{align}
For any given function $\kappa:\bigcup_{j=1}^{h}[m]^j\times[m]^j\to\Z_{\geq 0}$ satisfying \eqref{5se15d}, there are precisely
\begin{align}\label{29de15a}
\prod_{t=1}^h\frac{(\lambda_t-\lambda_{t+1})!}{\prod_{v,w\in[m]^t}\kappa(v,w)!}
\end{align}
pairs $\tau'\sim\tau$ and $\sigma'\sim\sigma$ having count function $\kappa$ (setting $\lambda_{h+1}:=0$).
(Note that \eqref{5se15d} implies $\lambda_t-\lambda_{t+1}=\sum_{v,w\in[m]^t}\kappa(v,w)$, for each $t$,
so that for each $t$, the factor in \eqref{29de15a} is a Newton multinomial coefficient.)
Hence
$$
p_{\tau,\sigma}(X)
=
|C_{\lambda}|
\sum_{\kappa}\prod_{t=1}^h(\lambda_t-\lambda_{t+1})!\prod_{v,w\in[m]^t}
\frac{\det((x_{v_j,w_{j'}})_{j,j'=1}^t)^{\kappa(v,w)}}{\kappa(v,w)!},
$$
where $\kappa$ ranges over functions $\kappa:\bigcup_{t=1}^h([m]^t\times [m]^t)\to\Z_{\geq 0}$ satisfying \eqref{5se15d}.
\endproof

This is the algorithm we gave in~$\cite{onsartikel}$. We now also state a different method (due to Gijswijt~\cite{gijswijt}) which is easy to implement. Define the operators
\begin{align} 
d_{s \to j}:= \sum_{i=1}^m x_{s,i} \frac{\partial}{\partial x_{j,i}}, \,\, \text{ and }  d_{j \to s}^*:= \sum_{i=1}^m x_{i,s} \frac{\partial}{\partial x_{i,j}}.
\end{align} \symlistsort{dsj}{$d_{s \to j}$}{operator} \symlistsort{djs}{$d_{j \to s}^*$}{operator}
Also define the polynomial (here~$\lambda_{m+1}:=0$)\symlistsort{Plambda(X)}{$P_{\lambda}(X)$}{polynomial defined in~(\ref{pdef})}
\begin{align} \label{pdef}
P_{\lambda}(X):= \prod_{k=1}^m \left( k! \, \det\left(  (x_{i,j})_{i,j=1}^{k} \right) \right)^{\lambda_{k}-\lambda_{k+1}},
\end{align}
which is a polynomial in the variables~$x_{i,j}$, where~$i,j=1,\ldots, m$. Then~$P_{\lambda}(X)$ can be computed in time polynomially bounded in~$n$ (for fixed~$m$, note that~$\det (  (x_{i,j})_{i,j=1}^m )$ has~$m!$ terms).   

Recall that~$r(s,j)$ and~$u(s,j)$ are defined in~$\eqref{rsiusi}$ for each $1 \leq j\leq h $ and each $s\in[m]$. We set~$r(s,j):=0$ and~$u(s,j):=0$ if~$h<j\leq m$ and~$s \in [m]$, so that~$r(s,j)$ and~$u(s,j)$ are defined for each~$j \in [m]$ and~$ s \in [m]$.
Now it holds, as is proved in~\cite[Theorem~7]{gijswijt}, that
\begin{align} \label{ultimatep}
    p_{\tau, \sigma }(X) = \left( \prod_{j=1}^{m-1} \prod_{s=j+1}^m   \frac{1}{r(s,j)!\,u(s,j)!}    (d_{s \to j})^{r(s,j)}  (d_{j \to s}^*)^{u(s,j)}\right) \cdot  P_{\lambda}(X).
\end{align}
Expression~$(\ref{ultimatep})$ gives a method to compute~$p_{\tau, \sigma}(X)$ in polynomial time (for fixed~$m$), using only methods for polynomial addition, multiplication and differentiation.


\subsection{Appendix 2: The Delsarte bound\label{delsil}}
To illustrate the method of this chapter, we consider the Delsarte bound  (cf.~McEliece, Rodemich and Rumsey~\cite{thetaprime2} and~Schrijver~\cite{thetaprime}). The analytic definition can be given as a semidefinite programming problem of large size, as a matrix of order~$q^n \times q^n$ is involved. We show how the matrix occurring in the semidefinite program can be reduced. A semidefinite programming description of the Delsarte bound $D_q(n,d)$ is the following:
\begin{align} \label{delsarteanalytical}
D_q(n,d) = \max\big\{ \mbox{$\sum_{u,v \in [q]^n} X_{u,v}$} \,\,  \big| \,\, &  X \in \R^{[q]^n \times [q]^n}_{\geq 0},  \,\,\, \trace(X)=1,  \notag \\ & \text{$X_{u,v}=0$ if~$0<d_H(u,v) < d$},\,\,\, X\succeq 0 \big\}.
\end{align}
If~$C \subseteq [q]^n$ is a code with~$|C|=A_q(n,d)$ and~$d_{\text{min}}(C)\geq d$, then the matrix~$X$ with~$X_{u,v}=|C|^{-1}$ if~$u,v \in C$ and~$X_{u,v}=0$ otherwise is a feasible solution (positive semidefiniteness of this~$X$ follows from~$X=|C|^{-1}\chi_C \chi_C\T$, with~$\chi_C \in \R^{V}$ the $0,1$-vector with~$(\chi_C)_v=1 \Longleftrightarrow v \in C$) with objective value~$|C|$. This shows that~$A_q(n,d) \leq D_q(n,d)$.

The group~$H:= S_q^n \rtimes S_n$ acts on~$[q]^n$ and preserves distances, which means that $d_H(u,v)=d_H(\pi(u),\pi(v))$ for all~$\pi \in H$ and~$u,v \in [q]^n$. If~$X$ is an optimum solution, also~$\pi \cdot X$ is an optimum solution (where~$\pi \cdot X$ denotes the matrix with~$(\pi \cdot X)_{u,v} = X_{\pi(u),\pi(v)}$) as~$\pi \cdot X$ is feasible with the same objective value as~$X$. Since the set of feasible solutions is convex, we may replace any optimum solution~$X$ by~$(1/|H|) \sum_{\pi \in H} \pi \cdot X$,  which is an~$H$-invariant optimum solution. So we may assume that the optimum in~$\eqref{delsarteanalytical}$ is~$H$-invariant. 

In that case, the value~$X_{u,v}$ only depends on the distance~$i$ between~$u$ and~$v$, so we can write~$z(\omega_i)$ for the common value of~$X_{u,v}$ for all~$u,v \in [q]^n$ with~$d_H(u,v)=i$.  With notation as in the previous subsections, set~$Z:=[q]$ and~$G:=S_q$, so that~$\Lambda:= ([q]\times [q])/S_q$.

As explained in~\eqref{natbijection}, there is a natural bijection between~$\Lambda^n/S_n$ and~$(Z \times Z)^n/H \cong (Z^n \times Z^n)/H $. In the present case, an~$H$-orbit of~$Z^n \times Z^n$ is determined by the distance~$i$ between two elements of~$Z^n$. So we can write~$\omega_i$ for the element in~$\Lambda^n/S_n$ that corresponds with the~$H$-orbit of two words at distance~$i$.\symlistsort{omega i}{$\omega_i$}{$H$-orbit of two words at distance~$i$} Then the bound~$D_q(n,d)$ from~\eqref{delsarteanalytical} is equal to
\begin{align} \label{delsartetussenstap}
  \max\big\{\mbox{$\sum_{i=0}^n \binom{n}{i} (q-1)^i q^nz(\omega_i)$} \,\, \big| \,\, &  z: \Lambda^n/S_n \to \R,\,\,  z(\omega_0) =q^{-n}, \,\, z(\omega_1)=\ldots = z(\omega_{d-1})=0,   \notag \\ &z(\omega_i) \geq 0  \text{ for all } d\leq i \leq  n,  \,\, M(z) \succeq 0 \big\},
\end{align}
where~$M(z):=   \sum_{i=0}^n z(\omega_i) K_{\omega_i}$. 

Now we compute~$\Phi(M(z))$ with the method from this chapter. First we find a representative set for the natural action of~$G=S_q$ on~$Z=[q]$. It is well known that~$V:=\C^Z = \C^{[q]}$ decomposes into pairwise orthogonal irreducible representations as~$V_{1,1} \oplus V_{2,1}$ where~$V_{1,1} = \{ \lambda\bm{1} \,\,|\,\, \lambda \in \C\}$ is the trivial representation (here~$\bm{1}$ denotes the all-ones vector in~$\C^{[q]}$) and~$V_{2,1}=V_{1,1}^{\perp}$ is the  $(q-1)$-dimensional standard representation \cite{sagan}. So a representative set for the action of~$G$ on~$V$ is~$\bm{B}=\{B_1,B_2\}$, where~$B_1 \in V_{1,1}$ and~$B_2 \in V_{2,1}$ arbitrary.  We take 
\begin{align}\label{reprsetsqdelsarte}
B_1 := \mbox{$\frac{1}{\sqrt{q}}$}\mathbf{1}, \,\,\,\,\,\, B_2 :=\mbox{$\frac{1}{\sqrt{2}}$} (e_0-e_1),
\end{align}
where~$e_i$ denotes the~$i$th standard basis vector in~$\C^{[q]}$,  for~$i \in [q]$.\symlistsort{ei}{$e_i$}{$i$th standard basis vector}

Now Proposition~\ref{prop2}  gives a representative set~\eqref{matset} for the action of~$H=S_q^n \rtimes S_n$ on~$Z^n=[q]^n$. For convenience of the reader, we restate the main definitions of Chapter~\ref{orbitgroupmon} and the main facts about representative sets  applied to the context of this section in a separate frame --- see Figure~\ref{factsdelsarte}.
\begin{figure}[ht]  
  \fbox{
    \begin{minipage}{15.5cm} 
     \begin{tabular}{l}
{\bf FACTS.}  \\
\\
$G:= S_q$. \\
$Z:= [q] $. \\
\\
$k:=2$, $\bm{m} = (m_1,m_2) =(1,1)$.\\
\\
A representative set for the action of~$G$ on~$Z$ is $\bm{B}=\{B_1,B_2\}$ from~\eqref{reprsetsqdelsarte}. \\
A representative set for the action of~$H$ on~$Z^n$ follows from Proposition~\ref{prop2}.
\\ \\
$\Lambda = ([q] \times [q])/S_q $. \\
$a_P = \sum_{(x,y) \in P} x \otimes y$ for~$P \in \Lambda$.\\
$A:= \{ a_P \, | \, P \in \Lambda\} $, a basis of~$W:= (\C{[q]} \otimes  \C{[q]})^{S_q}$.\\
  $K_{\omega} = \sum_{(P_1,\ldots,P_n) \in \omega}
a_{P_1} \otimes \cdots \otimes a_{P_n}$ for~$\omega \in \Lambda^n/S_n$. 
     \end{tabular}
    \end{minipage}} \caption{\label{factsdelsarte}\small{The main definitions of Chapter~\ref{orbitgroupmon}  and the main facts about representative sets  applied to the context of Section~\ref{delsil}.}}
\end{figure}

With the notation as in Section~\ref{reprmult}, we have in this case that~$\bm{m}=(m_1,m_2)=(1,1)$, and that~$\bm{N}=\{ (n-t,t) \,\, | \,\, t=0,\ldots,n\}$. For each~$(n-t,t) \in \bm{N}$, there is only one~$\bm{\lambda} \vdash (n-t,t)$ with $\height(\lambda_1) \leq m_1=1$ and~$\height(\lambda_2) \leq m_2=1$, namely~$\bm{\lambda}=(\lambda_1,\lambda_2)$ with~$\lambda_1=(n-t)$ and~$\lambda_2=(t)$. Moreover, for this~$\bm{\lambda}$, the set~$\bm{T_{\lambda,m}}=T_{\lambda_1,1} \times T_{\lambda_2,1}$ has only one element~$\bm{\tau}$, namely~$\bm{\tau}=(\tau_1,\tau_2)$ with
\newcommand{\puntjestab}{$\,\, \ldots \,\,$}
$$
\tau_1=\underbrace{\,\young(1) \,\ldots\, \young(1)\,}_{n-t} \,, \quad \tau_2=\underbrace{\,\young(1) \,\ldots\, \young(1)\,}_{t}\,,  
$$
i.e., both~$\tau_1$ and~$\tau_2$ are Young tableaux of height~$1$ which only contain ones.  So the blocks~$U_{\bm{\lambda}}\T M(z) U_{\bm{\lambda}}$ all have size~$1$, and $U_{\bm{\lambda}}\T M(z) U_{\bm{\lambda}} = \bm{u_{\tau,B}}\T M(z)  \bm{u_{\tau,B}}$ for the unique element~$\bm{\tau} \in \bm{T_{\lambda,m}}$ with~$\bm{\lambda} \vdash (n-t,t)$.  By Proposition~\ref{ptsprop} we have for this~${\bm{\tau}}$ that
$$
\sum_{i=0}^n  (\bm{u_{\tau,B}}\T K(\omega_i)  \bm{u_{\tau,B}}) \mu(\omega_i)= \prod_{i=1}^2 p_{\tau_i,\tau_i}(F_i) = ((F_1)_{1,1})^{n-t}((F_2)_{1,1})^{t}. 
$$
Recall that each $B_i \in \C^{[q]}$ is a linear function on $\C{[q]}$, and that each $a_P$ is an element of
$\C{[q]}\otimes\C{[q]}$, where $P \in \Lambda$.
We express each $(F_i)_{1,1} = (B_i\otimes B_i)|_W$ in the dual basis $A^*:=\{a^*_P\mid P\in\Lambda\}$ of~$A$.
The coefficient of $a^*_P$ is obtained by evaluating $(B_i\otimes B_i)(a_P)$. We denote an equivalence class in~$\Lambda=[q]^2/S_q$ by its lexicographically smallest element, representing vectors in~$[q]^2$ as words, i.e., a vector in~$[q]^2$  is represented as a string of symbols in~$[q]$ of length $2$. We find
\begin{align}
(F_1)_{1,1}&=(B_1(1)\otimes B_1(1))|_W= a^*_{00} +(q-1) a^*_{01},\notag\\
(F_2)_{1,1}&=(B_2(1)\otimes B_2(1))|_W= a^*_{00} - a^*_{01},
\end{align} 
as~$(B_1(1)\otimes B_1(1))(a_{00})=1$, $(B_1(1)\otimes B_1(1))(a_{01})=q-1$, $(B_2(1)\otimes B_2(1))(a_{00})=1$ and $(B_2(1)\otimes B_2(1))(a_{01})=-1$.
So we obtain
\begin{align}
((F_1)_{1,1})^{n-t}((F_2)_{1,1})^{t} &= (a^*_{00} + (q-1) a^*_{01})^{n-t} (a^*_{00} - a^*_{01})^{t}\notag\\
&= \sum_{l=0}^{n-t} \binom{n-t}{l} (a^*_{00})^{n-t-l} (q-1)^l (a^*_{01})^l \sum_{j=0}^{t}  \binom{t}{j}(a^*_{00})^{t-j} (-1)^j (a^*_{01})^j\notag\\
&=\sum_{l=0}^{n-t} \sum_{j=0}^t   \binom{t}{j}  \binom{n-t}{l}  (-1)^j(q-1)^l (a^*_{00})^{n-j-l} (a^*_{01})^{j+l}.\notag\\
\intertext{Now we replace each monomial~$\mu(\omega)$ by the variable~$z(\omega)$, so~$ (a^*_{00})^{n-i} (a^*_{01})^{i}$ gets replaced by~$z(\omega_i)$ for each~$0\leq i \leq n$. We obtain: }
U_{\bm{\lambda}}\T M(z) U_{\bm{\lambda}} &=\sum_{l=0}^{n-t} \sum_{j=0}^t   \binom{t}{j}  \binom{n-t}{l}  (-1)^j(q-1)^l z(\omega_{j+l}). \label{zomegauitdrukking}
\end{align}
Put~$a_i:=  \binom{n}{i} (q-1)^{i}q^n z(\omega_i)$ for~$0 \leq i \leq n$. Here we note that~$|\omega_i|=\binom{n}{i} (q-1)^i q^n$, so~$a_i=|\omega_i|z(\omega_i)$. Moreover, note that~$\binom{n}{t}\binom{t}{j} \binom{n-t}{l} = \binom{n}{j+l}\binom{n-j-l}{t-j} \binom{j+l}{j}$ for nonnegative integers~$n,t,j,l$ satisfying $j \leq t \leq n$ and~$l \leq n-t$. So we obtain from~\eqref{zomegauitdrukking} that
\begin{align} 
U_{\bm{\lambda}}\T M(z) U_{\bm{\lambda}} &=\frac{1}{\binom{n}{t}q^n}\sum_{l=0}^{n-t}  \sum_{j=0}^t   \binom{n-j-l}{t-j}  \binom{j+l}{j}  (-1)^j(q-1)^{-j} a_{j+l} \notag\\
&=\frac{1}{\binom{n}{t}(q-1)^t q^n}\sum_{i=0}^{n}\sum_{j=0}^t   \binom{n-i}{t-j}  \binom{i}{j}  (-1)^{j}(q-1)^{t-j} a_{i} \notag\\
&=\frac{1}{|\omega_t|} \sum_{i=0}^{n}K_t(i)a_{i},
\end{align}
where~$K_t$ is the~$t$-th \emph{Krawtchouk polynomial} defined in~\eqref{kraw}.
So~$M(z) \succeq 0$ if and only if $\sum_{i=0}^{n} K_t(i) a_{i} \geq 0$ for each~$t=0,\ldots,n$. This allows us to rewrite~$\eqref{delsartetussenstap}$ as a linear program as follows: 
\begin{align}
D_q(n,d) = \max\big\{ \mbox{$\sum_{i=0}^n a_i$} \,\, \big| \,\, & a_0=1,\,\, a_1 =\ldots=a_{d-1}=0, \,\,\, a_i \geq 0 \text{ if $d \leq i \leq n$}, \,\,\,\,  \notag  
\\&  \mbox{$\sum_{i=0}^{n} K_t(i) a_{i} \geq 0$} \text{ for all $ 0 \leq t \leq n$ }\big\}, 
\end{align}
which coincides with the definition  given in the introduction. Hence we obtain the Delsarte bound as an application of the method of this chapter. 

Note that the fact that we obtain a \emph{linear} program stems from the fact that the $G$-module~$\C^Z$ decomposes into irreducible modules each of multiplicity~$\leq 1$. In this case the matrices~$U_{\bm{\lambda}}\T M(z) U_{\bm{\lambda}}$ are of order~$1 \times 1$.

\chapter{Semidefinite programming bounds for unrestricted codes} \label{onsartchap}
\chapquote{The purpose of computation is insight, not numbers.}{Richard Hamming (1915--1998)}

\noindent For~$q,n,d \in \N$, let~$A_q(n,d)$ denote the maximum cardinality of a~$q$-ary code of word length~$n$ and minimum distance at least~$d$. We will consider a semidefinite programming upper bound on~$A_q(n,d)$ based on quadruples of codewords. By the symmetry of the problem, we can apply representation theory to reduce the problem to a semidefinite programming problem with order bounded by a polynomial in~$n$. The method yields the new upper bounds~$A_4(6,4)\leq 176$, $A_4(7,3) \leq 596$, $A_4(7,4) \leq 155$, $A_5(7,4) \leq 489$ and~$A_5(7,5) \leq 87$.  This chapter is based on joint work with Bart Litjens and Lex Schrijver~\cite{onsartikel} and relies on the method explained in Chapter~\ref{orbitgroupmon}.

\section{Introduction}

We will assume throughout that $q\geq 2$. For~$q=2$, semidefinite programming bounds based on quadruples of codewords have been studied by Gijswijt, Mittelmann and Schrijver~\cite{semidef}. While this chapter is mainly meant to handle the case $q\geq 3$, the results also hold for~$q=2$.  We will study the following upper bound on $A_q(n,d)$, sharpening
Delsarte's classical linear programming bound \cite{delsarte}.

For $k\in\Z_{\geq 0}$,
let $\CC_k$ be the collection of subsets $C$ of $[q]^n$ with $|C|\leq k$.\symlistsort{Ck}{$\mathcal{C}_k$}{collection of codes of cardinality at most~$k$}
For each $x:\CC_4\to\R$ define the $\CC_2\times\CC_2$ matrix $M(x)$ by\symlistsort{M(x)}{$M(x)$}{variable matrix}

\begin{align}
M(x)_{C,C'}:=x(C\cup C')
\end{align}
for $C,C'\in \CC_2$.
Then define\symlistsort{Bq(n,d)}{$B_q(n,d)$}{upper bound on~$A_q(n,d)$}

\begin{align} \label{8no15a}
B_q(n,d) :=  \max \big\{ \mbox{$\sum_{v \in [q]^n} x(\{v\})$}\,\, |\,\,&x:\mathcal{C}_4 \to \R_{\geq 0}, \,\, x(\emptyset )=1,  \,\,\,\,x(S)=0 \text{ if~$d_{\text{min}}(S)<d$},  \notag\\ 
& M(x)\succeq 0 \big\}. 
\end{align} 

\begin{proposition}
$A_q(n,d)\leq B_q(n,d)$.
\end{proposition}
\proof
Let $C\subseteq[q]^n$ have minimum distance at least $d$ and satisfy $|C|=A_q(n,d)$.
Define $x:\CC_4\to\R$ by $x(S)=1$ if $S\subseteq C$ and $x(S)=0$ otherwise, for~$S \in \CC_4$.
Then $x$ satisfies the conditions in~\eqref{8no15a}: the condition that~$M(x) \succeq 0$ follows from the fact that for this $x$ one has $M(x)_{S,S'}=x(S)x(S')$
for all $S,S'\in\CC_2$. Moreover, $\sum_{v\in[q]^n}x(\{v\})=|C|=A_q(n,d)$.
\endproof

The optimization problem \eqref{8no15a} is huge, but, with the representation theory of Chapters~\ref{prem} and~\ref{orbitgroupmon}, 
 it can be reduced to a size bounded by a polynomial in $n$, with entries (i.e.,
coefficients) being polynomials in $q$.

To explain the reduction, let $H$ be the wreath product $S_q^n\rtimes S_n$.
For each $k$, the group $H$ acts naturally on $\CC_k$, maintaining minimum distances and cardinalities
of elements of $\CC_k$ (being codes).
Then we can assume that $x$ in~\eqref{8no15a} is invariant under the $H$-action on $\CC_4$.
That is, we can assume that $x(C)=x(D)$ whenever $C,D\in\CC_2$ and $D=g \cdot C$ for some $g\in H$.
Indeed, the conditions in \eqref{8no15a} are maintained under replacing $x$ by $g\cdot x$.
(Note that $M(g\cdot x)$ is obtained from $M(x)$ by simultaneously permuting rows and columns.)
Moreover, the objective function does not change by this action.
Hence the optimum $x$ can be replaced by the average of all $g\cdot x$ (over all $g\in  H$),
by the convexity of the set of positive semidefinite matrices.
This makes the optimum solution $H$-invariant.

Let $\Omega_4$ be the set of $H$-orbits on $\CC_4$.\symlistsort{Omegak}{$\Omega_k$}{set of $H$-orbits on $\mathcal{C}_k$}
Note that $\Omega_4$ is bounded by a polynomial in $n$ (independently of $q$).
As there exists an $H$-invariant optimum solution, we can replace, for each $\omega\in\Omega_4$ and $C\in\omega$, each variable $x(C)$ by a variable $z(\omega)$.\symlistsort{z(omega)}{$z(\omega)$}{variable of reduced program}
In this way we obtain $M(z)$.\symlistsort{M(z)}{$M(z)$}{reduced variable matrix}

Then $M(z)$ is invariant under the simultaneous action of $H$ on the set~$\CC_2$ of its rows and columns.
Hence $M(z)$ can be block-diagonalized (as explained in Sect.~\ref{symint}) by $M(z)\mapsto U\T M(z) U$, where $U$ is a
matrix independent of $z$, such that the order of $ U\T M(z) U$ is polynomial in~$n$ and such that the original matrix $M(z)$ is positive semidefinite if and only if
each of the blocks is positive semidefinite.
The entries in each block are linear functions of the variables $z(\omega)$.

In this chapter we will describe the blocks that reduce the problem. We make use of the framework developed in Chapter~\ref{orbitgroupmon}.
With the reduced semidefinite program, we found the following   improvements on the known bounds for $A_q(n,d)$  (see Table~\ref{table41}), with thanks to Hans D.\ Mittelmann
for his help in solving the larger-sized programs.

\begin{table}[H]
\begin{center}
   \begin{tabular}{| r | r | r|| r| >{\bfseries}r| r| r|}
    \hline
      $q $ & $n$ & $d$  &   \multicolumn{1}{>{\raggedright\arraybackslash}b{15mm}|}{best lower bound known}    &\multicolumn{1}{>{\raggedright\arraybackslash}b{15mm}|}{\textbf{new upper bound}}  & \multicolumn{1}{>{\raggedright\arraybackslash}b{18.5mm}|}{best upper bound previously known} & \multicolumn{1}{>{\raggedright\arraybackslash}b{15mm}|}{Delsarte bound} \\  \hline 
4&6&3&164&{\bf 176}&179 & 179 \\
4&7&3&512&{\bf 596}&614&  614 \\
4&7&4&128&{\bf 155}&169& 179 \\ \hline
5&7&4&250&{\bf 489}&545& 625 \\
5&7&5&\hspace*{5.5pt}53&\hspace*{5.5pt}{\bf 87}&108& 125 \\
    \hline 
    \end{tabular}
\end{center}
\caption{\small The new upper bounds on~$A_q(n,d)$.}\label{table41}
\end{table}
The best upper bounds previously known for $A_4(6,3)$ and $A_4(7,3)$ are Delsarte's linear programming bound \cite{delsarte};
the other three best upper bounds previously known were given by Gijswijt, Schrijver, and Tanaka \cite{tanaka} (semidefinite programming bounds based on triples).
We refer to the most invaluable tables maintained by Andries Brouwer \cite{brouwertableq} with the best
known lower and upper bounds for the size of error-correcting codes (see also
Bogdanova, Brouwer, Kapralov, and \"Osterg{\aa}rd \cite{4ary} and Bogdanova and \"Osterg{\aa}rd \cite{5ary} for studies of bounds for codes over alphabets of
size $q=4$ and $q=5$, respectively).

\subsection{Comparison with earlier bounds}

The bound $B_q(n,d)$ described above is a sharpening of Delsarte's classical linear programming bound~\cite{delsarte}.
The value of the Delsarte bound is equal to our bound after replacing $\CC_4$ and $\CC_2$ by
$\CC_2$ and $\CC_1$, respectively, which generally yields
a less strict bound.

We can add to \eqref{8no15a} the condition that, for each $D\in\CC_4$, the $\CC_4(D)\times \CC_4(D)$ matrix
\begin{align}\label{20de15a}
\text{$(x(C\cup C'))_{C,C'\in \CC_4(D)}$ is positive semidefinite,}
\end{align}
where $\CC_4(D):=\{C\in\CC_4\mid C\supseteq D, |D|+2|C\setminus D|\leq 4\}$.
(So $M(x)$ in \eqref{8no15a} is the case $D=\emptyset$.)
Also the addition of \eqref{20de15a} allows a reduction of the optimization problem to polynomial
size as above.
(It can be seen that adding \eqref{20de15a} for $|D|=2$ suffices.)
For $q=2$ we obtain in this way the bound given by Gijswijt, Mittelmann and Schrijver \cite{semidef}.

A bound intermediate to the Delsarte bound and the currently investigated bound is based on
considering functions $x:\CC_3\to\R_{\geq 0}$ and the related matrices --- see 
Schrijver \cite{schrijver} for binary codes and
Gijswijt, Schrijver, and Tanaka \cite{tanaka} for nonbinary codes.

\section{Reduction of the optimization problem\label{16de15i}}

In this section we describe reducing the optimization problem \eqref{8no15a} conceptually.
In Section \ref{19de15b} we consider this reduction computationally.
For the remainder of this chapter we fix $n$ and $q$.

We consider the natural action of $H=S_q^n\rtimes S_n$ on ${\CC_2}$.
If $U_1,\ldots,U_k$ form a representative set of matrices for this action, then with \eqref{PhiC2}
we obtain a reduction of the size of the optimization problem to polynomial size.
To make this reduction explicit in order to apply semidefinite programming, we need to express each $m_i\times m_i$ matrix $U_i\T M(z)U_i$ as an explicit
matrix in which each entry is a linear combination of the variables $z(\omega)$ for $\omega\in\Omega_4$ (the set of $H$-orbits of $\CC_4$).

For $\omega\in\Omega_4$, let $N_{\omega}'$ be the $\CC_2\times\CC_2$ matrix with $0,1$ entries\symlistsort{Nomega'}{$N_{\omega}'$}{$\CC_2\times\CC_2$ matrix}
satisfying
\begin{align}
\text{$(N_{\omega}')_{\{\alpha,\beta\},\{\gamma,\delta\}}=1$ if and only if $\{\alpha,\beta,\gamma,\delta\}\in\omega$} \label{NomegaX}
\end{align}
for $\alpha,\beta,\gamma,\delta\in[q]^n$.
Then
$$
U_i\T M(z)U_i=\sum_{\omega}z(\omega)U_i\T N_{\omega}' U_i.
$$
So to get the reduction,
we need to obtain the matrices $U_i\T N_{\omega}' U_i$ explicitly,
for each $\omega\in\Omega_4$ and for each $i=1,\ldots,k$.
We do this in a number of steps, using the main symmetry reduction in Chapter~\ref{orbitgroupmon}.

We first describe in Section \ref{16de15gx} a representative set for the natural
action of $S_q$ on ${[q]\times [q]}$.
Proposition~\ref{prop2} (with~$V:=\C^Z$, where~$Z:=[q]\times [q]$, and~$G:=S_q$) then yields a representative set for the
action of the wreath product $H=S_q^n\rtimes S_n$ on the set
$([q]^n)^2$ of {\em ordered} pairs of words in $[q]^n$, in other words,
on $\C^{([q]^n)^2}\cong(\C^{[q]\times [q]})^{\otimes n}$.
From this we derive in Section \ref{28de15b} a
representative set for the action of $H$ on the set $\CC_2\setminus\{\emptyset\}$
of {\em unordered} pairs $\{v,w\}$ (including singleton) of words $v,w$
in $[q]^n$.
Then in Section \ref{28de15c} we derive a representative
set for the action of $H$ on the set $\CC_2^d\setminus\{\emptyset\}$,
where $\CC_2^d$ is the set of codes in $\CC_2$ of minimum distance at least $d$.
(So each singleton word belongs to $\CC_2^d$.)
Finally, in Section \ref{28de15c} we include the
empty set $\emptyset$, by an easy representation-theoretic argument.


\subsection{A representative set for the action of of~\texorpdfstring{$S_q$}{Sq} on~\texorpdfstring{${[q] \times [q]}$ }{[q]x[q]}\label{16de15gx}}

We now consider the natural action of $S_q$ on $\C^{[q]\times [q]}$.
Let $e_j$ be the $j$-th unit basis vector in $\C^q$,
$I_q$ be the $q\times q$ identity matrix, $J_q$ be the all-one
$q\times q$ matrix, $\bm{1}$ be the all-one column vector in $\C^q$,
$N:=(e_0-e_1)\bm{1}\T$, and $E_{i,j}:=e_ie_j\T$.\symlistsort{N}{$N$}{matrix $(e_0-e_1)\bm{1}\T$}\symlistsort{Eij}{$E_{i,j}$}{$e_ie_j\T$} 
We furthermore define the following matrices, where we consider matrices
in $\C^{q\times q}$ as {\em columns} of the matrices $B_i$:
\begin{align}\label{9no15a}
B_1&:=[I_q,J_q-I_q], \notag \\
B_2&:=[E_{0,0}-E_{1,1},
N-N\T, N+N\T-2(E_{0,0}-E_{1,1})], \notag \\
B_3&:=
[E_{0,1}+E_{1,2}+E_{2,0}-E_{1,0}-E_{2,1}-E_{0,2}], \notag \\
B_4&:=
[E_{0,2}-E_{2,1}+E_{1,3}-E_{3,0}+E_{2,0}-E_{1,2}+E_{3,1}-E_{0,3}].
\end{align}
The matrices in $\C^{q\times q}$ are elements of the dual space
$(\C ([q]\times [q]))^*$, so they are elements of the algebra
$\mathcal{O}(\C([q]\times [q]))$ of polynomials on the linear space $\C([q]\times [q])$.

\begin{proposition} \label{reprsettweetalsq}
The matrix set $\bm{B}:=\{B_1,\ldots,B_4\}$ is representative for the natural action of
$S_q$ on $\C^{[q]\times [q]}$, if $q\geq 4$.
If $q\leq 3$, we delete $B_4$, and if $q=2$ we moreover delete $B_3$ and the last column of $B_2$
(as this column is $0$ if $q=2$).
\end{proposition}
\proof
For $a\in \C^q$, let $\Delta_a$ be the $q\times q$ diagonal matrix with diagonal $a$.\symlistsort{Deltaa}{$\Delta_a$}{diagonal matrix with diagonal $a$} 
Define
\begin{align*}
V_{1,1}&:=\{\lambda I_q\mid \lambda\in\C\} , \\ 
V_{1,2}&:=\{\lambda(J_q-I_q)\mid \lambda\in\C\}, \\
V_{2,1}&:=\{\Delta_a\mid a\in\C^q, a\T\bm{1}=0\},\\
V_{2,2}&:=\{a\bm{1}\T-\bm{1}a\T\mid a\in\C^q, a\T\bm{1}=0\},\\
V_{2,3}&:=\{a\bm{1}\T+\bm{1}a\T-2\Delta_a\mid a\in\C^q, a\T\bm{1}=0\},\\
V_{3,1}&:=\{X\in\C^{q\times q}\mid X\text{ skew-symmetric}, X\bm{1}=0\},\\
V_{4,1}&:=\{X\in\C^{q\times q}\mid X\text{ symmetric}, X\bm{1}=0, X_{i,i}=0\text{ for all }i\in[q]\}.
\end{align*}

Observe that each $V_{i,j}$ is $S_q$-stable, and that $V_{i,j}$ and $V_{i',j'}$ are orthogonal whenever
$(i,j)\neq(i',j')$ (with respect to the inner product $X,Y\mapsto \tr(Y^*X)$).
Moreover $\lambda I_q\mapsto\lambda(J_q-I_q)$ gives an $S_q$-isomorphism $V_{1,1}\to V_{1,2}$,
$\Delta_a\mapsto a\bm{1}\T-\bm{1}a\T$ gives an $S_q$-isomorphism $V_{2,1}\to V_{2,2}$, and
$\Delta_a\mapsto a\bm{1}\T+\bm{1}a\T-2\Delta_a$ gives an $S_q$-isomorphism $V_{2,1}\to V_{2,3}$.

Let $q\geq 4$.
Then $\dim(V_{i,j})>0$ for all $i,j$.
Set, as before, $m_1=2$, $m_2=3$, $m_3=m_4=1$.
Then $\sum_{i=1}^4m_i^2=15$, which is equal to the number of partitions of $\{1,2,3,4\}$, hence to the
dimension of $(\C^{[q]\times [q]}\otimes \C^{[q]\times [q]})^{S_q}$.
This implies that the $V_{i,j}$ in fact form an orthogonal {\em decomposition} of $\C^{[q]\times [q]}$
into {\em irreducible} representations and that $V_{i,j}$ and $V_{i',j'}$ are equivalent representations
if and {\em only if} $i=i'$ (as any further representation,
or decomposition, or equivalence would yield that the sum of the squares of the multiplicities of the
irreducible representations is strictly larger than 15,
contradicting the fact that $\Phi$ in \eqref{PhiC2} is bijective).

Now $B_{1,1}$ and $B_{1,2}$ are the elements of $V_{1,1}$ and $V_{1,2}$ with $\lambda=1$.
Moreover, $B_{2,1}$, $B_{2,2}$, and $B_{2,3}$ are the elements of $V_{2,1}$, $V_{2,2}$, and $V_{2,3}$
with $a=e_0-e_1$.
Finally, $B_{3,1}$ and $B_{4,1}$ are (arbitrary) nonzero elements of $V_{3,1}$ and $V_{4,1}$.
This implies that $\{B_1,\ldots,B_4\}$ is a representative matrix set.

If $q=3$, then $\dim(V_{4,1})=0$, while the dimension of $(\C^{[3]\times [3]}\otimes\C^{[3]\times [3]})^{S_3}$ is
equal to the number of partitions of $\{1,2,3,4\}$ into at most 3 classes, which is $2^2+3^2+1^2=14$.
If $q=2$, then moreover $\dim(V_{2,3})=\dim(V_{3,1})=0$, while the dimension of $(\C^{[2]\times [2]}\otimes\C^{[2]\times [2]})^{S_2}$ is
equal to the number of partitions of $\{1,2,3,4\}$ into at most 2 classes, which is $2^2+2^2=8$.
Similarly as above,
this implies that also for $q\leq 3$, $B_1,\ldots,B_k$ form a representative matrix set.
\endproof
Note that the above representative set is real.

If $q\geq 4$, set $k:=4$ and $\bm{m}:=(m_1,\ldots,m_k) := (2,3,1,1)$.
If $q=3$, set $k:=3$ and $\bm{m}:=(m_1,\ldots,m_k) := (2,3,1)$.
If $q=2$, set $k:=2$ and $\bm{m}:=(m_1,m_2):=(2,2)$.
For the remainder of this chapter we fix $k$, $\bm{m}=(m_1,\ldots,m_k)$, and $\bm{B}=\{B_1,\ldots,B_k\}$.

\subsection{A representative set for the action of \texorpdfstring{$H$}{H} on~\texorpdfstring{$([q] \times [q])^{n}$ }{([q]x[q]) power n}\label{28de15a}}

Recall that $H=S_q^n\rtimes S_n$ and that we have fixed $k$, $\bm{m}=(m_1,\ldots,m_k)$, and $\bm{B}=\{B_1,\ldots,B_k\}$
in Section \ref{16de15gx}. We use the notation and definitions of Chapter~\ref{orbitgroupmon}. So ${\bm{N}}$ denotes the collection of all $k$-tuples $(n_1,\ldots,n_k)$ of nonnegative integers adding up
to $n$. Moreover, for $\bm{n}=(n_1,\ldots,n_k)\in{\bm N}$, by
$\bm{\lambda\vdash n}$ it is meant that $\bm{\lambda}=(\lambda_1,\ldots,\lambda_k)$ with
$\lambda_i\vdash n_i$ for $i=1,\ldots,k$.
(So each $\lambda_i$ is equal to a partition $(\lambda_{i,1},\ldots,\lambda_{i,h_i})$ of~$n_i$, for some $h_i$.)

Now Proposition~\ref{prop2} (with~$V:=\C^Z$, where~$Z:=[q]\times [q]$, and~$G:=S_q$) gives a representative set for the action of~$H$ on~$([q] \times [q])^n$: it is the set 
$$
\{~~
[\bm{u_{\tau,B}}
\mid
\bm{\tau}\in \bm{T_{\lambda,m}}]
~~\mid
\bm{n}\in\bm{N},\bm{\lambda\vdash n}
\}
$$
where~$\bm{T_{\lambda,m}}$ and~$\bm{u_{\tau,B}}$ are defined in~$\eqref{wlambda}$ and~$\eqref{vtau}$  respectively. 

For convenience of the reader, we restate the main definitions of Chapter~\ref{orbitgroupmon}  and the main facts about representative sets  applied to the context of this chapter --- see Figure~\ref{factsnonbinaryleeg}.

\begin{figure}[ht] 
  \fbox{
    \begin{minipage}{15.5cm}
     \begin{tabular}{l}
{\bf FACTS.}  \\
\\
$G:= S_q$. \\
$Z:= [q] \times [q]$. \\
\\
If $q\geq 4$, set~$k:=4$ and~$\bm{m}:=(m_1,\ldots,m_k):=(2,3,1,1)$.\\
If $q=3$, set $k:=3$ and~$\bm{m}:=(m_1,\ldots,m_k):=(2,3,1)$.\\
If $q=2$, set $k:=2$ and~$\bm{m}:=(m_1,m_2):=(2,2)$.\\
\\
A representative set for the action of~$G$ on~$Z$ is $\bm{B}\hspace{-.3pt}=\hspace{-.3pt}\{\hspace{-.5pt}B_1,\hspace{-.4pt}\ldots,\hspace{-.4pt}B_k\hspace{-.5pt}\}$ from Prop.~\ref{reprsettweetalsq}. \\
A representative set for the action of~$H$ on~$Z^n$ follows from Proposition~\ref{prop2}.
\\ \\
$\Lambda = ([q]^2 \times [q]^2)/S_q \cong [q]^4/S_q$. \\
$a_P = \sum_{ (x,y) \in P} x \otimes y$ for~$P \in \Lambda$.\\
$A:= \{ a_P \, | \, P \in \Lambda\} $, a basis of~$W:= ({\C[q]^2 \otimes \C[q]^2})^{S_q}$.  \\
  $K_{\omega'} = \sum_{(P_1,\ldots,P_n) \in \omega'}
a_{P_1} \otimes \cdots \otimes a_{P_n}$ for~$\omega' \in \Lambda^n/S_n$. 
     \end{tabular}
    \end{minipage}  
  }  \caption{\label{factsnonbinaryleeg}\small{The main definitions of Chapter~\ref{orbitgroupmon}  and the main facts about representative sets  applied to the context of Chapter~\ref{onsartchap}.}}
\end{figure}

\subsection{Unordered pairs\label{28de15b}}

We now go over from the set $([q]^n)^2$ of ordered pairs of code words to the set $\CC_2\setminus\{\emptyset\}$ of
unordered pairs (including singletons) of code words.
For this we consider the action of the group $S_2$ on
$\C^{[q]^n\times[q]^n}\cong\C^{([q]^n)^2}\cong(\C^{[q]\times [q]})^{\otimes n}$,
where the nonidentity element $\sigma$ in $S_2$ acts as taking the transpose.\symlistsort{sigma}{$\sigma$}{nonidentity element in $S_2$ (in Sect.~\ref{28de15b})}
The actions of $S_2$ and $H$ commute.

Let $L$ be the $(\CC_2\setminus\{\emptyset\})\times ([q]^n)^2$ matrix with $0,1$ entries \symlistsort{L}{$L$}{$(\mathcal{C}_2\setminus\{\emptyset\})\times ([q]^n)^2$ matrix}
satisfying
\begin{align*}
\text{$L_{\{\alpha,\beta\},(\gamma,\delta)}=1$ if and only if $\{\gamma,\delta\}=\{\alpha,\beta\}$,}
\end{align*}
for $\alpha,\beta,\gamma,\delta\in[q]^n$.
Then the function $x\mapsto Lx$ is an $H$-isomorphism
$(\C^{([q]^n)^2})^{S_2}\to\C^{\CC_2\setminus\{\emptyset\}}$.

Now note that each $B_i(j)$, as matrix in $\C^{q\times q}$, is $S_2$-invariant (i.e.,\ symmetric)
except for $B_2(2)$ and $B_3(1)$,
while $\sigma\cdot B_2(2)=-B_2(2)$ and $\sigma\cdot B_3(1)=-B_3(1)$
(as $B_2(2)$ and $B_3(1)$ are skew-symmetric).
So for any
${\bm{n}}\in\bm{N}$, $\bm{\lambda\vdash n}$, and $\bm{\tau}\in \bm{T_{\lambda,m}}$,
we have
$$
\sigma\cdot \bm{u_{\tau,B}}
=
(-1)^{|\tau_2^{-1}(2)|+|\tau_3^{-1}(1)|}
\bm{u_{\tau,B}}.
$$
Therefore, let $\bm{T'_{\lambda,m}}$ be the set of those $\bm{\tau}\in \bm{T_{\lambda,m}}$
with
$|\tau_2^{-1}(2)|+|\tau_3^{-1}(1)|$ even.\symlistsort{Tlambda,mbold'}{$\bm{T_{\lambda,m}'}$}{subset of $\bm{T_{\lambda,m}}$}
Then the matrix set
\begin{align}
\{~~
[ \bm{u_{\tau,B}} \mid \bm{\tau}\in \bm{T'_{\lambda,m}} ]
~~\mid~~
\bm{n}\in\bm{N},\bm{\lambda\vdash n}\}
\end{align}
is representative for the action of $H$ on $(\C^{([q]^n)^2})^{S_2}$.
Hence the matrix set
\begin{align}\label{21de15a}
\{~~
[L\bm{u_{\tau,B}} \mid \bm{\tau}\in \bm{T'_{\lambda,m}} ]
~~{\mid}~~
\bm{n}\in\bm{N}, \bm{\lambda\vdash n}\}
\end{align}
is representative for the action of $H$ on ${\CC_2\setminus\{\emptyset\}}$.

\subsection{Restriction to pairs of words at distance at least~\texorpdfstring{$d$}{d}\label{28de15c}}

Let $d\in\Z_{\geq 0}$, and let
$\CC^d_2$ be the collection of elements of $\CC_2$ of minimum distance at least $d$.\symlistsort{Ckd}{$\mathcal{C}_k^d$}{collection of elements of~$\mathcal{C}_k$ of minimum distance~$\geq d$}
Note that each singleton code word belongs to $\CC_2^d$, and that $H$ acts on $\CC^d_2$.
From \eqref{21de15a} we derive a representative set for the action of $H$ on ${\CC_2^d\setminus\{\emptyset\}}$.

To see this, let for each $t\in\Z_{\geq 0}$, $L_t$ be the subspace of $\C^{\CC_2}$ spanned by the elements $e_{\{\alpha,\beta\}}$ with\symlistsort{Lt}{$L_t$}{subspace of~$\mathbb{C}^{\mathcal{C}_2}$}
$\alpha,\beta\in[q]^n$ and $d_H(\alpha,\beta)=t$.
(For any $C\in\CC_2$, $e_C$ denotes the unit base vector in $\C^{\CC_2^d}$ for coordinate $C$.)\symlistsort{eC}{$e_C$}{unit basis vector corresponding to code~$C$}

Then for any
${\bm{n}\in\bm{N}}$, ${\bm{\lambda\vdash n}}$, and $\bm{\tau}\in  \bm{T'_{\lambda,m}}$,
the irreducible representation $H\cdot L\bm{u_{\tau,B}}$ is contained in $L_t$, where
$$
t:=n-|\tau_1^{-1}(1)|-|\tau_2^{-1}(1)|,
$$
since $B_1(1)=I_q$ and $B_2(1)=E_{0,0}-E_{1,1}$ are the only two entries $B_i(j)$ in the $B_i$ that have nonzeros on the diagonal
of the matrix $B_i(j)$.
Let $\bm{T''_{\lambda,m}}$ be the set of those $\bm{\tau}$ in $\bm{T'_{\lambda,m}}$
with\symlistsort{Tlambda,mbold''}{$\bm{T_{\lambda,m}''}$}{subset of $\bm{T_{\lambda,m}'}$}
$$
n-|\tau_1^{-1}(1)|-|\tau_2^{-1}(1)|\in \{0,d,d+1,\ldots,n\}.
$$
Then a representative set for the action of $H$ on $ {\CC^d_2\setminus\{\emptyset\}}$ is 
\begin{align}
\big\{~~
[L\bm{u_{\tau,B}}
\mid
\bm{\tau}\in \bm{T''_{\lambda,m}}
]
~~\mid~~
\bm{n}\in\bm{N}, \bm{\lambda\vdash n}
\big\}.
\end{align}

\subsection{Adding  \texorpdfstring{$\emptyset$}{emptyset} \label{28de15d}}

To obtain a representative set for the action of $H$ on ${\CC^d_2}$, note that $H$ acts trivially
on $\emptyset$.
So ${\emptyset}$ belongs to the $H$-isotypical component of $\C^{\CC_2}$ that consists
of $H$-invariant elements.
Now the $H$-isotypical component of $\C^{\CC_2\setminus\{\emptyset\}}$ that consists of
the $H$-invariant elements corresponds to the matrix in the representative set indexed by
indexed by $\bm{n}=(n,0,0,0)$ and $\bm{\lambda}=((n),(),(),())$, where $()\vdash 0$.
So to obtain a representative set for $\C^{\CC_2}$, we just add the vector $e_{\emptyset}$ as column to this matrix.

\section{How to compute \texorpdfstring{$(L\bm{u_{\tau,B}})\T N_{\omega}'L\bm{u_{\sigma,B}}$}{L(boldutauB)transpose Nomega' L (boldusigmaB)}\label{19de15b}}

We currently have a reduction of the original problem to matrix blocks with coefficients
$(L\bm{u_{\tau,B}})\T N_{\omega}' L\bm{u_{\sigma,B}}$  (where~$N_{\omega}' $ is as in~\eqref{NomegaX}),
for $\bm{n}\in\bm{N}$, $\bm{\lambda\vdash n}$, $\bm{\tau},\bm{\sigma}\in \bm{T_{\lambda,m}}$,
and $\omega\in\Omega_4$.
The number and orders of these blocks are bounded by a polynomial in $n$, but computing 
these coefficients still must be reduced in time, since the orders of
$L$, $\bm{u_{\tau,B}}$, $\bm{u_{\sigma,B}}$, and $N_{\omega}'$ are exponential in $n$.

Fix $\bm{n}\in\bm{N}$, $\bm{\lambda\vdash n}$, and $\bm{\tau},\bm{\sigma}\in \bm{T_{\lambda,m}}$.
For any $\omega\in\Omega_4$, let $N_{\omega}:=L\T N_{\omega}' L$.\symlistsort{Nomega}{$N_{\omega}$}{$([q]^n\times [q]^n)\times([q]^n\times [q]^n)$
matrix}
So $N_{\omega}$ is a $([q]^n\times [q]^n)\times([q]^n\times [q]^n)$
matrix with 0,1 entries satisfying
\begin{align*}
\text{$(N_{\omega})_{(\alpha,\beta),(\gamma,\delta)}=1$ if and only if $\{\alpha,\beta,\gamma,\delta\}\in\omega$,}
\end{align*}
for all $\alpha,\beta,\gamma,\delta\in[q]^n$.
By definition of $N_{\omega}$,
$$
(L\bm{u_{\tau,B}})\T N_{\omega}' L\bm{u_{\sigma,B}}
=
 \bm{u_{\tau,B}}\T N_{\omega}\bm{u_{\sigma,B}}.
$$
So it suffices to evaluate the latter value.

Recall that~$\Lambda=([q]^4)/S_q$. By~$\eqref{natbijection}$, there is a natural bijection between~$\Lambda^n/S_n$ and the set of~$H$-orbits on~$([q]^n)^4\cong ([q]^4)^n$. The function $([q]^n)^4\to\CC_4$ with $(\alpha_1,\ldots,\alpha_4)\mapsto\{\alpha_1,\ldots,\alpha_4\}$
then gives a surjective function $r:\Lambda^n/S_n\to\Omega_4\setminus\{\{\emptyset\}\}$.\symlistsort{r}{$r$}{surjective function}

\begin{lemma}\label{Lsomlemma}
For each $\omega\in\Omega_4$:
$$
N_{\omega}=\sum_{\substack{\omega' \in \Lambda^n/S_n \\ r(\omega')=\omega}}K_{\omega'}.
 $$
\end{lemma}
\proof
Choose $\alpha,\beta,\gamma,\delta\in [q]^n$.
Then
\begin{align*}
\sum_{\substack{\omega' \in \Lambda^n/S_n\\r(\omega')=\omega}}
(K_{\omega'})_{(\alpha,\beta),(\gamma,\delta)}
&=
\sum_{\substack{\omega' \in \Lambda^n/S_n\\r(\omega')=\omega}}
\sum_{\substack{P_1,\ldots,P_n\in\Lambda\\ (P_1,\ldots,P_n) \in \omega' }}
\left(\bigotimes_{i=1}^n a_{P_i}\right)_{\alpha,\beta,\gamma,\delta}
\\&=
\sum_{\substack{\omega' \in \Lambda^n/S_n\\r(\omega')=\omega}}
\sum_{\substack{P_1,\ldots,P_n\in\Lambda\\ (P_1,\ldots,P_n) \in \omega' }}
\prod_{i=1}^n(a_{P_i})_{\alpha_i,\beta_i,\gamma_i,\delta_i}.
\end{align*}
The latter value is 1 if
$r(\omega')=\omega$ where~$\omega'$ is the~$S_n$-orbit of~$(\pi(\alpha_1\beta_1\gamma_1\delta_1),\ldots, \pi(\alpha_n\beta_n\gamma_n\delta_n)) $ (here~$\pi(\alpha_i\beta_i\gamma_i\delta_i)$ denotes the equivalence class in~$\Lambda$ containing~$(\alpha_i,\beta_i,\gamma_i,\delta_i)$),
and is 0 otherwise.
So it is equal to $(N_{\omega})_{(\alpha,\beta),(\gamma,\delta)}$.
\endproof

By this lemma, it suffices to compute $\bm{u_{\tau,B}}\T K_{\omega'}\bm{u_{\sigma,B}}$ for each
$\omega' \in \Lambda^n/S_n$. By Proposition~\ref{ptsprop}, we have~$\sum_{\omega' \in \Lambda^n/S_n} \left( \bm{u_{\tau,B}}\T K_{\omega'}\bm{u_{\sigma,B}}\right)\mu(\omega') = \prod_{i=1}^k p_{\tau_i,\sigma_i}(F_i)$, where the matrices~$F_i \in (W^*)^{m_i \times m_i}$ are defined in~$\eqref{fformula}$ and the polynomial~$p_{\tau_i,\sigma_i}$ is defined in~\eqref{ptsdef}. So it suffices to calculate the matrices~$F_i$, that is, to  express each $(B_i(j)\otimes B_i(h))|_W$
as linear function into the dual basis $A^*$ of~$A$. So we must
calculate the numbers $(B_i(j)\otimes B_i(h))(a_P)$ for all $i=1,\ldots,k$,
$j,h=1,\ldots,m_i$, and $P\in\Lambda$
 --- see Appendix 1 (Section \ref{19de15a} below).
 
  Now one computes the entry $\hspace{-.2pt}\sum_{\omega \in \Omega_4}\hspace{-3.05pt} z(\omega) \bm{u_{\tau,B}}\T\hspace{-.1pt} N_{\omega}\hspace{-.1pt} \bm{u_{\sigma,B}}$ by first expressing $ \prod_{i=1}^k \hspace{-.95pt}p_{\tau_i,\sigma_i}\hspace{-.75pt}(\hspace{-.7pt}F_i)$ as a linear combination of degree $n$ monomials expressed in the dual basis~$A^*$ of~$A$ and subsequently replacing each monomial~$\mu(\omega')$ in~$ \prod_{i=1}^k p_{\tau_i,\sigma_i}(F_i)$ with the variable~$z(r(\omega'))$.   

We finally consider the entries in the row and column indexed by $\emptyset$ in
the matrix associated with $\bm{\lambda}=((n),(),(),())$
(cf.\ Section \ref{28de15d}).
Trivially, $e_{\emptyset}\T M(x)e_{\emptyset}=(M(x))_{\emptyset,\emptyset} =x(\emptyset)$, which is set to 1 in the optimization problem.
Any $\bm{\tau}\in \bm{T_{\lambda,m}}$ is determined by the number $t$ of 2's in the row of the Young shape $Y((n))$.
Then
\begin{align*}
\bm{u_{\tau,B}}=\sum_{\substack{v,w\in[q]^n\\ d_H(v,w)=t}}e_{(v,w)}
\text{~~~and hence~~~}
L\bm{u_{\tau,B}}=\sum_{\substack{v,w\in[q]^n\\ d_H(v,w)=t}}e_{\{v,w\}}.
\end{align*}
Hence, as $\emptyset\cup\{v,w\}=\{v,w\}$,
\begin{align}\label{tauleegentry}
e_{\emptyset}\T M(x)L\bm{u_{\tau,B}}
=
\sum_{\substack{v,w\in[q]^n\\ d_H(v,w)=t}}x({\{v,w\}})
=
\binom{n}{t}q^n(q-1)^tz(\omega),
\end{align}
where $\omega$ is the $H$-orbit of $\CC_4$ consisting of all pairs $\{\alpha,\beta\}$ with $d_H(\alpha,\beta)=t$.

\section{Appendices\label{App}}

\subsection{Appendix 1: The formulas \texorpdfstring{$(F_i)_{j,h}$}{(Fi)j,h} \label{19de15a}}

Recall that each $B_i(j) \in \C^{[q] \times [q]}$ is a linear function on $\C({[q]\times [q])}$, and that each $a_P$ is an element of
$\C{([q]\times [q])}\otimes\C({[q]\times [q]})$, where $P \in \Lambda$.
We compute
$$
(F_i)_{j,h}
:= (B_i(j) \otimes B_i(h))|_W = \sum_{P \in \Lambda} (B_i(j) \otimes B_i(h))(a_P) a_P^* \,\,\,\, \in W^*
$$
 for each~$i=1,\ldots,4$ and~$j,h=1,\ldots,m_i$.
The coefficient of $a^*_P$ is obtained by evaluating $(B_i(j)\otimes B_i(h))(a_P)$.
This is routine, but we display the expressions. We denote an equivalence class in~$\Lambda=[q]^4/S_q$ by its lexicographically smallest element, representing vectors in~$[q]^4$ as words, i.e., a vector in~$[q]^4$  is represented as a string of symbols in~$[q]$ of length $4$.

\begin{align*}
(F_1)_{1,1}&=
q a^*_{0000} +q(q-1) a^*_{0011},\\
(F_1)_{1,2}&=
q(q-1)(a^*_{0001} + a^*_{0010} +(q-2) a^*_{0012}),\\
(F_1)_{2,1}&=
q(q-1)(a^*_{0111} + a^*_{0100} +(q-2) a^*_{0122}),\\
(F_1)_{2,2}&=
q(q-1)(a^*_{0101} + a^*_{0110} +(q-2)(a^*_{0102} +a^*_{0120} + a^*_{0112} + a^*_{0121} +(q-3) a^*_{0123})).\\[10pt]
(F_2)_{1,1}&=
2 a^*_{0000} -2 a^*_{0011},\\
(F_2)_{1,2}&=
2q(a^*_{0001} - a^*_{0010}),\\
(F_2)_{1,3}&=
2(q-2)(a^*_{0010} + a^*_{0001} -2 a^*_{0012}),\\
(F_2)_{2,1}&=
2q(a^*_{0100} - a^*_{0111}),\\
(F_2)_{2,2}&=
2q(2a^*_{0101} -2 a^*_{0110} +(q-2)(a^*_{0102} - a^*_{0120} - a^*_{0112} + a^*_{0121})),\\
(F_2)_{2,3}&=
2q(q-2)(a^*_{0102} + a^*_{0120} - a^*_{0112} - a^*_{0121}),\\
(F_2)_{3,1}&=
2(q-2)(a^*_{0111} + a^*_{0100} -2 a^*_{0122}),\\
(F_2)_{3,2}&=
2q(q-2)(a^*_{0102} - a^*_{0120} + a^*_{0112} - a^*_{0120}),\\
(F_2)_{3,3}&=
2(q-2)(2 a^*_{0101} +2 a^*_{0110} +(q-4)(a^*_{0102} + a^*_{0120} + a^*_{0112} + a^*_{0121}) -4(q-3) a^*_{0123}).\\[10pt]
(F_3)_{1,1}&=
6(a^*_{0101} - a^*_{0110} - a^*_{0102} + a^*_{0120} + a^*_{0112} - a^*_{0121}).\\[10pt]
(F_4)_{1,1}&=
8(a^*_{0101} + a^*_{0110} - a^*_{0102} - a^*_{0120} - a^*_{0112} -a^*_{0121}) +16 a^*_{0123}.\\
\end{align*}
 
\subsection{Appendix 2: An overview of the program}

In this section we give a high-level overview of the program, to help the reader with implementing the method. 
See Figure~$\ref{pseudocodesq}$ for an outline of the method.  We write~$\omega_t \in \Omega_{4}$ for the (unique) $S_q^n \rtimes S_n$-orbit of a pair of code words of distance~$t$, and~$\omega_{\emptyset}$ for the orbit~$\{\emptyset\}$.\symlistsort{omegat}{$\omega_t$}{$S_q^n \rtimes S_n$-orbit of a pair of  code words at  distance~$t$}\symlistsort{omegaempty}{$\omega_{\emptyset}$}{the orbit~$\{\emptyset\}$} Also, we write~$\Omega_4^d \subseteq \Omega_4$ for the set of~$S_q^n \rtimes S_n$-orbits on~$\mathcal{C}_4^d$, where~$\mathcal{C}_4^d$ denotes the collection of elements of~$\mathcal{C}_4$ of minimum distance at least~$d$.\symlistsort{Omegakd}{$\Omega_k^d$}{collection of $H$-orbits of codes of minimum distance at least~$d$}

\begin{figure}[ht] 
  \fbox{
    \begin{minipage}{15.5cm} \small
     \begin{tabular}{l}
    \textbf{Input: } Natural numbers~$q,n,d$ \\
	\textbf{Output: }Semidefinite program to compute~$B_q(n,d)$ \\
	$\,$\\
    \printv \emph{Maximize} $q^n z(\omega_0)$  \\
        \printv \emph{Subject to:}  \\    
        	\foreachv $\bm{n} =(n_1,\ldots,n_k) \in \bm{N}$ \\
\hphantom{1cm}\foreachv $\bm{\lambda} =(\lambda_1,\ldots,\lambda_k) \vdash \bm{n}$ with~$\height(\lambda_i)\leq m_i$~$\,\, \forall \, i=1,\ldots,k$ \\
\hphantom{1cm}\hphantom{1cm} start a new block~$M_{\bm{\lambda}}$\\	
\hphantom{1cm}\hphantom{1cm} \foreachv $\bm{\tau} \in   \bm{T''_{\lambda,m}}$ from Section~$\ref{28de15c}$\\
\hphantom{1cm}\hphantom{1cm}\hphantom{1cm} \foreachv $\bm{\sigma} \in    \bm{T''_{\lambda,m}}$ from Section~$\ref{28de15c}$\\
			\hphantom{1cm}\hphantom{1cm}\hphantom{1cm}\hphantom{1cm} compute $ \prod_{i=1}^k p_{\tau_i,\sigma_i}(F_i)$  (cf. the definitions in~\eqref{ptsdef} and~\eqref{fformula}) \\	
	\hphantom{1cm}\hphantom{1cm}\hphantom{1cm}\hphantom{1cm}			replace each degree~$n$ monomial~$\mu(\omega')$ in variables~$a_{P}^*$  
	\\	\hphantom{1cm}\hphantom{1cm}\hphantom{1cm}\hphantom{1cm}\hphantom{1cm} by a variable~$z(r(\omega'))$ if~$r(\omega') \in \Omega_4^d$ and by~$0$ otherwise \\
\hphantom{1cm}\hphantom{1cm}\hphantom{1cm}\hphantom{1cm}	 $(M_{\bm{\lambda}})_{\bm{\tau},\bm{\sigma}}:=$  the resulting linear expression in~$z(\omega)$\\				
				\hphantom{1cm}\hphantom{1cm}\hphantom{1cm} \ndv\\
		\hphantom{1cm}\hphantom{1cm}\ndv\\
		\hphantom{1cm}\hphantom{1cm}\ifv $\bm{n}=(n,0,0,0)$ and $\bm{\lambda}=((n),(),(),())$  \text{ \color{blue}//Add $ \emptyset$.}\\
	\hphantom{1cm}\hphantom{1cm}\hphantom{1cm}	add a row and column to~$M_{\bm{\lambda}}$ indexed by~$\emptyset$
	\\
	\hphantom{1cm}\hphantom{1cm}\hphantom{1cm}	put~$(M_{\bm{\lambda}})_{\emptyset,\emptyset}:=1$ and the entries~$(M_{\bm{\lambda}})_{\emptyset,\bm{\tau}}$ and~$(M_{\bm{\lambda}})_{\bm{\tau},\emptyset}$ as in~$(\ref{tauleegentry})$ \\
		\hphantom{1cm}\hphantom{1cm}\ndv\\
		  \hphantom{1cm}\hphantom{1cm}\printv $M_{\bm{\lambda}}\succeq 0$   \\				
	\hphantom{1cm}\ndv 	\\ 
	\ndv  \\
		\foreachv $\omega \in \Omega_4^d$  	\text{\color{blue}//Now nonnegativity of all variables.}\\ 
	    \hphantom{1cm} \printv $ z(\omega) \geq 0$\\
	    \ndv 
     \end{tabular}  
    \end{minipage}    } \caption{\label{pseudocodesq}\small{Algorithm to generate semidefinite programs for computing~$B_q(n,d)$.}}
\end{figure}

\chapter{Nonbinary code bounds based on divisibility arguments} \label{divchap}

\chapquote{We all believe that mathematics is an art.}{Emil Artin (1898--1962)}

\noindent For~$q,n,d \in \N$, let again~$A_q(n,d)$ denote the maximum size of a code~$C \subseteq [q]^n$ with minimum distance at least~$d$. We give a divisibility argument resulting in the new upper bounds~$A_5(8,6) \leq 65$, $A_4(11,8)\leq 60$ and~$A_3(16,11) \leq 29$. These in turn imply the new upper bounds~$A_5(9,6) \leq 325$,~$A_5(10,6) \leq 1625$,~$A_5(11,6) \leq 8125$ and~$A_4(12,8) \leq 240$.

Furthermore, we prove that for~$\mu,q \in \N$,  there is a 1-1-correspondence between so-called symmetric~$(\mu,q)$-nets (which are certain designs) and codes~$C \subseteq [q]^{\mu q}$ of size~$\mu q^2$ with minimum distance at least~$\mu q - \mu$. From this, we derive the new upper bounds $A_4(9,6) \leq 120$ and $A_4(10,6) \leq 480$.

This chapter is based on~$\cite{divisibility}$.

\section{Introduction}
Recall that for~$q,n,d \in \N$, an~$(n,d)_q$\emph{-code} is a set~$C\subseteq [q]^n$ that satisfies~$d_{\text{min}}(C)\geq d$. As before, define
\begin{align}\label{aqnd}
    A_q(n,d) :=  \max \{ | C| \,\, | \,\, C  \text{ is an $(n,d)_q$-code} \}.
\end{align}
In this chapter we find new upper bounds on~$A_q(n,d)$ (for some~$q,n,d$), based on a \emph{divisibility}-argument.  In some cases, it will sharpen a combination of the following two well-known upper bounds on~$A_q(n,d)$. Fix~$q,n,d \in \N$. Then
\begin{align} \label{elementarybounds}
qd >(q-1)n\,\,\, \Longrightarrow\,\,\, A_q(n,d) \leq \frac{qd}{qd-(q-1)n}.
\end{align}
This is the~$q$-ary \emph{Plotkin bound}.\indexadd{Plotkin bound} Moreover,
\begin{align} \label{elementarybounds2}
A_q(n,d) \leq q \cdot A_q(n-1,d). 
\end{align}
 A proof of these statements can be found in~$\cite{sloane}$.  The Plotkin  bound can be proved by comparing the leftmost and rightmost terms in~$(\ref{detruc})$ below. The second bound follows from the observation that in a~$(n,d)_q$-code any symbol can occur at most~$A_q(n-1,d)$ times at the first position.

We view an~$(n,d)_q$-code~$C$ of size~$M$ as an~$M \times n$ matrix with the words as rows. Two codes~$C,D \subseteq [q]^n$ are \emph{equivalent} (or \emph{isomorphic}) if~$D$ can be obtained from~$C$ by first permuting the~$n$ columns of~$C$ and subsequently applying to each column a permutation of the~$q$ symbols in~$[q]$ (we will write `\emph{renumbering a column}' instead of `applying a permutation to the symbols in a column').\indexadd{renumbering a column}

\begin{table}[ht]
\begin{center}
    \begin{tabular}{| r|r|r || r |>{\bfseries}r |r|r|}
    \hline
  $q$ & $n$ & $d$  &   \multicolumn{1}{>{\raggedright\arraybackslash}b{20mm}|}{best lower bound known \cite{4ary,5ary,plotkin,vaessens}}    &\multicolumn{1}{>{\raggedright\arraybackslash}b{23mm}|}{\textbf{new upper bound}}  & \multicolumn{1}{>{\raggedright\arraybackslash}b{20mm}|}{best upper bound previously known \cite{4ary,5ary,  tanaka,vaessens}} & \multicolumn{1}{>{\raggedright\arraybackslash}b{20mm}|}{Delsarte bound} \\  \hline 
 $5$ & $8$ & $6$ &     50 & 65 & 75  & 75 \\ 
 $5$ & $9$ & $6$ &  135 & 325& 375   & 375 \\
 $5$ & $10$ & $6$ & 625  & 1625& 1855   & 1875 \\    
 $5$ & $11$ & $6$ &  3125 & 8125& 8840  & 9375 \\   
    \hline
 $4$ & $9$ & $6$ &   64  & 120 &  128   & 128  \\  
 $4$ & $10$ & $6$ &  256 &480   & 496  & 512  \\      
 $4$ & $11$ & $8$ & 48 & 60  & 64 & 64  \\    
 $4$ & $12$ & $8$ &  128  & 240 & 242 & 242  \\    
     \hline 
 $3$ & $16$ & $11$ &  18   & 29 & 30  & 33 \\   \hline  
    \end{tabular}
\end{center}
  \caption{\small An overview of the new upper bounds on~$A_q(n,d)$. The previous lower and upper bounds are taken from references~$\cite{4ary,5ary, vaessens}$, except for the upper bounds~$A_5(10,6)\leq 1855$, $A_5(11,6)\leq 8840$ and $A_4(10,6)\leq 496$ (cf.~\cite{tanaka}), and the lower bounds~$A_5(8,6)\geq 50$ and~$A_4(11,8) \geq 48$.\protect\footnotemark\, These lower bounds follow from the exact values~$A_5(10,8)=50$ and~$A_4(12,9)=48$ ($\cite{plotkin}$). For updated tables with all most recent code bounds, we refer to~$\cite{brouwertableq}$. }
\end{table} 
\footnotetext{In~$\cite{4ary,5ary}$, the lower bounds~$A_5(8,6)\geq 45$ and~$A_4(11,8) \geq 34$ are given.}

 If an~$(n,d)_q$-code~$C$ is given, then for~$j=1,\ldots,n$, let~$c_{\alpha,j}$ denote the number of times symbol~$\alpha \in [q]$ appears in column~$j$ of~$C$.\symlistsort{calphaj}{$c_{\alpha,j}$}{number of times symbol~$\alpha$ appears in column~$j$ of~$C$} For any two words~$u,v \in [q]^n$, we define~$g(u,v):=n-d_H(u,v)$. So~$g(u,v)$ is the number of~$i$ with $ u_i = v_i$.\symlistsort{g(u,v)}{$g(u,v)$}{number of~$i$ with $ u_i = v_i$}   In our divisibility arguments, we will use the following observations (which are well known and often used in coding theory and combinatorics). 
 
\begin{proposition} 
If~$C$ is an~$(n,d)_q$-code of size~$M$, then
\begin{align}\label{detruc}
\binom{M}{2}(n-d) \geq \sum_{\substack{\{u,v\} \subseteq C\\ u \neq v}} g(u,v) = \sum_{j=1}^n \sum_{\alpha \in [q]} \binom{c_{\alpha,j}}{2} \geq n \cdot \left( (q-r)\binom{m}{2} +r \binom{m-1}{2}  \right),
\end{align}
where~$m := \ceil{M/q}$ and~$ r:=qm-M$, so that~$M=qm-r$ and~$0 \leq r < q$. Moreover, writing~$L$ and~$R$ for the leftmost term and the rightmost term in~$(\ref{detruc})$, respectively, we have 
\begin{align} \label{boundnd} 
| \{ \{u,v\} \subseteq C \,\, | \,\, u \neq v, \,\, d_H(u,v)\neq d \}| \leq L - R,  
\end{align}
i.e., the number of pairs of distinct words~$\{u,v\} \subseteq C$ with distance \emph{unequal} to~$d$ is at most the leftmost term minus the rightmost term in~$(\ref{detruc})$.
\end{proposition}
\proof 
The first inequality in~$(\ref{detruc})$ holds because~$n-d \geq g(u,v)$ for all~$u,v \in C$. The equality is obtained by counting the number of equal pairs of entries in the same columns of~$C$ in two ways. The second inequality follows from the (strict) convexity of the binomial coefficient~$F(x):=x(x-1)/2$. Fixing a column~$j$, the quantity~$\sum_{\alpha \in [q]} F(c_{\alpha,j})$, under the condition that~$\sum_{\alpha \in [q]} c_{\alpha,j} = M$, is minimal if the~$c_{\alpha,j}$ are as equally divided as possible, i.e., if~$c_{\alpha,j} \in \{ \ceil{M/q}, \floor{M/q}\}$ for all~$\alpha \in [q]$. The desired inequality follows.

To prove the second assertion, note that~$(\ref{detruc})$ implies that~$\sum_{\{u,v\} \subseteq C, \, u \neq v} g(u,v) \geq R$, so
\begin{align}\label{detruc2}
| \{ \{u,v\} \subseteq C \,\, | \,\, u \neq v, \,\, d_H(u,v)\neq d \}|  &\leq \sum_{\substack{\{u,v\} \subseteq C\\ u \neq v}} (n-d-g(u,v)) 
\\&\leq  \binom{M}{2}(n-d) - R = L-R.\notag \qedhere
\end{align}
\endproof 
A code $C \subseteq [q]^n$ is  \emph{equidistant with distance~$d$} if the distance between any two distinct codewords in~$C$ is equal to~$d$.\indexadd{equidistant!with distance~$d$} A code is \emph{equidistant} if there exists a~$d \in \N$ such that~$C$ is equidistant with distance~$d$.\indexadd{equidistant}\indexadd{code!equidistant} Define, for~$q,n,d,M \in \N$,\symlistsort{h(q,n,d,M)}{$h(q,n,d,M)$}{leftmost term minus rightmost term in~(\ref{detruc})}
\begin{align}\label{LRfunctie}
h(q,n,d,M):= L-R,
\end{align}
where~$L$ and~$R$ denote the leftmost and rightmost terms in~$\eqref{detruc}$, respectively. So if~$C$ is an~$(n,d)_q$-code of size~$M$, then~\eqref{boundnd} states that
\begin{align}\label{hupperbound}
| \{ \{u,v\} \subseteq C \,\, | \,\, u \neq v, \,\, d_H(u,v)\neq d \}| \leq h(q,n,d,M).
\end{align} 
\begin{corollary}\label{cor} If~$h(q,n,d,M)=0$ for some~$q,n,d$ and $M$ (so the leftmost term equals the rightmost term in~$(\ref{detruc})$), then for any $(n,d)_q$-code~$C$ of size~$M$, 
\begin{enumerate}[(i)] 
\item \label{1} $d_H(u,v)=d$ for all~$u,v \in C$ with~$u \neq v$, i.e.,~$C$ is equidistant with distance~$d$, and
\item \label{2} for each column~$C_j$ of~$C$, there are $q-r$ symbols in~$[q]$ that occur~$m$ times in~$C_j$ and~$r$ symbols in~$[q]$ that occur~$m-1$ times in~$C_j$. 
\end{enumerate}
\end{corollary}
\noindent In the next sections we will use~(\ref{1}),~(\ref{2}) and the bound in~$(\ref{hupperbound})$ to give (for some~$q,n,d$) new upper bounds on~$A_q(n,d)$, based on divisibility arguments. Furthermore, in Section~$\ref{symsec}$, we will prove that, for~$\mu,q \in \N$, there is a 1-1-correspondence between symmetric~$(\mu,q)$-nets (which are certain designs) and $(n,d)_q=(\mu q, \mu q - \mu)_q$-codes~$C$ with~$|C|=\mu q^2 $. We derive some new upper bounds from these `symmetric net' codes.

\section{The divisibility argument}
In this section, we describe the divisibility argument and illustrate it by an example. Next, we show how the divisibility argument can be applied to obtain upper bounds on~$A_q(n,d)$ for certain~$q,n,d$. In subsequent sections, we will see how we can improve upon these bounds for certain fixed~$q,n,d$. We will use the following notation.\indexadd{block@($k$-)block} 

\begin{defn}[$k$-block]
Let~$C$ be an~$(n,d)_q$-code in which a symbol~$\alpha \in [q]$ is contained exactly~$k$ times in column~$j$. The~$k \times n$ matrix~$B$ formed by the $k$ rows of~$C$ that have symbol~$\alpha$ in column~$j$ is called a~($k$-)\emph{block} (for column~$j$). In that case, columns~$\{1,\ldots,n\} \setminus \{j\}$ of~$B$ form an~$(n-1,d)_q$-code of size~$k$. 
\end{defn}

  At the heart of the divisibility arguments that will be used throughout this chapter lies the following observation.
\begin{proposition}[Divisibility argument] \label{parity}
Suppose that~$C$ is an~$(n,d)_q$-code and that~$B$ is a block in~$C$ (for some column~$j$) containing every symbol exactly~$m$ times in every column except for column~$j$. If~$n-d$ does not divide~$m(n-1)$, then for each~$u \in C \setminus B$ there is a word~$v \in B$ with~$d_H(u,v) \notin \{d,n\}$.
\end{proposition}
\proof 
Let~$u \in C \setminus B$. We renumber the symbols in each column such that~$u$ is $\mathbf{0}:=0\ldots 0$, the all-zeros word. The total number of~0's in~$B$ is~$m(n-1)$ (as the block~$B$ does not contain~0's in column~$j$ since~$u\notin B$ and since~$B$ consists of all words in~$C$ that have the same symbol in column~$j$). Since~$n-d$ does not divide~$m(n-1)$, there must be a word~$v \in B$ that contains a number of~0's not divisible by~$n-d$. In particular, the number of~$0$'s in~$v$ is different from~$0$ and~$n-d$. So~$d_H(u,v) \notin \{d,n\}$.
\endproof 

\begin{examp}\label{586}
We apply Proposition~$\ref{parity}$ to the case~$(n,d)_q=(8,6)_5$. The best known upper bound\footnote{The Delsarte bound~$\cite{delsarte}$ on~$A_5(8,6)$, the bound based on~$\eqref{elementarybounds}$, and the semidefinite programming bound~$B_5(8,6)$ based on quadruples of codewords from~\eqref{8no15a} all are equal to 75.} is~$A_5(8,6) \leq 75$, which can be derived from~$(\ref{elementarybounds})$ and~$(\ref{elementarybounds2})$, as the Plotkin bound yields~$A_5(7,6) \leq 15$ and hence~$A_5(8,6) \leq 5 \cdot 15 = 75$. Since $h(5,7,6,15)=0$ (where~$h$ is defined in~\eqref{LRfunctie}),  any $(7,6)_5$-code~$D$ of size~$15$ is equidistant with distance~$6$ and each symbol appears exactly~$m=3$ times in every column of~$D$. 

Suppose there exists a~$(8,6)_5$-code~$C$ of size~$75$. As~$A_5(7,6) \leq 15$, each column of~$C$ yields a division of~$C$ into five 15-blocks. Let~$B$ be a~$15$-block for the~$j$th column and let~$u \in C \setminus B$. By the above, the other columns of~$B$ contain each symbol~$3$ times, and~$3(n-1)=3\cdot 7 =21$ is not divisible by~$n-d=2$. So by Proposition~\ref{parity}, there must be a word~$v \in  B$ with~$d_H(u,v) \notin \{6,8\}$.
 
However, since all~$(7,6)_5$-codes of size~15 are equidistant with distance~$6$, all distances in~$C$ belong to~$\{6,8\}$: either two words are contained together in some~$15$-block (hence their distance is~$6$) or there is no column for which the two words are contained in a~$15$-block (hence their distance is~$8$). This implies that an~$(8,6)_5$-code~$C$ of size~$75$ cannot  exist. Hence~$A_5(8,6) \leq 74$. Theorem~$\ref{importantth}$ and Corollary~$\ref{1mod4}$ below will imply that~$A_5(8,6)\leq 70$ and in Section~$\ref{sec586}$ we will show that, with some computer assistance, the bound can be pushed down to~$A_5(8,6) \leq 65$.
\end{examp}

 To exploit the idea of Proposition~$\ref{parity}$, we will count the number of so-called irregular pairs of words occurring in a code.\indexadd{irregular pair}

\begin{defn}[Irregular pair]
Let~$C$ be an~$(n,d)_q$-code and~$u,v \in C$ with~$u\neq v$. If~$d_H(u,v) \notin \{d,n\}$, we call~$\{u,v\}$ an \emph{irregular pair}. 
\end{defn}
  For any code~$C \subseteq  [q]^n$, we write\symlistsort{X}{$X$}{the set of irregular pairs in~$C$ (only in Chapter~\ref{divchap})} 
\begin{align}\label{X}
X:= \text{ the set of irregular pairs~$\{u,v\}$ for~$u,v \in C$}.
\end{align}

 Using Proposition~$\ref{parity}$, we can for some cases derive a lower bound on~$|X|$. If we can also compute an upper bound on~$|X|$ that is smaller than the lower bound, we derive that the code~$C$ cannot exist. The proof of the next theorem uses this idea. 
\begin{theorem} \label{importantth}
Suppose that~$q,n,d,m$ are positive integers with $q\geq 2$, such that~$d=m(qd-(q-1)(n-1))$, and such that~$n-d$ does not divide~$m(n-1)$. If~$r \in \{1,\ldots,q-1\}$ satisfies~
\begin{align}\label{kleinertruc}
n(n-1-d)(r-1)r <  (q-r+1)(qm(q+r-2)-2r),
\end{align}
 then~$A_q(n,d) < q^2m -r$. \end{theorem}
\proof 
 By the Plotkin bound~$(\ref{elementarybounds})$ we have
\begin{align} \label{part}
    A_q(n-1,d) \leq qm. 
\end{align}
Let~$D$ be an~$(n-1,d)_q$-code of size~$qm-t$ with~$0 \leq t<q$. Note that~$d=(q-1)m(n-1)/(qm-1)$, so
$$
n-1-d = \frac{(qm-1)n-(qm-1)-(q-1)m(n-1)}{qm-1} = \frac{(m-1)(n-1)}{qm-1} .
$$
Then we have (where~$h$ is defined cf.~\eqref{LRfunctie})
\begin{align*}
h(q,n-1,d,qm-t)
&= \binom{qm-t}{2}(n-1-d) - (n-1)\left( (q-t) \binom{m}{2} +t \binom{m-1}{2} \right)
\\&=   \binom{qm-t}{2}(n-1-d) 
\\&\phantom{==}-\mbox{$\frac{1}{2}$} (n-1-d)(qm-1)\left( (q-t) m +t (m-2) \right)
\\&=\mbox{$\frac{1}{2}$} (n-1-d) ((qm-t)(qm-t-1)-(qm-1)(qm-2t))
\\&= (n-1-d)\binom{t}{2}.
\end{align*}
Hence
\begin{align} \label{most}
    \text{$D$ contains at most~$(n-1-d)\binom{t}{2}$ pairs of words with distance~$\neq d$}.
\end{align}
Therefore, all~$(n-1,d)_q$-codes~$D$ of size~$qm$ are equidistant with distance~$d$ (then~$t=0$) and each symbol occurs~$m$ times in every column of~$D$.

 Now let~$C$ be an~$(n,d)_q$-code of size~$M:=q^2m-r$ with~$r \in \{1,\ldots,q-1\}$ and suppose that~\eqref{kleinertruc} holds. Consider a~$qm$-block~$B$ for some column of~$C$. As~$n-d$ does not divide~$m(n-1)$, by Proposition~$\ref{parity}$ we know 
\begin{align}\label{obs}
\text{if~$u \in C\setminus B$, then there exists~$v \in B$ with~$d_H(u,v) \notin \{d,n\}$.}
\end{align}
Let~$B_1,\ldots, B_{s}$ be $qm$-blocks~in~$C$ for some fixed column. Since~$|C|=q^2m-r$, the number of $qm$-blocks for any fixed column is at least~$q-r$ (so we can take~$s=q-r$). Then, with~$(\ref{obs})$, one obtains a lower bound on the number~$|X|$ of irregular pairs in~$C$. Every pair~$\{B_{i},B_{k}\}$ of $qm$-blocks gives rise to~$qm$ irregular pairs: for each word~$u \in B_{i}$, there is a word~$v \in B_k$ such that~$\{u,v\}\in X$. This implies that in~$\cup_{i=1}^s B_i \subseteq C$ there are at least~$\binom{s}{2}qm$ irregular pairs. Moreover, for each word~$u$ in~$C\setminus \cup_{i=1}^s B_i$ (there are~$M-qm\cdot s$ of such words) there is, for each~$i=1,\ldots,s$, a word~$v_i \in B_i$  with~$\{u,v_i\} \in X$. This gives an additional number of at least~$(M-qms)s$ irregular pairs in~$C$. Hence:
\begin{align} \label{lowerb}
|X| &\geq \binom{s}{2} qm + (M-qms)s \notag  
\\&= \mbox{$\frac{1}{2}$}s (qm(2q-s-1)-2r) =: \psi(s).
\end{align} 

On the other hand, note that the~$i$th block for the~$j$th column has size~$qm-r_{i,j}$ for some integer~$r_{i,j}\geq 0$ by~$(\ref{part})$, where~$\sum_{i=1}^q r_{i,j}=r\leq q-1$ (hence each~$r_{i,j} <q$). So by~$(\ref{most})$, the number of irregular pairs in~$C$ that have the same entry in column~$j$ is at most
\begin{align} 
(n-1-d)\sum_{i=1}^q \binom{r_{i,j}}{2}.
\end{align} 
 As each irregular pair~$\{u,v\}$ has~$u_j=v_j$ for at least one column~$j$, we conclude  
\begin{align} \label{upperb}
|X| \leq  (n-1-d)\sum_{j=1}^n \sum_{i=1}^q \binom{r_{i,j}}{2} \leq n (n-1-d)\binom{r}{2}.
\end{align} 
Here the last inequality follows by convexity of the binomial function, since (for fixed~$j$) the sum~$\sum_{i=1}^q \binom{r_{i,j}}{2}$ under the condition that~$\sum_{i=1}^q r_{i,j}=r$  is maximal if one of the~$r_{i,j}$ is equal to~$r$ and the others are equal to~$0$.

Suppose each~$r_{i,j}\in \{0,1\}$, then~$|X|=0$ by~$(\ref{upperb})$. However, as~$q-r\geq 1$, there is at least one~$qm$-block for any fixed column, so~$|X|\geq 1$ by~$(\ref{obs})$, which is not possible. Hence we can assume that~$r_{i,j} \geq 2$ for some~$i,j$ (this also implies~$A_q(n,d) \leq q^2m-2$). Then the number~$s$ of~$qm$-blocks for column~$j$ satisfies~$s \geq q-r+1$. This gives by~$(\ref{lowerb})$ and~$(\ref{upperb})$ that
 \begin{align} \label{newineq2}
\psi(q-r+1) \leq |X| \leq n(n-1-d) \binom{r}{2},
\end{align}
where the function~$\psi$ is defined in~\eqref{lowerb}. Multiplying~\eqref{newineq2} by~$2$ yields
$$
n(n-1-d)(r-1)r \geq  (q-r+1)(qm(q+r-2)-2r),
$$
contradicting~\eqref{kleinertruc}.  

So if~$\eqref{kleinertruc}$ holds, then~$A_q(n,d) < q^2m-r$, as was needed to prove.
\endproof

 We give two interesting applications of Theorem~$\ref{importantth}$.

\begin{corollary} \label{1mod4}
If~$q \equiv 1 \pmod{4}$ and~$q\neq 1$, then
 \begin{align}\label{q+3}
     A_q(q+3,q+1) \leq \mbox{$\frac{1}{2}$}q^2(q+1)-q =\mbox{$\frac{1}{2}$}(q-1)q (q+2). 
 \end{align}
\end{corollary}
\proof 
Apply Theorem~$\ref{importantth}$ to~$n=q+3$,~$d=q+1$, $m=d/(qd-(q-1)(n-1))=(q+1)/2  \in \N$ and~$r=q-1$. Then~$n-d=2$ does not divide~$m(n-1)=(q+1)(q+2)/2$, as~$q\equiv 1 \pmod{4}$. Furthermore, the left hand side minus the right hand side in~\eqref{kleinertruc} is equal to $-(q^3-q^2-2)<0$, so~\eqref{kleinertruc} is satisfied. Hence~$A_q(q+3,q+1) < q^2(q+1)/2-(q-1)$.
\endproof 

 Applying Corollary~$\ref{1mod4}$ to~$q=5$ gives~$A_5(8,6) \leq 70$. In Section~$\ref{sec586}$ we will improve this to~$A_5(8,6) \leq 65$.

\begin{remark}
Note that for bound~$(\ref{q+3})$ to hold it is necessary that~$q\equiv 1 \pmod{4}$. If~$q \equiv 3 \pmod{4}$ the statement does \emph{not} hold in general. For example,~$A_3(6,4)= 18$ (see~$\cite{brouwertableq}$), which is larger than bound~$(\ref{q+3})$. 
\end{remark}

Theorem~$\ref{importantth}$  also gives an upper bound on~$A_q(n,d)=A_q(kq+k+1,kq)$, where~$q\geq 2$ and~$k+1$ does not divide~$q(q+1)$ (which is useful for~$k <q-1$; for~$k \geq q+1$ the Plotkin bound gives a better bound). One new upper bound for such~$q,n,d$ is obtained:

\begin{proposition}\label{60} $A_4(11,8)\leq 60$. 
\end{proposition}
\proof 
This follows from Theorem~$\ref{importantth}$ with~$q=4$, $n=11$, $d=8$, $m=d/(qd-(q-1)(n-1)) = 4 \in \N$ and $r=3$. Then~$n-d=3$ does not divide~$m(n-1)=40$.  Furthermore, the left hand side minus the right hand side in~\eqref{kleinertruc} is equal to $-16<0$, so~\eqref{kleinertruc} is satisfied. Therefore~$A_4(11,8) < 61$.
\endproof  

 This implies the following bound, which is also new:

\begin{corollary} 
$A_4(12,8) \leq 240$.
\end{corollary}
\proof 
By Proposition~$\ref{60}$ and~$(\ref{elementarybounds2})$.
\endproof

\section{Kirkman triple systems and \texorpdfstring{$A_5(8,6)$}{A5(8,6)}}\label{sec586}

 In this section we consider the case $(n,d)_q=(8,6)_5$ from Example~$\ref{586}$. Corollary~\ref{1mod4} implies that $A_5(8,6) \leq 70$. Using small computer checks, we will obtain $A_5(8,6) \leq 65$. 
 
  As in the proof of Theorem~$\ref{importantth}$, we will compare upper and lower bounds on~$|X|$. But since an~$(8,6)_5$-code~$C$ of size at most~$ 70$ does not necessarily contain a~$15$-block (as~$70=5 \cdot 14$), we need information about~$14$-blocks. To this end we show, using an analogous approach as in~$\cite{bogzin}$ (based on occurrences of symbols in columns of an equidistant code):

\begin{proposition}\label{14block}
Any~$(7,6)_5$-code~$C$ of size~$14$ can be extended to a~$(7,6)_5$-code of size~$15$.
\end{proposition}
\proof 
Let~$C$ be a $(7,6)_5$-code  of size~$14$. Since~$h(5,7,6,14)=0$ (where~$h$ is defined cf.~\eqref{LRfunctie}), Corollary~\ref{cor} yields that~$C$ is equidistant with distance~$6$ and that for each~$j \in \{1,\ldots,7\}$ there exists a unique~$\beta_j \in [q]$ with~$c_{\beta_j,j} =2$ and~$c_{\alpha,j}=3$ for all~$\alpha \in [q] \setminus \{\beta_j \}$. Define a~$15$-th codeword~$u$ by putting~$u_j:= \beta_j$ for all~$j=1,\ldots,7$. We claim that~$C \cup \{ u \}$ is a~$(7,6)_5$-code of size~$15$.

To establish the claim we must prove that~$d_H(u,w) \geq 6$ for all~$w \in C$. Suppose that there is a word~$w \in C$ with~$d_H(u,w) <6$. We can renumber the symbols in each column of~$C$ such that~$w=\mathbf{0}$. Since~$C$ is equidistant with distance~$6$, each word in~$C \setminus \{ w \}$ contains precisely one~$0$. On the other hand, there are two column indices~$j_1$ and~$j_2$ with~$u_{j_1} = 0$ and~$u_{j_2} = 0$.  Then~$C\setminus\{w\}$ contains at most~$1+1+5\cdot 2 = 12$ occurrences of the symbol~$0$ (since in columns~$j_1$ and~$j_2$ there is precisely one~$0$ in~$C \setminus \{w\}$). But in that case, since~$|C\setminus \{ w \}|=13>12$, there is a row in~$C$ that contains zero occurrences of the symbol~$0$, contradicting the fact that~$C$ is equidistant with distance~$6$.
\endproof 

 Note that a code of size more than~$65$ must have at least one~15- or~$14$-block, and therefore it must have a subcode of size~$65$ containing at least one~$14$-block. We shall now prove that this is impossible because 
\begin{align} \label{only13}
\text{each~$(8,6)_5$-code of size~$65$ only admits~$13$-blocks.} 
\end{align} 
It follows that~$A_5(8,6) \leq 65$. In order to prove~$(\ref{only13})$, let~$C$ be a~$(8,6)_5$-code of size~$65$. We first compute a lower bound on the number of irregular pairs in~$C$. Define, for~$x,y \in \mathbb{Z}_{\geq 0}$,\symlistsort{f(x,y)}{$f(x,y)$}{function defined in~(\ref{lowX})}
\begin{align}\label{lowX}
    f(x,y) &:= (3 x+ y)(65-15  x-14 y) + 3 \cdot 15 \binom{x}{2} +  14  \binom{y}{2} + 3\cdot 14 xy
    \\& \phantom{={,}}  - 2 \cdot 21  x- 8 y + \mathbf{1}_{\{y>0 \text{ and } x=0\}}  (65-14-39). \notag 
\end{align}

\begin{proposition}[Lower bound on~$|X|$]\label{LBX} Let~$C$ be an~$(n,d)_q=(8,6)_5$-code of size~$65$ and let~$j \in \{1,\ldots,n\}$. Let~$x$ and~$y$ be the number of symbols that appear~$15$ and~$14$ times (respectively) in column~$j$. Then the number~$|X|$ of irregular pairs in~$C$ is at least~$f(x,y)$.
\end{proposition} 
\proof 
First consider a~$(7,6)_5$-code~$D$ of size~$15$ or size~$14$ and define 
\begin{align} 
S := \{ u \in [5]^7 \,\, | \,\,d_H(u,w) \geq 5 \,\,\,\, \forall\, w \in D \}. 
\end{align} 
For any~$u \in S$, define
\begin{align}
\alpha(u) := |\{w \in D\,\,| \, \, d_H(u,w)=6 \}|.
\end{align}
Then
\begin{align} \label{calign}
\text{if $|D|=15$, then} \phantom{aaai}   &      &   \text{if $|D|=14$, then}\phantom{aaai}& \notag  \\
| \{u \in S\,\, |\,\, \alpha(u)=0\}| &=0,  &  | \{u \in S\,\, |\,\, \alpha(u)=0\}| &\leq 8, \notag \\
| \{u \in S\,\, |\,\, \alpha(u)=1\}| &\leq 21,  &  |\{ u \in S\,\, |\,\, \alpha(u)\leq1\}| &\leq 39.  \\
| \{u \in S\,\, |\,\, \alpha(u)=2\}|&=0,    &  &\notag 
\end{align}
This can be checked efficiently with a computer\footnote{All computer tests in this chapter are small and can be executed within a minute on modern personal computers.} by checking all possible~$(7,6)_5$-codes  of size~$15$ and~$14$ up to equivalence. Here we note that a~$(7,6)_5$-code~$D$ of size~$15$ (which  is equidistant with distance~$6$, see Example~\ref{586})   corresponds to a solution to \emph{Kirkman's school girl problem}~$\cite{zinoviev}$.\footnote{Kirkman's school girl problem asks to arrange 15 girls 7 days in a row in groups of 3 such that no two girls appear in the same group twice. The 1-1-correspondence between~$(7,6)_5$-codes~$D$ of size~$15$ and solutions to Kirkman's school girl problem is given by the rule: girls~$i_1$ and~$i_2$ walk in the same triple on day~$j$  $\, \Longleftrightarrow D_{i_1,j} = D_{i_2,j}$. }
 So to establish~$(\ref{calign})$, it suffices to check\footnote{By `check' we mean that given a~$(7,6)_5$-code~$D$ of size~$14$ or~$15$, we first compute~$S$, then~$\alpha(u)$ for all~$u \in S$, and subsequently verify~$(\ref{calign})$.} all~$(7,6)_5$-codes of size~$15$, that is, Kirkman systems (there are 7 nonisomorphic Kirkman systems~$\cite{7sol}$), and all~$(7,6)_5$-codes of size~$14$, of which there are at most~$7 \cdot 15$ by Proposition~$\ref{14block}$. 
 
Recall that~$C$ is an~$(8,6)_5$-code of size~$65$. Let~$G=(C,X)$ be the graph with vertex set~$V(G):=C$ and edge set~$E(G):=X$. Consider a~$15$-block~$B$ determined by column~$j$.  By~$(\ref{calign})$, each~$u \in C \setminus B$ has~$\geq 1$ neighbor in~$B$. 

Furthermore,~$(\ref{calign})$ gives that all but~$\leq 21$ elements~$u\in C\setminus B$ have~$\geq 3$ neighbors in~$B$. So by adding~$\leq 2 \cdot 21$ new edges, we obtain that each~$u \in C \setminus B$ has~$\geq 3$ neighbors in~$B$. 

Similarly, for any~$14$-block~$B$ determined by column~$j$, by adding~$\leq 8$ new edges we achieve that each~$u \in C \setminus B$ has~$\geq 1$ neighbor in~$B$. Hence, by adding~$\leq ( 2 \cdot 21 \cdot x+  8\cdot y )$ edges to~$G$, we obtain a graph~$G'$ with
\begin{align}
|E(G')| \geq (3 x+ y)(65-15 x-14 y) + 3 \cdot 15  \binom{x}{2} +  14 \binom{y}{2}+ 3\cdot 14  xy.
\end{align}
 This results in the required bound, except for the term with the indicator function. That term can be added because $|\{ u \in S\,\, |\,\, \alpha(u)\leq1\}| \leq 39$ if~$|D|=14$, by~$(\ref{calign})$.
\endproof

\begin{theorem}[$A_{5}(8,6) \leq 65$]\label{65th} 
Suppose that~$C$ is an~$(n,d)_q=(8,6)_5$-code with~$|C|=65$. Then each symbol appears exactly~$13$ times in each column of~$C$. Hence,~$A_{5}(8,6) \leq 65$.
\end{theorem}
\proof 
If~$D$ is any~$(7,6)_5$-code of size~$k$, then~$h(5,7,6,k)$ (where~$h$ is defined cf.~\eqref{LRfunctie}) is an upper bound on the number of pairs~$\{u,v\} \subseteq D$ with~$u \neq v$ and~$d_H(u,v)\neq 6$ (hence~$d_H(u,v)=7$), cf.~\eqref{hupperbound}.  The values~$h(5,7,6,k)$ are given in Table~\ref{UB}.

\begin{table}[ht]
\centering
\begin{tabular}{l|lllllllllll}
$k$ & 15 & 14 & 13 & 12 &11  &10  &9 & 8 & 7 & 6 & 5  \\\hline
$h(5,7,6,k)$      &0  & 0 & 1 & 3 & 6 & 10  & 8 & 7 & 7 & 8 & 10 
\end{tabular}
\caption{\small Upper bound~$h(5,7,6,k)$ on the number of pairs~$\{u,v\} \subseteq D$ with~$d_H(u,v)=7$ for a~$(7,6)_5$-code~$D$ with~$|D|=k$.}
\label{UB}
\end{table}

Recall that~$C$ is an~$(n,d)_q=(8,6)_5$-code with~$|C|=65$. We now give an upper bound on~$|X|$.  Let~$a^{(j)}_{k}$ be the number of symbols that appear exactly~$k$ times in column~$j$ of~$C$. Then the number of irregular pairs that have the same entry in column~$j$ is at most $\sum_{k =5}^{15} a^{(j)}_{k}h(5,7,6,k)$. It follows that 
\begin{align}\label{contra}
    |X| \leq U := \sum_{j=1}^8 \sum_{k =5}^{15} a^{(j)}_{k} h(5,7,6,k).
\end{align}
One may check that if~$\mathbf{a}, \mathbf{b} \in \mathbb{Z}_{\geq 0}^{15}$ are $15$-tuples of nonnegative integers, with $\sum_k a_k k =65$, $\sum_k b_k k =65$, $\sum_k a_k=5$, $\sum_k b_k=5$, and $f(a_{15},a_{14}) \leq f(b_{15},b_{14}) \neq 0$, then
\begin{align} \label{tocheck}
\sum_{k=5}^{15} (7a_k +b_k) h(5,7,6,k) < f(b_{15},b_{14}).
\end{align}
(There are~$30$~$\mathbf{a} \in \mathbb{Z}_{\geq 0}^{15}$ with~$\sum_k a_k k =65$  and~$\sum_{k}a_k = 5$. So there are~$900$ possible pairs~$\mathbf{a},\mathbf{b}$. A computer now quickly verifies~$(\ref{tocheck})$.)    

By permuting the columns of~$C$ we may assume that $\max_j f( a^{(j)}_{15},a^{(j)}_{14})=f(a^{(1)}_{15},a^{(1)}_{14})$. Hence if $f(a^{(1)}_{15},a^{(1)}_{14})>0$, then
\begin{align}
U &= \sum_{j=1}^8 \sum_{k =5}^{15} a^{(j)}_{k} h(5,7,6,k) = \frac{1}{7}\sum_{j=2}^8\left( \sum_{k =5}^{15} \left(7a^{(j)}_{k}  + a^{(1)}_{k} \right) h(5,7,6,k) \right) 
\\&< f(a^{(1)}_{15},a^{(1)}_{14}) \leq |X| \notag 
\end{align}
(where we used Prop.~$\ref{LBX}$ in the last inequality), contradicting~$(\ref{contra})$. So $f( a^{(j)}_{15},a^{(j)}_{14})=0$ for all~$j$, which implies (for~$\mathbf{a}^{(j)} \in \mathbb{Z}_{\geq 0}^{15}$ with~$\sum_k a^{(j)}_k k =65$, $\sum_k a^{(j)}_k=5$) that~$a^{(j)}_{15}=a^{(j)}_{14}=0$ for all~$j$, hence each symbol appears exactly~$13$ times in each column of~$C$.
\endproof 

\begin{corollary}
$A_5(9,6) \leq 325$,~$A_5(10,6) \leq 1625$ and~$A_5(11,6) \leq 8125$. 
\end{corollary}
\proof 
By Theorem~$\ref{65th}$ and~$(\ref{elementarybounds2})$.
\endproof 

\section{Improved bound on~\texorpdfstring{$A_3(16,11)$}{A3(16,11)}}\label{4118}

We show that~$A_3(16,11) \leq 29$ using a surprisingly simple argument.  

\begin{proposition}
$A_3(16,11) \leq 29$. 
\end{proposition}
\proof 
Suppose that~$C$ is an~$(n,d)_q=(16,11)_3$-code of size~$30$. We can assume that~$\mathbf{0} \in C$. It is known that~$A_3(15,11)=10$, so the symbol~$0$ is contained at most~$10$ times in every column of~$C$. Since~$|C|=30$, the symbol~$0$ appears exactly~10 times in every column of~$C$, so the number of 0's in~$C$ is divisible by~$5$. On the other hand,
Corollary~\ref{cor} yields that any~$(15,11)_3$-code of size~$10$ is equidistant with distance~$10$, since $h(3,15,11,10)=0$. This implies that all distances in a~$(16,11)_3$-code of size~$30$ belong to~$\{11, 16\}$. So the number of~$0$'s in any code word~$\neq \mathbf{0}$ is~$0$ or~$5$. As~$\mathbf{0}$ contains~$16$ 0's, it follows that the total number of~0's is not divisible by~$5$, a contradiction.
\endproof

\section{Codes from symmetric nets}\label{symsec}

In this section we will show that there is a~$1$-$1$-correspondence between \emph{symmetric $(\mu ,q)$-nets} and $(n,d)_q=(\mu q,\mu q-\mu)_q$-codes of size~$\mu q^2$. From this, we derive in Section~$\ref{496}$ the new upper bound~$A_4(9,6) \leq 120$, implying~$A_4(10,6) \leq 480$.

\begin{defn}[Symmetric net] Let~$\mu,q \in \mathbb{N}$. A \emph{symmetric $(\mu, q)$-net} (also called \emph{symmetric transversal design}~$\cite{beth}$) is a set~$X$ of~$\mu q^2$ elements, called \emph{points}, together with a collection~$\mathcal{B}$ of subsets of~$X$ of size~$\mu q$, called \emph{blocks}, such that:\indexadd{symmetric net}\indexadd{symmetric transversal design}
\begin{enumerate}[(s1)]
  \setlength{\itemsep}{1pt}
  \setlength{\parskip}{0pt}
  \setlength{\parsep}{0pt}
\item $\mathcal{B}$ can be partitioned into~$\mu q$ partitions (\emph{block parallel classes}) of~$X$.
\item Any two blocks that belong to different parallel classes intersect in exactly~$\mu$ points.
\item $X$ can be partitioned into~$\mu q$ sets of~$q$ points (\emph{point parallel classes}), such that any two points from different classes occur together in exactly~$\mu$ blocks, while any two points from the same class do not occur together in any block.\footnote{That is, a symmetric~$(\mu,q)$-net is a~$1-(\mu q^2, \mu q, \mu q)$ design~$D$, which is resolvable (s1), affine (s2), and the dual design~$D^*$ of~$D$ is affine resolvable (s3).}
\end{enumerate} 
\begin{remark}  From the 1-1-correspondence between symmetric~$(\mu,q)$-nets and $(n,d)_q=(\mu q,\mu q - \mu)_q$-codes~$C$ of size~$\mu q^2$ in Theorem~\ref{1rel} below it follows that~(s2) and~(s3) can be replaced by the single condition: 
\begin{itemize} 
\item[(s')] Each pair of points is contained in at most~$\mu$ blocks,
\end{itemize} 
since the only condition posed on such a code is that~$g(u,v) \leq \mu$ for all distinct~$u,v \in C$. 
\end{remark}
\end{defn}

\begin{examp}
Let~$X=\{1,2,3,4\}$ and~$\mathcal{B}=\{\{1,3\},\,\{2,4\},\,\{1,4\},\,\{2,3\} \}$. Then $(X,\mathcal{B})$ is a symmetric $(1,2)$-net. The block parallel classes are $\{\{1,3\},\,\{2,4\}\}$ and $\{\{1,4\},\,\{2,3\}\}$. The point parallel classes are $\{1,2\}$ and~$\{3,4\}$. 
\end{examp}

\noindent By labeling the points as~$x_1,\ldots,x_{\mu q^2}$ and the blocks as~$B_1,\ldots,B_{\mu q^2}$,  the~\emph{$\mu q^2 \times \mu q^2$-incidence matrix}~$N$ of a symmetric $(\mu,q)$-net is defined by
\begin{align}
N_{i,j} := \begin{cases}1 &\mbox{if } x_i \in B_j, \\
0 &\mbox{otherwise}.
\end{cases}
\end{align}

 If~$(X,\mathcal{B})$ and~$(X',\mathcal{B}')$ are symmetric nets, an \emph{isomorphism} of~$(X,\mathcal{B})$ and~$(X',\mathcal{B}')$ is a bijection~$\phi\, :\, X \to X' $ with the property that~$\{\{\phi(x) \, | \, x \in B \} \, | \, B \in \mathcal{B} \} = \mathcal{B}'$. That is, two symmetric nets are isomorphic if and only if their incidence matrices are the same up to row and column permutations.\indexadd{symmetric net!isomorphism of symmetric nets}\indexadd{symmetric transversal design!isomorphism of symmetric transversal\\ designs} Symmetric nets are a generalization of \emph{generalized Hadamard matrices}. 

\begin{defn}[Generalized Hadamard matrix]\indexadd{generalized Hadamard matrix}\indexadd{matrix!generalized Hadamard matrix}
Let~$M$ be an~$n \times n$ matrix with entries from a finite group~$G$. Then~$M$ is called a \emph{generalized Hadamard matrix} GH$(n,G)$ (or~GH$(n,|G|)$) if for any two different rows~$i$ and~$k$, the~$n$-tuple $(M_{ij}M_{kj}^{-1})_{j=1}^n$ contains each element of~$G$ exactly~$n/|G|$ times. 
\end{defn}

\begin{figure}[ht] \footnotesize
\begin{align*} 
\left( \begin{array}{cccccccc}
e & e &e &e &e &e &e & e\\
e & e &  a &a &b &b &c &c  \\
e & b & e  &b & c& a&c &a  \\
e &  c&  c &e &a & b& b& a \\
e & a & b  &c &e &a & b& c \\
e &  c&  b & a&c &e &a & b \\
e &b  & a  & c& a& c&e & b \\
e &  a&  c &b &b &c & a&e  \\ \end{array} \right)
\end{align*}
    \caption{\small An incidence matrix of the unique (up to isomorphism) symmetric~$(2,4)$-net is obtained by writing the elements~$e,a,b,c$ as $4 \times 4$-permutation matrices in the generalized Hadamard matrix GH$(8,V_4)$ (with~$V_4$ the Klein 4-group). See Al-Kenani~$\cite{42net}$.}
\end{figure} 

 Each generalized Hadamard matrix~GH$(n,G)$ gives rise to a symmetric~$(n/|G|,|G|)$-net: by replacing each element $g$ of~$G$ by the~$G \times G$ (permutation) matrix of the linear map~$\C G \to \C G $ determined by   $h \mapsto gh$ ($\forall h \in G$), one obtains the incidence matrix of a symmetric net. Not every symmetric~$(n/q,q)$-net gives rise to a generalized Hadamard matrix~GH$(n,q)$, see~$\cite{matrixnet}$. But if the group of automorphisms (\emph{bitranslations}) of a symmetric $(n/q,q)$-net has order~$q$, then one can construct a generalized Hadamard matrix~GH$(n,q)$ from it. See~$\cite{beth}$ for details. 

\begin{assumption} 
In this section we consider triples~$(n,d)_q$ of natural numbers for which
\begin{align}
qd = (q-1)n,
\end{align} 
hence~$n-d=n/q=:\mu$ and~$\mu \in \N$. So~$(n,d)_q=(\mu q, \mu q - \mu)_q$. 
\end{assumption}

  The fact that a generalized Hadamard matrix~$\text{GH}(n,q)$ gives rise to an~$(n,d)_q$-code of size~$qn$, was proved in~$\cite{plotkin}$ and for some parameters it can also be deduced from~$\cite{zinoviev2}$. Using a result by Bassalygo, Dodunekov, Zinoviev and Helleseth~$\cite{granking}$ about the structure of~$(n,d)_q$-codes of size~$qn$,\footnote{Note that~$A_q(n,d)\leq qn$, since by the Plotkin bound~$(\ref{elementarybounds})$, $A_q(n-1,d) \leq n$, hence~$A_q(n,d) \leq qn =\mu q^2$ by~$(\ref{elementarybounds2})$.} we prove that such codes are in 1-1-relation with symmetric~$(n/q,q)$-nets. 

\begin{theorem} \label{1rel} Let~$\mu, q \in \N$. There is a~$1$-$1$-relation between  symmetric~$(\mu,q)$-nets (up to isomorphism) and $(n,d)_q=(\mu q,\mu q -\mu)_q $-codes~$C$ of size~$\mu q^2$ (up to equivalence). 
\end{theorem}
\proof 
Given an $(n,d)_q= (\mu q,\mu q -\mu)_q$-code~$C$ of size~$\mu q^2$, we construct a $\{0,1\}$-matrix~$M$ of order $\mu q^2 \times \mu q^2$ with the following properties:
\begin{enumerate}[(I)]
\item \label{I} $M$ is a~$\mu q^2 \times \mu q^2$ matrix that consists of~$q \times q$ blocks~$\sigma_{i,j}$ (so~$M$ is a~$\mu q \times \mu q$ matrix of blocks~$\sigma_{i,j}$), where each~$\sigma_{i,j}$ is a permutation matrix.
\item \label{II} $MM\T= M\T M= A$, where~$A$ is the~$\mu q^2 \times \mu q^2$ matrix that consists of~$q \times q$ blocks~$A_{i,j}$ (so~$A$ is an~$\mu q \times \mu q$ matrix of blocks~$A_{i,j}$), with
\begin{align} \label{M}
A_{i,j} =\begin{cases} \mu q \cdot I_q &\mbox{if } i =j, \\
\mu \cdot J_q & \mbox{if } i \neq j. \end{cases}
\end{align}
Here~$J_q$ denotes the~$q \times q$ all-ones matrix.
\end{enumerate}
 By Proposition 4 of~$\cite{granking}$, since~$d=n(q-1)/q$ and~$|C|=qn$, $C$ can be partitioned as
\begin{align} \label{partition} 
C = V_1 \cup V_2 \cup \ldots \cup V_{n},
\end{align}
where the union is disjoint,~$|V_i|=q$ for all~$i=1,\ldots,n$, and where~$d_H(u,v)=n$ if~$u,v \in C$ are together in one of the~$V_i$, and~$d_H(u,v)=d$ if~$u\in V_i$ and~$v \in V_j$ with~$i\neq j$.

Now we write each word~$w \in [q]^{n}$ as a $\{0,1\}$-row vector of size~$qn = \mu q^2$ by putting a~$1$ on positions~$(i,w_i) \in \{1,\ldots,n\} \times [q]$ (for~$i=1,\ldots,n$) and 0's elsewhere. The $q$~words in any of the~$V_i$ then form a~$q \times qn$ matrix consisting of~$n$ permutation matrices~$\sigma_{i,j}$ of size~$q \times q$. 

By placing the matrices obtained in this way from all~$n$ tuples~$V_1,\ldots,V_{n}$  underneath each other, we obtain a~$qn \times qn$ matrix~$M$ consisting of~$n^2$ permutation matrices of order~$q\times q$, so~(\ref{I}) is satisfied. Property~(\ref{II}) also holds, since for any~$u,v \in C$ written as row vectors of size~$qn$, with the~$V_i$ as in~$(\ref{partition})$, it holds that
\begin{align} 
\sum_{k \in  \{1,\ldots,n\} \times [q]}u_kv_k= g(u,v)=\begin{cases}n = \mu q &\mbox{if } u = v, \\
0  &\mbox{if } u \neq v \text{ and } u,v \in V_i,\\
n-d=\mu  &\mbox{if } u \neq v \text{ and } u \in V_i, v \in V_j \text{ with } i \neq j.\\
\end{cases}
\end{align}
So~$MM\T=A$. Moreover, for any~$j_1:=(j_1',a_1) \in \{1,\ldots,n\} \times [q]$ and~$j_2:=(j_2',a_2) \in \{1,\ldots,n\} \times [q]$,  
\begin{align} 
\sum_{k \in  \{1,\ldots,qn\}}M_{k,j_1} M_{k,j_2}&= | \{ w \in C \, | \, w_{j_1'}=a_1, \, w_{j_2'}=a_2 \}  | \notag
\\ &=\begin{cases}n = \mu q  &\mbox{if } j_1'=j_2' \text{ and } a_1=a_2, \\
0  &\mbox{if } j_1' = j_2' \text{ and }  a_1\neq a_2,\\
n/q = \mu  &\mbox{if } j_1' \neq j_2',\\
\end{cases}
\end{align}
where the statement that $  | \{ w \in C \, | \, w_{j_1'}=a_1, \, w_{j_2'}=a_2 \}  | =n/q$ if~$j_1' \neq j_2'$ follows by considering the words in~$C$ that have~$a_1$ at the~$j_1'$-th position. (The remaining columns form an~$n$-block for the~$j_1'$-th column. In this~$n$-block, each symbol occurs exactly~$n/q$ times at each position, since~$h(q,n-1,d,n)=0$ (where~$h$ is defined in~\eqref{LRfunctie}).) We see that also~$M\T M=A$. Hence,~$M$ is the incidence matrix of a symmetric~$(\mu,q)$-net (see~$\cite{beth}$, Proposition~I.7.6 for the net and its dual). 

Note that one can do the reverse construction as well: given a symmetric~$(\mu,q)$-net, the incidence matrix of~$M$ can be written (after possible row and column permutations) as a matrix of permutation matrices such that~$MM\T=M\T M=A$, with~$A$ as defined in~$(\ref{M})$. From~$M$ we obtain a code~$C$ of size~$\mu q^2$ of the required minimum distance by mapping the rows $(i,w_i) \in \{1,\ldots,\mu q\} \times [q]$ to~$w \in [q]^{\mu q}$. Observe that equivalent codes yield isomorphic incidence matrices~$M$ and vice versa.
\endproof 

\begin{figure}[ht]
{\footnotesize$$
\begin{blockarray}{c|ccc|}
    & & &  \\  \cline{1-4}
 \begin{block}{c|ccc|}
   w_0 &   0 & 0 & 0 \\ 
   w_1 &   1 & 1 & 1 \\
   w_2 &   2 & 2 & 2 \\ 
   w_4 &   0 & 2 & 1 \\ 
   w_5 &   1 & 0 & 2 \\ 
   w_6 &   2 & 1 & 0 \\ 
   w_7 &   0 & 1 & 2 \\ 
   w_8 &   1 & 2 & 0 \\  
   w_9 &   2 & 0 & 1 \\       \end{block}
\end{blockarray} \,\, \quad \longleftrightarrow
\,\,\quad 
\begin{blockarray}{c|ccc|ccc|ccc|}
 &  0 & 1 & 2 & 0 & 1 & 2 & 0 & 1 & 2  \\  \cline{1-10}
 \begin{block}{c|ccc|ccc|ccc|}
   w_1 &   1 & 0 & 0 & 1 & 0 &0  & 1 & 0 & 0\\
   w_2 &   0 & 1 & 0 & 0 & 1 & 0 & 0 & 1 & 0\\
   w_3 &   0 & 0 & 1 & 0 & 0 & 1 & 0 & 0 & 1\\\cline{1-10}
   w_4 &   1 &0  & 0 & 0 & 0 & 1 & 0 & 1 &0 \\
   w_5 &   0 & 1 & 0 & 1 & 0 & 0 & 0 & 0 &1 \\
   w_6 &   0 &  0&  1& 0 & 1 & 0 & 1 & 0 &0 \\\cline{1-10}
   w_7 &   1 & 0 &  0& 0 & 1 & 0 & 0 & 0 &1 \\
   w_8 &   0 & 1 &  0& 0 & 0 & 1 & 1 & 0 & 0\\
   w_9 &   0 & 0 & 1 & 1 & 0 & 0 & 0 & 1 &0 \\  \end{block}
\end{blockarray} 
$$ }
    \caption{\small An~$(n,d)_q=(3,2)_3$-code~$C=\{w_1,\ldots,w_9\}$ of size~$9$ (left table) gives rise to an incidence matrix of a symmetric~$(1,3)$-net (right table) and vice versa.}
\end{figure}

\section{New upper bound on~\texorpdfstring{$A_4(9,6)$}{A4(9,6)}}\label{496}

\noindent In this section we use the 1-1-correspondence between symmetric~$(\mu,q)$-nets and~$(n,d)_q=(\mu q, \mu q - \mu)_q$-codes of size~$\mu q^2$ in combination with a known result about symmetric~$(2,4)$-nets~$\cite{42net}$ to derive that~$A_4(9,6) \leq 120$.

As~$A_4(8,6)=32$, any~$(9,6)_4$-code of size more than~$120$ must contain at least one~$31$- or~$32$-block, and therefore it contains a subcode of size~$120$ containing at least one~$31$-block. We will show (using a small computer check) that this is impossible because
\begin{align} \label{only30}
\text{each~$(9,6)_4$-code of size~$120$ only admits~$30$-blocks.} 
\end{align} 
Therefore~$A_4(9,6) \leq 120$. To prove~\eqref{only30}, we need information about~$(8,6)_4$-codes of size~$31$.

\begin{proposition}\label{31code}
Let~$q,n,d \in \N$ satisfy~$qd=(q-1)n$. Any~$(n,d)_q$-code~$C$ of size~$qn-1$ can be extended to an~$(n,d)_q$-code of size~$qn$.
\end{proposition}
\proof 
Let~$C$ be an~$(n,d)_q$-code of size~$qn-1$. Recall that~$c_{\alpha,j}$ denotes the number of times symbol~$\alpha \in [q]$ appears in column~$j \in \{1,\ldots,n\}$ of~$C$. By the Plotkin bound, $A_q(n-1,d) \leq n$, so each symbol occurs at most~$n$ times in each column of~$C$, hence there exists for each~$j \in \{1,\ldots,n\}$ a unique~$\beta_j \in [q]$ with~$c_{\beta_j,j} =n-1$, implying~$c_{\alpha,j}=n$ for all~$\alpha \in [q] \setminus \{\beta_j \}$. We can define a~$qn$-th codeword~$u$ by putting~$u_j:= \beta_j$ for all~$j=1,\ldots,n$. We claim that~$C \cup \{ u \}$ is an~$(n,d)_q$-code of size~$qn$.

To establish the claim we must prove that~$d_H(u,w) \geq d$ for all~$w \in C$. Let~$w \in C$ with~$d_H(u,w) < n$. We can renumber the symbols in each column of~$C$ such that~$w=\mathbf{0}$. Then~$w$ is contained in an~$(n-1)$-block~$B$ for some column in~$C$ (otherwise~$d_H(u,w)=n$). The number of~0's in~$B$ is $n+(n-2)(n-d)=n+(n-2)n/q$ (since any $(n-1,d)_q$-code of size~$n-1$ is equidistant with distance~$d$ by Corollary~\ref{cor}, as~$h(q,n-1,d,n-1)=0$).

Moreover, the number of~0's in~$C \setminus B$ is~$(q-1)(n-1)n/q$ (since in any $(n-1,d)_q$-code of size~$n$, each symbol appears exactly~$n/q$ times in each column by Corollary~\ref{cor}, as~$h(q,n-1,d,n)=0$). Adding these two numbers we see that the total number of~0's in~$C$ is~$n^2-n/q$. Since~$C \cup \{u\}$ contains each symbol~$n^2$ times by construction,~$u$ contains symbol~$0$ exactly~$n/q$ times, hence~$d_H(u,w)=n-n/q=d$, which gives the desired result.
\endproof

\begin{proposition}\label{120}
$A_4(9,6) \leq 120$. 
\end{proposition} 
\proof 
The~$(n,d)_q=(8,6)_4$-code of size~$32$ is unique up to equivalence, since the symmetric~$(2,4)$-net is unique up to equivalence (see Al-Kenani~$\cite{42net}$). By checking all~$(8,6)_4$-codes~$D$ of size~$31$ (of which there are at most~$32$ up to equivalence since each~$(8,6)_4$-code of size~$31$ arises by removing one word from an~$(8,6)_4$-code of size~$32$ by Proposition~\ref{31code}) we find that for each~$(8,6)_4$-code~$D$ of size~$31$:
\begin{align} 
|\{ u \in [4]^{8} \,\, | \,\,d_H(u,w) \geq 5 \,\,\,\, \forall\, w \in D \}| \leq 25. 
\end{align} 
So any~$(9,6)_4$ code with a~$31$- or~$32$-block contains at most~$31+25=56$ words.
Hence a $(9,6)_4$-code of size~$120$ cannot contain a~$31$- or~$32$-block. So~\eqref{only30} holds, hence $A_4(9,6) \leq 120$.
\endproof 

\begin{corollary}
$A_4(10,6) \leq  480$.
\end{corollary}
\proof 
By Proposition~$\ref{120}$ and~$(\ref{elementarybounds2})$.
\endproof

\chapter{Semidefinite programming bounds for constant weight codes} \label{cw4chap}\vspace{-6pt}
\chapquote{Further study of this function  is being pursued.}{Selmer Martin Johnson (1916--1996)} \vspace{-4pt}

\noindent For nonnegative integers~$n,d,w$, let~$A(n,d,w)$ be the maximum size of a code~$C \subseteq \F_2^n$ with constant weight~$w$ and minimum distance at least~$d$. We consider two semidefinite programs based on quadruples of code words that yield several new upper bounds on~$A(n,d,w)$. The new upper bounds imply the exact values $A(22,8,10)=616$ and  $A(22,8,11)=672$. Lower bounds on~$A(22,8,10)$ and~$A(22,8,11)$ are obtained from the~$(n,d)_2=(22,7)_2$ shortened binary Golay code of size~$2048$. It can be concluded that the shortened binary Golay code is a union of constant weight~$w$ codes of sizes~$A(22,8,w)$. This chapter is based on~\cite{cw4}.

\section{Introduction}

In this chapter we consider two semidefinite programming upper bounds on $A(n,d,w)$. Both upper bounds sharpen the classical Delsarte linear programming bound~$\cite{delsarte}$, as well as Schrijver's semidefinite programming bound based on a block diagonalization of the Terwilliger algebra~$\cite{schrijver}$.

The chapter serves the following purposes. Firstly, the quadruple bound for unrestricted binary codes from~$\cite{semidef}$, which is a slight sharpening of the bound~$B_2(n,d)$ from~$\eqref{8no15a}$,  is adapted to a bound~$A_k(n,d,w)$ for binary constant weight codes. Subsequently, a relaxation~$B_k(n,d,w)$ is formulated, which might also be of interest for unrestricted binary codes. By studying~$A_4(n,d,w)$ and~$B_4(n,d,w)$,  a sharpening of the Schrijver bound~$\cite{schrijver}$ for constant weight codes is obtained that is in most cases sharper than the bound from~$\cite{KT}$ (in which linear inequalities were added to the Schrijver bound). The constructed semidefinite programs are very large, but a symmetry reduction (using representation theory of the symmetric group) based on the method of Section~\ref{prelimrep} and Chapter~\ref{orbitgroupmon} is given to reduce them to polynomial size. This finally leads to many new upper bounds on~$A(n,d,w)$, including the exact values~$A(22,8,10)=616$ and~$A(22,8,11)=672$ --- see Table~\ref{28}.

The once shortened binary Golay code (see Section~\ref{codingprem}) which is an~$(n,d)=(22,7)$-code of size 2048, contains the   numbers of words of a given weight~$w$ (and no words of other weights) which are given in  Table~\ref{ShortenedGolay}. 
\begin{table}[ht]  
\centering
\begin{tabular}{l|lllllll} \small
weight~$w$ & 0 & 7 & 8 & 11 &12 &15 &16 \\\hline
$\# $ words  &1  & 176 & 330 & 672 & 616 &176 &77
\end{tabular}
\caption{\small Number of words of a given weight~$w$ contained in the once shortened  binary Golay code.}
\label{ShortenedGolay}
\end{table}

 While it was already known that~$A(22,8,6)=77$, $A(22,8,7)=176$, and~$A(22,8,8)=330$ (here note that~$A(n,d,w)=A(n,d,n-w)$), the results of this chapter imply that $A(22,8,10)=616$ and~$A(22,8,11)=672$. So if one collects all words of a given weight~$w$ in the (once) shortened binary Golay code, the resulting code (if nonempty) is a constant weight code of maximum size. In other words, the shortened  binary Golay code   is a union of constant weight~$w$ codes of sizes~$A(22,8,w)$. The value~$A(22,8,10)=616$ together with the already known values implies that also the twice shortened extended binary   Golay code (which is an~$(n,d)=(22,8)$-code of size~1024) has this property, since it contains~$1$,~$330$,~$616$, and~$77$ words of weight~$0$,~$8$,~$12$, and~$16$, respectively.   It was already known that the  binary Golay code, the extended binary  Golay code and the once shortened extended binary  Golay code have this property, i.e., that they are unions of constant weight codes of sizes~$A(n,d,w)$.

 Several tables with bounds on~$A(n,d,w)$ have been given in the literature~$\cite{agrell,table3, table2}$.  Tables with best currently known upper and lower bounds can be found on the website of Andries Brouwer~$\cite{brouwertable}$.

\subsection{The upper bounds~\texorpdfstring{$A_k(n,d,w)$}{Ak(n,d,w)} and~\texorpdfstring{$B_k(n,d,w)$}{Bk(n,d,w)}}

We describe two upper bounds on~$A(n,d,w)$ based on quadruples of code words.  Fix $n,d,w \in \N$ and let~$N \subseteq \mathbb{F}_2^n$ be the set of all words of constant weight~$w$.\symlistsort{N}{$N$}{subset of~$\mathbb{F}_2^n$} For~$k \in \Z_{\geq 0}$, let~$\mathcal{C}_k$ be the collection of codes~$C \subseteq N$ with~$|C|\leq k$.\symlistsort{Ck}{$\mathcal{C}_k$}{collection of codes of cardinality at most~$k$}  For any~$D \in \mathcal{C}_k$, we define\symlistsort{Ck(D)}{$\mathcal{C}_k(D)$}{the set~$\{C \in \mathcal{C}_k \, \mid \, C \supseteq D, \, \lvert D\rvert +2\lvert C\setminus D\rvert \leq k  \}$}
\begin{align} 
\mathcal{C}_k(D) := \{C \in \mathcal{C}_k \,\, | \,\, C \supseteq D, \, |D|+2|C\setminus D| \leq k  \}.  
\end{align} 
Note that then~$|C \cup C'|= |C| + |C'| -|C\cap C'| \leq 2|D| + |C \setminus D| + |C' \setminus D| -|D| \leq k$ for all~$C,C' \subseteq \mathcal{C}_k(D)$.  
Furthermore, for any function~$x : \mathcal{C}_k \to \R$ and $D \in \mathcal{C}_k$, we define the~$\mathcal{C}_k(D) \times \mathcal{C}_k(D)$ matrix~$M_{k,D}(x)$ by 
\begin{align*}
M_{k,D}(x)_{C,C'} : = x(C \cup C'),
\end{align*}
for~$C,C' \in \mathcal{C}_k(D)$.\symlistsort{MkD}{$M_{k,D}(x)$}{variable matrix}
 Define the following number, which is an adaptation to constant weight codes of the upper bound for unrestricted codes based on quadruples of codewords from Gijswijt,  Mittelmann and Schrijver~\cite{semidef}:\symlistsort{Ak(n,d,w)}{$A_k(n,d,w)$}{upper bound on~$A(n,d,w)$}
\begin{align} \label{A4ndw}
A_k(n,d,w):=    \max\big\{ \mbox{$\sum_{v \in N} x(\{v\})$}\,\, |\,\,&x:\mathcal{C}_k \to \R, \,\, x(\emptyset )=1, x(S)=0 \text{ if~$d_{\text{min}}(S)<d$},\notag\\
&  M_{k,D}(x) \succeq 0 \text{ for each~$D$ in~$\mathcal{C}_k$}\big\}.  
\end{align}
In this chapter, we first consider~$A_4(n,d,w)$. Even after reductions (see the next subsection), the semidefinite program for computing~$A_{4}(n,d,w)$ is large in practice, although~$A_k(n,d,w)$ can be computed in polynomial time for fixed~$k$. In computing $A_k(n,d,w)$, the matrix blocks coming from the matrices~$M_{k,D}(x)$ for~$D = \emptyset$  if~$k$ is even, and for~$|D|=1$ if~$k$ is odd, are often larger (in size and importantly, more variables occur in each matrix entry, yielding large semidefinite programs) than the blocks coming from~$M_{k,D}(x)$ for~$D$ with~$|D|\geq 2$. This observation has suggested the following relaxation of~$A_k(n,d,w)$, which is sharper than~$A_{k-1}(n,d,w)$ for~$k\geq 4$ (while it equals~$A_{k-1}(n,d,w)$ for~$k=3$, so assume~$k \geq 3$ in the following definition).\symlistsort{Bk(n,d,w)}{$B_k(n,d,w)$}{upper bound on~$A(n,d,w)$}
\begin{align} \label{Bndw}
\hspace{-1pt}B_k(n,d,w):=   \max \big\{ \mbox{$\sum_{v \in N} x(\{v\})$}\,\, |\,\,&x:\mathcal{C}_k \to \R, \,\, x(\emptyset )=1, x(S)=0 \text{ if~$d_{\text{min}}(S)<d$},\notag\\
& M_{k-1,D}(x|_{\CC_{k-1}}) \succeq 0 \text{ for each~$D \in \mathcal{C}_{k-1}$ with~$|D|<2$}, \notag 
\\
& M_{k,D}(x) \succeq 0 \text{ for each~$D \in \mathcal{C}_{k}$ with~$|D|\geq 2$}\big\}. 
\end{align}
\begin{proposition} 
For all~$k,n,d,w \in \N$ with~$k \geq 3$,  we have that $ A_{k-1}(n,d,w) \geq B_k(n,d,w) \geq A_k(n,d,w) \geq A(n,d,w)$.
\end{proposition}
\proof
It is not hard to see that~$B_k(n,d,w) \geq A_k(n,d,w)$, as all constraints in~$(\ref{Bndw})$ follow from~$(\ref{A4ndw})$. Similarly, it follows that~$A_{k-1}(n,d,w)\geq B_k(n,d,w)$.

To see~$A_k(n,d,w) \geq A(n,d,w)$, let~$C \subseteq \mathbb{F}_2^n$ be a constant weight~$w$ code with $d_{\text{min}}(C)\geq d$ and $|C| = A(n,d,w)$. Define $x: \mathcal{C}_k \to \R$ by $x(S)=1$ if $S \subseteq C$ and $x(S)=0$ else, for~$S \in \mathcal{C}_k$. Then~$x$ satisfies the conditions in~$(\ref{A4ndw})$, where the condition that $M_{k,D}(x) \succeq 0$ is satisfied since $M_{k,D}(x)_{C,C'}=x(C)x(C')$ for all $C,C'\in \mathcal{C}_k(D)$. Moreover, the objective value equals $\sum_{v \in N} x(\{v\}) =|C|=A(n,d,w)$, which gives $A_k(n,d,w) \geq A(n,d,w)$.
\endproof

\noindent This chapter considers~$A_4(n,d,w)$ and~$B_4(n,d,w)$, that is,~$k=4$. By symmetry we assume throughout that~$w \leq n/2$ (otherwise, add the all-ones word to each word in~$\mathbb{F}_2^n$, where the addition is in the vector space~$\mathbb{F}_2^n$,  and replace~$w$ by~$n-w$). For computing~$A_4(n,d,w)$, it suffices to require that the matrices~$M_{4,D}(x)$ with~$|D|$ even are positive semidefinite. To see this, note that if~$D \subseteq C$ with~$|D|$ even and~$|C|=|D|+1$, then~$\mathcal{C}_4(C) \subseteq \mathcal{C}_4(D)$, i.e.,~$M_{4,C}(x)$ is a principal submatrix of~$M_{4,D}(x)$ and hence positive semidefiniteness of~$M_{4,D}(x)$ implies positive semidefiniteness of~$M_{4,C}(x)$. For computing~$B_4(n,d,w)$, it suffices to require that the matrices~$M_{3,D}(x|_{\CC_{3}})$ for each~$D \in \mathcal{C}_{3}$ with~$|D| \leq 1$ and the matrices~$M_{4,D}(x)$ for each~$D \in \mathcal{C}_{4}$ with~$|D| \in \{2,4\}$ are positive semidefinite.

If~$|D|=4$, then~$M_{4,D}(x)$ is a matrix of order~$1 \times 1$, so positive semidefiniteness of~$M_{4,D}(x)$ is equivalent to~$x(D) \geq 0$.
We can assume in~$(\ref{A4ndw})$ and~$(\ref{Bndw})$ that~$x \, : \, \mathcal{C}_4 \to \mathbb{R}_{\geq 0}$, since if~$|D|\leq 4$ then~$x(D)$ occurs on the diagonal of~$M_{4,D}(x)$ and if~$|D| \leq 1$  then~$x(D)$ occurs on the diagonal of~$M_{3,D}(x|_{\CC_{3}})$.

\begin{table}[!htb]\small
\begin{center}
    \begin{tabular}{| r | r | r|| r |>{\bfseries}r |r | r| r|}
    \hline
    $n$ & $d$ & $w$ &  \multicolumn{1}{>{\raggedright\arraybackslash}b{16mm}|}{best lower bound known}  & \multicolumn{1}{>{\raggedright\arraybackslash}b{21mm}|}{\textbf{new upper bound}} & \multicolumn{1}{>{\raggedright\arraybackslash}b{18mm}|}{best upper bound previously known} & \multicolumn{1}{>{\raggedright\arraybackslash}b{22mm}|}{$\floor{A_3(n,d,w)}$}& \multicolumn{1}{>{\raggedright\arraybackslash}b{14 mm}|}{Delsarte bound}\\\hline 
    17 & 6 & 7  & 166 & 206*  & 207 & 228 & 249 \\ 
    18 & 6 & 7  & 243 & 312*  & 318 & 353 & 408 \\ 
    19 & 6 & 7  & 338 & 463*  & 503 & 526 & 553\\ 
    19 & 6 & 8  & 408 & 693  & 718 & 718 & 751\\ 
    20 & 6 & 8  & 588 & 1084  &1106 &1136 & 1199\\  
    21 & 6 & 8  & 775 & 1665 &   1695 & 1772 & 1938\\
    21 & 6 & 9  & 1186& 2328 & 2359   &2359 & 2364\\     \hline 
    25 & 8 & 8  & 759 & 850 &  856  & 926  & 948 \\     
    21 & 8 & 9  & 280 & 294 &   302 & 314 & 358\\ 
    22 & 8 & 9  & 280 & 440 &  473 & 473 & 597 \\ 
    23 & 8 & 9  & 400 & 662 &  703  & 707  & 830\\
    24 & 8 & 9  & 640 & 968 &  1041 & 1041  & 1160\\
    25 & 8 & 9  & 829 & 1366 &  1486 & 1486 & 1626 \\
    26 & 8 & 9  & 887 & 1901 &  2104 &  2108 & 2282 \\
    27 & 8 & 9  & 1023 & 2616 & 2882   & 2918 & 3203 \\  
    22 & 8 & 10 & {\color{donkerrood}616} & {\color{donkerrood}616} &630   & 634 & 758 \\
    22 & 8 & 11 & {\color{donkerrood}672} & {\color{donkerrood}672} &  680  & 680 & 805\\    \hline
    27 & 10 & 9  &  118 & 291 &  293  &299 & 299  \\ 
    22 & 10 & 10  & 46 & 71 &   72  &  72 & 82 \\
    23 & 10 & 10  & 54 & 116 &  117 &117 & 117 \\   
    26 & 10 & 10  & 130 &  397&  406& 406  & 412 \\    
    27 & 10 & 10  & 162 & 555 & 571  & 571& 579  \\    
    22 & 10 & 11  & 46 & 79 &  80  &80 &88  \\\hline
    \end{tabular}
\end{center}
  \caption{\small An overview of the new upper bounds for constant weight codes. The best previously known bounds are taken from Brouwer's table~$\cite{brouwertable}$. The unmarked new upper bounds are instances of~$B_4(n,d,w)$  (rounded down), and the new upper bounds marked with ${}^*$ are instances of~$A_4(n,d,w)$ (rounded down). For comparison: $\floor{B_4(17,6,7)}=213$, $\floor{B_4(18,6,7)}=323$ and~$\floor{B_4(19,6,7)}=486$. The best previously known upper bounds on~$A(22,8,10)$ and~$A(22,8,11)$ are from~$\cite{KT}$ and~$\cite{schrijver}$, respectively. The new upper bounds imply that~$A(22,8,10)=616$ and~$A(22,8,11)=672$ (marked in red).\label{28}}
\end{table}

\subsection{Exploiting the symmetry of the problem}
 Fix~$k \in \N$ with~$k \geq 2$. The group~$H:=S_n$ acts naturally on~$\mathcal{C}_k$ by simultaneously permuting the indices~$1,\ldots,n$ of each code word in~$C \in \mathcal{C}_k$ (since the weight of each codeword is invariant under this action), and this action maintains minimum distances and cardinalities of codes~$C \in \mathcal{C}_k$. We can assume that the optimum~$x$ in~$(\ref{A4ndw})$ (or~$(\ref{Bndw})$) is $H$-invariant. To see this, let~$x$ be an optimum solution for~$(\ref{A4ndw})$ (or~$(\ref{Bndw})$). For each~$g \in H$, the function~$g \cdot x$ is again an optimum solution, since the objective value of~$g \cdot x$ equals the objective value of~$x$ and~$g \cdot x$ still satisfies all constraints in~$(\ref{A4ndw})$ (or~$(\ref{Bndw})$). Since the feasible region is convex, the optimum~$x$ can be replaced by the average of~$g \cdot x$ over all~$g \in H$. This yields an $H$-invariant optimum solution.  
 
 Let~$\Omega_k$ be the set of~$H$-orbits on~$\mathcal{C}_k$.\symlistsort{Omegak}{$\Omega_k$}{set of $H$-orbits on $\mathcal{C}_k$} Then~$|\Omega_k|$ is bounded by a polynomial in~$n$. Since there exists an~$H$-invariant optimum solution, we can replace, for each~$\omega \in \Omega_k$ and~$C \in \omega$, each variable~$x(C)$ by a variable~$z(\omega)$.\symlistsort{z(omega)}{$z(\omega)$}{variable of reduced program} Hence, the matrices~$M_{k,D}(x)$ become matrices~$M_{k,D}(z)$ and we have considerably reduced the number of variables in~$(\ref{A4ndw})$ and~$(\ref{Bndw})$.\symlistsort{MkDz}{$M_{k,D}(z)$}{reduced variable matrix}

It is only required that we check positive semidefiniteness of~$M_{k,D}(z)$ for one code~$D$ in each~$H$-orbit of~$\mathcal{C}_k$, as for each~$g \in H$, the matrix~$M_{k,g(D)}(z)$ can be obtained by simultaneously permuting rows and columns of~$M_{k,D}(z)$.

 We sketch how to reduce these matrices in size. For~$D \in \mathcal{C}_k$, let~$H_D'$ be the subgroup of~$H$ consisting of all~$g \in H$ that leaves each element of~$D$ invariant.\symlistsort{HD'}{$H_D'$}{subgroup of~$H$ that leaves each element of~$D$ invariant} Then the action of~$H$ on~$\mathcal{C}_k$ gives an action of $H_D'$ on~$\mathcal{C}_k(D)$.  The simultaneous action of~$H_D'$ on the rows and columns of~$M_{k,D}(z)$ leaves~$M_{k,D}(z)$ invariant. Therefore, there exists a block-diagonalization (as explained in Section~\ref{symint}) given by $M_{k,D}(z) \mapsto U^*  M_{k,D}(z) U$ of~$M_{k,D}(z)$, for a matrix~$U$ depending on~$H_D'$ but  not depending on~$z$, such that the order of~$ U^* M_{k,D}(z) U$ is polynomial in~$n$ and such that~$M_{k,D}(z)$ is positive semidefinite if and only if each of the blocks of~$ U^*  M_{k,D}(z) U$ is positive semidefinite. In our case the matrix~$U$ can be taken to be real; so~$U^*=U\T$. The entries in each block of~$U\T  M_{k,D}(z) U$ are linear functions in the variables~$z(\omega)$ (with coefficients bounded polynomially in~$n$). Hence, we have reduced the size of the matrices involved in our semidefinite program to polynomial size.

Note that, after reductions, the number of variables involved in the semidefinite programs for computing~$A_k(n,d,w)$ and~$B_k(n,d,w)$ are the same. However, the program for computing~$B_k(n,d,w)$ contains fewer blocks than the program for computing~$A_k(n,d,w)$, and the blocks are smaller and contain fewer variables per matrix entry. This is important, as the semidefinite programs for computing~$A_4(n,d,w)$  for moderate values of~$n,d,w$   ---for example for the cases given in Table~\ref{28} below---  even after reductions turn out to be very large in practice (although they are of polynomial size).

For particular weights~$w$ (in the case of constant weight codes), the group of distance-preserving permutations of~$\mathcal{C}_k$ can be larger than~$S_n$. If~$w=n/2$ there is a further action of~$S_2$ on~$\mathcal{C}_k$ by adding the all-ones word~$\bm{1}$ to each word in each code~$S$ in~$\mathcal{C}_k$. (This operation is called `taking complements'; if~$S \in \CC_k$, we write~$S^c:=\{u+\bm{1} \, | \, u \in S \}$.)\symlistsort{Sc}{$S^c$}{the code~$S+\bm{1}$}\symlistsort{1}{$\bm{1}$}{all-ones word}  Since the corresponding reduction of the semidefinite program can only be used for specific weights~$w$, we do not consider the reduction in this chapter, although it was used for reducing the number of variables in computing~$B_4(22,8,11)$  (as we can assume~$x(S)=x(S^c)$ in~\eqref{Bndw}).

The reductions of the optimization problem will be described in detail in Section~$\ref{red}$. Table~$\ref{28}$ contains new upper bounds for~$n\leq 28$, which  is the range of~$n$ usually considered.  We did not compute bounds for all cases of~$n,d,w$ with~$n\leq 28$. Only  cases for which finding~$B_4(n,d,w)$  did not require excessive computing time or memory are  considered in the present work.  Since some tables on Brouwer's website~$\cite{brouwertable}$ also consider~$n$ in the range $29 \leq n \leq 32$, we give some new bounds for these cases (many of which are computed with the smaller program~$A_3(n,d,w)$) in Section~\ref{conccw4}  in Table~$\ref{32}$.  All improvements have been found using multiple precision versions of SDPA~\cite{sdpa}, where the largest program (for computing~$B_4(22,8,10)$) took approximately three weeks to compute on a modern desktop~pc. 

\subsection{Comparison with earlier bounds}

In this chapter we will consider~$B_4(n,d,w)$ and~$A_4(n,d,w)$, that is,~$k=4$. It can be proved that~$A_2(n,d,w)$ is equal to the Delsarte bound~$\cite{delsarte}$. The bound~$A_4(n,d,w)$ is an adaptation of the bound for unrestricted binary codes based on quadruples considered in~$\cite{semidef}$. (In~\cite{semidef}, this quadruple bound on~$A(n,d)$ is called~$A_4(n,d)$, which must not be confused with the parameter~$A_q(n,d)$ from~\eqref{aqndprem}.) The semidefinite programming bound for constant weight codes introduced by Schrijver in~$\cite{schrijver}$ is a slight sharpening of~$A_3(n,d,w)$ (in almost all cases it is equal to~$A_3(n,d,w)$).  The bound~$B_4(n,d,w)$, which is based on quadruples of code words, is a bound `in between'~$A_3(n,d,w)$ and~$A_4(n,d,w)$: it is the bound~$A_3(n,d,w)$ with constraints for matrices~$M_{4,D}(x)$ with~$|D|=2$ (based on quadruples of code words) added. Or it can be seen as a bound obtained from~$A_4(n,d,w)$ by removing the positive semidefiniteness of the (large) matrix~$M_{4,\emptyset}(x)$ and replacing it by the positive semidefiniteness of~$M_{3,\emptyset}(x|_{\CC_{3}})=M_{2,\emptyset}(x|_{\CC_{2}})$.

Recently, Kim and Toan~$\cite{KT}$ added linear inequalities to the Schrijver bound~$\cite{schrijver}$.  As it restricts to triples of codewords, their method has the advantage that the semidefinite programs are small and can be solved fast.  The bound~$B_4(n,d,w)$ is often sharper than their bound, but it takes much more time to compute.

\section{Reduction of the optimization problem}\label{red}
In this section we give the reduction of optimization problem~$(\ref{Bndw})$ for computing the bound $B_4(n,d,w)$, using the representation theory from Chapters~\ref{prem} and~\ref{orbitgroupmon}.  Also, we give a reduction for computing $A_4(n,d,w)$. First we consider block diagonalizing $M_{4,D}(z)$ for $D \in \mathcal{C}_4$ with $|D|=1$ or $|D|=2$, applied in computing both $B_4(n,d,w)$ and $A_4(n,d,w)$.\footnote{Note that if $|D|=1$ then $M_{4,D}(z)=M_{3,D}(z|_{\Omega_3})$. For computing $A_4(n,d,w)$, it is not necessary to consider the case $|D|=1$  separately, as $M_{4,D}(z)$ for $|D|=1$ is a principal submatrix of $M_{4,\emptyset}(z)$.} 
 Subsequently we consider the cases $M_{3,\emptyset}(z|_{\Omega_3}) =M_{2,\emptyset}(z|_{\Omega_2})$ or $M_{4,\emptyset}(z)$, applied in computing $B_4(n,d,w)$ or $A_4(n,d,w)$, respectively. Note that for the cases $|D|=3$ and $|D|=4$ the matrix $M_{4,D}(z) = (z(D))$  is its own block diagonalization, so then $M_{4,D}(z)$ is positive semidefinite if and only if $z(D) \geq 0$. 
 
 \begin{remark}
 If~$z : \Omega_k \to \R$,~$j \leq k$ and~$D \in \mathcal{C}_j$, we will from now on omit the restriction sign in~$M_{j,D}(z|_{\Omega_j})$. That is, we will write~$M_{j,D}(z)$ to denote the matrix~$M_{j,D}(z|_{\Omega_j})$.
 \end{remark}

\subsection{The cases~\texorpdfstring{$|D|=1$}{|D|=1} and~\texorpdfstring{$|D|=2$}{|D|=2}} \label{D2}
In this section we consider one code~$D \in \mathcal{C}_4$  with~$|D|=1$ or~$|D|=2$. We can assume that~$D=\{v_1,v_2\}$ with
\begin{align}
    v_1=&\overbrace{1\ldots1\,1\ldots1}^{w}\overbrace{0\ldots0\,\,0\ldots0}^{n-w} 
    \\v_2= &\underbrace{0\ldots0}_{t}\underbrace{1\ldots1\,1\ldots1}_{w}\underbrace{\,0\ldots0}_{n-t-w},    \notag 
\end{align}
where~$t \in \Z_{\geq 0}$ with~$t=0$ or~$d/2 \leq t \leq w$. For the remainder of this section, fix~$t \in \{0\} \,\cup\, \{t \,\,|\,\, d/2 \leq t \leq w \}$ (recall that~$w \leq n/2$, so~$t \leq n-w$). If~$t=0$, then~$|D|=1$ and if~$d/2 \leq t \leq w$, then~$|D|=2$. The rows and the columns of~$M_{4,D}(z)$ are parametrized by codes~$C\supseteq D$ of size at most~$3$ (if~$|D|=2$) or size at most~$2$ (if~$|D|=1$). 

Let~$H_D'$ be the group of distance-preserving permutations of~$\mathcal{C}_4$ that fix~$v_1$ and~$v_2$. So
\begin{align}
    H_D' \cong S_t \times S_{w-t} \times S_t \times S_{n-t-w}.
\end{align}
We first describe a representative set for the action of~$H_D' $ on~$(\C^{\mathbb{F}_2})^{\otimes j_1} \otimes \ldots \otimes (\C^{\mathbb{F}_2})^{\otimes j_4} \cong \mathbb{C}^{\mathbb{F}_2^n}$ and then restrict to words of weight~$w$ and distance at least~$d$ to both words in~$D$. Let~$e_j$ denote the~$j$-th standard basis vector of~$\mathbb{C}^{\mathbb{F}_2}$, for~$j=0,1$. 

We use the notation of Section~\ref{prelimrep} and Chapter~\ref{orbitgroupmon}, most importantly of Section~\ref{genmult}. Let $$(j_1,j_2,j_3,j_4):=(t,w-t,t,n-t-w).$$ For~$i=1,\ldots,4$, set~$Z_i:=\F_2$,~$G_i:=\{1\}$ (the trivial group) and~$B_1^{(i)}:=[e_0, e_1]$. For convenience of the reader, we restate  the main definitions of Section~\ref{genmult}  and the main facts about representative sets  applied to the context of this section --- see Figure~\ref{factscwd12 }.

\begin{figure}[ht]
  \fbox{
    \begin{minipage}{15.5cm}
     \begin{tabular}{l}
{\bf FACTS.}  \\
\\
$s=4$, $(j_1,j_2,j_3,j_4):=(t,w-t,t,n-t-w)$.
\\\\
for~$i=1,\ldots,4$:\\
$G_i:= \{1\}$, the trivial group. \\
$Z_i:= \F_2$. \\
$m_i:=2$, so~$\bm{m}:=(m_1,\ldots,m_4)=(2,2,2,2)$.\\
\\
A representative set for the action of~$G_i$ on~$Z_i$ is $\{B_1^{(i)}\}=\{[e_0,e_1]\}$. \\
A representative set for the action of~$H_D' $ on~$(\C^{Z_1})^{\otimes j_1} \otimes \ldots \otimes (\C^{Z_4})^{\otimes j_4} $ follows 
\\ from~\eqref{reprsetdef} and~\eqref{prodreprset}.
\\ \\
$\Lambda_i := (Z_i \times  Z_i )/G_i=\F_2 \times \F_2$. \\
$a_P := x \otimes y$ for~$P=(x,y) \in \Lambda_i$.\\
$A_i:= \{ a_P \, | \, P \in \Lambda_i\} $, a basis of~$W_i:= \C { \F_2 \otimes \C \F_2}$.  \\
  $K^{(i)}_{\omega_i} := \sum_{(P_1,\ldots,P_{j_i}) \in \omega_i}
a_{P_1} \otimes \cdots \otimes a_{P_{j_i}}$ for~$\omega_i \in \Lambda_i^{j_i}/S_{j_i}$. 
     \end{tabular}
    \end{minipage}    } \caption{\label{factscwd12 }\small{The main definitions of Section~\ref{genmult}  and the main facts about representative sets  applied to the context of Section~\ref{D2}.}}
\end{figure}

A representative set for the action of~$H_D' $ on~$(\C^{\mathbb{F}_2})^{\otimes j_1} \otimes \ldots \otimes (\C^{\mathbb{F}_2})^{\otimes j_4} \cong \mathbb{C}^{\mathbb{F}_2^n}$ is, cf.~\eqref{reprsetdef} and~\eqref{prodreprset}, given by
\begin{align}
\left\{~
\left[ \bigotimes_{i=1}^4  u_{\tau_i,B_1^{(i)}}
\mid
\forall \,i=1,\ldots, 4 \, : \, \tau_i \in T_{\lambda_i,2}  \right]
~~\big|~~
 \forall \, i=1,\ldots, 4   \,: \, \lambda_i \vdash j_i \,\,
\right\}.
\end{align}
We restrict to words of weight~$w$ and distance contained in~$\{0,d,d+1,\ldots,n\}$ to both words in~$D$. For~$d_1,d_2\in \{0,1,\ldots,n\}$, let~$L_{w,d_1,d_2}$ denote the linear subspace of~$\mathbb{C}^{\mathbb{F}_2^n}$ spanned by the unit vectors~$e_v$, with~$v$ a word of weight~$w$ and distances~$d_1$ and~$d_2$ to~$v_1$ and~$v_2$, respectively.\symlistsort{Lwd1d2}{$L_{w,d_1,d_2}$}{subspace of $\mathbb{C}^{\mathbb{F}_2^n}$} Note that~$L_{w,d_1,d_2}$ is~$H_D' $-invariant.  If for all~$i=1,\ldots,4$ we have~$\lambda_i \vdash j_i $ and $ \tau_i \in  T_{\lambda_i,2}$, then the irreducible representation~$\C H_D'  \cdot  \bigotimes_{i=1}^4  u_{\tau_i,B_1^{(i)}}$ is contained in~$L_{w,d_1,d_2}$, where
\begin{align}
  w&=  |\tau_1^{-1}(2)|+|\tau_2^{-1}(2)|+|\tau_3^{-1}(2)|+|\tau_4^{-1}(2)|,\notag\\ d_1 &=|\tau_1^{-1}(1)|+|\tau_2^{-1}(1)|+|\tau_3^{-1}(2)|+|\tau_4^{-1}(2)|, \notag  \\ 
    d_2&=  |\tau_1^{-1}(2)|+|\tau_2^{-1}(1)|+|\tau_3^{-1}(1)|+|\tau_4^{-1}(2)|. 
\end{align}
 Let for~$\bm{ \lambda }:=(\lambda_1,\ldots,\lambda_4) \vdash (j_1,\ldots,j_4)$,\symlistsort{Tlambda,mbold'}{$\bm{T_{\lambda,m}'}$}{subset of $\bm{T_{\lambda,m}}$}
\begin{align} \label{wlambdaprimegebeuren}
\hspace{-4pt}\bm{T_{\lambda, m}'}\hspace{-2pt} := \hspace{-2pt}\big\{ \hspace{-1pt}(\tau_1,\ldots,\tau_4) \hspace{1pt} \big| \hspace{-1pt}  & \,\hspace{2pt}\text{for all $i=1,\ldots,4$ :}\,\,  \tau_i \in T_{\lambda_i,2},\notag \\
&|\tau_1^{-1}(2)|+|\tau_2^{-1}(2)|+|\tau_3^{-1}(2)|+|\tau_4^{-1}(2)|\hspace{-.5pt}=\hspace{-.5pt}w, \notag \\ & |\tau_1^{-1}(1)|+|\tau_2^{-1}(1)|+|\tau_3^{-1}(2)|+|\tau_4^{-1}(2)| \hspace{-.5pt}\in\hspace{-.5pt} \{0,d,d+1,\hspace{-.5pt}\ldots,\hspace{-.5pt}n\}, \notag  \\ 
    &  |\tau_1^{-1}(2)|+|\tau_2^{-1}(1)|+|\tau_3^{-1}(1)|+|\tau_4^{-1}(2)| \hspace{-.5pt} \in \hspace{-.5pt}\{0,d,d+1,\hspace{-.5pt}\ldots,\hspace{-.5pt}n\} \hspace{-.5pt}\big\}.\hspace{-1pt}
\end{align}
(Note that~$    \bm{T_{\lambda, m}'}$ is a subset of~$\bm{T_{\lambda, m}}$ from~$\eqref{wlambda}$.) Then 
\begin{align} \label{setcw}
\left\{ \left[\mbox{$ \bigotimes_{i=1}^4 u_{\tau_i,B_1^{(i)}}$}  \,\, | \,\, (\tau_1,\ldots,\tau_4) \in       \bm{T_{\lambda, m}'} \right]\,\, \big| ~~\bm{\lambda}=(\lambda_1,\ldots,\lambda_4) \vdash(j_1,\ldots,j_4)    \right\}
\end{align}
  is a representative set for the action of~$H_D' $ on~${N'}$, where~$N'$ denotes the set of words~$v \in N$ such that~$C:=\{v_1,v_2,v\} \in  \mathcal{C}_4(D)$ with~$d_{\text{min}}(C) \geq d$.

\subsubsection{Computations for~\texorpdfstring{$|D|=1$ or~$|D|=2$}{D=1 or D=2}} \label{d12comp}
Fix~$D = \{v_1,v_2\} \in \mathcal{C}_4$. Let~$\Omega_4(D)$ denote the set of all~$S_n$-orbits of codes in~$\mathcal{C}_4$ containing~$D=\{v_1,v_2\}$. For each~$\omega \in \Omega_4(D)$, we define the~$\mathbb{F}_2^n \times \mathbb{F}_2^n$ matrix~$N_{\omega}$ by\symlistsort{Nomega}{$N_{\omega}$}{$\mathbb{F}_2^n \times \mathbb{F}_2^n$ matrix}
\begin{align} \label{nomegaproof}
    (N_{\omega})_{\alpha,\beta} := \begin{cases} 1 &\mbox{if } \{v_1,v_2,\alpha,\beta\} \in \omega,  \\ 
0 & \mbox{otherwise,} \end{cases} 
\end{align}
for~$\alpha,\beta \in \F_2^n$.
Then, for each~$z : \Omega_4(D) \to \mathbb{R}$, one has~$M_{4,D}(z) \succeq 0$ if and only if $\sum_{\omega \in \Omega_4(D)}z(\omega) N_{\omega} \succeq 0$. (The implication``$\Longleftarrow$'' is trivial, as~$M_{4,D}(z)$ is a principal submatrix of~$\sum_{\omega \in \Omega_4(D)}z(\omega) N_{\omega} $. The implication``$\Longrightarrow$'' follows from the fact that~$L\T M_{4,D}(z) L=\sum_{\omega \in \Omega_4(D)}z(\omega) N_{\omega} $, where~$L$ is the $\mathcal{C}_4(D) \times \mathbb{F}_2^n$ matrix with~$0,1$ entries satisfying
$    L_{C,\alpha} = 1 \,\, \text{ if and only if } C =   \{v_1,v_2,\alpha\},$
for~$C \in \mathcal{C}_4(D)$ and~$\alpha \in \mathbb{F}_2^n$.) 

Then we obtain with~$(\ref{PhiR})$ and~$(\ref{setcw})$ that, for each~$z : \Omega_4(D) \to \mathbb{R}$,
\begin{align} \label{blocks1cw}
    \Phi \left(\sum_{\omega \in \Omega_4(D)}z(\omega) N_{\omega} \right ) =\bigoplus_{\substack{ \bm{\lambda} =(\lambda_1,\ldots,\lambda_4) \\ \lambda_i \vdash j_i \,\, \forall \,\, 1 \leq i \leq 4 } }  \sum_{\omega \in \Omega_4(D)} z(\omega)  U_{\bm{\lambda}} \T N_{\omega}  U_{\bm{\lambda}},
\end{align}
where $U_{\bm{\lambda}}$ is the matrix in~\eqref{setcw} that corresponds with $\bm{\lambda}:=(\lambda_1,\ldots,\lambda_4)$. For each~$i=1,\ldots,4$, the number of $\lambda_i \vdash j_i$, and the numbers~$|    \bm{T_{\lambda, m}'}|$ and~$|\Omega_4(D)|$  are all bounded by a polynomial in~$n$. This implies that the number of blocks in~$(\ref{blocks1cw})$, the size of each block and the number of variables occurring in all blocks are polynomially bounded in~$n$. 
Next we will show how to compute each entry~$ (\bigotimes_{i=1}^4  u_{\tau_i,B_1^{(i)}})\T N_{\omega} ( \bigotimes_{i=1}^4 u_{\sigma_i,B_1^{(i)}}
)$ in polynomial time, for~$(\tau_1,\ldots,\tau_4)$ and $(\sigma_1,\ldots,\sigma_4)$ in $   \bm{T_{\lambda, m}'}$.

 As in Section~\ref{genmult}, let~$R:=Z_1^{j_1} \times \ldots \times Z_4^{j_4}$, so $R = \mathbb{F}_2^t \times \mathbb{F}_2^{w-t} \times \mathbb{F}_2^t \times  \mathbb{F}_2^{n-w-t}$. Write~$\mathcal{C}_4' $ for the collection of codes~$C \subseteq \F_2^n$ of size~$\leq 4$ (so not necessarily of constant weight~$w$). Then the function
\begin{align}
 R^2=   \left( \mathbb{F}_2^t \times \mathbb{F}_2^{w-t} \times \mathbb{F}_2^t \times  \mathbb{F}_2^{n-w-t} \right)^2 &\to \mathcal{C}_4',\\
((\alpha^1,\alpha^2,\alpha^3,\alpha^4),(\beta^1,\beta^2,\beta^3,\beta^4)) &\mapsto \{v_1,v_2,\alpha^1\alpha^2\alpha^3\alpha^4, \beta^1\beta^2\beta^3\beta^4\},    \notag 
\end{align}
induces a surjective function~$  r\, : \, R^2/H_D'  \to \Omega_4'(D)$, where~$\Omega_4'(D)\supseteq \Omega_4(D)$ denotes the set of all~$H_D' $-orbits of codes in~$\mathcal{C}_4'$ that contain~$D=\{v_1,v_2\}$.\symlistsort{r}{$r$}{surjective function}

For any $\omega' \in R^2/H_D'$, the matrix~$K_{\omega' }$ is defined in~\eqref{KomegaIH}.
If~$\omega \in \Omega_4(D)$, then a basic calculation as in  Lemma~$\ref{Lsomlemma}$ gives
\begin{align}
    N_{\omega} = \sum_{\substack{\omega' \in R^2/H_D'  \\ r(\omega') = \omega  }} K_{\omega'}.
\end{align}
So it suffices to compute  $\left(\bigotimes_{i=1}^4 u_{\tau_i,B_1^{(i)}}\right)\T K_{\omega'}  \left( \bigotimes_{i=1}^4 u_{\sigma_i,B_1^{(i)}}\right)$. Note that~$K_{\omega'}=K_{\omega_1}^{(1)} \otimes \ldots \otimes K_{\omega_4}^{(4)}$, cf.~\eqref{Komegaprod}, where for each~$\omega' \in R^2/H_D' $ we write~$\omega'=\omega_1\times \ldots \times \omega_4$, with~$\omega_i \in \Lambda_i^{j_i}/S_{j_i}$ for~$1 \leq i \leq 4$.  Then, as in~\eqref{tensorcomp},
  \begin{align}\label{fundamentalproductcw}
&  \left(\mbox{$\bigotimes_{i=1}^4$}  u_{\tau_i,B_1^{(i)}}\right)\T   \left( \mbox{$\bigotimes_{i=1}^4$} K_{\omega_i}^{(i)} \right)    \left( \mbox{$\bigotimes_{i=1}^4$} u_{\sigma_i,B_1^{(i)}} \right) = \prod_{i=1}^4 \left( u_{\tau_i,B_1^{(i)}}\T   K_{\omega_i}^{(i)}  u_{\sigma_i,B_1^{(i)}}\right).
\end{align}
Hence
\begin{align*}
&\sum_{\omega' \in R^2/H_D' } \left( \left(\mbox{$\bigotimes_{i=1}^4$} u_{\tau_i,B_1^{(i)}}\right)\T   \left( \mbox{$\bigotimes_{i=1}^4$} K_{\omega_i}^{(i)} \right)    \left( \mbox{$\bigotimes_{i=1}^4$} u_{\sigma_i,B_1^{(i)}}\right)\right)\mu(\omega_1)\cdots \mu(\omega_4) 
\\&= \prod_{i=1}^4 \sum_{\omega_i \in \Lambda^{j_i}/S_{j_i} } \left(u_{\tau_i,B_1^{(i)}}\T   K_{\omega_i}^{(i)}   u_{\sigma_i,B_1^{(i)}}\right)\mu(\omega_i).  
\end{align*} 
Moreover, by Proposition~$\ref{prop2}$ , for~$1 \leq i \leq 4$, one has 
$$
 \sum_{\omega_i \in \Lambda_i^{j_i}/S_{j_i}} \left(u_{\tau_i,B_1^{(i)}} K_{\omega_i}^{(i)}u_{\sigma_i,B_1^{(i)}}\right) \mu(\omega_i) = p_{\tau_i,\sigma_i}(F_1^{(i)}),
 $$
where~$p_{\tau_i,\sigma_i}$ is defined in~\eqref{ptsdef} and where~$F_1^{(i)}$ is the~$\F_2 \times \F_2$ matrix with 
\begin{align}\label{f1icw}
(F_1^{(i)}):= \left( B_1^{(i)}(j) \otimes B_1^{(i)}(h) \right)\big|_{W_i} = a_{(j,h)}^*, \,\,\, \text{ for $j,h \in \F_2$ (so that~$(j,h) \in \Lambda_i$)}.
\end{align}
So one computes the entry $\sum_{\omega \in \Omega_4(D)} z(\omega) \left(\bigotimes_{i=1}^4 u_{\tau_i,B_1^{(i)}}\right)\T N_{\omega}  \left( \bigotimes_{i=1}^4 u_{\sigma_i,B_1^{(i)}} \right)$ by first expressing each polynomial $ p_{\tau_i,\sigma_i}(F_1^{(i)}) $ as a linear combination of~$\mu(\omega_i)$ and subsequently replacing each $\mu(\omega_1)\cdots\mu(\omega_4)$ in the product $\prod_{i=1}^4  p_{\tau_i,\sigma_i}(F_1^{(i)})$ with the variable~$z(r(\omega_1 \times \ldots \times \omega_4))$ if~$r(\omega_1 \times \ldots \times \omega_4) \in \Omega_4$ and with zero otherwise. (Here~$\omega_1 \times \ldots \times \omega_4 \in R^2/H_D' $, cf.~\eqref{omegaschrijfwijze}.)  

\subsection{The case~\texorpdfstring{$D=\emptyset$}{D=empty}} \label{Dempty}
Next, we consider how to block diagonalize~$M_{3,\emptyset}(z)=M_{2,\emptyset}(z)$ for computing~$B_4(n,d,w)$. We also give a reduction of the matrix~$M_{4,\emptyset}(z)$ for computing~$A_4(n,d,w)$. So we will reduce the matrices~$M_{2\xi,\emptyset}(z)$ for~$\xi \in \{1,2\}$, where we consider~$\xi=1$ for computing~$B_4(n,d,w)$ and~$\xi=2$ for computing~$A_4(n,d,w)$. We start by giving a representative set for the natural action of~$S_n$ on~$(\C^{\F_2^{\xi}})^{\otimes n} \cong \C^{(\F_2^n)^{\xi}}$, using the results described in Section~\ref{prelimrep}. For the computations, we will use the framwork from Chapter~\ref{orbitgroupmon}.

Let~$B_1:=[e_j \,\, | \,\, j \in \F_2^{\xi}]$ be an~$\F_2^{\xi} \times \F_2^{\xi}$ matrix containing the unit basis vectors of~$\C^{\F_2^{\xi}}$ as columns, in lexicographic order (from left to right). So if~$\xi=1$ then~$B_1=[e_0,e_1]$ and if~$\xi=2$ then~$B_1=[e_{(0,0)}, e_{(0,1)}, e_{(1,0)}, e_{(1,1)}]$. Then~$B_1$ is an ordered basis of~$ \C^{\F_2^{\xi}}$. 
For convenience of the reader, we restate  the main definitions of Chapter~\ref{orbitgroupmon}  and the main facts about representative sets  applied to the context of this section --- see Figure~\ref{factscwdleeg }.

\begin{figure}[ht]
  \fbox{
    \begin{minipage}{15.5cm}
     \begin{tabular}{l}
{\bf FACTS.}  \\
\\
$G:= \{1\}$, the trivial group. \\
$Z:= \F_2^{\xi}$. \\
\\
$k:=1$,~$m_1:=2^{\xi}$ and~$m:=2^{\xi}$.\\
\\
A representative set for the action of~$G$ on~$Z$ is $\{B_1\}:=\{[e_j \,\, | \,\, j \in \F_2^{\xi}]\}$. \\
A representative set for the action of~$S_n$ on~$Z^n$ is given by~\eqref{reprsetdef}.
\\ \\
$\Lambda := (Z \times  Z)/G=\F_2^{\xi} \times \F_2^{\xi}$. \\
$a_P := x \otimes y$ for~$P=(x,y) \in \Lambda$.\\
$A:= \{ a_P \, | \, P \in \Lambda\} $, a basis of~$W:= \C{ \F_2^{\xi} \otimes \C \F_2^{\xi}} $. \\
  $K_{\omega'} := \sum_{(P_1,\ldots,P_n) \in \omega'}
a_{P_1} \otimes \cdots \otimes a_{P_n}$ for~$\omega' \in \Lambda^n/S_n$. 
     \end{tabular}
    \end{minipage}    }  \caption{\label{factscwdleeg }\small{The main definitions of Chapter~\ref{orbitgroupmon}  and the main facts about representative sets  applied to the context of Section~\ref{Dempty}.}}
\end{figure}

 Proposition~\ref{reprsetdef} gives a representative set for the action of~$S_n$ on~$(\C^{\F_2^{\xi}})^{\otimes n}$: it is the set 
\begin{align} \label{reprtot}
\{~~
[u_{\tau,B_1}
\mid
\tau \in T_{\lambda,2^{\xi}}]
~~\mid
\lambda\vdash n
\}.
\end{align}

For computing~$B_4(n,d,w)$,\hspace{-.41pt} we   consider~${\xi}=1$.\hspace{-1.4pt}  We restrict the representative set~$(\ref{reprtot})$ for the action of~$S_n$ on~$(\C^{\F_2})^{\otimes n} \cong \C^{\F_2^n}$ to~$\C^N$ (recall:~$N \subseteq \F_2^n$ is the set of all words of constant weight~$w$). Note that for any~$\tau \in T_{\lambda,2}$, the number of times~$B_1(2)=e_1$ appears in the tensor in each summand of~$u_{\tau,B_1}$ (cf.~\eqref{utau}) is equal to~$|\tau^{-1}(2)|$. Similarly,~$|\tau^{-1}(1)|$ is equal to the number of times $B_1(1)=e_0$ appears in the tensor in each summand of~$u_{\tau,B_1}$.
 Let\symlistsort{Tlambda,21}{$T_{\lambda,2}^{(1)}$}{subset of $T_{\lambda,2}$}
\begin{align} \label{wset1}
    T_{\lambda,2}^{(1)}:= \{\tau \in T_{\lambda,2} \, | \,\,\,    |\tau^{-1}(2)|=w\},\,\,\,\,\, \text{for $\lambda \vdash n$}.
\end{align}

Then    
\begin{align} \label{set1}
 \left\{\,\,\left[u_{\tau,B_1} \,\, | \,\, \tau \in  T_{\lambda,2}^{(1)} \right] \,\, | \,\, \lambda \vdash n  \right\}
\end{align}
is representative for the action of~$S_n$ on~$N=\mathcal{C}_1 \setminus \{ \emptyset\} \subseteq {\F_2^n} $. 

For computing~$A_4(n,d,w)$, we consider~${\xi}=2$. We proceed by restricting the representative set~$(\ref{reprtot})$ of the action of~$S_n$ on~$\C^{(\F_2^n)^2}$ to pairs of words in~$N^2$ with distance contained in~$\{0,d,d+1,\ldots,n\}$. Given~$w_1,w_2,d_1\in \{0,1,\ldots,n\}$, let~$L_{w_1,w_2,d_1}^{(2)}$ denote the linear subspace of~$\mathbb{C}^{(\mathbb{F}_2^n)^2}$ spanned all the unit vectors~$e_{\alpha,\beta}$, with~$\alpha$ and~$\beta$ words of weight~$w_1$ and~$w_2$ respectively, and~$d_H(\alpha,\beta)=d_1$.\symlistsort{Lw1w2d1}{$L_{w_1,w_2,d_1}^{(2)}$}{subspace of $\mathbb{C}^{(\mathbb{F}_2^n)^2}$}
 Then~$L_{w_1,w_2,d_1}^{(2)}$ is~$S_n$-invariant. Moreover, for any~$\lambda \vdash n$ and~$\tau \in   T_{\lambda, 4}$, the irreducible representation~$\C S_n \cdot u_{\tau,B_1} $ is contained in~$L_{w_1,w_2, d_1}^{(2)}$, with
\begin{align}\label{134tau}
  w_1&=  |\tau^{-1}(3)|+|\tau^{-1}(4)|, \notag \\ w_2 &=|\tau^{-1}(2)|+|\tau^{-1}(4)|, \notag 
  \\ d_1 &=|\tau^{-1}(2)|+|\tau^{-1}(3)|.   
  \end{align}
  So let\symlistsort{Tlambda,42}{$T_{\lambda,4}^{(2)}$}{subset of $T_{\lambda,4}$}
  \begin{align} \label{wset2}
       T_{\lambda,4}^{(2)}:= \{\tau \in T_{\lambda,4}\, | \,\,\,    |\tau^{-1}(2)|+|\tau^{-1}(4)|&=w, \,|\tau^{-1}(3)|+|\tau^{-1}(4)| =w,\notag\\ |\tau^{-1}(2)|+|\tau^{-1}(3)|&\in \{0,d,d+1,\ldots,n\}\},  \,\,\,\,\, \text{ for $\lambda \vdash n$}.
  \end{align}
Then 
\begin{align} \label{set2}
 \left\{\,\,\left[u_{\tau,B_1} \,\, | \,\, \tau \in  T_{\lambda,4}^{(2)}  \right] \,\, | \,\, \lambda \vdash n  \right\}
\end{align}
is representative for the action of~$S_n$ on~${(N^2)_d} \subseteq {\F_2^n \times \F_2^n} $, where~$(N^2)_d$ denotes the set of all pairs of words in~$N \times N$ with distance contained in~$\{0,d,d+1,\ldots,n\}$. 

It is possible to further reduce the program (by a factor~$2$) by giving a reduction from ordered pairs to \emph{unordered} pairs of words.  We will not consider this reduction in the present chapter. Regardless of a further reduction by a factor 2, the programs for computing~$A_4(n,d,w)$ are considerably larger (although they have polynomial size) than the ones for computing~$B_4(n,d,w)$. 

Note that~$S_n$ acts trivially on~$\emptyset$. The~$S_n$-isotypical component of~$\mathbb{C}^{N^{\xi}}$ that consists of the~$S_n$-invariant elements corresponds to the matrix in the representative set~\eqref{set1} or~\eqref{set2} indexed by~$\lambda = (n)$. So to obtain a representative set for the action of~$S_n$ on~$(N^{\xi})_d \cup \{\emptyset\}$ (here~$(N^1)_d := N$), we add a new unit base vector~$e_{\emptyset}$ to this matrix (as a column). 

\subsubsection{Computations for~\texorpdfstring{$D=\emptyset$}{D=empty}} 
We consider~${\xi}=1$ and~${\xi}=2$ for computing~$B_4(n,d,w)$ and~$A_4(n,d,w)$, respectively. If~${\xi}=1$, then for all~$\omega \in \Omega_{2} \subseteq \Omega_4$, we define the~$\mathbb{F}_2^n \times \mathbb{F}_2^n$ matrix~$N_{\omega}^{(1)}$ by \symlistsort{Nomega1}{$N_{\omega}^{(1)}$}{$\mathbb{F}_2^n \times \mathbb{F}_2^n$ matrix}
\begin{align}\label{blocks2}
    (N_{\omega}^{(1)})_{\alpha,\beta} := \begin{cases} 1 &\mbox{if } \{\alpha,\beta\} \in \omega,  \\ 
0 & \mbox{otherwise,} \end{cases} 
\end{align}
for~$\alpha,\beta \in \F_2^n$.
Similarly, if~${\xi}=2$, then for all 
all~$\omega \in \Omega_{4}$, we define the~$(\mathbb{F}_2^n \times \mathbb{F}_2^n) \times(\mathbb{F}_2^n \times \mathbb{F}_2^n) $ matrix~$N_{\omega}^{(2)}$ by \symlistsort{Nomega2}{$N_{\omega}^{(2)}$}{$(\mathbb{F}_2^n \times \mathbb{F}_2^n) \times(\mathbb{F}_2^n \times \mathbb{F}_2^n) $ matrix}
\begin{align}\label{blocksundef}
    (N_{\omega}^{(2)})_{(\alpha,\beta),(\gamma,\delta)} := \begin{cases} 1 &\mbox{if } \{\alpha,\beta,\gamma,\delta\} \in \omega,  \\ 
0 & \mbox{otherwise,} \end{cases} 
\end{align}
for~$\alpha,\beta,\gamma,\delta \in \F_2^n$.
Let~$M_{2{\xi},\emptyset}'(z)$ denote the matrix~$M_{2{\xi},\emptyset}(z)$ with the row and column indexed by~$\emptyset$ removed. Then, for each~$z: \Omega_{2{\xi}} \to \R$, one has~$M_{2{\xi},\emptyset}'(z) \succeq 0$ if and only if $\sum_{\omega \in \Omega_{2{\xi}} \setminus\{\{ \emptyset \} \}}z(\omega) N_{\omega}^{({\xi})} \succeq 0$. With~$(\ref{PhiR})$ and~$(\ref{set1})$ or~$(\ref{set2})$ we obtain that, for any $z : \Omega_{2{\xi}} \setminus \{ \{ \emptyset \} \} \to \R$,
\begin{align} \label{blocks3}
 \Phi \left(\sum_{\omega \in \Omega_{2{\xi}} \setminus\{\{ \emptyset \} \}}z(\omega) N_{\omega}^{({\xi})} \right) = \bigoplus_{\lambda \vdash n}  \sum_{\omega \in \Omega_{2{\xi}} \setminus\{\{ \emptyset \} \}}  z(\omega) U_{\lambda}\T N_{\omega}^{({\xi})} U_{\lambda},
\end{align}
where~$U_{\lambda}$ is the matrix in~$(\ref{set1})$ or~$(\ref{set2})$ that corresponds with~$\lambda$. 
The number of~$ \lambda \vdash n$ with~$\height(\lambda_1)\leq 2^{\xi}$, and the numbers~$|T_{\lambda,2^{\xi}}^{(\xi)}|$, $|\Omega_{2{\xi}}|$, for~${\xi}=1$ and~${\xi}=2$, respectively, are all polynomially bounded in~$n$. Hence the number of blocks in~$(\ref{blocks3})$, as well as the size of each block and the number of variables occurring in all blocks are bounded by a polynomial in~$n$.
We now explain how to compute the coefficients~$u_{\tau,B_1}\T N_{\omega}^{({\xi})}u_{\sigma,B_1} $ in polynomial time, for~$\tau,\sigma \in T_{\lambda,2^{\xi}}^{(\xi)}$ and~${\xi} \in \{1,2\}$.

Recall that~$\Lambda=\F_2^{\xi} \times \F_2^{\xi}$. Since naturally~$(\F_2^n)^{2{\xi}} \cong (\F_2^{\xi} \times \F_2^{\xi})^n$, there is a natural bijection between~$\Lambda^n/S_n$ and the set of~$S_n$-orbits on~$(\F_2^n)^{2{\xi}}$. The function $(\F_2^n)^{2{\xi}} \to \mathcal{C}_{2{\xi}}$ with $(x_1,\ldots,x_{2{\xi}}) \mapsto \{x_1,\ldots,x_{2{\xi}}\} $
then induces a surjective function $r' :\Lambda^n/S_n\to \Omega_{2{\xi}}' \setminus \{\{\emptyset\}\}$, where~$\Omega_{2{\xi}}' \supseteq \Omega_{2{\xi}}$ denotes the set of all~$S_n$-orbits of codes in~$\mathcal{C}_{2{\xi}}'$ (so not necessarily of constant weight~$w$) and~${\xi} \in \{1,2\}$.\symlistsort{r'}{$r'$}{surjective function}

In view of Lemma~$\ref{Lsomlemma}$, we have for each~$\omega \in \Omega_{2{\xi}} \setminus\{\{ \emptyset \} \}$ that
\begin{align}
    N_{\omega}^{({\xi})}  = \sum_{\substack{\omega'  \in \Lambda^n/S_n\\ r'(\omega') = \omega  }} K_{\omega'}.
\end{align}
So it suffices to compute $\left( u_{\tau,B_1}\T K_{\omega'}u_{\sigma,B_1}\right)$ for each
$\omega' \in \Lambda^n/S_n$ and~$\tau,\sigma \in T_{\lambda,2^{\xi} }^{(\xi)}$. By Proposition~\ref{ptsprop}, it holds that $\sum_{\omega' \in \Lambda^n/S_n} \left( u_{\tau,B_1}\T K_{\omega'}u_{\sigma,B_1}\right)\mu(\omega') =  p_{\tau,\sigma}(F_1)$, where the polynomial~$p_{\tau ,\sigma}$ is defined in~\eqref{ptsdef} and where~$F_1$ is an~$\F_2^{\xi} \times \F_2^{\xi}$ matrix with
\begin{align}\label{f1cw}
(F_1)_{j,h} =(B_1(j)\otimes B_1(h))|_W = a_{(j,h)}^*, \,\,\, \text{ for $j,h \in \F_2^{\xi}$ (so that~$(j,h) \in \Lambda$)}.
\end{align}
 Hence one computes the entry $\sum_{\omega \in \Omega_{2{\xi}} \setminus\{\{ \emptyset \} \}} z(\omega) u_{\tau,B_1}\T N_{\omega}^{({\xi})} u_{\sigma,B_1}$ by first expressing~$ p_{\tau,\sigma}(F_1)$ as a linear combination of degree $n$ monomials expressed in the dual basis~$A^*$ of~$A$ and subsequently replacing each monomial~$\mu(\omega')$ in~$ p_{\tau,\sigma}(F_1)$ with the variable~$z(r'(\omega'))$ if~$r'(\omega')\in \Omega_{2\xi}$ and with zero otherwise.

At last, we compute the entries in the row and column indexed by~$\emptyset$ in the matrix for~$\lambda= (n)$. Then~$e_{\emptyset}\T M_{2{\xi},\emptyset}(z) e_{\emptyset} =M_{2{\xi},\emptyset}(z)_{\emptyset,\emptyset}= x(\emptyset)=1$ by definition, see~$(\ref{A4ndw})$ and~$(\ref{Bndw})$. For computing the other entries we distinguish between the cases~${\xi}=1$ and~${\xi}=2$. If~${\xi}=1$, then for~$\lambda = (n)$, we have~$|T_{\lambda,2}^{(1)}|=1$, so there is only one coefficient to compute. If~$\tau$ is the unique element in~$T_{\lambda,2}^{(1)}$ (note that the tableau~$\tau$ contains~$w$ times symbol~$2$ and~$n-w$ times symbol~$1$), then~$ u_{\tau,B_1}= \sum_{v \in \mathbb{F}_2^n, \text{wt}(v)=w} e_v$, so
\begin{align}\label{leeg1}
e_{\emptyset}\T M_{2,\emptyset}(z)   u_{\tau,B_1} = \binom{n}{w} z(\omega_0),
\end{align}
where~$\omega_0 \in \Omega_{2}$ is the (unique) $S_n$-orbit of a code of size~$1$.\symlistsort{omega0}{$\omega_0$}{orbit of a code of size~$1$} 

If~${\xi}=2$, then  for~$\lambda = (n)$, any~$\tau \in T_{\lambda,4}^{(2)}$ is determined by the number~$t$ of~$2$'s in the row of the Young shape~$Y((n))$ of tableau~$\tau$ (this determines also the number of~$1$'s, $3$'s and~$4$'s, by~\eqref{134tau}). Then
\begin{align}
    u_{\tau,B_1}  = \sum_{\substack{ v_1,v_2 \in \mathbb{F}_2^n,\\ \text{wt}(v_1)=\text{wt}(v_2)=w \\ d_H(v_1,v_2)=2t}} e_{(v_1,v_2)}. 
\end{align}
Hence
\begin{align} \label{leeg2}
 e_{\emptyset}\T M_{4,\emptyset}(z)   u_{\tau,B_1}  = \sum_{\substack{ v_1,v_2 \in \mathbb{F}_2^n,\\ \text{wt}(v_1)=\text{wt}(v_2)=w \\ d_H(v_1,v_2)=2t}} x(\{v_1,v_2\}) = \binom{n}{w}\binom{w}{t}\binom{n-w}{t} z(\omega_t),
\end{align}
where~$\omega_t \in \Omega_{4}$ is the (unique) $S_n$-orbit of a pair of constant weight code words at  distance~$2t$.\symlistsort{omegat}{$\omega_t$}{$S_n$-orbit of a pair of constant weight code words at  distance~$2t$}

\section{Concluding remarks}\label{conccw4}

Recently, also new upper bounds for constant weight codes for some values of~$n$ larger than~$28$ have been given, see~$\cite{largen}$ and~$\cite{brouwertable}$. We therefore also provide a table with improved upper bounds for~$n$ in the range~$29 \leq n \leq 32$ and~$d \geq 10$ (cf.~Brouwer's table~$\cite{brouwertable}$) --- see Table~\ref{32}. Most of our new upper bounds for these cases are instances of~$A_3(n,d,w)$, which can computed by using the block diagonalization of~$M_{2,\emptyset}(z)$ from Section~$\ref{Dempty}$ and from~$M_{3,D}(z)=M_{4,D}(z)$ for~$|D|=1$ from Section~$\ref{D2}$ (use only the blocks for~$t=0$ from this section). The bound~$A_3(n,d,w)$ is in almost all cases equal to the Schrijver bound~$\cite{schrijver}$.

\begin{table}[!htb] \footnotesize
\begin{center}
   \begin{minipage}[t]{.495\linewidth}
    \begin{tabular}{| r | r | r|| r |>{\bfseries}r | r| r|}
    \hline
    $n$ & $d$ & $w$ &  \multicolumn{1}{>{\raggedright\arraybackslash}b{8.29mm}|}{\scriptsize best lower bound known}  & \multicolumn{1}{>{\raggedright\arraybackslash}b{10.95mm}|}{\textbf{new upper bound}} & \multicolumn{1}{>{\raggedright\arraybackslash}b{9.0mm}|}{\scriptsize best upper bound previously known} & \multicolumn{1}{>{\raggedright\arraybackslash}b{8.3mm}|}{\scriptsize Del-sarte bound}\\\hline 
        31& 10 & 8  & 124 & $\mathbf{322^B}$ & 329  & 344 \\    
        32& 10 & 8  & 145 & $\mathbf{402^B}$& 436  &  442\\  
        29& 10 & 9  & 168 & $\mathbf{523^B}$& 551  & 565  \\       
        30& 10 & 9  & 203 & $\mathbf{657^B}$& 676  &  725\\   
        31& 10 & 9  & 232 & $\mathbf{822^B}$& 850 & 936\\           
        30& 10 & 10  & 322 & 1591 &1653  & 1696  \\    
        31& 10 & 10  & 465 & 2074& 2095  & 2247\\    
        32& 10 & 10  & 500 & 2669& 2720 &  2996 \\    
        29& 10 & 11  & 406 & 2036& $2055$ & $2055$  \\ 
        30& 10 & 11  & 504 & 2924& $2945$ & $2945$  \\   
        31& 10 & 11  & 651 & 4141& $4328$ & $4328$ \\   
        32& 10 & 11  & 992 & 5696& 6094 &  6538\\   
        29& 10 & 12  & 539 & 3091&  $3097$& $3097$ \\   
        30& 10 & 12  & 768 & 5008&  $5139$ &$5139$   \\   
        31& 10 & 12  & 930 & 7259&  $7610$& $7610$  \\
        32& 10 & 12  & 1395 & 10446& $11541$ & $11541$ \\   
        29& 10 & 13  & 756 & 4282&  $4420$&   $4420$ \\    
        30& 10 & 13  & 935 & 6724&  $7149$&$7149$  \\   
        31& 10 & 13  & 1395 & 10530& $12254$ &$12254$  \\   
        32& 10 & 13  & 1984 & 16755& $18608$&  $18608$ \\  
        29& 10 & 14  & 1458 & 4927&  $5051$& $5051$ \\   
        30& 10 & 14  & 1458 & 8146&  $9471$& $9471$ \\  
        31& 10 & 14  & 1538 & 13519& $15409$ &   $15409$\\  
        32& 10 & 14  & 2325 & 22213& $24679$ & $24679$   \\      
        30& 10 & 15  & 1458 & 8948&  $10053$&  $10053$ \\   
        31& 10 & 15  & 1922 & 15031& $17337$ & $17337$ \\  
        32& 10 & 15  & 2635 & 26361& $29770$ &$29770$  \\
        32& 10 & 16  & 3038 & 27429& $30316$&$30316$  \\          
    \hline 
    \end{tabular}\end{minipage} \,
   \begin{minipage}[t]{.487\linewidth}
    \begin{tabular}{| r | r | r|| r |>{\bfseries}r | r|r|}
    \hline
    $n$ & $d$ & $w$ &  \multicolumn{1}{>{\raggedright\arraybackslash}b{8.29mm}|}{\scriptsize best lower bound known}  & \multicolumn{1}{>{\raggedright\arraybackslash}b{10.95mm}|}{\textbf{new upper bound}} & \multicolumn{1}{>{\raggedright\arraybackslash}b{8.3mm}|}{\scriptsize best upper bound previously known}& \multicolumn{1}{>{\raggedright\arraybackslash}b{8.3mm}|}{\scriptsize Del-sarte bound}\\\hline 
            29& 12 & 9  &  42& $\mathbf{59^B}$& 66 & 67\\  
        30& 12 & 9  &  42& $\mathbf{74^B}$& 94 &96  \\ 
       31& 12 & 9  &  50& 94& 103 & 112\\ 
        29& 12 & 10  &  66& 126& $129$ & $129$\\ 
        32& 12 & 11  & 186 & 573& $574$ &  $574$\\    
        30& 12 & 12  & 190 & 492& $493$ &$493$  \\ 
   31& 12 & 12  & 310 & 679& $692$ &$692$ \\ 
   32& 12 & 12  & 496 & 952& $1014$ & $1014$ \\ 
        30& 12 & 13  & 236 & 642& $689$ &$689$ \\  
        31& 12 & 13  & 400 & 958& $1177$ &  $1177$ \\ 
        32& 12 & 13  & 434 & 1497& $1669$ & $1669$\\   
        29& 12 & 14 &  173& 492& $507$&  $507$\\  
        30& 12 & 14 & 288 &801 & $952$ & $952$ \\ 
        31& 12 & 14 & 510 &1238 & $1455$  &$1455$ \\ 
 32& 12 & 14 & 900 &2140 & $2143$ & $2143$  \\     
        30& 12 & 15 &  302& 894& $1008$ & $1008$ \\  
        31& 12 & 15 & 572 & 1435& $1605$ & $1605$ \\         \hline
    32& 14 & 11  & 39 & 68 & 89& 90  \\ 
    29 & 14 & 12  & 29 & 47 & 50 &52 \\ 
    30 & 14 & 12  & 36 & 62 & 72  &  80 \\ 
    31 & 14 & 12  & 45 & 80 & 103  &  104 \\ 
    32 & 14 & 12  & 55 & 118 & $134$  & $134$ \\ 
    29 & 14 & 13  & 35 & 58 & 66 & 74\\  
    30 & 14 & 13  & 45 & 78 & $101$ & $101$ \\ 
    31 & 14 & 13 &  60   &129    &$137$ & $137$ \\ 
    29 & 14 & 14 &  58   &   63 & $82$& $82$ \\ 
    30 & 14 & 14 &  58   &   95 & $116$ & $116$ \\ 
    30 & 14 & 15 &  62  &   104 & $122$& $122 $\\ \hline  
    \end{tabular}\end{minipage}    
\end{center}
  \caption{\small An overview of the new upper bounds for constant weight codes for~$29 \leq n \leq 32$ and~$d \geq 10$. The unmarked new upper bounds are instances of~$A_3(n,d,w)$ (rounded down), the ones marked with~${}^B$ are instances of~$B_4(n,d,w)$ (rounded down). The best previously known bounds are taken from Brouwer's table~$\cite{brouwertable}$.   \label{32}}
\end{table} 

Upper bound~$(\ref{Bndw})$ could also be useful for unrestricted  binary codes: one can define~$B_k(n,d)$ just as~$B_k(n,d,w)$ in~$(\ref{Bndw})$, where~$\mathcal{C}_k$ now is defined to be the collection of all codes~$C \subseteq \mathbb{F}_2^n$ of size~$\leq k$ and~$N:=\mathbb{F}_2^n$.\symlistsort{N}{$N$}{subset of~$\mathbb{F}_2^n$}\symlistsort{Bk(n,d)}{$B_k(n,d)$}{upper bound on~$A(n,d)$} Note that~$B_k(n,d)$ in this context has a different meaning than~$B_q(n,d)$ from~\eqref{8no15a}. To compute~$B_k(n,d)$ in case~$k=5$, one must block diagonalize~$M_{4,\emptyset}(x|_{\CC_4})$ (the block diagonalization can be found explicitly in~$\cite{semidef}$ or more conceptually in Chapter~\ref{onsartchap})  and the matrices~$M_{5,D}(x)$ with~$|D|=3$ and~$|D|=2$. One can assume that~$D=\{v_1,v_2,v_3\}$ with
\begin{align}
    v_1=&\overbrace{0\ldots0\,0\ldots0}^{w}\overbrace{0\ldots0\,\,\,0\ldots0}^{n-w} \\
    v_2=& \,1\ldots1\,1\ldots1 \,0\ldots0\,\,\,0\ldots0 \notag   
    \\v_3= &\underbrace{0\ldots0}_{t_1}\underbrace{1\ldots1}_{w-t_1}\underbrace{1\ldots1}_{t_2}\underbrace{\,0\ldots0}_{n-w-t_2},    \notag 
\end{align}
for~$0 \leq w \leq n$,~$t_1\leq w$,~$t_2 \leq n-w$ and such that the weight of~$v_3$ is at least the weight of~$v_2$, so~$t_2 \geq t_1$. Then $S_{t_1} \times S_{w-t_1} \times S_{t_2} \times S_{n-w-t_2}$ acts on~$\mathcal{C}_5$ fixing~$D$, and a block diagonalization of~$M_{5,D}(x)$ for~$|D|=3$ and~$|D|=2$ for non-constant weight codes is obtained by a straightforward adaptation of the block diagonalization of~$M_{4,D}(x)$ for~$|D|=2$ and~$|D|=1$ for constant weight codes given in Section~$\ref{D2}$.\footnote{To compute the bound from~$\cite{semidef}$ for $k=5$, we additionally must compute a block diagonalization of the (large) matrix~$M_{5,D}(x)$ for~$|D|=1$ (so we can assume~$D=\{0\ldots 0\}$, the singleton zero word). This block diagonalization can be obtained by adapting the block diagonalization of~$M_{4,\emptyset}$ for constant weight codes given in Section~$\ref{Dempty}$.}   
However, for small cases in which~$A(n,d)$ is unsettled (with~$n$ around 20), the program~$B_5(n,d)$ is large in practice, even  after reductions. The program to compute the bound from~\cite{semidef} (which is defined analogously to~$A_k(n,d,w)$ in~$\eqref{A4ndw}$) for~$k=5$ is still larger after reductions (although of size polynomial in~$n$).
Using a lot of computing time, one may be able to compute~$B_5(n,d)$  for some small unknown cases of~$n,d$, possibly sharpening recent semidefinite programming bounds for binary codes~$\cite{semidef, KT}$.  This is material for further research.

\section{Appendix: An overview of the program}

\begin{figure}[!htb]\fbox{
    \begin{minipage}{15.5cm} {\small
     \begin{tabular}{l}
    \textbf{Input: } Natural numbers~$n,d,w$ and~$\xi \in \{1,2\}$ \\
	\textbf{Output: }Semidefinite program to compute~$B_4(n,d,w)$ (${\xi}=1$) or~$A_4(n,d,w)$ (${\xi}=2$) \\
	$\,$\\
    \printv \emph{Maximize} $\binom{n}{w} z(\omega_0)$  \\
        \printv \emph{Subject to:}  \\    
\text{\color{blue}//Start with~$|D|=1$ and~$|D|=2$.}\\
\foreachv $t \in \Z_{\geq 0}$ with~$t=0$ or~$d/2 \leq t \leq w$  \text{\color{blue}//Recall: \hspace{-1pt}$(j_1,\ldots,j_4) =(t,w-t,t,n-w-t)$.}\\
	\hphantom{1cm} \foreachv $\bm{\lambda}=(\lambda_1,\ldots,\lambda_4) \vdash (j_1,\ldots,j_4)$ with height$(\lambda_i) \leq 2$ for all~$i$  \\
			\hphantom{1cm} \hphantom{1cm} start a new block~$M_{\bm{\lambda}}$\\
			\hphantom{1cm} \hphantom{1cm} \foreachv $(\tau_1,\ldots,\tau_4)  \in     \bm{T_{\lambda, m}'}$ from~$(\ref{wlambdaprimegebeuren})$\\
					\hphantom{1cm} \hphantom{1cm}\hphantom{1cm} \foreachv $(\sigma_1,\ldots,\sigma_4)  \in      \bm{T_{\lambda, m}'}$  from~$(\ref{wlambdaprimegebeuren})$\\
					\hphantom{1cm} \hphantom{1cm}\hphantom{1cm}\hphantom{1m} compute~$p_{\tau_i,\sigma_i}(F_1^{(i)})$ for~$i=1,\ldots,4$ (with $p_{\tau_i,\sigma_i}$\hspace{-1pt} cf.~$(\ref{ptsdef})$,~$F_1^{(i)}$\hspace{-2pt} cf.~\eqref{f1icw})\\					\hphantom{1cm} \hphantom{1cm}\hphantom{1cm}\hphantom{1m}	replace each $\mu_1(\omega_1)\cdots\mu(\omega_4)$ in~$\prod_{i=1}^4 p_{\tau_i,\sigma_i}(F_1^{(i)})$ 	by a variable \\
			\hphantom{1cm} \hphantom{1cm}\hphantom{1cm}\hphantom{1cm}\hphantom{1m}  $z(r(\omega_1\times \ldots\times \omega_4))$ if $r(\omega_1 \times \ldots\times \omega_4) \in \Omega_4^d$ and by~$0$ otherwise\\		
	\hphantom{1cm} \hphantom{1cm}\hphantom{1cm}\hphantom{1m}	 $(M_{\bm{\lambda}})_{(\tau_1,\ldots,\tau_4),(\sigma_1,\ldots,\sigma_4)}:= $ the resulting linear expression in~$z(\omega)$ \\				
					\hphantom{1cm} \hphantom{1cm}\hphantom{1cm} \ndv\\
			\hphantom{1cm} \hphantom{1cm} \ndv\\
		  	\hphantom{1cm} \hphantom{1cm}   \printv $M_{\bm{\lambda}} \succeq 0$   \\
		\hphantom{1cm} \ndv 	\\   
 \ndv\\		
		\text{\color{blue}//Now~$D= \emptyset$.}\\
	\foreachv $\lambda \vdash n$ with~$\height(\lambda)\leq2^{\xi}$  \\
		\hphantom{1cm} start a new block~$M_{\lambda}$\\	
		\hphantom{1cm} \foreachv $\tau \in   T_{\lambda,2^{\xi}}^{(\xi)}$ from~$(\ref{wset1})$ or~$(\ref{wset2})$\\
				\hphantom{1cm}\hphantom{1cm} \foreachv $\sigma \in     T_{\lambda,2^{\xi}}^{(\xi)}$ from~$(\ref{wset1})$ or~$(\ref{wset2})$ \\
				\hphantom{1cm}\hphantom{1cm}\hphantom{1cm} compute~$p_{\tau,\sigma}(F_1)$  in variables~$a_P^*$ (with~$p_{\tau,\sigma}$ cf.~$(\ref{ptsdef})$ and~$F_1$ cf.~\eqref{f1cw})\\	\hphantom{1cm}\hphantom{1cm}\hphantom{1cm}			replace each monomial~$\mu(\omega')$ by a variable~$z(r'(\omega))$ if $r'(\omega') \in \Omega_4^d$\\
	\hphantom{1cm}\hphantom{1cm}\hphantom{1cm}\hphantom{1cm}				 and by~$0$ otherwise	  \\
\hphantom{1cm}\hphantom{1cm}\hphantom{1cm}	 $(M_{\lambda})_{\tau,\sigma}:=$  the resulting linear expression in~$z(\omega)$\\				
				\hphantom{1cm}\hphantom{1cm} \ndv\\
		\hphantom{1cm} \ndv\\
		\hphantom{1cm}\ifv $\lambda=(n)$ \text{\color{blue} //Add a row and a column corresponding to $\emptyset$.}\\
	\hphantom{1cm}\hphantom{1cm}	add a row and column to~$M_{\lambda}$ indexed by~$\emptyset$
	\\
	\hphantom{1cm}\hphantom{1cm}	put~$(M_{\lambda})_{\emptyset,\emptyset}:=1$ and the entries~$(M_{\lambda})_{\emptyset,\tau}$ and~$(M_{\lambda})_{\tau,\emptyset}$ as in~$(\ref{leeg1})$ or~$(\ref{leeg2})$ \\
		\hphantom{1cm}\ndv\\
		  \hphantom{1cm}\printv $M_{\lambda} \succeq 0$   \\				
	\ndv 	\\  
	 	 \text{\color{blue}//Now nonnegativity of all variables (this includes the cases~$|D|=3$ and~$|D|=4$).}  \\ 
		\foreachv $\omega \in \Omega_4^d$ \\
	    \hphantom{1cm} \printv $ z(\omega) \geq 0$\\
	    \ndv 
     \end{tabular}}
    \end{minipage} }   \caption{\label{pseudocode}\small{Algorithm to generate semidefinite programs for computing $A_4(n,d,w)$ and $B_4(n,d,w)$.}}
\end{figure}
See Figure~$\ref{pseudocode}$ for an outline of the method. Here we write~$\omega_0 \in \Omega_{4}$ for the (unique) $S_n$-orbit of a  constant weight~$w$ code of size~$1$, and~$\omega_{\emptyset}$ for the orbit~$\{\emptyset\}$. Also, we write~$\Omega_4^d$ for the set orbits in~$\Omega_4$ that correspond to a code with minimum distance at least~$d$.

\chapter{Uniqueness of codes using semidefinite programming} \label{cu17chap}

\chapquote{Writing briefly takes far more\\time than writing at length.}{Carl Friedrich Gauss (1777--1855)}
\vspace{-3pt}

\noindent For~$n,d,w \in \N$, let~$A(n,d,w)$ again denote the maximum size of a binary code of word length~$n$, minimum distance~$d$ and constant weight~$w$. Schrijver recently showed, using semidefinite programming, that $A(23,8,11)=1288$, and we obtained in Chapter~$\ref{cw4chap}$ that~$A(22,8,11)=672$ and~$A(22,8,10)=616$.  In this chapter we show uniqueness of the codes achieving these bounds.

Let~$A(n,d)$ again denote the maximum size of a binary code of word length~$n$ and minimum distance~$d$. Gijswijt, Mittelmann and Schrijver showed that~$A(20,8)=256$, implying that the quadruply shortened extended  binary Golay code is optimal. We show that there are several nonisomorphic codes achieving this bound, and classify all such codes under the additional condition that all distances are divisible by 4.  

This chapter is based on joint work with Andries Brouwer~\cite{uniqueness}.

\section{Introduction}
Recall that a  binary constant weight~$w$ code~$C\subseteq \F_2^n$ with~$d_{\text{min}}(C)\geq d$ is called an~$(n,d,w)$-code. Moreover, a binary code~$C\subseteq \mathbb{F}_2^n$ and~$d_{\text{min}}(C) \geq d$ is called an~$(n,d)$-code.

Using semidefinite programming, some upper bounds on~$A(n,d,w)$ have recently been obtained that are equal to the best known lower bounds: it has been established that $A(23,8,11)=1288$ (see~\cite{schrijver}), and that $A(22,8,11)=672$ and $A(22,8,10)=616$ (see Chapter~$\ref{cw4chap}$, which is based on~$\cite{cw4}$). We show using the output of the corresponding semidefinite programs that the constant weight codes codes of maximum size are unique (up to coordinate permutations) for these~$n,d,w$.

For unrestricted (non-constant weight) binary codes, the bound~$A(n,d)=A(20,8)\leq256$ was obtained by Gijswijt,  Mittelmann and Schrijver in~\cite{semidef}, implying that the quadruply shortened extended binary Golay code is optimal.  The quadruply shortened extended binary Golay code is a linear~$(n,d)=(20,8)$-code of size~$256$ and has all distances divisible by~$4$. 

Up to equivalence, there is a unique optimum $(24-i,8)$-code of size~$2^{12-i}$, for $i=0,1,2,3$, namely the~$i$ times shortened extended binary Golay code \cite{brouwer2}. We show that the $4$~times shortened extended binary Golay code is not the only $(20, 8)$-code of size 256. We classify such codes under the additional condition that all distances are divisible by~$4$, and find that there exist~$15$ such codes.

\section{The semidefinite programming upper bound}

Following~\cite{semidef, schrijver} and Chapter~\ref{cw4chap}, we recall  semidefinite programming upper bounds on $A(n,d)$ and $A(n,d,w)$.  Fix~$n,d,w\in \N$ and let~$N$ be either~$\mathbb{F}_2^n$ or the set of words in~$\mathbb{F}_2^n$ of weight~$w$.\symlistsort{N}{$N$}{subset of~$\mathbb{F}_2^n$} For~$k \in \Z_{\geq 0}$, let~$\mathcal{C}_k$ be the collection of codes~$C \subseteq N$ with~$|C|\leq k$.\symlistsort{Ck}{$\mathcal{C}_k$}{collection of codes of cardinality at most~$k$}  For any~$D \in \mathcal{C}_k$, we define\symlistsort{Ck(D)}{$\mathcal{C}_k(D)$}{the set~$\{C \in \mathcal{C}_k \, \mid \, C \supseteq D, \, \lvert D\rvert +2\lvert C\setminus D\rvert \leq k  \}$}
\begin{align} 
\mathcal{C}_k(D) := \{C \in \mathcal{C}_k \,\, | \,\, C \supseteq D, \, |D|+2|C\setminus D| \leq k  \}.  
\end{align} 
Note that then~$|C \cup C'|= |C| + |C'| -|C\cap C'| \leq 2|D| + |C \setminus D| + |C' \setminus D| -|D| \leq k$ for all~$C,C' \subseteq \mathcal{C}_k(D)$.  
Furthermore, for any function~$x : \mathcal{C}_k \to \R$ and $D \in \mathcal{C}_k$, we define the~$\mathcal{C}_k(D) \times \mathcal{C}_k(D)$ matrix~$M_{k,D}(x)$ by $M_{k,D}(x)_{C,C'} : = x(C \cup C')$, for~$C,C' \in \mathcal{C}_k(D)$.\symlistsort{MkD}{$M_{k,D}(x)$}{variable matrix} Recall the parameter~$A_k(n,d,w)$ from~\eqref{A4ndw}:
\begin{align} \label{A4ndwu}
A_k(n,d,w):=    \max \big\{ \mbox{$\sum_{v \in N} x(\{v\})$}\,\, |\,\,&x:\mathcal{C}_k \to \R_{\geq 0}, \,\, x(\emptyset )=1, x(S)=0 \text{ if~$d_{\text{min}}(S)<d$}, \notag \\
 \,\, &M_{k,D}(x) \succeq 0 \text{ for each~$D$ in~$\mathcal{C}_k$}\big\}, 
\end{align}
where~$N$ is the set of words  in~$\mathbb{F}_2^n$ of weight~$w$.
(Here~$X \succeq 0$ means:~$X$ positive semidefinite.) Then~$A_k(n,d,w)$ is an upper bound on $A(n,d,w)$. Similarly, one obtains an upper bound $A_k(n,d)$ on~$A(n,d)$ by setting~$N:=\mathbb{F}_2^n$ in~$\eqref{A4ndwu}$, so that~$\mathcal{C}_k$ is the collection of unrestricted (not necessarily constant weight) codes of size at most~$k$. It can be proved that~$A_2(n,d)$ and~$A_2(n,d,w)$ are equal to the classical Delsarte linear programming bound in the Hamming and Johnson schemes respectively~\cite{delsarte}. 

\begin{remark}
In this chapter, we use~$A_k(n,d)$ to denote the described bound on~$A(n,d)$. It must not be confused with the parameter~$A_q(n,d)$ from~\eqref{aqndprem}.\symlistsort{Ak(n,d)}{$A_k(n,d)$}{upper bound on~$A(n,d)$ (only in Chapter~\ref{cu17chap})}
\end{remark}

 Let~$H$ be the group of distance-preserving permutations of~$N$. In case of constant weight-codes,~$H=S_n$, except if~$n=2w$;  in this case the group $H$ is twice as large, since then taking complements is also a distance-preserving permutation of~$N$. (Here the \emph{complement} of a word~$v \in N$ is the word~$v+\bm{1}$, where the addition is in~$\F_2^n$ and where~$\bm{1}$ denotes the all-ones word.) In case of non-constant weight codes,~$H= S_2^n \rtimes S_n$, where~$S_2^n$ denotes the direct product of~$n$ copies of~$S_2$. 
 
 Let~$\Omega_k$ be the set of $H$-orbits of codes in~$\mathcal{C}_k$ and let~$\Omega_k^d \subseteq \Omega_k$ be those orbits that correspond to codes with minimum distance at least~$d$.\symlistsort{Omegakd}{$\Omega_k^d$}{collection of $H$-orbits of codes of minimum distance at least~$d$}  By averaging an optimum~$x$ over all~$g \cdot x$ for~$g\in H$, one obtains the existence of a~$H$-invariant optimum solution to~$\eqref{A4ndwu}$. 
 
 The original problem~\eqref{A4ndwu} is equivalent to the much smaller problem in which the constraint is added that~$x$ is~$H$-invariant. We will write~$z({\omega})$ for the common value of a $H$-invariant function~$x$ on codes~$C$ in orbit~$\omega$. Hence, the matrices~$M_{k,D}(x)$ become matrices~$M_{k,D}(z)$ and we have considerably reduced the number of variables in~$(\ref{A4ndwu})$.\symlistsort{MkDz}{$M_{k,D}(z)$}{reduced variable matrix}  Moreover, a block diagonalization~$M_{k,D}(z) \mapsto U_{k,D}\T M_{k,D}(z) U_{k,D}$ can be obtained reducing the sizes of the matrices involved to make the computations in~$(\ref{A4ndwu})$ tractable (see Chapter~$\ref{cw4chap}$ and~$\cite{semidef, schrijver}$ for the reductions).
 
 It can be seen (cf.~\cite{semidef, cw4}) that the nonnegativity condition on~$x$, and hence on~$z$, is already implied by the positive semidefiniteness of all matrices~$M_{k,D}(x)$. When solving the semidefinite program with a computer, we add separate~$1\times 1$ blocks~$(z(\omega))$  which are required to be positive semidefinite as constraints (although nonnegativeness of~$z$ is already implied by positive semidefiniteness of the matrices~$M_{k,D}(x)$). As explained in the next section, this will allow us to easily determine which variables~$z(\omega)$ will be necessarily zero in any optimum solution. This may yield necessary conditions on all optimal codes, which may classify the optimal codes, in some cases implying uniqueness.

\subsection{Information about maximum size codes}\label{procuniq}
 Suppose that we have an instance of~$n,d,w$ or~$n,d$ for which $A_k(n,d,w)=A(n,d,w)$ or $A_k(n,d)=A(n,d)$. We want to obtain information about codes attaining these bounds from the semidefinite programming output. The semidefinite program~$\eqref{A4ndwu}$ can be written as follows:\symlistsort{F omega}{$F_{\omega}$}{matrix in semidefinite programs~(\ref{primal}) and~(\ref{dual})}\symlistsort{F empty}{$F_{\emptyset}$}{matrix in semidefinite programs~(\ref{primal}) and~(\ref{dual})}\symlistsort{b omega}{$b_{\omega}$}{constant in semidefinite programs~(\ref{primal}) and~(\ref{dual})}
 \begin{align}\label{primal}
  A_k(n,d,w) = \max \left\{ \sum_{\omega \in \Omega_k^d  \setminus \{\{ \emptyset \}\} } b_{\omega} z(\omega) \,\,\big{|}\,\, M= F_{\emptyset} - \sum_{\omega \in  \Omega_k^d   \setminus \{\{ \emptyset \}\} } F_{\omega} z({\omega})  \succeq 0 \right\}.
 \end{align}
Here~$b_{\omega_0}=|N|$, where~${\omega_0} \in \Omega_k^d$ corresponds to the orbit of a code of size~$1$ in~$\mathcal{C}_k$, and~$b_{\omega} =0$ for all other~$\omega \in \Omega_k^d  \setminus \{\{ \emptyset \}\} $. Moreover,~$M$ is a (large) block diagonal matrix that consists of blocks $U_{k,D}\T M_{k,D}(z)U_{k,D}$ (which are reduced versions of the blocks~$M_{k,D}(z)$ that are required to be positive semidefinite in~$\eqref{A4ndwu}$) and blocks~$(z({\omega}))$.  For each orbit~$\omega \in \Omega_k^d  \setminus \{\{ \emptyset \}\} $,  the matrix~$F_{\omega}$ is a matrix of the same size as~$M$ with entries the coefficients of~$-z(\omega)$ in the corresponding entries of~$M$. The matrix~$F_{\emptyset}$  is a matrix of the same size as~$M$ with entries the constant coefficients in the corresponding entries of~$M$. Recall that for two real-valued square matrices~$A,B$ of the same size, we write~$\langle A, B\rangle := \text{tr}(B\T A)$.  The dual program of~$\eqref{primal}$ then reads   
 \begin{align}\label{dual}
   \min \left\{ \langle F_{\emptyset}, X \rangle \,\,\big{|}\,\,  \langle F_{\omega}, X\rangle = b_{\omega}  \text{ for all~$\omega \in  \Omega_k^d\setminus \{\{ \emptyset \}\} $},\,~X\succeq 0 \right\}.
 \end{align}
If~$(M,z)$ is any optimum solution for~\eqref{primal} and~$X$ is an optimum solution for~$\eqref{dual}$ with the same values, then~$\langle M, X\rangle =0$.\symlistsort{X}{$X$}{feasible (or optimum) solution to~(\ref{dual})}\symlistsort{Mz}{$(M,z)$}{feasible (or optimum) solution to~(\ref{primal})} This is called \emph{complementary slackness}.\indexadd{complementary slackness} As~$M \succeq 0$ and~$X\succeq 0$, we have in particular~$z({\omega}) X_{\omega}=0$ for the separate~$1 \times 1$ blocks~$(z({\omega}))$ in~$M$ and~$(X_{\omega})$ in~$X$, where~$(X_{\omega})$ denotes the~$1\times 1$ block in~$X$ corresponding to the~$1\times 1$ block~$(z({\omega}))$ in~$M$.\symlistsort{Xomega}{$(X_{\omega})$}{$1\times 1$ block in~$X$ corresponding to the~$1\times 1$ block~$(z({\omega}))$ in~$M$} Thus
\begin{align*}
&\text{$X_{\omega}>0$ in some optimum solution to~$\eqref{primal}$ } \, 
\\ &\quad\quad \quad\quad\quad\quad\quad\, \,\,\,\Longrightarrow \, \,\,\,
\\&\text{$z({\omega})=0$ in each optimum solution to~$\eqref{dual}$}. 
\end{align*} 
 If~$z({\omega})=0$ for all solutions to~$\eqref{A4ndwu}$ with objective value~$A_k(n,d,w)=A(n,d,w)$, then for any code~$C$ of maximum size there is no subcode~$D\subset C$ with~$D \in \omega$. (Suppose otherwise; then one constructs a feasible solution to~$\eqref{A4ndwu}$ by putting~$x(S)=1$ for~$S \in \CC_k$ with~$S \subseteq C$ and~$x(S)=0$ else, and hence by averaging over~$H$ there exists a feasible $H$-invariant solution with~$z({\omega}) >0$, a contradiction.)  So orbit~$\omega$ does not appear in any code of maximum size. Hence we can identify some orbits~$\omega \in \Omega_k^d$ that cannot occur in any code of maximum size. We will call these orbits~\emph{forbidden orbits}.\indexadd{forbidden orbit} 
 
 We used the solver SDPA-GMP~$\cite{ nakata, sdpa}$ to conclude which orbits are forbidden. The semidefinite programming solver does not produce exact solutions, but approximations up to a certain precision. In our case the approximations are precise enough to verify (with certainty) that certain orbits are forbidden. See the appendix for details.

\section{Self-orthogonal codes} 
If~$u,v \in \F_2^n$, we define~$(u\cap v) \in \F_2^n$ to be the word that has~$1$ at position~$i$ if and only if~$u_i=v_i=1$.\symlistsort{u cap v}{$u \cap v$}{word in $\mathbb{F}_2^n$ that has~$1$ at position~$i$ if and only if~$u_i=v_i=1$} The following easy equality is well-known and will be used often throughout this chapter:
\begin{align} \label{wellknown}
d_H(u,v)= \wt(u) + \wt(v) -2 \wt(u \cap v), \,\,\, \text{ for all~$u,v \in \F_2^n$}. 
\end{align}
The function~$(u,v) \mapsto \wt(u \cap v) \pmod{2}$ is a non-degenerate symmetric $\F_2$-bilinear form on~$\F_2^n$. 
If~$(u,v)=0$, then~$u$ and~$v$ are called \emph{orthogonal}. A code~$C$ is \emph{self-orthogonal} if~$(u,v)=0$ for all~$u,v \in C$.\indexadd{code!self-orthogonal}
Given a code~$C \subseteq \F_2^n$, the \emph{dual code}~$C^{\perp}$ is the set of all~$v \in \F_2^n$ that are orthogonal to all~$u\in C$. A code~$C$ is called \emph{self-dual} if~$C=C^{\perp}$. For small~$n$, self-dual codes are classified by Pless and Sloane~\cite{sloanepless}.\symlistsort{Cperp}{$C^{\perp}$}{dual code of code~$C$}\indexadd{code!self-dual} 
  
 \section{Constant weight codes}
 
 With semidefinite programming three exact values of~$A(n,d,w)$ have been obtained. In~\cite{schrijver}, it is found that~$A_3(23,8,11)\leq 1288$, matching the known lower bound and thereby proving that~$A(23,8,11)=1288$.  Similarly, in Chapter~\ref{cw4chap} (which is based on~$\cite{cw4}$), the upper bounds $A(22,8,10)\leq 616$ and $A(22,8,11)\leq 672$ are obtained, which imply~$A(22,8,10)=616$ and $A(22,8,11)=672$. The latter two upper bounds are in fact instances of the bound~$B_4(n,d,w)$ defined in~\eqref{Bndw}, which is a bound (not necessarily strictly) in between~$A_3(n,d,w)$ and $A_4(n,d,w)$. 
 \begin{defn}[$B_4(n,d,w)$]
 The bound~$B_4(n,d,w)$ from~\eqref{Bndw} is defined by replacing in the definition of~$A_4(n,d,w)$ from~$\eqref{A4ndwu}$ the matrix~$M_{4,\emptyset}(x)$ by (the much smaller matrix)~$M_{2,\emptyset}(x|_{\CC_{2}})$.  
 \end{defn} 
 \noindent In this section we show that the codes attaining these bounds are unique up to coordinate permutations, using the information about forbidden orbits obtained from the semidefinite programming output. In order to prove uniqueness of the~$(23,8,11)$-code of maximum size, we start by proving uniqueness of the~$(24,8,12)$-code of maximum size. The uniqueness of this code can already be obtained from the classical linear programming bound by complementary slackness.
Recall that the \emph{distance distribution}~$(a_i)_{i=0}^n$ of a code~$C \subseteq \F_2^n$ is the sequence given by
$$
a_i := |C|^{-1} \cdot |\{(u,v) \in C \times C \,\, | \, \, d_H(u,v)=i\}|, \,\,\,\, \text{ for~$i=0,\ldots,n$}.
$$
  The complementary slackness based on computational results that we used to conclude uniqueness of the mentioned codes, is displayed (in~$(\ref{lpfact})$,~$(\ref{sdpfact1})$,~$(\ref{sdpfact2})$ and $(\ref{sdpfact3})$ below) at the beginning of each proof.

\subsection{\texorpdfstring{$A(24,8,12)$}{A(24,8,12)}}
\begin{proposition} \label{prop24}
Up to coordinate permutations there is a unique $(24,8,12)$-code of size~$2576$. It is given by the set of words of weight~$12$ in the extended binary Golay code.
\end{proposition}
\proof 
Let~$C$ be a~$(24,8,12)$-code of size~$2576$. The classical Delsarte linear programming bound in the Johnson scheme (which is equal to~$A_2(n,d,w)$) gives maximum 2576. Moreover, one has by complementary slackness,\footnote{We used SDPA-GMP to solve this LP. The approximate dual solution allows us, with a computation similar to the computation in the Appendix in~$(\ref{error})$-$(\ref{upperbound})$ below, to give a very small $\varepsilon >0$ such that in any optimum solution,~$a_i < \varepsilon$ for~$i \notin \{0,8,12,16,24\}$. But if~$a_i>0$ for some~$(n,d,w)=(24,8,12)$-code~$C$ with~$|C|=2576$, then for this code~$a_i \geq 2 / 2576$, by definition of the distance distribution of~$C$. Since~$\varepsilon <10^{-90} < 2/2576$ for~$i \notin \{0,8,12,16,24\}$ it follows that~$a_i=0$ for any~$(24,8,12)$-code~$C$ of size~$2576$.}
\begin{align} \label{lpfact}
a_i =0 \text{ for }~i \notin \{0,8,12,16,24\}.
\end{align}
 
To derive uniqueness, consider the $\F_2$-linear span~$F:=\langle C \rangle$  of~$C$.\symlistsort{C langle}{$\langle C \rangle$}{$\F_2$-linear span of code~$C \subseteq \F_2^n$} Note that~$C$ (by~\eqref{lpfact}), hence~$F$, is self-orthogonal. So~$|F|\leq 2^{24/2}=2^{12}$. Since~$|F|\geq |C| >2^{11}$ and since~$F$ is linear, we must have~$|F|=2^{12}$. So~$F$ is self-dual. Let $u \in F$, $u \ne 0$. The sets~$\{ u + x \, |\, x \in C\} \subseteq F$ and~$C\subseteq F$ have non-empty intersection, because both sets have size~$2576> |F|/2$. So~$u+x=y$ for some~$x,y \in C$. But then~$\text{wt}(u) = d_H(x,y) \ge 8$, as~$C$ has minimum distance~$8$. It follows that~$F$ has minimum distance 8,
and we conclude that~$F$ is the extended binary Golay code. So~$C$ is the set of weight~$12$ words in the extended binary Golay code.
\endproof

\subsection{\texorpdfstring{$A(23,8,11)$}{A(23,8,11)}}
\begin{proposition}
Up to coordinate permutations there is a unique $(23,8,11)$-code of size~$1288$. It is given by the set of words of weight~$11$ in the binary Golay code.
\end{proposition}
\proof 
Let~$C$ be a~$(23,8,11)$-code of size~$1288$. With the solution of the semidefinite program $A_3(23,8,11)$ (which is~$1288$) from~$\cite{schrijver}$ one obtains, by considering the forbidden orbits from the semidefinite programming output:\footnote{Note that the LP does not give this information: the Delsarte bound is 1417, which is not optimal.} 
\begin{align} \label{sdpfact1}
\text{if~$x,y \in C$ then~$ d_H(x,y) \leq 16$.}
\end{align}
 Construct a code~$D$ of length~$24$, weight~$12$ and size~$2576$ as follows: add a symbol~$1$ to every codeword of~$C$, put it in~$D$ and put also the complement of the resulting word in~$D$. Then~$D$ has minimum distance~$8$ by~$\eqref{sdpfact1}$. Hence~$D$ is the set of weight~$12$ words in the extended Golay code~$F$, by Proposition~$\ref{prop24}$. The automorphism group of the extended binary Golay code acts transitively on the coordinate positions~\cite{sloane}. Hence,~$C$ is the set of weight~11 words in the binary Golay code.  
\endproof

\subsection{\texorpdfstring{$A(22,8,11)$}{A(22,8,11)}}
\begin{proposition}
Up to coordinate permutations there is a unique $(22,8,11)$-code of size~$672$.
\end{proposition}
\proof 
Let~$C$ be a~$(22,8,11)$-code of size~$672$. First one concludes that~$a_{14}=0$ using the semidefinite program $B_4(22,8,11)$ from~$\cite{cw4}$. This is explained in more detail in the appendix: if~$a_{14}>0$, then~$a_{14} \geq 2/672$ and~$a_{10}+a_{14}+a_{18}+a_{22} \geq 318/672$ (by Proposition~\ref{extraprop2} below). We add these two constraints to the program~$B_4(22,8,11)$. The resulting bound is strictly smaller than~$672$, so~$a_{14}=0$ in any~$(22,8,11)$-code of size~$672$. 

Subsequently, by considering the forbidden orbits in the solution of the semidefinite program $A_3(22,8,11)$ from~$\cite{schrijver}$ with the added constraint that~$a_{14}=0$ (the solution of $A_3(22,8,11)$ with this added constraint is~$672$) one obtains:
  \begin{align} \label{sdpfact2}
\text{if~$x,y,z \in C$ then~$d_H(x,y) \in \{0,8,12,16\}$ and~$\text{wt}(x+y+z)\in \{7,11,15\}$}.
\end{align} 
Let~$D$ be the collection of $672+672 = 1344$ codewords of length~$24$ of the form~$10x$ with~$x\in C$ together with their complements, and let $F := \langle D \rangle $ be the~$\mathbb{F}_2$-linear span of~$D$. All distances in~$D$ belong to~$\{8,12,16\}$ by~(\ref{sdpfact2}), so~$D$, and hence also~$F$, is self-orthogonal, which implies~$|F|\leq 2^{24/2} = 2^{12}$. Since all words in~$D$ have weight divisible by~$4$ and~$F$ is self-orthogonal, all words in~$F$ also have weight divisible by 4.

The code $F$ contains words of forms~$01x$,~$10y$,~$11z$ and~$00u$. Each form occurs at least~$672$ times, so $|F| \geq 4 \cdot 672 > 2^{11}$, hence $|F| = 2^{12}$ and~$F$ is self-dual.

To show that~$F$ is the extended binary Golay code, it suffices to prove that all words in~$F$ have weight~$\geq 8$, i.e., that no word in~$F$ has weight~$4$. Words of~$F$ are sums of words~$10x$ with~$x \in C$, possibly together with the all-ones word. So we must prove that sums of words~$10x$ do not have weight~$4$ or~$20$. A sum of words~$10x$ starts with~$00$ or~$10$ and is the sum of an even or odd number of words~$10x$, respectively. 

Words in~$F$ starting with~$00$ form a subcode~$F_{00}$ of~$F$ of size~$2^{10} = 1024$. If~$u \in F_{00}$, then~$\{u+10y \,|\, y \in C\} \cap \{10x  \,|\, x \in C\} \neq \emptyset$, as there are~$1024$ words in~$F$ starting with~$10$ but~$|C|=672> 1024/2$. So~$u=10x+10y$ for some~$x,y \in C$, hence~$F_{00} = \{10x + 10y \, | \, x,y \in C\}$. However, distances~$4$ and~$20$ do not occur in~$C$, so words in~$F_{00}$ do not have weight~$4$ or~$20$. 

The~$1024$ words in~$F$ starting with~$10$ are formed by the coset~$10x+ F_{00}$ (with~$x \in C$ arbitrary but fixed) and hence are a sum of three elements of the form~$10x$ with $x \in C$. But such a sum has weight~$8,12$ or~$16$ by~(\ref{sdpfact2}), implying that words in~$F$ starting with~$10$ do not have weight~$4$ or~$20$. 

Therefore weights~$4$ and~$20$ do not appear in~$F$, so~$F$ is indeed the extended binary Golay code. As the automorphism group of the extended binary Golay code~$F$ acts 2-transitively on the coordinate positions~\cite{sloane}, this implies that~$C$ is unique. 
\endproof

\subsection{\texorpdfstring{$A(22,8,10)$}{A(22,8,10)}}

\begin{proposition}
Up to coordinate permutations there is a unique $(22,8,10)$-code of size~$616$.
\end{proposition}
\proof 
Let~$C$ be a~$(22,8,10)$-code of size~$616$. First one concludes that~$a_{14}=0$ using the semidefinite program $B_4(22,8,10)$ from~$\cite{cw4}$. This is explained in more detail in the appendix: if~$a_{14}>0$, then~$a_{14} \geq 2/616$ and~$a_{10}+a_{14}+a_{18} \geq 208/616$ (by Proposition~\ref{extraprop1} below). We add these two constraints to the program~$B_4(22,8,10)$. The resulting bound is strictly smaller than~$616$, so~$a_{14}=0$ in any~$(22,8,10)$-code of size~$616$. 

Subsequently, by considering the forbidden orbits in the solution of the semidefinite program $A_3(22,8,10)$ from~$\cite{schrijver}$ with the added constraint that~$a_{14}=0$ (the solution of $A_3(22,8,10)$ with this added constraint is~$616$) one obtains:
\begin{align} \label{sdpfact3}
\text{if~$x,y,z \in C$ then~$d_H(x,y) \in \{0,8,12,16\}$ and~$\text{wt}(x+y+z)\in \{6,10,14,22\}$}.
\end{align} 
Let~$F:=\langle C \rangle$. Since~$C$ is self-orthogonal and has words of weights divisible by~$2$ but not by~$4$,~$F$ is self-orthogonal and has half of the weights divisible by~$4$ and half of the weights divisible by~$2$ but not by~$4$. Both halves of~$F$ have size~$\geq |C|=616$, but~$F$ has size~$\leq 2048$ as it is self-orthogonal. So~$|F|=2048$ and~$F$ is self-dual.

Let~$E\subseteq F$ be the subcode of~$F$ consisting of all words with weight divisible by~$4$. For each~$u \in E$, we have~$C \cap \{u+y \,|\, y \in C\} \neq \emptyset$  (as $|C|=616 >1024/2$), so~$u=x+y$ for some~$x,y \in C$. Hence~$E=\{x+y \, | \, x,y \in C\}$. By~(\ref{sdpfact3}), no word in~$E$ has weight~$4$. So weight~$4$ does not occur in~$F$.

If any word~$u$ in~$F$ has weight~$2$ then it is in~$F \setminus E = x+ E$ (with~$x \in C$ arbitrary). So it is the sum of three words in~$C$.  But such sums do not have weight~$2$ by~$(\ref{sdpfact3})$,  hence no word in~$F$ has weight~$2$. So~$F$ is a self-dual code of minimum distance~$6$.  As the self-dual~$(n,d)=(22,6)$-code is unique (cf.~$\cite{sloanepless}$),~$F$ is unique. Hence also~$C$ is unique, as it is the collection of weight~$10$ words of~$F$. (Note that two weight~$10$ words in~$F$ have distance~$0 \pmod{4}$, so distance at least~$8$, since~$d_H(u,v)=\wt(u)+\wt(v)-2\wt(u \cap v)$ for any two words~$u,v \in F$.)
\endproof

\section{The quadruply shortened binary Golay code}

Recently, Gijswijt, Mittelmann and Schrijver \cite{semidef} proved that~$A(20,8)=256$, with the semidefinite program~$A_4(n,d)$ from~$\eqref{A4ndwu}$. An example of a code attaining this bound is the four times shortened extended binary Golay code, which has distance distribution
\begin{align}\label{distgol}
a_0=1, \,\, a_8=130,\,\, a_{12}=120, \,\,a_{16}=5,\,\,\text{$a_i=0$ for all other~$i$}. 
\end{align}
This code is formed by the words starting with~$0000$ in the extended binary Golay code with these first four coordinate positions removed. 

Recall that two binary codes~$C,D \subseteq \F_2^n$ are \emph{equivalent} if~$D$ can be obtained from~$C$ by first permuting the~$n$ coordinates and by subsequently permuting the alphabet~$\{0,1\}$ in each coordinate separately.

Up to equivalence there are unique $(24-i,8)$-codes of size~$2^{12-i}$ for $i=0,1,2,3$, namely the~$i$ times shortened extended binary Golay codes \cite{brouwer2}. In this section we show that there exist several nonisomorphic~$(20,8)$-codes of size~$256$. First we show that there exist such codes with different distance distributions. Subsequently we classify such codes under the additional condition that all distances are divisible by~$4$.

We start by recovering information about possible distance distributions from the semidefinite program~$A_4(20,8)$. Write~$\omega_t \in \Omega_4$ for the orbit of two words at Hamming distance~$t$.  From a code~$C$ with distance distribution~$(a_i)$, one constructs a feasible solution to~$\eqref{A4ndwu}$ by putting~$x(S)=1$ for~$S \in \CC_k$ with~$S \subseteq C$ and~$x(S)=0$ else, and hence by averaging over~$H$ one obtains a feasible $H$-invariant solution with variables~$z({\omega})$. This solution has
\begin{align*}
z({\omega_t}) &= \frac{1}{|H|} \sum_{g \in H} g \cdot x(\{y_1,y_2\}) =  \frac{t!(20-t)!}{|H|}|\{(u,v) \in C^2 \,\, : \,\, d_H(u,v)=t \}|  \notag 
\\&= \frac{|\{(u,v) \in C^2 \,\, : \,\, d_H(u,v)=t \}|}{2^{20} \binom{20}{t}}   = \frac{|C| a_t}{2^{20} \binom{20}{t}}, 
\end{align*}
where~$\{y_1,y_2\}$ is any pair of words with distance~$t$ and~$H=S_2^{20} \rtimes S_{20}$. So we can add linear constraints on the~$a_i$ as linear constraints on the variables~$z({\omega_t})$ to our semidefinite program.

The distance distribution~$(a_i)$ is not determined uniquely by the  requirement that the sequence~$(z(\omega))_{\omega \in \Omega_4}$ is an optimal solution of
the semidefinite program~$A_4(20,8)$ (from~\eqref{A4ndwu} and~\eqref{primal}).\footnote{By contrast, in all constant weight cases considered in this chapter, the values of the~$z({\omega_t})$ give the unique distance distribution of the (unique up to coordinate permutations)~$(n,d,w)$-codes of maximum size.} We find minimum possible values for some of the~$a_i$ for the case where all distances are even as follows. For any code~$C$, the~$a_i$ ($i\neq 0$) are integer multiples of~$2/|C|$. So for any~$(20,8)$-code of size~$256$,
$$
 \text{if~$a_{16}<1$ then~$a_{16}\leq 254/256$.}
 $$
   With the constraint~$a_{16} \leq 254/256$ the semidefinite program returns an objective value strictly smaller than~$256$. So~$a_{16} \geq 1$. Similarly, we find~$a_8 \geq 126$ and~$a_{12} \geq 96$. If we simultaneously add the constraints~$a_8 \leq 126$, $a_{12} \leq 96$ and~$a_{16} \leq 1$, the semidefinite program returns~$256$ as objective value, and the values of~$z({\omega_{10}})$ and~$z({\omega_{14}})$ force~$a_{10}=a_{14}=16$. Therefore, apart from the 4 times shortened extended binary Golay code, also a code with
\begin{align}\label{dist20}
a_1=1,\,\,a_{8}=126,\,\, a_{10}=16, \,\, a_{12}=96,\,\, a_{14}=16,\,\, a_{16}=1,\,\,\text{$a_i=0$ for all other~$i$},
\end{align}
is allowed by the program~$A_4(20,8)$. Such a code exists, as the following construction demonstrates. 


Start with the extended binary Golay code~$F$ containing the weight~$8$ word~$u$ with all 1s in the first eight positions. As~$A(24-8,8)=A(16,8)=32$ (see~$\cite{brouwertableand}$), there can be at most~$32$ words in~$F$ starting with~$8$ zeros. These form a linear subcode~$E$ of~$F$. As any word in~$F$ has an even number of~$1$s at the first eight positions, there are at most~$2^7=128$ distinct cosets~$E+v$ in~$F$. As~$32 \cdot 128=2^{12}=|F|$ it follows that~$|E|=32$ and there are exactly~$128$ distinct cosets~$E+v$.
 
 So if we specify a string of~$8$ symbols with an even number of ones and take all words in~$F$ having these~$8$ fixed symbols in the first~$8$ positions, we obtain a subcode~$D$ of~$F$ of size~$32$ and minimum distance at least~$8$.  Choose the following~$8$ specifications, each giving a subcode~$D$ of size~$32$ and minimum distance~$8$. 
\begin{align*} 
    \begin{matrix} 
00000000 \\
11000000 \\
10100000 \\
10010000 \\
10001000 \\
10000100 \\
10000010 \\
10000001
\end{matrix} \,\,\, \quad \text{ and then replace the first~$8$ coordinates by }\quad \,\,\begin{matrix} 
0000 \\
1100 \\
1010 \\
1001 \\
0110 \\
0101 \\
0011 \\
1111
\end{matrix}.
\end{align*}
This yields a~$(20,8)$-code of size~$8\cdot 32= 256$ in which distances~$10$ and~$14$ occur. Note that this code indeed has minimum distance at least~$8$: first observe that each code~$D$ has minimum distance at least~$8$. Then note that for two different specifications the first part (the first~$8$ positions) had distance at most~$2$ before the replacement, so the second part has distance at least~$6$. After the replacement of the first part, the first~$4$ positions have distance at least~$2$, so two words obtained from different specifications have, after the replacement, in total distance at least~$2+6=8$.  One verifies by computer that its distance distribution is given by~$\eqref{dist20}$. So, there exist~$(20,8)$-codes of maximum size with distance distribution~$\eqref{distgol}$ as well as with~$\eqref{dist20}$. 

\subsection{Unrestricted~\texorpdfstring{$(20,8)$}{(20,8)}-codes of maximum size with all distances divisible by~4}\label{unresdsect}

In this section we give a classification of the~$(20,8)$-codes of size~$256$ with all distances divisible by~$4$. An example of such a code is the quadruply shortened extended binary Golay code~$B$, which is linear. \symlistsort{B n}{$B$}{quadruply shortened extended binary Golay code (only in Sect.~\ref{unresdsect})} 
There is, up to equivalence, only one such code, since the automorphism group of the extended Golay code acts 5-transitively on the coordinate positions~\cite{brouwer2}. (Moreover, Dodunekov and Encheva~$\cite{dodunekov}$ have proved that there exists, up to equivalence, only one \emph{linear}~$(20,8)$-code of size~$256$.) The quadruply shortened extended binary Golay code contains~$5$ words of weight~$16$, forming a subcode~$D$. Since the minimum distance is~$8$, each of those words must have the~$0$'s at different positions. So we can assume that \symlistsort{D}{$D$}{the code defined in (\ref{Dmat}) (only in Sect.~\ref{unresdsect})}
\begin{align} \label{Dmat}
D=     \begin{matrix} 
00001111111111111111 \\
11110000111111111111 \\
11111111000011111111 \\
11111111111100001111 \\
11111111111111110000
\end{matrix},
\end{align}
 with linear span~$\langle D \rangle \subseteq B$ of dimension~$4$. So~$B$ is a union of~$16$ cosets~$u+\langle D\rangle $. If we replace a coset~$u + \langle D \rangle$ by its complement~$\mathbf{1}+u + \langle D \rangle$, we obtain another~$(20,8)$-code of maximum size that is not linear, so this is really a different code. Note that all distances remain divisible by ---but not equal to--- $4$, as~$d_H(x,y)\in \{8,12\}$ for any~$x \in u + \langle D \rangle$ and~$y \in B \setminus (u + \langle D \rangle) $, so~$d_H(\mathbf{1}+x,y) \in \{8,12\}$. By replacing any of the~$16$ cosets~$u+\langle D \rangle$ in~$B$ with~$\mathbf{1}+u+\langle D \rangle$, we obtain~$2^{16}=65536$ codes with all distances divisible by~$4$. 
 
 In this section we will first prove that any maximum-size~$(20,8)$-code  with all distances divisible by~$4$ is equivalent to one of the~$2^{16}$ thus obtained codes. Secondly, we will obtain (by computer) that these~$2^{16}$ codes can be partitioned into~$15$ equivalence classes. 
 
 In order to prove the first result, we start by proving two auxiliary propositions. Let~$C$ be any~$(20,8)$-code of size~$256$ with all distances divisible by~$4$ and containing~$\mathbf{0}$, the zero word. Define~$E:=\langle C \cup \{ \mathbf{1}\} \rangle$ to be the linear span of~$C$ together with the all-ones vector.
\begin{proposition} \label{Eprop}
Up to a permutation of the coordinate positions, the codes~$E$ and $\langle B \cup \{  \mathbf{1} \} \rangle $ are the same.
\end{proposition}
\proof 
 After the constraints~$a_i=0$ if~$4 \nmid i$ and~$a_{20} \geq 2/256$ are added to the Delsarte bound for~$(n,d)=(20,8)$, the linear program (LP) returns a solution strictly smaller than~$256$. Therefore~$a_{20}=0$ in any~$(20,8)$-code of size~$256$ with all distances divisible by~$4$. As~$A(20,8,8)=130$ (cf.~\cite{brouwertable}) and $A(20,8,4)=\floor{20/4}=5$, one has~$a_{8}\leq 130$ and~$a_{16} \leq 5$. Moreover, the LP-bound contains the inequalities $a_8-a_{12}-3a_{16}+5 \geq 0$ and~$-a_{8} -a_{12} +31 a_{16} +95 \geq 0$  (given that~$a_{20}=0$). We add those two equations and use that~$a_{16} 
 \leq 5$ to obtain that~$a_{12}\leq 120$. As~$256=1+a_8+a_{12}+a_{16}$, the distance distribution of any~$(20,8)$-code of size 256 with all distances divisible by 4 is given by~$(\ref{distgol})$.  Moreover, the values in~$(\ref{distgol})$ are not mere averages: as~$A(20,8,8)=130$ and $A(20,8,4)=5$, the number of words at distance~$i$ from \emph{any} word~$u \in C$ is specified by~$(\ref{distgol})$.

Since~$C$ is self-orthogonal and has all distances divisible by~$4$, also~$E$ is self-orthogonal and has all distances divisible by~$4$. Furthermore,~$E$ has dimension~$9$. To see this, note that~$\mathbf{1} \notin C$ since~$a_{20} = 0$, so ~$|E|\geq 257$, so~$|E| \geq 512$. On the other hand,~$\dim E <10$, as~$E$ is self-orthogonal with all distances divisible by~$4$, but there does not exist a self-dual code of length~$20$ with all distances divisible by~$4$ (cf.~$\cite{pless}$). So~$\dim E =9$ and~$|E|=512$, implying that
\begin{align} \label{ofof}
\text{for every word~$u \in E$ one has~$u\in C$ or~$\mathbf{1}+u \in C$.} 
\end{align}
For any code we write~$A_i$ for the number of words of weight~$i$.\symlistsort{Ai}{$A_i$}{number of words of weight~$i$ in code (only in Prop.~\ref{Eprop})} Since~$C$ has weights $A_0=1$, $A_8=130$, $A_{12}=120$, $A_{16}=5$, we conclude that~$E$ has weights 
\begin{align*} 
A_0=1,\, A_4=5, \,A_8=250,\, A_{12}=250, \,A_{16}=5,\, A_{20}=1.
\end{align*}
The orthogonal complement~$E^\perp$ of~$E$ has dimension 11, and is a union $E \cup (a+E) \cup (b+E) \cup (c+E)$.
Here $a,b,c$ have even weight (because $\mathbf{1} \in E$), so each of $\langle \{a\} \cup E\rangle $ and $\langle \{b\} \cup E\rangle $
and $\langle  \{c\} \cup E \rangle$ is self-dual. This means that $a,b,c$ are mutually non-orthogonal.

Look at Pless~$\cite{pless}$ to find the self-dual codes of length~$n=20$ and dimension~$10$.
There are 16 such codes, but we can forget about those with $A_8 < 250$.
There is a unique self-dual code of length~$n=20$ and dimension~$10$ with $A_8 \geq 250$, namely $M_{20}$ with weight enumerator 
\begin{align*} 
A_0 = 1,\, A_4 = 5,\, A_6 = 80,\, A_8 = 250,\, A_{10} = 352, \,A_{12}=250,\, A_{14}=80,\, A_{16}=5,\, A_{20}=1,
\end{align*} 
and~$E$ is the subcode of $M_{20}$ consisting of the words of weight divisible by~$4$, hence is unique. This means that~$E$ does not depend on the particular choice of~$C$ (up to a permutation of the coordinates), proving the desired result. 
\endproof 

\begin{proposition} \label{Dprop}
Let~$C$ be any~$(20,8)$-code of size~$256$ with all distances divisible by~$4$ containing~$\mathbf{0}$. Then~$C$ is invariant under translations by weight~$16$ words from~$C$.
\end{proposition}
\proof 
Clearly, if $a,b,c \in \F_2^{20}$ with $\wt(a)=16$ and $a+b+c=\mathbf{1}$, then $d_H(b,c) = 4$.
Now let~$b \in C$ be arbitrary, and~$a \in C$ a weight~$16$ vector. Then we have~$a+b+c \neq \mathbf{1}$ for all~$c \in C$. So~$\mathbf{1}+a+b \notin C$ while~$\mathbf{1}+a+b \in \langle C \cup \{\mathbf{1}\} \rangle$. Hence,~$a+b \in C$, by~$\eqref{ofof}$. So indeed,~$C$ is invariant under translations by weight~$16$ words from~$C$. (This in particular implies that~$C$ contains the $4$-dimensional linear span of the weight~$16$ vectors.)
\endproof 
Fix representatives~$u_1,\ldots,u_{16}$ for the cosets~$u_i+\langle D \rangle$ of the linear quadruply shortened binary extended Golay code~$B$ (see Table~$\ref{tab1}$ for a possible choice). Using Propositions~$\ref{Eprop}$ and~$\ref{Dprop}$, we obtain the first main result of this section.

\begin{proposition}\label{216}
Let~$C$ be any~$(20,8)$-code of size~$256$ with all distances divisible by~$4$. Then~$C$ is equivalent to~$B$ with some of the cosets~$u_i+\langle D \rangle$ replaced by~$\mathbf{1}+u_i+\langle D \rangle$. 
\end{proposition}
\proof 
We may assume,  by applying a distance-preserving permutation to~$C$,  that~$C$ contains~$\mathbf{0}$ and that~$\langle C \cup \{ \mathbf{1}\} \rangle=\langle B \cup \{\mathbf{1}\} \rangle$  (by Proposition~$\ref{Eprop}$). Then~$C$ contains~$5$ weight~$16$ vectors of the form~$\eqref{Dmat}$. By Proposition~$\ref{Dprop}$, the code~$C$ is a union of~$16$ cosets~$u+\langle D \rangle$, for some vectors~$u$. The code~$ \langle B \cup \{ \mathbf{1}\} \rangle=\langle C \cup \{\mathbf{1}\} \rangle $ is a union of cosets~$u_i+\langle D \rangle $ together with their complements~$\mathbf{1}+u_i + \langle D \rangle$. This implies by~$\eqref{ofof}$ that each coset of~$C$ has the form~$u_i+\langle D \rangle$ or~$\mathbf{1}+u_i+\langle D \rangle$ (and~$C$ cannot contain both~$u_i$ and~$\mathbf{1}+u_i$ at the same time as~$a_{20}=0$), as required. 
\endproof

\begin{table}[ht]
{\small
  \setlength{\abovedisplayskip}{6pt}
  \setlength{\belowdisplayskip}{\abovedisplayskip}
  \setlength{\abovedisplayshortskip}{0pt}
  \setlength{\belowdisplayshortskip}{3pt}
  \setlength{\tabcolsep}{.16666667em}
\begin{align*}
{\begin{tabular}{|l| *{15}{c}|}\hline
   Coset representative &  $C_1$ & $C_2$ & $C_3$ & $C_4$ & $C_5$ & $C_6$  & $C_7$& $C_8$ & $C_9$ & $C_{10}$ & $C_{11}$ & $C_{12}$ & $C_{13}$  & $C_{14}$ & $C_{15}$    \\
\hline
$u_{1\phantom{0}} = 00000000000000000000$ \,\,\,\,\,\,\,\,\,\,    & &1&1& 1&1&1&   1&1&1&1&  & & & 1&    \\
$u_{2\phantom{0}} = 00000101010101011010$                         & & &1& 1&1&1&   1&1&1&1& 1&1&1&  &1   \\
$u_{3\phantom{0}} = 00001001011001101100$                         & & & & 1&1& &   1& & &1& 1&1&1& 1&    \\
$u_{4\phantom{0}} = 00001100001100110110$                         & & & &  &1& &    & & & & 1& & &  &1   \\
$u_{5\phantom{0}} = 10100000010101101001$                         & & & &  & &1&   1&1&1&1& 1&1&1&  &    \\
$u_{6\phantom{0}} = 10100101000000110011$                         & & & &  & & &    &1& & &  &1& & 1&1   \\
$u_{7\phantom{0}} = 10101001001100000101$                         & & & &  & & &    & &1& &  & &1& 1&1   \\
$u_{8\phantom{0}} = 10101100011001011111$                         & & & &  & & &    & & & &  & & &  &    \\
$u_{9\phantom{0}} =  11000000011000110101$                        & & & &  & & &    & & &1& 1&1&1& 1&1   \\
$u_{10} = 11000101001101101111$                                   & & & &  & & &    & & & &  & &1& 1&1   \\
$u_{11} = 11001001000001011001$                                   & & & &  & & &    & & & &  & & &  &    \\
$u_{12} = 11001100010100000011$                                   & & & &  & & &    & & & &  & & &  &    \\
$u_{13} = 01100000001101011100$                                   & & & &  & & &    & & & &  & & &  &    \\
$u_{14} = 01100101011000000110$                                   & & & &  & & &    & & & &  & & &  &    \\
$u_{15} = 01101001010100110000$                                   & & & &  & & &    & & & &  & & &  &    \\
$u_{16} = 01101100000001101010$                                   & & & &  & & &    & & & &  & & &  &    \\
\hline  
\end{tabular}               
              }
  \end{align*}
}%
\caption{\label{tab1}\small The~$(20,8)$-codes of size~$256$ with all distances divisible by~$4$. The quadruply shortened extended binary Golay code~$B=C_1$ is the union~$\cup_{i=1}^{16} (u_i+ \langle D \rangle)$. The other codes~$C_j$ ($j=2,\ldots,15$) are obtained from~$B$ by replacing the coset~$u_i+\langle D \rangle$ by~$\mathbf{1}+u_i + \langle D \rangle$ if there is a~$1$ in entry~$(u_i,C_j)$ in the above table. }
\end{table}

It remains to classify the~$2^{16}=65536$ codes obtained from~$B$ by replacing some of the cosets~$u_i+\langle D \rangle$ by~$\mathbf{1}+u_i+\langle D \rangle$. For this we use the graph isomorphism program \texttt{nauty}~$\cite{dreadnaut}$. For any code~$C$ of word length~$n$ containing~$m$ codewords, a graph with~$2n+m$ vertices is created: one vertex for each codeword~$u \in C$ and two vertices~$0_i$ and~$1_i$ for each coordinate position. Each code word~$u$ has neighbor~$0_i$ if~$u_i=0$ and~$1_i$ if~$u_i=1$ ($i=1,\ldots,n)$. Moreover, there are edges~$\{0_i,1_i\}$ ($i=1,\ldots,n$). 

All code words have degree~$n$ and the coordinate positions have (in this case) larger degree. An automorphism of this graph permutes the codewords and permutes the coordinate positions. In this way one finds a subgroup of~$S_2^n \rtimes S_n$ that fixes~$C$ and the question of code equivalence is transformed into a question of graph isomorphism. With the program \texttt{nauty} we compute a canonical representative for each of the~$2^{16}$ mentioned codes. In this way we find that the~$2^{16}$ codes from Proposition~$\ref{216}$ can be partitioned into~$15$ equivalence classes. See Table~$\ref{tab1}$ for the classification.

\begin{proposition}
There are $15$ different $(20,8)$-codes of size~$256$ with all distances divisible by~$4$ up to equivalence. \qed
\end{proposition}

 \section{Appendix: Approximate solutions}
 
 The semidefinite programming solver does not produce exact solutions, but approximations up to a certain precision. Here we show that we have enough precision to conclude~$\eqref{sdpfact1}$,~$\eqref{sdpfact2}$ and~$\eqref{sdpfact3}$. 
 
 Let~$(M,z)$ be feasible for~$\eqref{primal}$ with optimum value at least~$A(n,d,w)$ and suppose that~$X$ is an approximation of the dual program. That is, we have~$X \succeq 0$\footnote{We used a separate java program to verify that~$X \succeq 0$ (in fact,~$X \succ 0$) in the SDP-outputs used in this chapter.} and for all~$\omega \in \Omega_k^d \setminus \{\{ \emptyset \}\}$:
\begin{align} \label{error}
    \langle X, F_{\omega} \rangle = b_{\omega} + \epsilon_{\omega}, 
\end{align}
 for small~$\epsilon_{\omega}$.\symlistsort{epsilonomega}{$\epsilon_{\omega}$}{small number (error)}  Consider one particular~$\omega \in \Omega_k^d\setminus \{\{ \emptyset \}\}$. There is a~$1 \times 1$ block~$(z({\omega}))$ in~$M$ and hence also a corresponding~$1 \times 1$ block~$(X_{\omega})$ in~$X$. Remove these~$1\times1$ blocks from~$M$ and~$X$ and call the resulting matrices~$M'$ and~$X'$. Then
 \begin{align}
     0 &\leq \langle M', X' \rangle = \langle M,X\rangle -z({\omega})X_{\omega}  \notag 
     \\&= \langle F_{\emptyset},X\rangle - \sum_{\omega \in \Omega_k^d\setminus \{\{ \emptyset \}\}} z({\omega})b_{\omega}  + \sum_{\omega \in \Omega_k^d\setminus \{\{ \emptyset \}\}} z({\omega}) \epsilon_{\omega} -X_{\omega} z({\omega}), \label{errorcomp}
 \end{align}
 as~$M'$ and~$X'$ are positive semidefinite, where we used~$\eqref{error}$ and the definition of~$M$ from~$\eqref{dual}$ in the second equality. Note that~$\sum_{\omega \in \Omega_k^d\setminus \{\{ \emptyset \}\}} z({\omega})b_{\omega}$ is bounded from below by~$|C|$, and each~$z({\omega})$ is bounded from above by~$z({\omega_0})$ (this can be done since all~$2 \times 2$ principal submatrices in the semidefinite program~$(\ref{primal})$ are positive semidefinite), which is bounded from above by~$1$. Hence
 \begin{align} \label{errorcomp2}
 X_{\omega}z({\omega}) \leq \langle F_{\emptyset},X \rangle - |C| +  \sum_{\omega \in \Omega_k^d\setminus \{\{ \emptyset \}\}}\epsilon_{\omega}.
 \end{align}
 The numbers~$\epsilon_{\omega}$ are easily calculated from the dual solution, just as the dual approximate objective value~$\langle F_{\emptyset},X\rangle$. So we find  a constant~$c_{\omega}$ from the semidefinite programming dual approximation~$X$ such that\symlistsort{comega}{$c_{\omega}$}{constant obtained from the dual SDP approximation}
 \begin{align}\label{errorlast}
 X_{\omega}z({\omega}) \leq c_{\omega}.
 \end{align}
In the case of~$A(23,8,11)$  one can conclude with the semidefinite program~$A_3(23,8,11)$, which can be solved with SDPA-GMP~$\cite{nakata,sdpa}$ within minutes, that
\begin{align} \label{upperbound} 
z({\omega}) \leq 10^{-90}
\end{align} 
for all orbits corresponding to codes not satisfying~$(\ref{sdpfact1})$. Let~$\omega$ be an orbit for which~$(\ref{upperbound})$ holds. If there exists a code~$C$ of maximum size containing a subcode~$D\subset C$ with~$D \in \omega$, then one constructs a feasible solution to~$\eqref{A4ndwu}$ by putting~$x(S)=1$ for~$S \in \CC_k$ with~$S \subseteq C$ and~$x(S)=0$ else, and hence by averaging over~$H$ there exists a feasible $H$-invariant solution with~$z(\omega ) \geq 1/|H|$ (this lower bound is not best possible, but sufficient). In our case,~$H=S_{23}$, so~$z({\omega}) \geq 1/23! > 10^{-23}$, which gives a contradiction with~$\eqref{upperbound}$. In this way, one verifies that all orbits not satisfying~$(\ref{sdpfact1})$ are forbidden, thereby establishing~$(\ref{sdpfact1})$. We used a separate java program to check that~$X\succeq 0$ (in fact,~$X \succ 0$) and to compute the error terms as in~$(\ref{error})$ and~$(\ref{errorcomp2})$.

Next, we consider the cases~$A(22,8,10)=616$ and~$A(22,8,11)=672$. First we assume that~$a_{14}=0$ for all maximum-size~$(22,8,10)=616$ and~$(22,8,11)$-codes~$C$.  We write~$\omega_t$ for the orbit of two words at Hamming distance~$t$.  Adding the constraint~$z({\omega_{14}})=0$ to the programs~$A_3(n,d,w)$  for these two cases of~$n,d,w$ gives~$A_3(n,d,w)=A(n,d,w)$. In this way one shows, in the same way as in the previous paragraph, that all orbits not satisfying~$(\ref{sdpfact2})$ and~$(\ref{sdpfact3})$ are forbidden, provided that~$a_{14}=0$. So in order to establish~$(\ref{sdpfact2})$ and~$(\ref{sdpfact3})$, it remains to prove that 
\begin{align}\label{rtp}
\text{if~$C$ is a maximum-size~$(22,8,10)$- or~$(22,8,11)$-code, then~$a_{14}=0$.} 
\end{align} 
Suppose to the contrary that~$C$ is a code as in~$\eqref{rtp}$, yet~$a_{14}>0$. Then~$a_{14} \geq 2/|C|$. We will show that this is not possible by adding constraints to the (large) program~$B_4(n,d,w)$. The semidefinite program will then give~$\floor{B_4(n,d,w)} < A(n,d,w)$ and we will arrive at a contradiction, as~$B_4(n,d,w)$ is an upper bound for~$A(n,d,w)$. To find a better lower bound on some of the~$a_i$, we use the following two propositions. We use in both propositions that for two words~$u,v$ in a constant weight~$w$ code~$C$,
\begin{align} \label{oddeven}
d_H(u,v) \equiv 2 \pmod{4} \,\,\,\,\, \Longleftrightarrow \,\,\,\,\, \wt(u \cap v) \not\equiv w \pmod{2},
\end{align} 
which follows from~\eqref{wellknown}. 
In the next two propositions we will call~$\wt(u\cap v)$ the \emph{inner product} of~$u$ and~$v$.\indexadd{inner product}
\begin{proposition} \label{extraprop1}
Let~$C$ be a~$(22,8,10)$-code of size~$616$ with~$a_{14} \geq 2 / 616$. Then~$a_{10}+a_{14}+a_{18} \geq 208/616$. 
\end{proposition}
\proof 
Suppose that~$a_{10}+a_{14}+a_{18} < 208/616$ (note that~$a_{22}=0$ as two weight~$10$ words cannot have distance~$22$). Then, by~$\eqref{oddeven}$, there are at most~$206/2=103$ pairs of words in~$C$ with odd inner product. Let~$\{b,b'\} \subseteq C$ be such a pair of words. Starting with~$b\in C$, and greedily picking vectors~$d \in C$ such that the inner product of~$d$ with the already chosen vectors is even, we end with a self-orthogonal subcode~$B$ of~$C$ of size~$\geq 616-103=513$. Starting with~$b'$, we repeat the same process to end up with a self-orthogonal subcode~$B'$ of~$C$ (containing~$b'$) of size~$\geq 513$. Furthermore,~$D:= B \cap B'$ has~$|D| \geq 512$. (To see this, note that $b,b'\notin D$ with odd inner product. Every word~$v \in C \setminus (D\cup \{b,b'\})$ has odd inner product with some word in~$C$ so there are at most~$103-1=102$ of such words~$v$, as~$b,b'$ is already a pair with odd inner product.) Write~$F:=\langle D \rangle$. Then~$F$ is self-orthogonal, as~$D$ is self-orthogonal.

Since~$\langle B \rangle$ and~$\langle B' \rangle$ are self-orthogonal codes, they have dimension at most~$11$. Since $\langle B \rangle \neq \langle B'\rangle$, as~$b \notin \langle B'\rangle $, we have~$\dim F =  \dim \langle B  \cap B'\rangle \leq \dim( \langle B \rangle \cap \langle B'\rangle) \leq 10$. On the other hand, we have~$ |F| >512$ (as~$D \subseteq F$ and the zero word is contained in~$F$). So~$\dim F=10$. Moreover,~$F$ has minimum distance~$8$. To see this, let~$u$ be any nonzero word in~$F \setminus D$. Then~$D \cap (u +D) \neq \emptyset$, as both~$D$ and~$u+D$ have size~$512$ and are contained in~$F \setminus \{ \mathbf{0}\}$, a set of size~$1023$. So~$u$ is the sum of two words of~$D$, and hence has weight at least~$8$ (as~$D$ has minimum distance at least~$8$).

So~$F$ is a self-orthogonal code of word length~$22$, dimension~$10$ and minimum distance~$8$. Such a code is the twice shortened extended binary Golay code (see~$\cite{brouwer2}$), which does not contain words of weight~$10$. But all words in~$D \subseteq F$ have weight~$10$, a contradiction.
\endproof 

\begin{proposition} \label{extraprop2}
Let~$C$ be an~$(22,8,11)$-code of size~$672$ with~$a_{14} \geq 2 / 672$. Then~$a_{10}+a_{14}+a_{18}+a_{22} \geq 318/672$. 
\end{proposition}
\proof 
Suppose that~$a_{10}+a_{14}+a_{18}+a_{22} < 318/616$. Then, by~$\eqref{oddeven}$, there are at most~$316/2=158$ pairs of words in~$C$ with even inner product. Let~$\{b,b'\} \subseteq C$ be such a pair of words. Starting with~$b\in C$, and greedily picking vectors~$d \in C$ such that the inner product of~$d$ with the already chosen vectors is odd, we end with a subcode~$B$ of~$C$ of size~$\geq 672-158=514$. Now, add an extra symbol~$1$ to every codeword in~$B$, to obtain a self-orthogonal code~$D$ of length~$23$. As~$D$ is self-orthogonal,~$\dim \langle D \rangle \leq \floor{23/2}=11$. 

Starting with~$b' \in C$, we repeat the same process to end up with a subcode~$B'$ of~$C$ (containing~$b'$) of size~$\geq 514$ such that all pairs of words in~$B'$ have odd inner product. Add an extra symbol~$1$ to every code word in~$B'$ to obtain a self-orthogonal code~$D'$ of length~$23$, so~$\dim \langle D'\rangle \leq 11$. Note that~$\langle D \rangle  \neq \langle D'\rangle$, as~$1b' \notin \langle D\rangle $.

Furthermore,~$E:= D \cap D'$ has~$|E| \geq 513$ and all words start with~$1$. (To see this, note that $b,b'\notin B \cap B' $ with even inner product. Every word~$v \in C \setminus ((B \cap B' ) \cup \{b,b'\})$ has even inner product with some word in~$C$ so there are at most~$157$ of such words~$v$.)  Hence~$|\langle E\rangle | \geq 2 \cdot 513$, so~$\dim \langle E\rangle  \geq 11$. But~$\langle D\rangle $ and~$\langle D'\rangle $ are distinct codes of dimension~$\leq 11$, so their intersection has dimension~$<11$, hence
$$
\dim\langle E \rangle=\dim\langle D \cap D'\rangle \leq \dim (\langle D\rangle \cap \langle D' \rangle) <11,
$$
a contradiction. 
\endproof 

From a code~$C$ with distance distribution~$(a_i)$, one constructs a feasible solution to~$\eqref{A4ndwu}$ by putting~$x(S)=1$ for~$S \in \CC_k$ with~$S \subseteq C$ and~$x(S)=0$ else, and hence by averaging over~$H$ there exists a feasible $H$-invariant solution. This solution has
\begin{align*}
z({\omega_t}) &= \mbox{$\frac{1}{|H|}$} \sum_{g \in H} x\cdot g(\{x,y\}) =  \frac{(\frac{t}{2})!(w-\frac{t}{2})!(\frac{t}{2})!(22-w-\frac{t}{2})!}{|H|}|\{(u,v) \in C^2 \,\hspace{-.55pt} : \,\hspace{-.55pt} d_H(u,v)=t \}|  \notag 
\\&= \frac{|\{(u,v) \in C^2 \,\, : \,\, d_H(u,v)=t \}|}{\binom{22}{w} \cdot\binom{22-w}{t/2} \binom{w}{t/2}}   = \frac{|C| a_t}{\binom{22}{w} \binom{22-w}{t/2} \binom{w}{t/2}} = \frac{z({\omega_0}) a_t}{\binom{22-w}{t/2} \binom{w}{t/2}}, 
\end{align*}
where~$\{x,y\}$ is any pair of constant weight~$w$ words with distance~$t$ and~$H=S_{22}$, the symmetric group on 22 elements. 

So we can add linear constraints on the~$a_i$ as linear constraints on the variables~$z({\omega_t})$ to our semidefinite program. To the program~$B_4(22,8,10)$ we add the constraints~$a_{14} \geq 2/616$ and~$a_{10}+a_{14}+a_{18} \geq 208/616$. To~$B_4(22,8,11)$ we add the constraints~$a_{14} \geq 2/672$ and~$a_{10}+a_{14}+a_{18}+a_{22} \geq 318/672$. Write~$B_4^*(n,d,w)$ for the resulting bound after adding these constraints. We find~$B_4^*(n,d,w)<A(n,d,w)$ in both cases (which we verified using the dual solution), which is not possible.\footnote{The SDP-solutions show $B_4^*(22,8,11)< 671.885 <672$ and~$B_4^*(22,8,10)<615.935<616$.} This establishes~$(\ref{rtp})$ and hence completes the verification of~$(\ref{sdpfact2})$ and~$(\ref{sdpfact3})$.

The time needed to solve the semidefinite programs~$B_4^*(22,8,11)$ and~$B_4^*(22,8,10)$ varied from one to three weeks with sufficient precision to conclude that~$B_4^*(n,d,w) < A(n,d,w)$ in these two cases (with SDPA-DD). By contrast, the semidefinite programs for~$A_3(22,8,10)$ and~$A_3(22,8,11)$ can be solved withvery high precision within minutes (with SDPA-GMP). 

The computer programs we used to generate input for the SDP-solver can be found in~\cite{programs}. Also, the input and output files for the SDP solver can be found in this folder, and a java program to inspect the outputs.

\chapter{Semidefinite programming bounds for Lee codes}\label{leechap}
\vspace{-6pt}
\chapquote{Nothing takes place in the universe in which some relation\\of maximum and minimum does not appear.}{Leonhard Euler (1707--1783)}\vspace{-3pt}

\noindent For nonnegative integers~$q,n,d$, let $A_q^L(n,d)$ denote the maximum cardinality of a code $C \subseteq \Z_q^n$ with minimum Lee distance at least~$d$. We consider a semidefinite programming upper bound  on $A_q^L(n,d)$ based on triples of codewords, which bound can be computed efficiently using symmetry reductions, resulting in  new upper bounds for several triples $(q,n,d)$.  We also give constructions for obtaining lower bounds and prove uniqueness of two instances for which upper bounds were already known: $A_5^L(7,9)=15$ and $A_6^L(4,6)=18$, using the semidefinite programming output. The code achieving the first value is obtained by arranging 15 girls  7 days in succession into triples in a special way, which is different from but connects to Kirkman's schoolgirl problem.


\hspace{2pt}This chapter is based on~\cite{leeartikel}, except for the constructions and uniqueness for $A_5^L(7,9)=15$ and $A_6^L(4,6)=18$, which are new. 

\section{Introduction}

Fix~$q,n \in \N$. Recall that the \emph{Lee distance} of two words~$u, v \in \Z_q^n$ is
\begin{align} \label{leedisth8}
\mbox{$ d_L(u,v):= \sum_{i=1}^n  \min\{ |u_i-v_i|,\, q-|u_i-v_i|\}$},
\end{align}
 where we consider~$u_i$ and~$v_i$ as integers in~$\{0,\ldots,q-1\}$. The \emph{minimum Lee distance} $d_{\text{min}}^L(C)$  of a code~$C\subseteq \Z_q^n$ is the minimum of~$d_L(u,v)$ taken over distinct~$u,v \in C$. If~$|C| \leq 1$, we set~$d_{\text{min}}^L(C) = \infty$. For~$d \in \N$, recall the definition of~$A^L_q(n,d)$ cf.~\eqref{aleeqndprem}:
\begin{align} \label{aleeqnd}
A^L_q(n,d):= \max \{ |C| \, \, | \,\, C \subseteq \Z_q^n, \,\, d_{\text{min}}^L(C) \geq d \}.
\end{align}
 Generally, it is an interesting and nontrivial problem to determine~$A_q^L(n,d)$ for given $q,n,d$. Quistorff made a table of upper bounds on~$A_q^L(n,d)$ based on analytic arguments \cite{quistorff}. H.\ Astola and I.\ Tabus calculated several new upper bounds by linear programming \cite{astola2},  using the classical Delsarte bound in the Lee scheme based on pairs of codewords~\cite{astola1, delsarte}.  In~$\cite{invariant}$, the possibility of applying semidefinite programming to Lee codes is mentioned and it is stated that to the best knowledge of the authors, such  bounds for Lee codes using triples have not yet been studied.

In this chapter, we describe how to efficiently compute a semidefinite programming upper bound~$B_3^L(q,n,d)$ on~$A^L_q(n,d)$ based on triples of codewords, using symmetry reductions, and we calculate this  bound for several values of~$q,n,d$. We find several new upper bounds on~$A^L_q(n,d)$, see Table~$\ref{tabellee}$. We only consider~$q \geq 5$, since~$A^L_4(n,d)=A_2(2n,d)$ --- this follows by applying the Gray map~\cite{gray}.

 \begin{table}[htb] \small
\centering
   \begin{minipage}[t]{.445\linewidth}
   \centering
   \begin{tabular}{| r | r | r|| >{\bfseries}r| r|}
    \hline
    $q$ & $n$ & $d$ &  \multicolumn{1}{>{\raggedright\arraybackslash}b{20mm}|}{\textbf{new upper bound}} & \multicolumn{1}{>{\raggedright\arraybackslash}b{20mm}|}{best upper bound previously known} 
    \\\hline 
    5 & 4 & 3  & 62 & $64^l$    \\
    5 & 4 & 4  & 27 & $30^l$    \\
    5 & 4 & 5  & 10 & $11^l$    \\
    5 & 5 & 3  & 270 & $276^l$   \\
    5 & 5 & 5  & 36 & $39^l$    \\
    5 & 5 & 6  & 15 & $18^l$    \\
    5 &6 & 3  & 1170& $1176^l$   \\      
    5 & 6 & 4  & 494 & $520^b$    \\
    5 & 6 & 5  & 149 & $155^l$    \\
    5 & 6 & 6  & 60 & $63^l$    \\                        
    5 & 6 & 7  & 25 & $28^l$    \\                  
         5 & 7 & 3 &  5180 & $5208^{bl}$     \\ 
      5 & 7 & 4  &2183  & $2232^b$     \\                 
    5 & 7 & 5  &590  & $608^l$     \\      
    5 & 7 & 6  & 250  & $284^l$     \\      
    5 & 7 & 7  & 79 & $81^l$     \\              
    5 & 7 & 8  &35  &  $41^l$   \\                   
    \hline
    6 & 3 & 3  & 27 & $29^l$    \\
    6 & 3 & 4  & 14 & $17^l$    \\
     6 & 4 & 4  & 78 &$79^l$     \\        
         6 & 4 & 5  & 22 & $26^l$    \\
    6 & 5 & 3  & 693 &  $699^l$   \\
        6 & 5 & 4  & 366 &  $378^l$   \\ 
             6 & 5 & 5  & 107 &  $114^l$   \\
    \hline 
    \end{tabular}\end{minipage}    \,\,\,\,\,\,
   \begin{minipage}[t]{.445\linewidth}
   \centering
   \begin{tabular}{| r | r | r|| >{\bfseries}r| r|}
    \hline
    $q$ & $n$ & $d$ &  \multicolumn{1}{>{\raggedright\arraybackslash}b{20mm}|}{\textbf{new upper bound}} & \multicolumn{1}{>{\raggedright\arraybackslash}b{20mm}|}{best upper bound previously known} 
    \\\hline 
     6 & 5 & 6  & 61 &  $67^l$   \\       
    6 & 5 & 7  & 22 &  $24^{bl}$   \\ 
            6 & 6 & 6 & 273 & $293^l$\\    
        6 & 6 & 7 & 79 & $85^l$\\    
    6 & 6 & 8 & 48 & $52^l$\\    
    6 & 6 & 9 & 16 & $17^l$\\
    \hline 
     7 & 3 & 4  & 21 &$24^{bl}$     \\
    7 & 3 & 5  & 10 &  $11^l$   \\               
     7 & 4 & 3  & 256 & $263^l$   \\
          7 & 4 & 4  & 121 & $128^b$   \\
    7 & 4 & 5 & $\bm{49^*}$ &  $50^l$   \\          
     7 & 4 & 6  & 23 & $27^l$    \\
     7 & 4 & 7  & 11 & $13^l$    \\
      7 & 4 & 8  & 6 &  $7^{bl}$   \\
           7 & 5 &  3 & 1499 &  $1512^l$   \\         
     7 & 5 &  4 &  686&  $720^b$   \\         
     7 & 5 &  5 & 240 &  $249^l$   \\         
     7 & 5 &  6 & 116  &  $130^l$   \\         
     7 & 5 &  7 & 49 &  $54^l$   \\      
     7 & 5 &  8 & 25 &  $28^l$   \\           
          7 & 5 &  9 & 13 & $14^l$   \\ 
                       7 & 6 & 10 & 26 & $31^l$  \\
                  7 & 6 & 11 & 13 & $14^b$   \\   
                  &&&&\\
                  \hline 
    \end{tabular}\end{minipage}    
  \caption{\label{tabellee}\small An overview of the new upper bounds for Lee codes. The new upper bounds are instances of the bound~$B_3^L(q,n,d)$ from~$(\ref{Bleend})$ below. The superscript $^l$ refers to a bound obtained by Astola and Tabus  using linear programming~\cite{astola2}. The superscript~$^b$ refers to a bound from Quistorff~\cite{quistorff}. The superscript~$^*$ refers to an upper bound matching the known lower bound:~$A_7^L(4,5)=49$ is achieved by a linear code~\cite{astolalinear}. No tables of lower bounds are given in~\cite{astola2,quistorff}. }
\end{table}  


Recall that two codes~$C,D \subseteq \Z_q^n$ are \emph{Lee equivalent} if~$D$ can be obtained from~$C$ by first permuting the~$n$ coordinates and by subsequently acting on the alphabet~$\Z_q$ with an element of the dihedral group~$D_q$ in each coordinate separately. Here the dihedral group~$D_q$ of order~$2q$ is the group of symmetries of a regular~$q$-gon with vertices~$0,\ldots,q-1$. It was already known that~$A_5^L(7,9)\leq15$ and~$A_6^L(4,6)\leq18$, cf.~\cite{quistorff}. In Sections~\ref{579} and~\ref{646} we will give codes achieving those bounds, and we also show that in both cases there is only one code (up to Lee equivalence) of maximum size, using the semidefinite programming output.

\subsection{The semidefinite programming bound}\label{semprogboundleesect}
We define a hierarchy of semidefinite programming upper bounds on~$A^L_q(n,d)$, which is an adaptation of the semidefinite programming hierarchy for binary codes defined by Gijswijt, Mittelmann and Schrijver in~$\cite{semidef}$. For~$k \in \Z_{\geq 0}$, let~$\mathcal{C}_k$ be the collection of codes~$C \subseteq \Z_q^n$ with~$|C|\leq k$.\symlistsort{Ck}{$\mathcal{C}_k$}{collection of codes of cardinality at most~$k$}   For any~$D \in \mathcal{C}_k$, we define\symlistsort{Ck(D)}{$\mathcal{C}_k(D)$}{the set~$\{C \in \mathcal{C}_k \, \mid \, C \supseteq D, \, \lvert D\rvert +2\lvert C\setminus D\rvert \leq k  \}$}
\begin{align} 
\mathcal{C}_k(D) := \{C \in \mathcal{C}_k \,\, | \,\, C \supseteq D, \, |D|+2|C\setminus D| \leq k  \}.  
\end{align} 
Note that then~$|C \cup C'|= |C| + |C'| -|C\cap C'| \leq 2|D| + |C \setminus D| + |C' \setminus D| -|D| \leq k$ for all~$C,C' \subseteq \mathcal{C}_k(D)$.  
Furthermore, for any function~$x : \mathcal{C}_k \to \R$ and $D \in \mathcal{C}_k$, we define the~$\mathcal{C}_k(D) \times \mathcal{C}_k(D)$ matrix~$M_{k,D}(x)$ by\symlistsort{MkD}{$M_{k,D}(x)$}{variable matrix} 
$$
M_{k,D}(x)_{C,C'} : = x(C \cup C'),
$$
for~$C,C' \in \mathcal{C}_k(D)$. 
  Now define the following number:\symlistsort{BLk(q,n,d)}{$B^L_k(q,n,d)$}{upper bound on~$A_q^L(n,d)$}
\begin{align} \label{Bleend}
B^L_k(q,n,d):=    \max \big\{ \mbox{$\sum_{v \in \Z_q^n} x(\{v\})\,\,$} |\,\,&x:\mathcal{C}_k \to \R, \,\, x(\emptyset )=1, \,\, x(S)=0 \text{ if~$d_{\text{min}}^L(S)<d$}, \notag \\
 \,\, &M_{k,D}(x) \succeq 0 \text{ for each~$D$ in~$\mathcal{C}_k$}\big\}. 
\end{align}
 \begin{proposition} \label{trivleeprop}
For all~$k,q,n,d \in \N$, we have~$A^L_q(n,d) \leq B^L_k(q,n,d)$.
\end{proposition}
\proof
Let~$C \subseteq \Z_q^n$ be a code with~$d_{\text{min}}^L(C)\geq d$ and~$|C| = A_q^L(n,d)$. Define~$x: \mathcal{C}_k \to \R$ by~$x(S)=1$ if~$S \subseteq C$ and~$x(S)=0$ else, for~$S \in \CC_k$. Then~$x$ satisfies the conditions in~$(\ref{Bleend})$, where the condition that~$M_{k,D}(x) \succeq 0$ is satisfied since~$M_{k,D}(x)_{C,C'}=x(C)x(C')$ for all~$C,C'\in \mathcal{C}_k(D)$. Moreover, the objective value equals~$\sum_{v \in  \Z_q^n} x(\{v\}) =|C|=A^L_q(n,d) $, which gives~$ B^L_k(q,n,d) \geq A^L_q(n,d) $.
\endproof
 It can be shown that the bound~$B_2^L(q,n,d)$ is equal to the Delsarte bound in the Lee scheme, which was calculated for several instances by Astola and Tabus in~$\cite{astola2}$. In this chapter we consider the bound~$B_3^L(q,n,d)$. The method for obtaining a symmetry reduction, using representation theory of the dihedral and symmetric groups, is an adaptation of the method in Chapter~\ref{orbitgroupmon}.

 \subsection{Symmetry reductions}
 Fix~$k \in \N$.  Let~$D_q$ be the dihedral group of order~$2q$ and let~$S_n$ be the symmetric group on~$n$ elements. The group~$H:=D_q^n \rtimes  S_n $ acts naturally on~$\mathcal{C}_k$, and this action maintains minimum distances and cardinalities of codes~$C \in \mathcal{C}_k$. We can assume that the optimum~$x$ in~$(\ref{Bleend})$  is $H$-invariant, i.e., $g \cdot x = x$ for all~$g \in H$. Indeed, if~$x$ is any optimum solution for~$(\ref{Bleend})$, then for each~$g \in H$, the function~$g \cdot x $ is again an optimum solution, since the objective value of~$g \cdot x $ equals the objective value of~$x$ and~$g \cdot x $ still satisfies all constraints in~$(\ref{Bleend})$. Since the feasible region is convex, the optimum~$x$ can be replaced by the average of~$g \cdot x $ over all~$g \in H$. This gives an $H$-invariant optimum solution.

 Let~$\Omega_k$ be the set of~$H$-orbits on~$\mathcal{C}_k$.\symlistsort{Omegak}{$\Omega_k$}{set of $H$-orbits on $\mathcal{C}_k$} Then~$|\Omega_k|$ is bounded by a polynomial in~$n$, for fixed~$k$ and~$q$. Since there exists an~$H$-invariant optimum solution, we can replace, for each~$\omega \in \Omega_k$ and~$C \in \omega$, each variable~$x(C)$ by a variable~$z(\omega)$.\symlistsort{z(omega)}{$z(\omega)$}{variable of reduced program} Hence, the matrices~$M_{k,D}(x)$ become matrices~$M_{k,D}(z)$ and we have considerably reduced the number of variables in~$(\ref{Bleend})$.\symlistsort{MkDz}{$M_{k,D}(z)$}{reduced variable matrix}

We only have to check positive semidefiniteness of~$M_{k,D}(z)$ for one code~$D$ in each~$H$-orbit of~$\mathcal{C}_k$, as for each~$g \in H$, the matrix~$M_{k,g(D)}(z)$ can be obtained by simultaneously permuting rows and columns of~$M_{k,D}(z)$. 

We sketch how to reduce these matrices in size. For~$D \in \mathcal{C}_k$, let~$H_D$ be the subgroup of~$H$ consisting of all~$g \in H$ with~$g(D)=D$.\symlistsort{HD}{$H_D$}{subgroup of~$H$ that leaves~$D$ invariant} Then the action of~$H$ on~$\mathcal{C}_k$ induces an action of $H_D$ on~$\mathcal{C}_k(D)$.  The simultaneous action of~$H_D$ on the rows and columns of~$M_{k,D}(z)$ leaves~$M_{k,D}(z)$ invariant.  Therefore, there exists a block-diagonalization (as explained in Section~\ref{symint}) given by~$M_{k,D}(z) \mapsto U^*  M_{k,D}(z) U$ of~$M_{k,D}(z)$, for a matrix~$U$ depending on~$H_D$ but not depending on~$z$, such that the order of~$ U^*  M_{k,D}(z) U$ is polynomial in~$n$ and such that~$M_{k,D}(z)$ is positive semidefinite if and only if each of the blocks of~$ U^*  M_{k,D}(z) U$ is positive semidefinite. In our case the matrix~$U$ can be taken to be real; so~$U^*=U\T$. The entries in each block of~$U\T  M_{k,D}(z) U$ are linear functions in the variables~$z(\omega)$ (with coefficients bounded polynomially in~$n$). Hence, we have reduced the size of the matrices involved in our semidefinite program to polynomial size.

The reductions of the optimization problem will be described in detail in Section~$\ref{reduct}$. Table~$\ref{tabellee}$ contains the new upper bounds. All improvements have been found using multiple precision versions of SDPA~\cite{nakata}.

\section{Reduction of the optimization problem\label{reduct}}

In this section we give the reduction of optimization problem $(\ref{Bleend})$ for computing the bound $B^L_3(q,n,d)$, using the representation theory from Section~\ref{prelimrep} and Chapter~\ref{orbitgroupmon}. First we consider block diagonalizing $M_{3,D}(z)$ for $D \in \mathcal{C}_3$ with $|D|=1$.  Subsequently we consider the case $D=\emptyset$. Note that for the cases $|D|=2$ and $|D|=3$ the matrix $M_{3,D}(z) = (z(D))$ has order $1 \times 1$, so it is its own block diagonalization. Hence, in those cases, $M_{3,D}(z)$ is positive semidefinite if and only if $z(D) \geq 0$. 

\subsection{The case~\texorpdfstring{$|D|=1$}{|D|=1}\label{D1}}
The Lee isometry group~$H=D_q^n \rtimes S_n$ acts transitively on~$\Z_q^n$, so we may assume that~$D=\{\bm{0}$\}, where~$\bm{0}=0\ldots0$ is the all-zero word. The rows and columns of~$M_{3,D}(z)$ are indexed by sets of the form~$\{\bm{0}, \alpha\}$ for~$\alpha \in \Z_q^n$. Then the subgroup~$H_D$ of~$H$ that leaves~$D$ invariant is equal to  $S_2^n \rtimes  S_n$, as the zero word must remain fixed (so we cannot apply a rotation of the alphabet in any coordinate position). Here the non-identity element of~$G:=S_2$ acts on~$Z:=\Z_q$, where we consider~$0,\ldots,q-1$ as vertices of a regular~$q$-gon, as a \emph{reflection} switching vertices~$i$ and~$q-i$ (for~$i=1,\ldots,\floor{\frac{q-1}{2}}$). So vertex~$0$ is fixed if~$q$ is odd, and vertices~$0$ and~$q/2$ are fixed if~$q$ is even. For~$i=0,\ldots,q-1$, let~$e_i$ be the~$i$th unit vector of~$\C^{\Z_q}$ (taken as column vector).
\begin{proposition}
 A representative matrix set for the reflection action of~$G=S_2$ on~$\C^Z=\C^{\Z_q}$ is
\begin{align} \label{oddrepr}
\text{$\bm{B}:=\{B_1,B_2\}$, with } \,\,\,\, B_1:=\left[e_0, \left( e_{i} + e_{q-i} \right)_{i=1}^{\floor{\frac{q}{2}}}\right], \,\,\,\, B_2:= \left[ \left(e_{i}- e_{q-i}\right)_{i=1}^{\floor{\frac{q-1}{2}}}\right].
\end{align}
\end{proposition}
\proof
 For~$j=1,\ldots,\floor{q/2}+1$, define~$W_{1,j}$ to be the~$1$-dimensional vector space spanned by the~$j$th column~$w_{1,j}$ of~$B_1$. Moreover, for~$j=1,\ldots,\floor{(q-1)/2}$, define~$W_{2,j}$ to be the~$1$-dimensional vector space spanned by the~$j$th column~$w_{2,j}$ of~$B_2$.  Note that each~$W_{i,j}$ is~$S_2$-stable and that~$W_{i,j}$ and~$W_{i',j'}$ are orthogonal whenever~$(i,j) \neq (i',j')$ (with respect to the inner product $u,v \mapsto v^*u$). Observe that, for~$j,j'$ and~$l,l'$ the maps~$W_{1,j}\to W_{1,j'}$ and~$W_{2,l} \to W_{2,l' }$ defined by~$w_{1,j}\mapsto w_{1,j'}$ and~$w_{2,l}\mapsto w_{2,l'}$, respectively, are~$S_2$-equivariant.  Note that the number of~$W_{1,j}$ we have defined is $\floor{q/2}+1$, the number of~$W_{2,j}$ is $\floor{(q-1)/2}$, and
\begin{align*}
(\floor{q/2}+1)^2 + \floor{(q-1)/2}^2 = \begin{cases}
\mbox{$\frac{1}{2}$}q^2+2 &  \text{if~$q$ is even}, \\
(q^2+1)/2  & \text{if~$q$ is odd}, 
\end{cases}
\end{align*}
which is equal to~$|(\Z_q \times \Z_q)/S_2| = \dim ( \C^{\Z_q} \otimes \C^{\Z_q})^{S_2}$. (If~$q$ is even, the points~$(0,0)$, $(q/2,0)$, $(0,q/2)$ and $(q/2,q/2)$ in $\Z_q \times \Z_q$ are fixed by the nonidentity element in~$S_2$. If~$q$ is odd, only the point~$(0,0)$ in~$\Z_q \times \Z_q$ is fixed by the nonidentity element in~$S_2$.)  It follows that the~$W_{1,j}$ and~$W_{2,j}$ form a \emph{decomposition} of~$\C^{\Z_q}$ into \emph{irreducible representations} (as any further representation, or decomposition, or equivalence among the~$W_{i,j}$ would yield that the sum of the squares of the multiplicities of the irreducible representations is strictly larger than~$\dim ( \C^{\Z_q} \otimes \C^{\Z_q})^{S_2}$, which contradicts the fact that~$\Phi$ in~$\eqref{PhiC2}$ is bijective). So the matrix set~$\bm{B}$ from \eqref{oddrepr} is indeed representative for the action of~$S_2$ on~$\C^{\Z_q}$. 
\endproof 
Note that the representative set is real. Set~$k:=2$,~$\bm{m}=:(m_1,m_2):=(\floor{q/2}+1,\floor{(q-1)/2})$. 
Now Proposition~\ref{prop2} (with~$V:=\C^Z$, where~$Z=\Z_q$, and~$G=S_2$) gives a representative set for the action of~$S_2^n \rtimes S_n$ on~$\Z_q^n$: it is the set 
$$
\{~~
[\bm{u_{\tau,B}}
\mid
\bm{\tau}\in \bm{T_{\lambda,m}}]
~~\mid
\bm{n}\in\bm{N},\bm{\lambda\vdash n}
\}
$$
where~$\bm{T_{\lambda,m}}$ and~$\bm{u_{\tau,B}}$ are defined in~$\eqref{wlambda}$ and~$\eqref{vtau}$  respectively. For convenience of the reader, we restate  the main definitions of Chapter~\ref{orbitgroupmon}  and the main facts about representative sets  applied to the context of this section --- see Figure~\ref{factslee0}.

\begin{figure}[ht]
  \fbox{
    \begin{minipage}{15.5cm}
     \begin{tabular}{l}
{\bf FACTS.}  \\
\\
$G:= S_2$. \\
$Z:= \Z_q$. \\
The non-identity element of~$S_2$ switches~$i$ and~$q-i$ (for~$i=1,\ldots,\floor{\frac{q-1}{2}}$). \\
\\
$k:=2$,~$\bm{m}=:(m_1,m_2):=(\floor{q/2}+1,\floor{(q-1)/2})$.\\
\\
A representative set for the action of~$G$ on~$ Z$ is $\bm{B}=\{B_1,B_2\}$ from~\eqref{oddrepr}. \\
A representative set for the action of~$G^n \rtimes S_n$ on~$ Z^n$ follows from Prop.~\ref{prop2}.
\\ \\
$\Lambda := (\Z_q \times \Z_q)/S_2  $. \\
$a_P := \sum_{ (x,y) \in P} x \otimes y$ for~$P \in \Lambda$.\\
$A:= \{ a_P \, | \, P \in \Lambda\} $, a basis of~$W := (\C{\Z_q \otimes \C\Z_q})^{S_2}$.  \\
  $K_{\omega'} := \sum_{(P_1,\ldots,P_n) \in \omega'}
a_{P_1} \otimes \cdots \otimes a_{P_n}$ for~$\omega' \in \Lambda^n/S_n$. 
     \end{tabular}
    \end{minipage}    } \caption{\label{factslee0}\small{The main definitions of Chapter~\ref{orbitgroupmon}  and the main facts about representative sets  applied to the context of Section~\ref{D1}.}}
\end{figure}

\subsubsection{Computations for~\texorpdfstring{$|D|=1$ }{|D|=1}} \label{d1comp}
Let~$D = \{ \bm{0} \} \in \mathcal{C}_3$ and let~$\Omega_3$ denote the set of all~$D_q^n \rtimes S_n$-orbits of codes in~$\mathcal{C}_3$. For each~$\omega \in \Omega_3$, we define the~$\mathcal{C}_3(D) \times \mathcal{C}_3(D)$ matrix~$N_{\omega}$ with entries in~$\{0,1\}$ by \symlistsort{Nomega}{$N_{\omega}$}{$\mathcal{C}_3(D) \times \mathcal{C}_3(D)$
matrix}
\begin{align}
    (N_{\omega})_{\{\bm{0},\alpha\},\{\bm{0},\beta\}} := \begin{cases} 1 &\mbox{if } \{\bm{0},\alpha,\beta\} \in \omega,  \\ 
0 & \mbox{otherwise,} \end{cases} 
\end{align}
for~$\alpha, \beta \in \Z_q^n$.
Given~$\bm{n}= (n_1,n_2) \in \bm{N}$, for each~$\bm{\lambda} \vdash \bm{n}$ we write~$U_{\bm{\lambda}}$ for the matrix in~$(\ref{matset})$ that corresponds with~$\bm{\lambda}$.  For each~$z : \Omega_3 \to \mathbb{R}$ we obtain with~$(\ref{PhiR})$ that
\begin{align} \label{blocks1tt}
    \Phi(M_{3,D}(z)) = \Phi \left(\sum_{\omega \in \Omega_3}z(\omega) N_{\omega} \right ) = \bigoplus_{\bm{n} \in \bm{N}} \bigoplus_{\bm{\lambda} \vdash \bm{n}} \sum_{\omega \in \Omega_3}z(\omega) U_{\bm{\lambda}}\T N_{\omega} U_{\bm{\lambda}}. 
\end{align}
The number of~$\bm{n} \in \bm{N}$, $\bm{\lambda} \vdash \bm{n}$, and the numbers~$|\bm{T_{\lambda,m}}|$ and~$|\Omega_3|$  are all bounded by a polynomial in~$n$. This implies that the number of blocks in~$(\ref{blocks1tt})$, the size of each block and the number of variables occurring in all blocks are polynomially bounded in~$n$. We now show how to compute the entries of the matrix~$U_{\bm{\lambda}}\T N_{\omega} U_{\bm{\lambda}}$, for all~$\omega \in \Omega_3$,~$\bm{n} \in \bm{N}$, $\bm{\lambda} \vdash \bm{n}$, in polynomial time. That is, we show how to compute the coefficients~$\bm{u_{\tau,B}}\T N_{\omega}  \bm{u_{\sigma,B}}$, for~$\bm{\tau}, \bm{\sigma} \in \bm{T_{\lambda,m}}$, in the blocks~$\sum_{\omega \in \Omega_3} z(\omega)U_{\bm{\lambda}}\T N_{\omega} U_{\bm{\lambda}}$ in polynomial time.

Recall that~$\Lambda=(\Z_q \times \Z_q)/S_2$. By~\eqref{natbijection} there is a natural bijection between~$\Lambda^n/S_n$ and the set of~$H$-orbits on~$(\Z_q \times \Z_q)^n$. Since naturally~$(\Z_q^n)^2\cong (\Z_q \times \Z_q)^n$,  there is a natural bijection between~$\Lambda^n/S_n$ and the set of~$H$-orbits on~$(\Z_q^n)^2$. The function $(\Z_q^n)^2\to\CC_3$ with $(\alpha,\beta)\mapsto\{{\bm{0}},\alpha,\beta\}$
then gives a surjective function $r:\Lambda^n/S_n\to\Omega_3\setminus\{\{\emptyset\}\}$.\symlistsort{r}{$r$}{surjective function}

Then (cf.~Lemma~\ref{Lsomlemma}) for each $\omega\in\Omega_3$,
$$
 N_{\omega}=\sum_{\substack{\omega' \in \Lambda^n/S_n\\ r(\omega')=\omega}}K_{\omega'}.
$$
Hence, it suffices to compute $\bm{u_{\tau,B}}\T K_{\omega'}\bm{u_{\sigma,B}}$ for each
$\omega' \in \Lambda^n/S_n$. 

 By Proposition~\ref{ptsprop}, we have~$\sum_{\omega' \in \Lambda^n/S_n} \left( \bm{u_{\tau,B}}\T K_{\omega'}\bm{u_{\sigma,B}}\right)\mu(\omega') = \prod_{i=1}^2 p_{\tau_i,\sigma_i}(F_i)$, where the matrices~$F_i \in (W^*)^{m_i \times m_i}$  are defined in~$\eqref{fformula}$ and the polynomial~$p_{\tau_i,\sigma_i}$ is defined in~\eqref{ptsdef}. So it suffices to calculate the matrices~$F_i$, that is, to  express each $(F_i)_{j,h}=(B_i(j)\otimes B_i(h))|_W$
as linear function into the dual basis $A^*$ of~$A$. So we must
calculate the numbers $(B_i(j)\otimes B_i(h))(a_P)$ for all $i=1,2$,
$j,h=1,\ldots,m_i$, and $P\in\Lambda$
 --- see Appendix 1 (Section \ref{formulaslee} below).

Now one computes the entry $\hspace{-.2pt}\sum_{\omega \in \Omega_3}\hspace{-3pt} z(\hspace{-.06pt}\omega\hspace{-.06pt}) \bm{u_{\tau,B}}\T\hspace{-.1pt} N_{\omega}\hspace{-.1pt} \bm{u_{\sigma,B}}$ by first expressing $ \prod_{i=1}^2 \hspace{-.95pt}p_{\tau_i,\sigma_i}\hspace{-.75pt}(\hspace{-.7pt}F_i)$ as a linear combination of degree $n$ monomials expressed in the dual basis~$A^*$ of~$A$ and subsequently replacing each monomial~$\mu(\omega')$ in~$ \prod_{i=1}^2 p_{\tau_i,\sigma_i}(F_i)$ with the variable~$z(r(\omega'))$.

\subsection{The case~\texorpdfstring{$D=\emptyset$}{D=empty}\label{Demptylee}}
 Let~$D=\emptyset$. The rows and columns of~$M_{3,\emptyset}(z)$ (which is equal to $M_{2,\emptyset}(z)$) are indexed by words in~$\Z_q^n$ together with the empty set, and~$H_D$ is equal to  $D_q^n \rtimes  S_n$. Here~$G:=D_q$ acts on~$Z:={\Z_q}$. To compute the block diagonalization of~$M_{2,\emptyset}(z)$, one can use the Delsarte formulas in the Lee scheme~\cite{astola1,astola2}. Here we give the reduction in terms of representative sets.
 
  Let~$\zeta = e^{2 \pi i /q}$ be a primitive~$q$th root of unity.
For each~$j=0,\ldots,\floor{q/2}$, define the vectors $a_j := (1,\zeta^j, \zeta^{2j}, \ldots, \zeta^{(q-1)j})\T$, $b_j :=  (1, \zeta^{-j}, \zeta^{-2j},\ldots, \zeta^{-(q-1)j})\T  \in \C^{\Z_q}$ and set~$V_j := \text{span}\{a_j,b_j\}$. Furthermore, put
$$
c_j:=\frac{\sqrt{\dim V_j} }{2} (a_j+b_j)= \sqrt{\dim V_j}(1, \cos(2j\pi/q), \ldots,\cos(2(q-1)j\pi/q))\T \in \R^{\Z_q} \subseteq \C^{\Z_q}.
$$
 \begin{proposition}  
A representative set for the action of~$G=D_q$ on~$\C^Z= \C^{\Z_q}$ is given by
\begin{align} \label{linrepr}
\bm{B} := \left\{B_1,\ldots,B_{\floor{\frac{q}{2}}+1}\right\}, \text{ where $B_j:=c_{j-1}$,  for~$j=1,\ldots,\floor{q/2}+1$}.
\end{align}
\end{proposition}
\proof 
 Observe that each~$V_j$ is~$D_q$-stable and that~$c_j \in V_j$. Moreover,~$V_l$ and~$V_j$ are orthogonal (with respect to the inner product $u,v \mapsto v^*u$) if~$l \neq j$. To see this, note that~$x:=\zeta^{\pm j \pm l}$ is a~$q$th root of unity unequal to~$1$ if~$j \neq l \in \{ 0,\ldots ,\floor{ q/2 }\}$, so~$1+x+x^2+\ldots+x^{q-1}=0$. This implies that $a_j^*a_l=b_j^*a_l=a_j^*b_l=b_j^*b_l=0$, so~$V_l$ and~$V_j$ are orthogonal. Note that~$\sum_{j=0}^{\floor{q/2}} 1^2 =  \floor{q/2}+1$, which is the number of distinct~$V_j$, is equal to the dimension of~$(\C^{\Z_q} \otimes \C^{\Z_q})^{D_q}$. So the~$V_j$ form an orthogonal \emph{decomposition} of~$\C^{\Z_q}$ into \emph{irreducible} representations (as any further representation, or decomposition, or equivalence among the~$V_j$ would yield that the sum of the squares of the multiplicities of the irreducible representations is strictly larger than~$\floor{q/2}+1$, which contradicts the fact that~$\Phi$ in~$\eqref{PhiC2}$ is bijective). As~$B_{j+1}$ is an element of~$V_j$ for~$j=0,\ldots,\floor{q/2}$, this implies that~$\bm{B} $  from \eqref{linrepr} is a representative matrix set.
\endproof 
Note that the representative set is real, and that each~$B_i$ is a~$q \times 1$ matrix. For convenience of the reader, we restate  the main definitions of Chapter~\ref{orbitgroupmon}  and the main facts about representative sets  applied to the context of this section --- see Figure~\ref{factsleeleeg}.  
\begin{figure}[ht]
  \fbox{
    \begin{minipage}{15.5cm}
     \begin{tabular}{l}
{\bf FACTS.}  \\
\\
$G:= D_q$. \\
$Z:= \Z_q$. \\
\\
$k:=\floor{\frac{q}{2}}+1$, $\bm{m}:=(m_1,\ldots,m_k):=(1,\ldots,1)$.\\
\\
A representative set for the action of~$G$ on~$ Z$ is $\bm{B}$ from Prop.~\ref{linrepr}. \\
A representative set for the action of~$G^n \rtimes S_n$ on~$  Z^n$ follows from Prop.~\ref{prop2}.
\\ \\
$\Lambda := (\Z_q \times \Z_q)/D_q  $. \\
${a}_P := \sum_{ (x,y) \in P} x \otimes y$ for~$P \in {\Lambda}$.\\
${A}:= \{ {a}_P \, | \, P \in {\Lambda}\} $, a basis of~${W}:= (\C\Z_q \otimes \C\Z_q)^{D_q}$. \\
  ${K}_{\omega'} : = \sum_{(P_1,\ldots,P_n) \in \omega'}
{a}_{P_1} \otimes \cdots \otimes {a}_{P_n}$ for~$\omega' \in  {\Lambda}^n/S_n$. 
     \end{tabular}
    \end{minipage}    }  \caption{\label{factsleeleeg}\small{The main definitions of Chapter~\ref{orbitgroupmon}  and the main facts about representative sets  applied to the context of Section~\ref{Demptylee}.}}
\end{figure}

 Let $\bm{{N}}$ be the collection of all ${k}$-tuples~$(n_1,\ldots,n_{k})$ of nonnegative integers adding up to~$n$ (here~$k=\floor{q/2}+1$). For each~${\bm{n}} \in \bm{N}$, there is only one~$\bm{\lambda} \vdash \bm{n}$ with $\height(\lambda_i) \leq {m}_i= 1$ for all~$i=1,\ldots,{k}$. Moreover, for this~$\bm{\lambda}$, the set~$\bm{T_{\lambda, {m}}}=T_{\lambda_1,1} \times \ldots \times T_{\lambda_{{k}},1}$ from~\eqref{wlambda} has only one element~$\bm{\tau}=(\tau_1,\ldots,\tau_{{k}})$. Here each~$\tau_i$ is a Young tableau of height~$1$ which only contains ones (for~$i=1,\ldots,{k}$). Write~$\bm{\tau_n}$ for this element~$\bm{\tau}$.\symlistsort{taun}{$\bm{\tau_n}$}{unique element of~$\bm{T_{\lambda, {m}}}$ with~$\bm{m}=(1,\ldots,1)$} Then, for each~$\bm{n} \in \bm{{N}}$, the element~$\bm{u_{\tau_n, {B}}}$ from~$\eqref{vtau}$  is equal to
\begin{align}
\bm{u_{\tau_n, {B}}} = {B}_1^{\otimes n_1} \otimes {B}_2^{\otimes n_2} \otimes \ldots \otimes {B}_{{k}}^{\otimes n_{{k}}}.
\end{align} 
So a representative set for the action of~$D_q^n \rtimes S_n$ on~$ Z^n$ is, cf.\ Proposition~\ref{prop2},
\begin{align}
\{ \bm{u_{\tau_n, {B}}}\,\,\, | \,\,\, {\bm{n}} \in \bm{{N}} \}.
\end{align}
Observe that~$D_q^n \rtimes S_n$ acts trivially on~$\emptyset$. The~$D_q^n \rtimes S_n$-isotypical component of~$\mathbb{C}^{\Z_q^n}$ consisting of the~$D_q^n \rtimes S_n$-invariant elements corresponds to the matrix in the representative set indexed by~$\bm{n} = (n,0,\ldots,0)$. Hence  we add a new unit base vector~$e_{\emptyset}$ to this matrix (as a column) in order to obtain a representative set for the action of~$D_q^n \rtimes S_n$ on~$ \C^{\Z_q^n \cup \{\emptyset\}} = \C^{\mathcal{C}_3(\emptyset)}$.

\subsubsection{Computations for~\texorpdfstring{$D=\emptyset$ }{D=empty}} \label{d0comp}

In this section we explain how to compute the coefficients in the block diagonalization of~$M_{2,\emptyset}(z)$. First we give a reduction of~$M_{2,\emptyset}(z)$ without the row and column indexed the empty code. Later we explain how the empty code is added. 

For each~$\omega \in \Omega_2$, we define the~$\Z_q^n\times \Z_q^n$ matrix~$\widetilde{N}_{\omega}$ with entries in~$\{0,1\}$ by $ (\widetilde{N}_{\omega})_{\alpha,\beta} = 1$ if $\{\alpha,\beta\} \in \omega$ and $(\widetilde{N}_{\omega})_{\alpha,\beta} =0 $ otherwise, for~$\alpha, \beta \in \Z_q^n$.\symlistsort{Nomega widetilde}{$\widetilde{N}_{\omega}$}{$\Z_q^n\times \Z_q^n$ matrix}
 For each~$z : \Omega_2 \to \mathbb{R}$ we obtain with~$($\ref{matset}$ )$ and~$(\ref{PhiR})$ that $ \Phi \left(\sum_{\omega \in \Omega_2}z(\omega)\widetilde{N}_{\omega} \right ) = \bigoplus_{\bm{n} \in \bm{{N}}} \sum_{\omega \in \Omega_2} z(\omega) \bm{u_{\tau_n,{B}}}\T \widetilde{N}_{\omega} \bm{u_{\tau_n,{B}}}$.  
This shows that $\Phi \left(\sum_{\omega \in \Omega_2}z(\omega) \widetilde{N}_{\omega}\right)$ is a diagonal matrix. Note that~$|\Omega_2|$ and~$|\bm{{N}}|$ are polynomially bounded in~$n$.  Now we show how to compute~$\bm{u_{\tau_n,{B}}}\T \widetilde{N}_{\omega}  \bm{u_{\tau_n,{B}}}$, for~$\bm{n} \in \bm{{N}}$, in polynomial time.

 By~$\eqref{natbijection}$, there is a natural bijection between~$\Lambda^n/S_n$ and the set of~$H$-orbits on~$ (\Z_q \times \Z_q)^n$. Since naturally~$(\Z_q^n)^2 \cong (\Z_q \times \Z_q)^n$, there is a natural bijection between~$\Lambda^n/S_n$ and~$(\Z_q^n)^2/H$. The function~$(\Z_q^n)^2 \to \mathcal{C}_2$ that maps~$(\alpha,\beta)$ to $\{\alpha,\beta\}$ then induces a surjective function $\widetilde{r} \, : \, \Lambda^n/S_n \to \Omega_2 \setminus \{ \{ \emptyset \}\}$.\symlistsort{r tilde}{$\widetilde{r}$}{surjective function} Then (cf.~Lemma~\ref{Lsomlemma}) for each $\omega\in\Omega_2$,
$$
 \widetilde{N}_{\omega}=\sum_{\substack{\omega' \in  \Lambda^n/S_n\\ \widetilde{r}(\omega')=\omega}}K_{\omega'}.
$$

So it remains to compute $\bm{u_{\tau_n,B}}\T K_{\omega'}\bm{u_{\tau_n,B}}$ for each
$\omega' \in \Lambda^n/S_n$.  By  Proposition~\ref{ptsprop} we find that $\sum_{\omega' \in \Lambda^n/S_n} \left( \bm{u_{\tau_n,B}}\T K_{\omega'}\bm{u_{\tau_n,B}} \right)\mu(\omega') =\prod_{i=1}^{ {k}} p_{(\bm{\tau_n})_i,(\bm{\tau_n})_i}  (F_i) $, where~$ p_{(\bm{\tau_n})_i,(\bm{\tau_n})_i}$ is defined in~\eqref{ptsdef} and~$F_i$ is an~$1 \times 1$ matrix in~$({W}^*)^{1 \times 1}$ with $(F_i)_{1,1}:=( {B}_i \otimes  {B}_i)|_{ {W}} = \sum_{P \in  {\Lambda}} ( {B}_i \otimes  {B}_i)({a}_P) {a}_P^* \in {W}^*$, for~$i=1,\ldots,k$. We compute  the formulas~$(F_i)_{1,1}$ in Appendix 1 (Section \ref{formulaslee}) below. 

  Now one computes the entry $\sum_{\omega \in \Omega_2} z(\omega) \bm{u_{\tau_n, {B}}}\T  \widetilde{N}_{\omega} \bm{u_{\tau_n, {B}}}$ by first expressing the polynomial $ \prod_{i=1}^{ {k}} p_{(\bm{\tau_n})_i,(\bm{\tau_n})_i}(F_i) $ as a linear combination of degree $n$ monomials expressed in the dual basis~$ {A}^*$ of~$ {A}$ and subsequently replacing each monomial~$\mu(\omega')$ in~$\prod_{i=1}^{ {k}} (({F_i})_{1,1})^{n_i}$ with the variable~$z(\widetilde{r}(\omega'))$.   

To add the empty code, we add an extra row and column corresponding to the vector $e_{\emptyset}$  to the matrix in the representative set indexed by ${\bm{n}} = (n, 0, \ldots , 0)$, cf.\ Section \ref{Demptylee}. We compute the remaining entries:
$$
e_{\emptyset}\T M_{2,\emptyset}(z)e_{\emptyset} =M_{2,\emptyset}(z)_{\emptyset,\emptyset}= x(\emptyset)=1
$$
 by definition, see~$(\ref{Bleend})$. Since~$\bm{u_{\tau_n, {B}}}= {B}_1^{\otimes n}$ is the all-ones vector, we have $ e_{\emptyset}\T M_{\emptyset}(z)   \bm{u_{\tau_n, {B}}} = q^n z({\omega_0})$,
where~$\omega_0 \in \Omega_{2}$ is the (unique) $D_q^n \rtimes S_n$-orbit of a code of size~$1$.

\section{The unique code achieving~\texorpdfstring{$A^L_5(7,9)=15$}{AL,5(7,9)}\label{579}} 
An interesting case is~$A^L_5(7,9)$. The SDP-bound~$B^L_3(5,7,9)$, the linear programming bound (cf.~\cite{astola2}) and the bound from~\cite{quistorff} all give~$A^L_5(7,9) \leq 15$ as upper bound, but it was not known whether a~$(7,9)_5^L$-code of size~$15$ exists. Here any code~$C\subseteq \Z_q^n$ with~$d_{\text{min}}^L(C)\geq d$ is called an\emph{~$(n,d)_q^L$-code} and similarly, any code~$C\subseteq \Z_q^n$ with~$d_{\text{min}}^H(C)\geq d$ is an\emph{~$(n,d)_q^H$-code}. We first derive some conditions that any~$(7,9)_5^L$-code of size~$15$ satisfies.

\subsection{Conditions on an~\texorpdfstring{$A^L_5(7,9)=15$}{AL,5(7,9)}-code and Kirkman schoolgirls}
Let~$C$ be an~$(n,d)_q^L$ code of size~$M$. By estimating~$\sum_{\substack{\{u,v\} \subseteq C  }} d^L(u,v)  $ one obtains:
\begin{align}\label{leeterms}
\binom{M}{2}d \leq \sum_{\substack{\{u,v\} \subseteq C  }} d^L(u,v)  = \sum_{i=1}^n \sum_{\substack{\{u,v\} \subseteq C  }} d^L(u_i,v_i) \leq n P_q(M), 
\end{align}
where
\begin{align} \label{Pqm}
P_q(M) := \max \left\{ \sum_{ \substack{ \{i,j\} \subseteq \{1,\ldots,M\}  }}  d^L(u_i,u_j) \,\, \big| \,\, (u_1,\ldots,u_M) \in \Z_q^M  \right\}.
\end{align}
Wyner and Graham~$\cite{wyner}$ proved that if~$q$ is odd, then~$P_q(M) \leq M^2(q^2-1)/(8q)$ and moreover that if~$q$ divides~$M$ and the $(q-1) \times (q-1)$ matrix~$(d_L(0,i)+d_L(0,j)-d_L(i,j))_{i,j=1,\ldots,q-1}$ is positive definite, then the maximum in~\eqref{Pqm} is achieved only by vectors~$(u_1,\ldots,u_M)$ in which each symbol in~$\Z_q$ appears exactly~$M/q$ times. 

For the case~$(n,d)_q^L = (7,9)_5^L$ and~$M=15$, we find that~$P_5(15)=135$, that the maximum in~\eqref{Pqm} is only achieved by vectors in which each symbol in~$\Z_5$  appears exactly~$3$ times (as the matrix from the previous paragraph is positive definite in this case), and that the rightmost term in~$\eqref{leeterms}$ is equal to the leftmost term in~$\eqref{leeterms}$. This implies that if~$C$ is a $(7,9)_5^L$ code of size~$15$, then
\begin{speciaalenumerate}
 \item each symbol in~$\Z_5$ appears exactly $3$ times in each column of~$C$ (interpreting~$C$ as a matrix with the words as rows), and 
 \item any two distinct words in~$C$ have Lee distance exactly~$d=9$, i.e.,~$C$ is \emph{equidistant} with Lee distance~$9$.
 \end{speciaalenumerate}
 So a~$(7,9)_5^L$-code~$C$ of size~$15$ is equivalent to an arrangement of~$15$ girls $7$ days in succession into triples, where the triples are placed at the corners of a pentagon, such that the total Lee distance between any two girls over the $7$ days is~$9$. This problem is similar to, but different from, \emph{Kirkman's school girl problem}, which asks for an arrangement of~$15$ girls $7$ days in a row into triples such that no two girls appear in the same triple twice. Such an arrangement is equivalent to a~$(7,6)_5^H$-code of size~$15$ (see~\cite{zinoviev} and Chapter~\ref{divchap} of this thesis), and there are~$7$ nonisomorphic arrangements possible \cite{7sol}. 
 
Consider the following~$\Z_7 \times \Z_7$ matrix with entries in~$\Z_5$:
 \begin{align}\label{fanocode}\mbox{\small $
\widetilde{M}:=
 \begin{pmatrix}
0  & 1&1&3&1&3&3 \\
 1&1&3&1&3&3 &0 \\
1&3&1&3&3 &0  & 1\\
3&1&3&3&0  & 1&1 \\
1&3&3&0  & 1&1&3 \\
3&3&0  & 1&1&3&1 \\
3&0  & 1&1&3&1&3 \\
 \end{pmatrix}, \,\,\,\, \text{ so } \widetilde{M}_{i,j} = \begin{cases}
0, &  \text{if~$i+j \equiv 0 \pmod{7}$}, \\
1, & \text{if~$i+j$ is a nonzero  square mod~$7$}, \\
3, & \text{if~$i+j$ is not a square mod~$7$}.\footnotemark
\end{cases}$}
 \end{align}
\footnotetext{Note that the $\{0,1\}$-matrix of order $7 \times 7$ which has 1's precisely in the positions where~$\widetilde{M}$ is $1$ and  the $\{0,1\}$-matrix of order $7 \times 7$ which has $1$'s precisely where~$\widetilde{M}$ is~$3$, both are incidence matrices of Fano planes~\cite{beth} (with disjoint support).}We will see that every $(7,9)_5^L$-code~$C$ of size~$15$ is equivalent to the code~$\widetilde{M} \cup  -\widetilde{M} \cup \{\bm{0} \}$, where the matrix~$\widetilde{M}$ is interpreted as a~$(7,9)_5^L$-code of size~$7$ (i.e., the rows are the codewords). Note that this code has minimum Hamming distance~$5$, so it follows that there is no~$(7,9)_5^L$-code of size~$15$ which is also a~$(7,6)_5^H$-code, i.e., there is no arrangement of $15$ girls~$7$ days in succession into triples which is a solution simultaneously to both problems given above.

\subsection{Information from semidefinite programming}

 Since~$B_3^L(5,7,9)\approx 15.000..$, we hoped to obtain information from the semidefinite programming output about a~$(7,9)_5^L$-code of size~$15$ in a way analogous to Chapter~\ref{cu17chap}: we hoped to find a list of orbits~$\omega \in \Omega_3^d$  that are forbidden in a~$(7,9)_5^L$-code of size~$15$. 
  However, the output of~$B_3^L(5,7,9)=15$ does not appear to give more information than we already have. Only orbits that contain Lee distances~$\neq 9$ are forbidden by this output in a~$(7,9)_5^L$-code of size~$15$. So the code is equidistant, but this follows already from the above considerations with~\eqref{leeterms}. In the next section, we give a larger semidefinite program that yields more information about~$(7,9)_5^L$ codes of size~$15$. The new information is described in Section~\ref{sixorbits}, where we also briefly describe the procedure from Section~\ref{procuniq}.

\subsubsection{A quadruple SDP-bound}
We consider the following bound, which is a bound in between~$B^L_{k-1}(q,n,d)$ and~$B^L_k(q,n,d)$ (cf.~\cite{cw4}):\symlistsort{bLk(q,n,d)}{$b^L_k(q,n,d)$}{upper bound on~$A_q^L(n,d)$}
\begin{align} \label{bleend}
b^L_k(q,n,d):=    \max \{ \mbox{$\sum_{v \in \Z_q^n} x(\{v\})\,\,$} |\,\,&x:\mathcal{C}_k \to \R, \,\, x(\emptyset )=1, x(S)=0 \text{ if~$d_{\text{min}}^L(S)<d$}, \notag \\
& M_{k-1,D}(x|_{\CC_{k-1}}) \succeq 0 \text{ for each~$D \in \mathcal{C}_{k-1}$ with~$|D|<2$}, \notag 
\\
& M_{k,D}(x) \succeq 0 \text{ for each~$D \in \mathcal{C}_{k}$ with~$|D|\geq 2$}\}.  
\end{align}
We consider the bound $b^L_4(q,n,d)$. In the definition, it can be assumed that~$x:\mathcal{C}_4 \to \R_{\geq 0}$. The reductions for the cases~$D=\emptyset$ and~$|D|=1 $ can be found in Section~\ref{reduct}. In Appendix~\ref{D2lee} we describe in detail how to block diagonalize~$M_{4,D}(z)$ if~$|D|=2$. The matrices~$M_{4,D}(z)$ for~$|D|\in\{3,4\}$ are~$1 \times 1$ matrices, hence they form their own block diagonalization.

\subsubsection{Six orbits\label{sixorbits}}
The semidefinite program for computing~$b_4^L(5,7,9)$ is very large, even after symmetry reductions. It contains~$292,483$ variables after reductions, i.e., $D_5^7 \rtimes S_7$-orbits of nonempty codes in~$\mathcal{C}_4$ with minimum Lee distance at least~$9$. Since any~$(7,9)_5^L$-code of size~$15$  is equidistant with Lee distance~$9$, we add to~$b_4^L(5,7,9)$ the constraint that~$y(\omega)=0$ if~$\omega$ is an orbit of a code in which some pair of words has Lee distance unequal to~$ 9$.  Write~$b_4^{L}(5,7,9)_{\text{eq}}$ for the resulting semidefinite program. It contains only~$1,632$ variables. 

The program for computing~$b_4^{L}(5,7,9)_{\text{eq}}$ contains constraints~$z({\omega}) \geq 0$ for every~$\omega \in \Omega_4^d \setminus \{\{\emptyset\}\}$ that corresponds to an equidistant code with Lee distance~$9$, and we write~$(X_{\omega})$ for the $1 \times 1$ block corresponding to~$(z({\omega}))$ in the dual semidefinite program. Then~$X_{\omega} >0$ in any solution to the dual program implies that~$z(\omega)=0$ in all primal solutions to the semidefinite program, as explained in Chapter~\ref{cu17chap}, which implies that no code of size~15 can contain a subcode~$D$ with~$D \in \omega$. (Suppose otherwise; then one constructs a feasible solution to~$b_4^{L}(5,7,9)_{\text{eq}}$ by putting~$x(S)=1$ for~$S \in \CC_k$ with~$S \subseteq C$ and~$x(S)=0$ else, and hence by averaging over~$H=D_5^7\rtimes S_7$ we find a feasible $H$-invariant solution with~$z(\omega) \geq 1/|H| >0$, a contradiction.)  If~$X_{\omega}>0$ in any dual solution, we call~$\omega$ a \emph{forbidden} orbit. 

We used the solver SDPA-GMP~$\cite{nakata, sdpa}$ to conclude which orbits are forbidden. The semidefinite programming solver does not produce exact solutions, but approximations up to a certain precision. All matrices in our approximate solution to the dual program were verified to be positive semidefinite (even strictly positive definite). This allows us to use equations~\eqref{error}-\eqref{errorlast} in~Chapter~\ref{cu17chap} to  verify with certainty which orbits are forbidden. (To use~\eqref{errorlast}, we need a specific lower bound on~$z(\omega)$ for non-forbidden orbits~$\omega$; in this case the lower bound of~$1/|H|>10^{-11}$ from the previous paragraph suffices.)

There are only six non-forbidden  orbits in~$\Omega_4$ of equidistant codes with Lee distance~$9$ of cardinality~$4$ for which there is one coordinate in which three of the codewords have the same symbol, say~$\alpha$, and the last codeword has a symbol at Lee distance~$1$ from~$\alpha$.  Here we list one representative~$N_i$ from each of these six orbits: 
\begingroup
\allowdisplaybreaks
\begin{align*}
N_1&:= \{ 0000000, 
0012222, 
0230234,  
1123414 \}, \\
N_2&:=  \{ 
0000000,  
0012222, 
0230234,
1124413\}, \\
N_3&:= \{
0000000,
0012222,
0230234,
1232002\}, \\
N_4&:= \{
0000000,
0012222,
0230234,
1303424\},\\ 
N_5&:=  \{
0000000, 
0012222, 
0230234, 
1424341\}, \\ 
N_6&:= \{
0000000,
0111222,  
0222444, 
1033023\}. 
\end{align*}
\endgroup
So any~$(7,9)_5^L$-code of size~$15$ is equivalent to a~$(7,9)_5^L$-code containing one of the~$N_i$ as a subcode.

\subsection{Uniqueness of~\texorpdfstring{$A^L_5(7,9)=15$}{AL,5(7,9)}}
We now descibe how one can obtain all possible~$(7,9)_5^L$-codes of size~$15$ that contain~$N_i$, for~$i=1,\ldots,6$. For each~$i =1,\ldots,6$, we do the following. 

Let~$R_i$ be the set of all words in~$\Z_5^7$ that have Lee distance~$9$ to all words in~$N_i$. Let~$R_{j,i}$, for each~$j \in \{2,3,4\}$, be the collection of all triples of codewords in $R_i$ with the property that each word starts with symbol~$j$ and that the triple is equidistant with Lee distance~$9$. Also, let~$R_{1,i}$ be the collection of all pairs of codewords in~$R_i$ that start with symbol~$1$ and have Lee distance~$9$. Now observe that each~$(7,9)_5^L$-code~$C$ of size~$15$ that contains~$N_i$ is a union of~$N_i$ and exactly one element in~$R_{j,i}$ for each~$j=1,\ldots,4$ (here note that each element in~$R_{j,i}$ is a set consisting of~$3$ words if~$j\in \{2,3,4\}$, and consisting of~$2$ words if~$j=1$), as~$C$ is equidistant with Lee distance~$9$ and each symbol occurs exactly~$3$ times at the first (more generally: at any) position. The computer finds
$$
\max_{i=1,\ldots,6} \mbox{$\left\{ \prod_{j=1}^4 |R_{j,i}| \right\} = 528,384$},
$$    
so a computer quickly enumerates all possible four-tuples~$(T_1,T_2,T_3,T_4) \in R_{1,i} \times R_{2,i} \times R_{3,i} \times R_{4,i}$ and verifies whether~$N_i \cup (\cup_{j=1}^4 T_j)$ indeed has minimum Lee distance~$9$. In this way one obtains all possible $(7,9)_5^L$-codes~$C$ of size~$15$ that contain~$N_i$. It is found that for each of~$N_1$, $N_4$, $N_5$ there is only one $(7,9)_5^L$-code of size~$15$ containing~$N_i$ ($i=1,4,5$), while for each of~$N_2$, $N_3$ and~$N_6$, there are exactly two $(7,9)_5^L$-codes of size~$15$  containing~$N_i$ ($i=2,3,6$).

It remains to classify the~$9$ obtained codes. To do this, the graph isomorphism program \texttt{nauty}~$\cite{dreadnaut}$ is used. For each $(7,9)_5^L$-code~$C$ of size~$15$, a graph with~$5\cdot 7+15$ vertices is created: one vertex for each codeword~$u \in C$ and five vertices~$0_k,1_k,2_k,3_k,4_k$ for each coordinate position~$k=1,\ldots,7$. Each code word~$u$ has neighbor~$a_k$ if~$u_i=a$ ($i=1,\ldots,15)$. Also, there are edges~$\{0_k,1_k\},\{1_k,2_k\},\{2_k,3_k\},\{3_k,4_k\},\{4_k,0_k\}$ ($k=1,\ldots,7$).

The codewords have degree~$7$ and the coordinate positions have degree~$2+3=5$ (as each symbol occurs exactly~$3$ times at each coordinate position).
An automorphism of this graph permutes the codewords and permutes the coordinate positions. So two $(7,9)_5^L$-codes~$C$ of size~$15$ are equivalent if and only if the corresponding graphs are equivalent.  In this way we find that all nine codes are equivalent, so the  $(7,9)_5^L$-code of size~$15$  is unique. One of the two extensions of~$N_6$ is for example
\begin{align*}
C:=\{
0000000,
0111222,
0222444,
1033023,
1303302,
1330230,
2142141,
2214114,\phantom{\}}\\
2421411,
3013330,
3130303,
3301033,
4244421,
4424142,
4442214
\}.
\end{align*}
The reader may check that it is equivalent to the code~$\widetilde{M} \cup  -\widetilde{M} \cup \{\bm{0} \}$, with~$\widetilde{M}$ as in~(\ref{fanocode}).


\section{The unique code achieving~\texorpdfstring{$A^L_6(4,6)=18$}{AL,6(4,6)}\label{646}}
 Let~$C$ be any~$(4,6)_6^L$-code of cardinality~$18$. The solution of $B_2^L(6,4,6)\approx18.000..$ shows, by considering the forbidden orbits from the semidefinite programming output in a similar way as in the previous section, that any pair of codewords in~$C$ is contained in the same $D_6^4 \rtimes S_4$-orbit as one of the following codes:
\begin{align} \label{orbitpairs}
M_1&:= \{
0000, 
0222\}, \quad 
M_2:= \{0000,
1113\}, \quad 
M_3:= \left\{ 0000,
3333\right\}.
\end{align} 
Assume that~$\bm{0} \in C$.  By~$(\ref{orbitpairs})$, each word in~$C$ contains only symbols in~$\{0,2,4\}$ or only symbols in~$\{1,3,5\}$. As~$A_3^L(4,3)=9$, there are at most~$9$ words containing only symbols~$\{0,2,4\}$ in~$C$ and there are at most~$9$ words containing only symbols~$\{1,3,5\}$ in~$C$. As~$|C|=18$, the code~$C$ contains precisely~$9$ words with only symbols in~$\{0,2,4\}$ and precisely~$9$ words with only symbols in~$\{1,3,5\}$. It is well known that the~$(4,3)_3^H=(4,3)_3^L$-code of size~$9$ is unique up to equivalence~\cite{terclas}. It is (up to equivalence) the code
$$
B:=\{0000, 0111,0222,1012,1120,1201,2021,2102,2210\}.
$$
 As~$D_6$ contains a subgroup that acts as~$D_3=S_3$ on~$\{0,2,4\}$, we may assume that~$C':=2B \pmod{6}$ is a subcode of~$C$. Now we enumerate all~$33$ words~$v \in \Z_6^4$ such that~$\{\bm{0},v\}$ is in the same $D_6^4 \rtimes S_4$-orbit as~$M_2$ or~$M_3$. There are only~$9$ of those words with Lee distance~$\geq 6$ to all words in~$C'$. So these are the only~$9$ words that could be contained in a~$(4,6)_6^L$-code of size~$18$ together with all words from~$C'$. Let us call the code formed by these~$9$ words~$C''$. So
 \begin{align*}
C'&= \{0000,0222,0444,2024,2240,2402,4042,4204,4420\},\\
C'' &= \{1153,1315,1531,3111,3555,5135,5351,5513,3333\} = 3333+C'.
\end{align*}
 It is not hard to verify that~$C' \cup C''$ has minimum Lee distance~$6$.  So~$C' \cup C''$ is a~$(4,6)_6^L$-code of cardinality~$18$, and it is the only such code (up to Lee equivalence).

\section{Appendices}

\subsection{Appendix 1: The formulas~\texorpdfstring{$(F_i)_{j,h}$}{(Fi)j,h}}\label{formulaslee}

\paragraph{Formulas for Section~\ref{D1}.} 
Let all notation be as in Section~\ref{D1} and Figure~\ref{factslee0}. For the computations in Section~\ref{d1comp}, we calculate the linear expressions 
$$
(F_i)_{j,h} =(B_i(j)\otimes B_i(h))|_W = \sum_{P \in \Lambda} (B_i(j)\otimes B_i(h))(a_P) a_P^* \in W^*,
$$
for~$i=1,2$ and~$j,h=1,\ldots,m_i$. This is routine, but we display the expressions. We denote an equivalence class~$P$ in~$\Lambda = (\Z_q \times \Z_q) /S_2$ by its lexicographically smallest element. We find
\begin{align} 
(F_1)_{1,1} &= 1 a_{00}^*, \notag   \\ 
(F_1)_{1,j+1}  &= 2 a_{0j}^*,    \text{ for~$j=1,\ldots,\floor{q/2}$} \notag \\
(F_1)_{j+1,1}  &= 2 a_{j0}^*,    \text{ for~$j=1,\ldots,\floor{q/2}$}  \notag\\
(F_1)_{j+1,h+1}  &= 2 a_{jh}^* + 2 a_{j(q-h)}^*,    \text{ for~$j,h \in \{1,\ldots,\floor{q/2}\}$}, \notag  \\ 
(F_2)_{j,h}  &= 2 a_{jh}^* - 2 a_{j(q-h)}^*,    \text{ for~$j,h \in \{1,\ldots,\floor{(q-1)/2}\}$}.   \label{formuleslee1}
\end{align}  

\paragraph{Formulas for Section~\ref{Demptylee}.} 
Now, let all notation be as in Section~\ref{Demptylee} and Figure~\ref{factsleeleeg}. For the computations in Section~\ref{d0comp}, we calculate the linear expressions 
$$
(F_i)_{1,1} =( B_i\otimes B_i)|_{{W}}= \sum_{P \in {\Lambda}} ({B}_i \otimes {B}_i)({a}_P) {a}_P^* \in {W}^*,
$$ for $i=1,\ldots,k$, where~${k}=\floor{q/2}+1$.  We denote an equivalence class~$P$ in~${\Lambda} = (\Z_q \times \Z_q)/D_q$ by its lexicographically smallest element. We find
  for even~$q$, for~$i \in \{1,\ldots, {k} \}$,
\begin{align} \label{cos1}
(F_i)_{1,1}&  = q \left({a}_{00}^*+(-1)^{(i-1)} {a}_{0(q/2)}^* + 2\sum_{j=1}^{\floor{q/2}-1} \cos(2 \pi j(i-1)/q) {a}_{0j}^* \right),
\intertext{and for odd~$q$, for~$i \in \{1,\ldots, {k} \}$,}
(F_i)_{1,1}& = q \left( {a}_{00}^* + 2\sum_{j=1}^{(q-1)/2} \cos(2 \pi j(i-1)/q) {a}_{0j}^*\right) .
 \label{formuleslee2compact}
\end{align} 

Note that the formulas in~$\eqref{cos1}$ and~$\eqref{formuleslee2compact}$  contain irrational numbers for~$q \notin \{2,3,4,6\}$. In the symmetry reduction of~$M_{2,\emptyset}(z)$, it is possible to obtain matrix blocks which only contain integers: one may use  the representative set for the reflection action~$S_2$ on~$\C^{\Z_q}$ instead of the representative set of the action of~$D_q$ on~$\C^{\Z_q}$ to reduce the matrix~$M_{2,\emptyset}(z)$.  This results in much larger matrix blocks, but they only contain integers. We leave the details to the reader,  see~\cite{leeartikel}.

\subsection{Appendix 2: Block diagonalizing~\texorpdfstring{$M_{4,D}(z)$ if~$|D|=2$}{M4D(z) if |D|=2}\label{D2lee}}
Here we sketch how to block diagonalize~$M_{4,D}(z)$ if~$|D|=2$, for computing~$b_4^L(q,n,d)$. We only consider the case~$q=5$, but the reductions can be easily generalized to higher~$q$. (However, although the semidefinite programs are of polynomial size, they get very large in practice.) 
We can assume that~$D=\{\bm{0},v_1\}$ with
\begin{align}
    \bm{0}=& \,0\ldots0\,0\ldots0\,\,0\ldots0
    \\v_1= &\underbrace{0\ldots0}_{j_1}\underbrace{1\ldots1}_{j_2}\underbrace{\,2\ldots2}_{j_3},    \notag 
\end{align}
where~$j_2+2j_3 \geq d$. The rows and the columns of~$M_{4,D}(y)$ are parametrized by codes~$C\supseteq D$ of size at most~$3$.

Let~$H_D'$ be the group of distance-preserving permutations of~$\mathcal{C}_4$ that fix~$\bm{0}$ and~$v_1$. So
\begin{align}
    H_D' \cong (S_2^{j_1}\rtimes S_{j_1}) \times S_{j_2} \times S_{j_3}.
\end{align}
We first describe a representative set for the action of~$H_D'$ on~$\C^{\Z_5^n}$. We use the notation of Section~\ref{genmult}. Fix~$s=3$. For~$i=1,\ldots,3$, set~$Z_i:=\Z_5$. Set~$G_1=S_2$ and~$G_i:=\{1\}$ (the trivial group) for~$i \in \{2,3\}$. Furthermore, set~$k_1:=2$, $k_2:=1$, $k_3:=1$. 
 Set~$B_1^{(1)}:=B_1$  from~\eqref{oddrepr} and~$B_2^{(1)}:= B_2$ from~\eqref{oddrepr}, and set~$\bm{B^{(1)}}:=\{B_1^{(1)},B_2^{(1)}\}$. Finally, let, for~$i=2,3$,~$B_1^{(i)}:=[e_0,e_1,e_2,e_3,e_4]$ be the matrix with as columns the standard basis vectors in~$\mathbb{C}^{\Z_5}$, i.e., the~$5 \times 5$ identity matrix, and set~$\bm{B^{(i)}}:=\{B_1^{(i)}\}$. For convenience of the reader, we restate  the main definitions of Section~\ref{genmult}  and the main facts about representative sets  applied to the context of this section --- see Figure~\ref{factslee12}.

\begin{figure}[htb]
  \fbox{
    \begin{minipage}{15.5cm}
     \begin{tabular}{l}
{\bf FACTS.}  \\
\\
$s=3$, $(j_1,j_2,j_3)$ are fixed with~$j_1+j_2+j_3=n$ and $j_2+2j_3 \geq d$.\\\\
$G_1:= S_2$, $G_i := \{1 \}$, the trivial group, for~$i=2,3$. \\
$Z_i:= \Z_5$, for $i=1,2,3$. \\
\\
$k_1:=2$, $k_2:=1$, $k_3:=1$. \\
$\bm{m^{(1)}}:=(3,2)$, $\bm{m^{(2)}}:=(5)$, $\bm{m^{(3)}}:=(5)$.\\
\\
A representative set for the action of~$G_i$ on~$ Z_i$ is $\bm{B^{(i)}} $ from Sect.~\ref{D2lee}. \\
A representative set for the action of~$H_D'$ on~$(\C^{Z_1})^{\otimes j_1} \otimes \ldots \otimes (\C^{Z_3})^{\otimes j_3} $ follows 
\\ from~\eqref{matset} and~\eqref{prodreprset}.
\\ \\
$\Lambda_i := (Z_i \times  Z_i )/G_i $. \\
$a_P := \sum_{(x,y) \in P} x \otimes y$ for~$P \in \Lambda_i$.\\
$A_i:= \{ a_P \, | \, P \in \Lambda_i\} $, a basis of~$W_i:= (\C{ Z_i} \otimes \C Z_i )^{G_i}$. \\
  $K^{(i)}_{\omega_i} := \sum_{(P_1,\ldots,P_{j_i}) \in \omega_i}
a_{P_1} \otimes \cdots \otimes a_{P_{j_i}}$ for~$\omega_i \in \Lambda_i^{j_i}/S_{j_i}$. 
     \end{tabular}
    \end{minipage}    }  \caption{\label{factslee12}\small{The main definitions of Section~\ref{genmult}  and the main facts about representative sets  applied to the context of Section~\ref{D2lee}.}}
\end{figure}

With Proposition~$\ref{prop2}$, we obtain representative sets for the separate actions of~of $G_i^{j_i}\rtimes S_{j_i}$ on $\C^{\Z_5^{ j_i}}$, for each~$i=1,2,3$. Subsequently, with~$\eqref{prodreprset}$ we obtain a representative set for the action of $H_D'$ on $(\C^{\Z_5^{j_1}}) \otimes \ldots \otimes (\C^{\Z_5^{j_3}}) \cong \C^{\Z_5^n}$.

\subsubsection{Sketch of the computations for~\texorpdfstring{$|D|=2$}{D=2}} 

Fix~$D = \{\bm{0}, v_1 \} \in \mathcal{C}_4$. Let~$\Omega_4(D)$ denote the set of all~$D_5^n \rtimes S_n$-orbits of codes in~$\mathcal{C}_4$ containing~$D= \{\bm{0}, v_1 \} $. For each~$\omega \in \Omega_4(D)$, we define the~$\Z_5^n \times \Z_5^n$ matrix~$N_{\omega}'$ with entries in~$\{0,1\}$ by\symlistsort{Nomega'}{$N_{\omega}'$}{$\Z_5^n \times \Z_5^n$ matrix}
\begin{align}
    (N_{\omega}')_{\alpha,\beta} := \begin{cases} 1 &\mbox{if } \{\bm{0},v_1,\alpha,\beta\} \in \omega,  \\ 
0 & \mbox{otherwise.} \end{cases} 
\end{align}
Then, for each~$z : \Omega_4(D) \to \mathbb{R}$, one has~$M_{4,D}(z) \succeq 0$ if and only if~$\sum_{\omega \in \Omega_4(D)}z(\omega) N_{\omega}' \succeq 0$.  (The implication``$\Longrightarrow$'' follows from the fact that~$L\T M_{4,D}(z) L=\sum_{\omega \in \Omega_4(D)}z(\omega) N_{\omega}' $, where~$L$ is the $\mathcal{C}_4(D) \times \mathbb{Z}_5^n$ matrix with~$0,1$ entries satisfying
$    L_{C,\alpha} = 1 \,\, \text{ if and only if } C =   \{\mathbf{0},v_1,\alpha\},$
for~$C \in \mathcal{C}_4(D)$ and~$\alpha \in \mathbb{Z}_5^n$.) 

As in Sect.~\ref{genmult}, let~$R=Z_1^{j_1} \times \ldots \times Z_3^{j_3}$, so $R = \Z_5^{j_1} \times \Z_5^{j_2} \times  \Z_5^{j_3}\cong \Z_5^n$. Then the function
\begin{align}
  R^2= \left( \Z_5^{j_1} \times \Z_5^{j_2} \times  \Z_5^{j_3} \right)^2 &\to \mathcal{C}_4,\\
((\alpha^1,\alpha^2,\alpha^3),(\beta^1,\beta^2,\beta^3)) &\mapsto \{\bm{0},v_1,\alpha^1\alpha^2\alpha^3, \beta^1\beta^2\beta^3\},    \notag 
\end{align}
induces a surjective function~$ r' \, : \, R^2/H_D' \to \Omega_4(D)$.\symlistsort{r'}{$r'$}{surjective function}
For any $\omega' \in R^2/H_D'$, the matrix~$K_{\omega' }$ is defined in~\eqref{KomegaIH}. If~$\omega \in \Omega_4(D)$, then in view of Lemma~$\ref{Lsomlemma}$ we obtain
\begin{align}\label{Nomegasom}
    N_{\omega}' = \sum_{\substack{\omega' \in R^2/H_D' \\  r'(\omega') = \omega  }} K_{\omega'}.
\end{align}
With the help of the representative set of the action of~$H_D'$ on $\C^{\Z_5^n}$, equation~$\eqref{PhiR}$ gives a symmetry reduction of $\sum_{\omega \in \Omega_4(D)}z(\omega) N_{\omega}'$ to size polynomially bounded in~$n$. The computation of the matrix entries in the reduced matrix can also be done in polynomial time, using \eqref{Nomegasom}, \eqref{Komegaprod}, \eqref{tensorcomp} and Proposition~\ref{ptsprop}. See also Section~\ref{d12comp} about the computations for the case~$|D|=2$ for constant weight codes, which is a highly similar case.

\subsection{Appendix 3: An overview of the program}\label{pseudocodeleesect}
In this section we give a high-level overview of the program for generating semidefinite programs to compute~$B_3^L(q,n,d)$ (or~$B_3^{L_{\infty}}(q,n,d)$ from Section~\ref{circbounds} below).
See Figure~$\ref{pseudocodelee}$ for an outline of the method.\symlistsort{omega0}{$\omega_0$}{orbit of a code of size~$1$}  Here  we write~$\omega_0$ for the unique~$D_q^n\rtimes S_n$-orbit corresponding to a code of size~$1$.\footnote{The programs we used to generate input for the SDP-solver can be found at \url{https://drive.google.com/open?id=1-XRbfc4TYhoySC33GRWfvNEMOZEltg6X}. (also accessible via the author's website).}

\begin{figure}[htb]
  \fbox{
    \begin{minipage}{15.5cm}\small
     \begin{tabular}{l}
    \textbf{Input: } Natural numbers~$q,n$ \\
	\textbf{Output: }Semidefinite program to compute~$B_3^L(q,n,1)$  \\
	$\,$\\
    \printv \emph{Maximize} $q^n z(\omega_0)$  \\
        \printv \emph{Subject to:}  \\    
\text{\color{blue}//Start with~$|D|=1$. The definitions are as in Sect.~\ref{D1}.}\\
\foreachv $\bm{n}=(n_1,n_2) \in \bm{N}$ 
\\ \text{\color{blue}//Recall: $\bm{m}=(m_1,m_2)=( \floor{q/2}+1 , \floor{(q-1)/2} )$.}   \\
	\hphantom{1cm} \foreachv $\bm{\lambda}=(\lambda_1,\lambda_2) \vdash \bm{n}$ with height$(\lambda_1) \leq m_1 $,   height$(\lambda_2) \leq m_2$\\
			\hphantom{1cm} \hphantom{1cm} start a new block~$M_{\bm{\lambda}}$\\
			\hphantom{1cm} \hphantom{1cm} \foreachv $\bm{\tau} \in \bm{T_{\lambda,m}}$ from~$(\ref{wlambda})$\\
					\hphantom{1cm} \hphantom{1cm}\hphantom{1cm} \foreachv $\bm{\sigma} \in \bm{T_{\lambda,m}}$ from~$(\ref{wlambda})$\\
					\hphantom{1cm} \hphantom{1cm}\hphantom{1cm}\hphantom{1cm}  compute $ \prod_{i=1}^2 p_{\tau_i,\sigma_i}(F_i)$  (cf. the definitions in~\eqref{ptsdef} and~\eqref{fformula})\\			
							\hphantom{1cm} \hphantom{1cm}\hphantom{1cm}\hphantom{1cm}		replace each degree~$n$ monomial~$\mu$ in variables~$a_{P}^*$ by a variable~$z(r(\mu))$\\
	\hphantom{1cm} \hphantom{1cm}\hphantom{1cm}\hphantom{1cm}	 $(M_{\bm{\lambda}})_{\bm{\tau},\bm{\sigma}}:= $ the resulting linear polynomial in variables~$z(\omega)$ \\				
					\hphantom{1cm} \hphantom{1cm}\hphantom{1cm} \ndv\\
			\hphantom{1cm} \hphantom{1cm} \ndv\\
		  	\hphantom{1cm} \hphantom{1cm}   \printv $M_{\bm{\lambda}} \succeq 0$   \\
			\hphantom{1cm} \ndv\\  	
 \ndv\\
 \text{\color{blue}//Now~$D=\emptyset$. The definitions are as in Sect.~\ref{Demptylee}.}\\
\foreachv $\bm{n} \in \bm{{N}}$  \\
	 \hphantom{1cm} start a new $(1 \times 1)$ block~$M_{\bm{n}}$  \\
      \hphantom{1cm} compute~$\prod_{i=1}^{{k}} p_{(\bm{\tau_n})_i,(\bm{\tau_n})_i}  ({F}_i) $  (cf.~the definitions in~\eqref{ptsdef} and Sect.~\ref{Demptylee}) \\			
							\hphantom{1cm}		replace each degree~$n$ monomial~$\mu$ in variables~${a}_{P}^*$ by a variable~$z(\widetilde{r}(\mu))$\\
	                           \hphantom{1cm}	 $(M_{\bm{n}}):= $ (the resulting linear polynomial in variables~$z(\omega)$) \\	
	                           	                           	\hphantom{1cm}   \ifv $\bm{n}=(n,0,\ldots,0)$    \text{\color{blue} //Add a row and a column corresponding to $\emptyset$.}\\		
		                          \hphantom{1cm} 	\hphantom{1cm} add a row and column to $M_{\bm{n}}$ indexed by~$\emptyset$   \\		   
		                         \hphantom{1cm}  	\hphantom{1cm} put $(M_{\bm{n}})_{\emptyset,\emptyset}:=1$ and $(M_{\bm{n}})_{\bm{\tau_n},\emptyset} =(M_{\bm{n}})_{\emptyset,\bm{\tau_n}}:=q^n z({\omega_0})$  \\	                        	                           	
	                           	\hphantom{1cm} \ndv   \\			                           	                           	
	                           	\hphantom{1cm}   \printv $M_{\bm{n}}\succeq 0$   \\			
					\ndv\\
\text{\color{blue}//Now nonnegativity of all variables.}\\
		\foreachv $\omega \in \Omega_3$  \\
	    \hphantom{1cm} \printv $ z(\omega) \geq 0$\\
	    \ndv 
     \end{tabular}
    \end{minipage}    } \caption{\label{pseudocodelee}\small{Algorithm to generate semidefinite programs for computing~$B_3^L(q,n,1)$. To compute~$B_3^L(q,n,d)$ (or~$B_3^{L_{\infty}}(q,n,d)$ from Sect.~\ref{circbounds} below), one must set all variables~$z(\omega)$  with~$\omega \in \Omega_3$ an orbit corresponding to a code of minimum Lee (respectively, Lee\textsubscript{$\infty$}) distance~$<d$ to zero. 
If rows and columns of matrix blocks  consist solely of zeros after the replacement, one can remove these rows and columns.}}
\end{figure}

\chapter{The Shannon capacity of circular graphs}\label{shannonchap}

\chapquote{We may have knowledge of the past but cannot control it;\\ we may control the future but have no knowledge of it.}{Claude Shannon (1916--2001)}

\noindent In this chapter we consider the Shannon capacity~$\Theta(C_{d,q})$ of  circular graphs~$C_{d,q}$. The circular graph $C_{d,q}$ is the graph with vertex set~$\Z_q$ (the cyclic group of order~$q$) in which two distinct vertices are adjacent if and only if their distance (mod~$q$) is strictly less than~$d$. The value of~$\Theta(C_{d,q})$  can be seen to only depend on the quotient~$q/d$.  

We show that the function~$	q/d \mapsto \Theta(C_{d,q})$ is continuous at \emph{integer} points~$q/d \geq 3$. We also prove that the independent set achieving~$\alpha(C_{5,14}^{\boxtimes 3})=14$, one of the independent sets used in our proof, is unique up to Lee equivalence. Furthermore, we adapt the SDP bound of Chapter~\ref{leechap} to an upper bound on~$\alpha(C_{d,q}^{\boxtimes n})$  based on triples of codewords. Finally, we give a new lower bound  of~$367^{1/5}$  on the Shannon capacity of the~$7$-cycle. This chapter is based on joint work with Lex Schrijver, part of which can be found in~$\cite{shannonc7}$.

\section{Introduction}\label{shanchapintro}

Let~$G=(V,E)$ and~$H=(V',E')$ be graphs. Recall that the \emph{strong product} $G \boxtimes H$ is the  graph  with vertex set~$V \times V'$ in which two distinct vertices~$(u,u' )$ and~$(v,v')$ are adjacent if ($u=v$ or~$uv \in E$) and ($u'=v'$ or~$u'v' \in E'$)  (cf.\ Sect.~\ref{indprod}). For~$n \in \N$,   write
$$
G^{\boxtimes n}:= \underbrace{G  \boxtimes  G \boxtimes \ldots \boxtimes G}_{\text{$n$ times}}.
$$
So two distinct vertices~$(u_1,\ldots,u_n)$ and~$(v_1,\ldots,v_n)$ of~$G^{\boxtimes n}$ are adjacent if and only if for each~$i \in \{1,\ldots,n\}$ one has either~$u_i  = v_i$ or~$u_i v_i \in E$. Recall that an \emph{independent set} in~$G$ is a subset $U $ of~$V$ such that~$e \nsubseteq U$ for all~$e \in E$, and that~$\alpha(G)$ denotes the \emph{independent set number} of~$G$: the maximum cardinality of an independent set in~$G$. If~$U$ and~$U'$ are independent sets in~$G$ and~$H$, respectively, then~$U \times U'$ is independent in~$G \boxtimes H$. So~$\alpha(G\boxtimes H) \geq \alpha(G) \alpha(H)$.

The \emph{Shannon capacity} of~$G$ is defined as
\begin{align}
    \Theta(G):= \sup_{n \in \N} \sqrt[n]{\alpha(G^{\boxtimes n})}.
\end{align}
As~$\alpha(G^{\boxtimes(n_1+n_2)}) \geq \alpha(G^{\boxtimes n_1})\alpha(G^{\boxtimes n_2})$ for any two positive integers~$n_1$ and~$n_2$, Fekete's lemma (cf.~$\cite{fekete}$, see also \cite[Corollary 2.2a]{schrijverpolyhedra}) implies that $\Theta(G) = \lim_{n \to \infty} \sqrt[n]{\alpha(G^{\boxtimes n})}$. 

The Shannon capacity was introduced by Shannon~$\cite{shannon}$ and is an important and widely studied  parameter in information theory (see e.g.,~\cite{ alon,bohman, haemers,lovasz, zuiddam}). It is the effective size of an alphabet in an information channel represented by the graph~$G$. The input is a set of letters~$V(G)=\{0,\ldots,q-1\}$ and two letters are `confusable' when transmitted over the channel if and only if there is an edge between them in~$G$. Then~$\alpha(G)$ is the maximum size of a set of pairwise non-confusable single letters. Moreover,~$\alpha(G^{\boxtimes n})$ is the maximum size of a set of pairwise non-confusable~$n$-letter words. (Here two $n$-letter words $u,v$ are confusable if and only if~$u_i$ and~$v_i$ are confusable for all~$ 1 \leq i \leq n$.) Taking~$n$-th roots and letting~$n$ go to infinity, we find the effective size of the alphabet in the information channel:~$\Theta(G)$. 

A classical upper bound on~$\Theta$ is Lov\'asz's $\vartheta$-function~\cite{lovasz}.  For any graph~$G$, the number~$\vartheta(G)$ is defined as\symlistsort{theta(G)}{$\vartheta(G)$}{Lov\'asz's theta number of graph~$G$}\indexadd{Lov\'asz's theta number}
\begin{align}\label{varthetadef}
\vartheta(G) := \max\left\{ \mbox{$\sum_{u,v \in V} X_{u,v}$} \,\,  \big| \,\,  X \in \R^{V \times V},  \,\,\, \trace(X)=1,  \text{ $X_{u,v}=0$ if~$uv \in E$},\,\,\, X\succeq 0 \right\}.
\end{align}
The Shannon capacity of~$C_5$, the cycle on~$5$ vertices, was already discussed by Shannon in 1956~\cite{shannon}, who showed~$\sqrt{5} \leq \Theta(C_5) \leq 5/2$. The  easy lower bound of~$\sqrt{5}$ is obtained from the independent set $\{(0,0),(1,2),(2,4),(3,1),(4,3)\}$ in~$C_5^{\boxtimes 2}$. More than twenty years later, Lov\'asz~$\cite{lovasz}$ determined that~$\Theta(C_5)=\sqrt{5}$ by proving an upper bound matching this lower bound: he showed~$\Theta(C_5) \leq \vartheta(C_5)= \sqrt{5}$.  More generally, for odd~$q$,
\begin{align}
    \Theta(C_q) \leq \vartheta(C_q) = \frac{q\cos(\pi/q)}{1+\cos(\pi/q)}.
\end{align}
For~$q$ even it is not hard to see that~$\Theta(C_q)=q/2$. 

The Shannon capacity of~$C_7$ is still unknown and its determination is a notorious open problem in extremal combinatorics~\cite{bohman, godsil}. Many lower bounds have been given by explicit independent sets in some fixed power of~$C_7$ \cite{baumert, matos, veszer}, while the best known upper bound is~$\Theta(C_7)\leq \vartheta(C_7) < 3.3177$. We give an independent set of size~$367$ in~$C_7^{\boxtimes 5}$, which yields~$\Theta(C_7)\geq 367^{1/5} > 3.2578$. The best previously known lower bound on~$\Theta(C_7)$ is~$\Theta(C_7) \geq 350^{1/5} > 3.2271$, found by Mathew and \"Osterg{\aa}rd~$\cite{matos}$. They proved that $\alpha(C_7^{\boxtimes 5}) \geq 350$ using stochastic search methods that utilize the symmetry of the problem. We obtain our new lower bound on~$\alpha(C_7^{\boxtimes 5})$ and on~$\Theta(C_7)$ by considering circular graphs.

For positive integers~$q,d$ with~$q \geq 2d$, recall that the \emph{circular graph} $C_{d,q}$ is the graph with vertex set~$\Z_q$ (the cyclic group of order~$q$) in which two distinct vertices are adjacent if and only if their distance (mod~$q$) is strictly less than~$d$.  So~$C_{2,q} = C_q$.   The circular graphs have the property that~$\alpha(C_{d,q}^{\boxtimes n})$ (for fixed~$n$), $\vartheta(C_{d,q})$ and~$\Theta(C_{d,q})$ only depend on the fraction~$q/d$. Moreover, the three mentioned quantities are nondecreasing functions in~$q/d$ (see Section~\ref{shannonprem} for details). Lov\'asz's sandwich theorem~\cite{lovasz} (see also~\cite{sandwich}) implies that
\begin{align}\label{sandwich}
\Theta(C_{d,q}) \leq \vartheta(C_{d,q}) \leq q/d \quad \quad \text{for all } q,d \in \N \text{ with } q \geq 2d.
\end{align}
Here~$q/d$ is equal to the `fractional clique cover number' of~$C_{d,q}$, which is an upper bound on~$\vartheta(C_{d,q})$ cf.~\cite{lovasz} (see Section~\ref{shannonprem}).  
Currently, the only known values of~$\Theta(C_{d,q})$ for~$q\geq 2d$ are for~$q/d$ integer and~$q/d=5/2$. 
\begin{figure}[ht]
\centering
\begin{tikzpicture}[scale=1.575]
    \pgfmathsetlengthmacro\MajorTickLength{
      \pgfkeysvalueof{/pgfplots/major tick length} * 0.65
    }
    \begin{axis}[enlargelimits=false,axis on top,xlabel ={\scriptsize\color{red}$q/d$}, ylabel = {\scriptsize\color{red}$\vartheta(C_{d,q})$}, 
                 xtick={2,2.5,3,3.5,4,4.5,5},ytick={2,2.5,3,3.5,4,4.5,5},
  major tick length=\MajorTickLength,
        x label style={
        at={(axis description cs:0.5,-0.05)},
        anchor=north,
      },
      y label style={
        at={(axis description cs:-0.08,.5)}, 
        anchor=south,
      }, 
                ]
                                
       \addplot graphics
       [xmin=2,xmax=5,ymin=2,ymax=5,
      includegraphics={trim=5cmm 10.5cm 4.5cm 10.105cm,clip}]{lexplot2-5_step5000box_noaxis.pdf};
    \end{axis}
\end{tikzpicture}
\caption{\small A graph of the function~$q/d \mapsto \vartheta(C_{d,q})$.\label{thetafigchap}}
\end{figure}
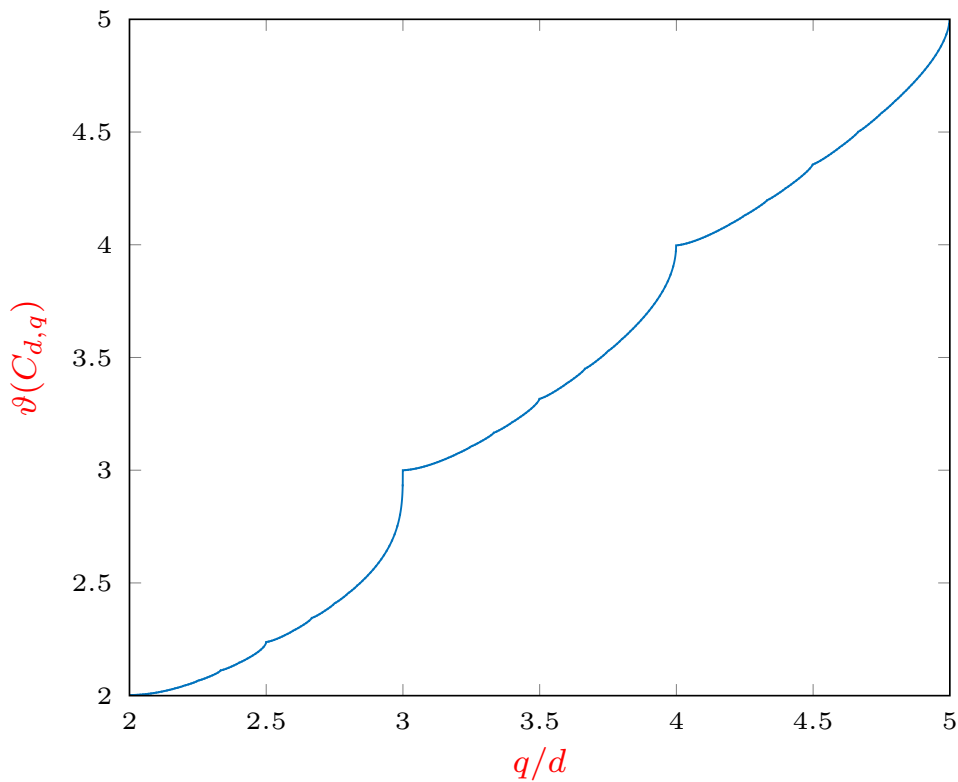

The main result of this chapter is that the function~$q/d \mapsto \Theta(C_{d,q})$ is continuous  at \emph{integers}~$q/d=r \geq 3$, which implies that also~$q/d \mapsto \vartheta(C_{d,q})$ is continuous at these points  (by~\eqref{sandwich} and by the fact that~$\Theta(C_{d,q})=q/d$ if~$q/d$ is integer). 
Right-continuity of the two mentioned functions at integers~$q/d=r \geq 2$ follows immediately from~\eqref{sandwich} in combination with the fact that~$\Theta(C_{d,q})=r$ if~$q/d=r $ is an integer and the fact that~$q/d \mapsto \Theta(C_{d,q})$ is a nondecreasing function in~$q/d$.   In order to prove left-continuity at integers~$q/d=r \geq 3$, we show the following.
 \begin{theoremnn}[Theorem \ref{shmaxth}]
For each~$r,n\in \N$ with~$r \geq 3$, 
$$
\max_{\frac{q}{d} < r } \alpha(C_{d,q}^{\boxtimes n}) = \frac{1+r^n(r-2)}{r-1}.
$$
\end{theoremnn}
The theorem will be used to prove that $\lim_{\frac{q}{d} \uparrow r} \Theta(C_{d,q})  = r$. As~$\Theta(C_{d,q})=r$ if~$q/d=r$, it follows that~$q/d \mapsto \Theta(C_{d,q})$ is left-continuous at~$r$  (and hence also~$q/d \mapsto \vartheta(C_{d,q})$, by~\eqref{sandwich}).  Continuity of~$q/d \mapsto \Theta(C_{d,q})$ at \emph{non}-integer points~$q/d$ seems hard to prove and even continuity of~$q/d \mapsto \vartheta(C_{d,q})$ at these points is not yet known, although the latter function \emph{looks} continuous (cf.\ Figure~\ref{thetafigchap}). 
We also prove that the independent set achieving~$\alpha(C_{5,14}^{\boxtimes 3})=14$, one of the independent sets used in our proof  of Theorem~\ref{shmaxth}, is unique up to Lee equivalence.

In Section~\ref{circbounds}, we show how to adapt the bound  $B_3^{L}(q,n,d)$  from Chapter~\ref{leechap} to an upper bound $B_3^{L_{\infty}}(q,n,d)$ on~$\alpha(C_{d,q}^{\boxtimes n})$. The new bound does not seem to improve significantly over the bound obtained from Lov\'asz theta-function,  except for very small~$n$. 

 Finally, we give  the mentioned independent set of size~$367$ in~$C_7^{\boxtimes 5}$, demonstrating the new lower bound on~$\Theta(C_7)$. It is obtained by adapting an independent set of size~$382$ in~$C_{108,382}^{\boxtimes 5}$ which we found by computer, inspired by the construction used to prove Theorem~\ref{shmaxth}.

\section{Preliminaries on the Shannon capacity and circular graphs}\label{shannonprem}

In this section we give the preliminaries on the Shannon capacity and on circular graphs used throughout this chapter. 

\paragraph{Graph homomorphisms and circular graphs.}

A homomorphism from a graph $G_1 = (V_1,E_1)$ to a graph $G_2 = (V_2,E_2)$ is a function $f\, :\, V_1 \to V_2$ such that if~$ij \in E_1$
then~$f(i)f(j) \in E_2$ (in particular,~$f(i) \neq f(j)$).\indexadd{homomorphism} If there exists a homomorphism~$f \, : \, G_1 \to G_2$ we write~$G_1 \to G_2$.\symlistsort{G1toG2}{$G_1 \to G_2$}{there exists a graph homomorphism $G_1 \to G_2$} For any graph~$G$, we write~$\overline{G}$ for the complement of~$G$.\indexadd{graph!complement}\symlistsort{G overline}{$\overline{G}$}{complement graph of graph~$G$} If~$\overline{G} \to \overline{H}$, then~$\alpha(G^{\boxtimes n}) \leq \alpha(H^{\boxtimes n})$ (for any~$n$) and~$\Theta(G) \leq \Theta(H)$. The circular graphs have the property that~$\overline{C_{d',q'}} \to \overline{C_{d,q}}$ if and only if~$q'/d' \leq q/d$~\cite{starnote}. (The ``if''-part follows from the following three facts, which are not hard to prove. Let~$q,d \in \N$ with~$q \geq 2d$, and let~$t \in \N$. Then~$\overline{C_{td,tq}} \to \overline{C_{d,q}}$, $\overline{C_{d,q}} \to \overline{C_{td,tq}}$ and~$\overline{C_{d,q}} \to \overline{C_{d,q+1}}$.)  

 So if~$q'/d' \leq q/d$, then~$\alpha(C_{d',q'}^{\boxtimes n}) \leq \alpha(C_{d,q}^{\boxtimes n})$ (for any~$n$) and~$\Theta(C_{d',q'}) \leq \Theta(C_{d,q})$. Moreover,~$\alpha(C_{d,q}^{\boxtimes n})$ and~$\Theta(C_{d,q})$ only depend on the fraction~$q/d$ (fixing~$n$).

\paragraph{Zuiddam's dual characterization of~$\Theta$.}

Let~$\mathcal{G}$ denote the set of all finite, simple\footnote{A graph is \emph{simple} if it has no loops or parallel edges.} graphs.\symlistsort{G mathcal}{$\mathcal{G}$}{set of all finite, simple graphs} Let~$\Delta$ be the set of all functions~$f \, : \, \mathcal{G} \to \R_{\geq 0}$ satisfying for all~$G,H \in \mathcal{G}$:\symlistsort{Delta}{$\Delta$}{set of all functions~$\mathcal{G}\to \R_{\geq 0}$ satisfying certain properties}
\begin{speciaalenumerate} 
\item \label{d1} $f( G \sqcup H ) = f(G) + f(H)$,
\item \label{d2} $ f(G \boxtimes H) = f(G)f(H)$,
\item \label{d3} If~$\overline{G} \to \overline{H}$, then~$f(G) \leq f(H)  $,
\item \label{d4} $f(K_1) = 1$. 
\end{speciaalenumerate}
Here~$G \sqcup H$ is the graph with vertex set~$V(G) \sqcup V(H)$ and edge set~$E(G) \sqcup E(H)$, where~$\sqcup$ denotes disjoint union.  Note that~$\Delta$ is compact in the Tychonoff product space~$\R^{\mathcal{G}}$, since~$\eqref{d1}-\eqref{d4}$ are closed conditions and~$0 \leq f(G) \leq |V(G)|$ for all~$G \in \mathcal{G}$, $f \in \Delta$.  


If~$f \in \Delta$ and~$G \in \mathcal{G}$, then~$\Theta(G) \leq f(G)$. This is not hard to see, by~\eqref{d4} and \eqref{d1} we have~$\alpha(G) =f(\overline{K_{\alpha(G)}}) $, so~\eqref{d3} gives~$\alpha(G) =f(\overline{K_{\alpha(G)}}) \leq f(G)$ (using that $ K_{\alpha(G)} \to \overline{G}$), and~\eqref{d2} then gives~$\Theta(G) \leq f(G)$. Jeroen Zuiddam~\cite{zuiddam} proved that for any graph~$G$,
\begin{align}
\Theta(G) = \min_{f \in \Delta} f(G).
\end{align}
Zuiddam derived his theorem from Strassen's semiring theorem~\cite{strassen}, using that~$(\mathcal{G},\sqcup,\boxtimes)$ is a commutative semiring with unit~$K_1$ and with the empty graph as zero element. It is known that $\Theta \notin \Delta$, see Haemers~\cite{haemers}:  there exists a~$G \in \mathcal{G}$ (\emph{the Schl\"afli graph}) satisfying $\Theta(G \boxtimes \overline{G} ) \geq 27 > 3 \cdot 7 \geq \Theta(G)\Theta(\overline{G})$.  

One function contained in~$\Delta$ is Lov\'asz $\vartheta$-function~\cite{lovasz}, which is defined in~\eqref{varthetadef}.
There is a closed form formula for~$\vartheta(C_{d,q})$, proved by Bachoc,  P\^{e}cher and Thi\'ery~\cite{bachoc}: 
\begin{align}\label{bachocform}
\vartheta(C_{d,q}) = \frac{q}{d}\sum_{i=0}^{d-1} \prod_{j=1}^{d-1} \frac{\cos\left(\frac{2i\pi}{d}\right)-\cos\left(\floor{\frac{qj}{d}}\frac{2\pi}{q}\right)}{1-\cos\left(\floor{\frac{qj}{d}}\frac{2\pi}{q}\right)},
\end{align}
Another element of~$\Delta$ is the fractional clique cover number, which was already observed to be an upper bound on~$\Theta$ by Shannon~\cite{shannon}. Let~$G \in \mathcal{G}$, and write~$\mathcal{K}(G)$ for the collection of all \emph{cliques} in~$G$ (a \emph{clique} is a set of vertices, every two of which are adjacent). Then the fractional clique cover number $\overline{\chi}^*(G)$ admits the following definition:\symlistsort{gistarbar(G)}{$\overline{\chi}^*(G)$}{fractional clique cover number of graph~$G$}\indexadd{fractional clique cover number}\symlistsort{K(G)}{$\mathcal{K}(G)$}{the collection of all cliques in~$G$}
\begin{align}
\overline{\chi}^*(G):= \min \left\{ \mbox{$\sum_{K \in \mathcal{K}(G)}  y(K)$} \,\,  \big| \,\,  y \, : \, \mathcal{K}(G) \to \R_{\geq 0}, \,\, \forall \, v \in V :    \mbox{$\sum_{ K \ni v }  y(K) \geq 1 $} \right\}.
\end{align}
It is known that~$\overline{\chi}^*(G) \in \Delta$ (see \cite{schrijverpolyhedra}). It is also known (cf.~\cite{mcchrom}, see also~\cite[Theorem 67.17]{schrijverpolyhedra}) that 
\begin{align}\label{chistreepth}
\overline{\chi}^*(G) = \inf_{n \in \N} \sqrt[n]{\overline{\chi}(G^{\boxtimes n})}= \lim_{n \to \infty} \sqrt[n]{\overline{\chi}(G^{\boxtimes n})}.
\end{align}
 Here~$\overline{\chi}(G)$ denotes the minimum number of cliques in~$G$ needed to cover the vertex set of~$G$, the \emph{clique cover number} of~$G$.\indexadd{clique cover number} It is equal to  $\min\{r \in \N \, : \, \overline{G} \to K_r\}$.  If~$f \in \Delta$, we have by~\eqref{d4}, \eqref{d1},~\eqref{d3} that~$ f(G) \leq \overline{\chi}(G)$, and~\eqref{d2} and~\eqref{chistreepth} then give~$ f(G) \leq \overline{\chi}^*(G)  $. So
\begin{align}
\overline{\chi}^*(G)= \max_{f \in \Delta} f(G).
\end{align}
 Since the circular graph~$C_{d,q}$ is vertex-transitive\footnote{A graph~$G=(V,E)$ is \emph{vertex-transitive} if for all~$u,v \in V$ there exists an automorphism~$g: V \to V$ of~$G$ such that~$g(u)=v$.} and every vertex is contained in a clique of size~$d$ (but no vertex is contained in a larger clique), one can see that~$\overline{\chi}^*(C_{d,q}) = q/d$.   So if~$f \in \Delta$, then 
  \begin{align} \label{sandwichtotal}
  &\Theta(C_{d,q}) \leq f(C_{d,q}) \leq \overline{\chi}^*(G) = q/d \,\,\,\,\,\,\, \text{ for all~$q,d \in \N$ with~$q \geq 2d$, } \notag 
  \\ &\text{with equality throughout if~$q/d$ is an integer}.  
\end{align}
 Note that~\eqref{sandwich} is a special case of~\eqref{sandwichtotal}.

 Other known elements of~$\Delta$, which we do not discuss in detail here (see~\cite{zuiddam,zuiddamthesis}), are the \emph{fractional orthogonal rank} (also called \emph{projective rank})~\cite{fracor} and the \emph{fractional Haemers bound}~\cite{blasiak, bukh}. The latter parameter can be defined for any field~$\F$, and a separation result of Bukh and Cox~\cite{bukh} shows that different fields~$\F$ yield different parameters. So~$\Delta$ contains infinitely many elements.
 
\paragraph{The Lee\textsubscript{$\infty$} distance.}

By definition there is an edge between two distinct vertices~$x,y$ of~$C_{d,q}^{\boxtimes n}$ if and only if there is an~$1 \leq i \leq n $ such that~$x_i-y_i \pmod{q}$ is either strictly smaller than~$d$ or strictly larger than~$q-d$.  For distinct~$u,v$ in~$\Z_q^n$, define their  \emph{Lee\textsubscript{$\infty$} distance} to be the maximum over the distances of~$u_i$ and~$v_i$ (mod~$q$), where~$i$ ranges from~$1$ to~$n$. The \emph{minimum Lee\textsubscript{$\infty$} distance}~$d_{\text{min}}^{L_{\infty}} (D)$ of a set~$D \subseteq \Z_q^n$ is the minimum distance  between any pair of distinct elements of~$D$. If~$|D|\leq 1$, set~$d_{\text{min}}^{L_{\infty}}(D)=\infty$. Then~$d_{\text{min}}^{L_{\infty}} \geq k$ if and only if~$D$ is independent in~$C_{d,q}^{\boxtimes n}$. So 
\begin{align} \label{strongalphaprodcdqchap}
A_q^{L_{\infty}}(n,d):=\alpha(C_{d,q}^{\boxtimes n})= \max \{ |C| \, \, | \,\, C \subseteq \Z_q^n, \,\, d_{\text{min}}^{L_{\infty}}(C) \geq d \}.
\end{align}
 It is well known that~$\alpha(C_{d,q}^{\boxtimes n})$ is equal to the number of $n$-dimensional hypercubes of side~$d$ that can be packed in a discrete~$n$-dimensional torus of width~$q$~\cite{baumert}.  


\section{Continuity of~\texorpdfstring{$q/d \mapsto \Theta(C_{d,q})$}{q/d to Theta(Cd,q)} at~\texorpdfstring{$q/d \in \N$}{integers}}
 
  Let~$C(\Delta,\infty)$ denote the Banach algebra of continuous functions~$\Delta \to \R$, equipped with the~$\norm{.}_{\infty}$-norm.\symlistsort{CDelta,infty}{$C(\Delta,\infty)$}{Banach algebra of continuous functions~$\Delta \to \R$ with the~$\norm{.}_{\infty}$-norm} Then the evaluation map\symlistsort{ev}{$\ev$}{evaluation map}\indexadd{evaluation map}
 \begin{align}\label{ev}
 \ev:  \mathcal{G} &\to \,\,C(\Delta,\infty) \notag \\
                       G &\mapsto \,\, (f \mapsto f(G))
 \end{align}
 embeds~$\mathcal{G}$ into~$C(\Delta,\infty)$. A natural question to ask is whether the image of $\ev$ is closed.  Since if not, the closure contains new `graph like' objects not being graphs.
 
 Natural candidates to investigate this question are the circular graphs~$C_{d,q}$. By property~\eqref{d3} of the functions in~$\Delta$, and by the fact that~$\overline{C_{d',q'}} \to \overline{C_{d,q}}$ if and only if~$q'/d' \leq q/d$, the image~$\ev(C_{d,q})$ only depends on the quotient~$q/d$. The fractional clique cover number~$\overline{\chi}^*(G)  $ of any~$G \in \mathcal{G}$ is a rational number, since it is the solution of a linear program with all input data rational~\cite{schrijverlp}.  Suppose that~$q/d \mapsto \ev(C_{d,q})$ is a continuous function in each~$r \in \Q$ with~$r\geq 2$. If~$2 < s \in \R\setminus \Q$ with~$q/d\to s $, then the limit of~$(\ev(C_{d,q}))$  is not the image~$\ev(G)$ of any~$G \in \mathcal{G}$, as there is no~$G \in \mathcal{G}$ with~$\overline{\chi}^*(G)=s$. So in this case the image of~$\ev$ is not closed.
 
 So we aimed at proving continuity of~$q/d \mapsto \ev(C_{d,q})$ in each~$r \in \Q$ with~$r\geq 2$. However, we could not prove this. But for \emph{integer} points~$r \geq 3$, the function~$q/d \mapsto \ev(C_{d,q})$ is continuous at $q/d=r$. Right-continuity at integers~$q/d =r  \geq 2$ follows from~\eqref{sandwichtotal} and the fact that~$q/d \mapsto f(C_{d,q})$ is a nondecreasing function in~$q/d$ for each~$f \in \Delta$. To prove left-continuity of~$q/d \mapsto \ev(C_{d,q})$ at integers~$q/d =r  \geq 3$, it suffices (again by~\eqref{sandwichtotal}) to prove that~$q/d \mapsto \Theta(C_{d,q})$ is left-continuous at these points.
 
 Let~$r \geq 3$ be an integer. We prove that if~$q/d$ tends to~$r$ from below, then~$\Theta(C_{d,q}) $ tends to~$ r$ by proving the following:

 \begin{theorem} \label{shmaxth}
For each~$r,n\in \N$ with~$r \geq 3$,
 \begin{align} \label{shmax}
\max_{\frac{q}{d} < r } \alpha(C_{d,q}^{\boxtimes n}) = \frac{1+r^n(r-2)}{r-1}.
\end{align}
 \end{theorem}
The theorem implies that
\begin{align*}
\sup_{\frac{q}{d} < r} \Theta(C_{d,q}) &= \sup_{\frac{q}{d} < r} \sup_{n\in \N} \sqrt[n]{\alpha(C_{d,q}^{\boxtimes n}) } =\sup_{n \in \N}  \sup_{\frac{q}{d} < r} \sqrt[n]{\alpha(C_{d,q}^{\boxtimes n}) } = \lim_{n \to \infty} \left(\frac{1+r^n(r-2)}{r-1}\right)^{1/n} = r,
\end{align*} 
so the function~$q/d \mapsto \Theta(C_{d,q})$ is left-continuous at~$r$ (using the fact that the function~$q/d \mapsto  \Theta(C_{d,q})$ is nondecreasing). Right-continuity  of this function at integers~$\geq 2$ follows immediately from~\eqref{sandwichtotal} and the fact that~$q/d \mapsto \Theta(C_{d,q})$ is a nondecreasing function in~$q/d$.

\subsection{Proof of Theorem~\ref{shmaxth}}

In this section we prove Theorem~\ref{shmaxth}. Throughout this section, fix an integer~$r \geq 3$. Furthermore, we define for~$n \in \Z_{\geq 0}$, the number\symlistsort{qn}{$q_n$}{number defined in~(\ref{qndef})}
\begin{align} \label{qndef}
q_n:=\frac{1+r^n(r-2)}{r-1}.
\end{align}
First we prove:
\begin{proposition}\label{shanub}
For all~$n \in \N$: $\max_{\frac{q}{d} < r } \alpha(C_{d,q}^{\boxtimes n}) \leq q_n $.
\end{proposition}
\proof
Choose~$q,d \in \N$ with~$q \geq 2d$ and~$q/d <r$. For each~$n \in \N$, define~$\alpha_n :=  \alpha(C_{d,q}^{\boxtimes n}) $, and set~$\alpha_0:=1$.\symlistsort{alphan}{$\alpha_n$}{number defined and used in the proof of Prop.~\ref{shanub}}  
Let~$A \subseteq \Z_q^n$ such that~$A$ is independent in~$C_{d,q}^{\boxtimes n}$ with~$|A|=\alpha_n$. We interprete~$A$ as an~$\alpha_n \times n$ matrix with the words (elements of~$\Z_q^n$) as rows. For~$i \in \Z_{q}$, let~$c_i$ be the number of times symbol~$i$ occurs in column~$1$. Then~$\sum_{i=t+1}^{t+d} c_i \leq \alpha_{n-1} $ for any~$t \in \Z_{q}$, where we take indices of~$c$ mod~$q$. So
\begin{align}
d \alpha_n= d \sum_{i=0}^{q-1} c_i = \sum_{t=0}^{q-1}\sum_{i=t+1}^{t+d} c_i \leq q \alpha_{n-1}.
\end{align}
This implies that~$ \alpha_n \leq \floor{ q \alpha_{n-1} /d} \leq r \alpha_{n-1}-1$. By induction we conclude
\begin{align*}
\alpha_n &\leq  r \alpha_{n-1}-1 \leq  r q_{n-1} -1 = q_n.   \qedhere
\end{align*}
\endproof 
To prove Theorem~\ref{shmaxth}, it remains to show that for all~$n \in \N$, 
\begin{align}\label{maxgeq}
\max_{\frac{q}{d} < r } \alpha(C_{d,q}^{\boxtimes n}) \geq q_n.
\end{align}
Note that
$$
\frac{q_n}{q_{n-1}} = \frac{1+r^n(r-2)}{1+r^{n-1}(r-2)} = r \left(  \frac{1+r^n(r-2)}{r+r^{n}(r-2)} \right) < r. 
$$
So in order to prove~\eqref{maxgeq}, it suffices to give an explicit independent set of size~$q_n$ in~$C_{q_{n-1},q_n}^{\boxtimes n}$.  In this chapter we use the notation~$[a,b]:=\{a,a+1,\ldots,b\}$ for any two integers~$a,b$.\symlistsort{a,b}{$[a,b]$}{the set $\{a,a+1,\ldots,b\}$}
\begin{theorem}\label{explindth}
The set $S:= \{ t \cdot (1,r,\ldots, r^{n-1}) \, | \, t \in \Z_{q_n} \} \subseteq \Z_{q_n}^n $ is independent in~$C_{q_{n-1},q_n}^{\boxtimes n}$.
\end{theorem} 
\begin{proof}
Note that~$x,y \in S$ implies that also~$x-y \in S$. So in order to prove the theorem, it suffices to prove that for all nonzero~$x =t \cdot (1,r,\ldots, r^{n-1}) \in S$ there exists~$i \in [1,n]$ such that~$x_i \in [q_{n-1}, q_n - q_{n-1}] = [q_{n-1},(r-1) q_{n-1}-1 ]$.  So we must prove:
\begin{align} \label{tqi}
&\text{for all~$t \in [1,q_n-1]$ there exists~$i \in [0,n-1]$ such that} \notag \\
&tr^i  \pmod{q_n} \in [q_{n-1},(r-1) q_{n-1}-1 ].
\end{align}
  In order to prove \eqref{tqi} we first prove:
 \begin{proposition} \label{hulpshanprop}
Let~$k \in [1,n]$,~$j \in [0,(q_{n-k}-1)/(r-2)]$. If~$t \in [q_{k-1}+j(r-2)r^k,(r-1)q_{k-1}-1+j(r-2)r^k]$, then
\begin{align} 
(r-1)jq_{n} \in[ A(t),B(t)]:= [tr^{n-k}-(r-1)q_{n-1}+1 , tr^{n-k}-q_{n-1}]. \notag 
\end{align} 
In particular, $tr^{n-k}  \pmod{q_n} \in [q_{n-1},(r-1)q_{n-1}-1 ]$.
\end{proposition} 
\begin{proof}
Let~$t_1=q_{k-1}+j(r-2)r^k$ and~$t_2=(r-1)q_{k-1}-1+j(r-2)r^k$. So~$t_1\leq t \leq t_2$. Note that
\begin{align*} 
    B(t) \geq B(t_1) &=(q_{k-1}+j(r-2)r^k) r^{n-k}-q_{n-1} =\frac{r^{n-k}-1}{r-1}  +j(r-2)r^n
    \\&\geq j+j(r-2)r^n =(r-1)jq_n ,
\end{align*}
as~$j \leq(q_{n-k}-1)/(r-2) = (r^{n-k}-1)/(r-1)$. Furthermore,
\begin{align*}
   A(t)\leq  A(t_2) &= ((r-1)q_{k-1}-1+j(r-2)r^k) r^{n-k}-(r-1)q_{n-1}+1 
    \\ &= j(r-2)r^n  \leq j + j(r-2)r^n=(r-1)jq_{n}.
    \end{align*}
    Hence~$(r-1)jq_n \in [A(t),B(t)]$. 
\end{proof}
We now continue the proof of~\eqref{tqi}, hence of Theorem~\ref{explindth}.  Let~$t \in [1,q_n-1]$. To finish the proof of~\eqref{tqi}, we show that 
\begin{align} \label{nieuwalign}
&\text{there exists~$k \in [1,n]$ and~$j \in [0,(q_{n-k}-1)/(r-2)]$} \notag 
\\ &\text{such that~$t \in [q_{k-1}+j(r-2)r^k,(r-1)q_{k-1}-1+j(r-2)r^k]$}.
\end{align} 
Assuming~\eqref{nieuwalign}, Proposition~\ref{hulpshanprop} implies that~$tr^{n-k}  \pmod{q_n} \in [q_{n-1},(r-1)q_{n-1}-1 ]$, hence~$(\ref{tqi})$ is satisfied with~$i=n-k$, as desired.
So it remains to prove~(\ref{nieuwalign}) for each~$t \in [1,q_n-1]$. To do that, we must show:
\begin{align} \label{tebewijzen2}
    \bigcup_{k \in [1,n]} \bigcup_{j \in [0,(q_{n-k}-1)/(r-2)]} [q_{k-1}+j(r-2)r^k,(r-1)q_{k-1}-1+j(r-2)r^k] \notag \\\supseteq [1,q_n-1].
\end{align}
We use induction on~$n$. If~$n=0$ then the assertion is clear. Assume that~$\eqref{tebewijzen2}$ holds for~$n-1$ instead of~$n$, i.e., 
\begin{align} \label{IHstableset}
    \bigcup_{k \in [1,n-1]} \bigcup_{j \in [0,(q_{n-1-k}-1)/(r-2)]} \notag  [q_{k-1}+j(r-2)r^k,(r-1)q_{k-1}-1+j(r-2)r^k] \\\supseteq [1,q_{n-1}-1].
\end{align} 
Then adding $(r-2)r^{n-1}$ to both sides in~$(\ref{IHstableset})$ gives 
\begin{align}\label{tussenstap} 
  &\bigcup_{k \in [1,n-1]} \bigcup_{j \in [0,(q_{n-1-k}-1)/(r-2)]} [q_{k-1}+j(r-2)r^{k}+(r-2)r^{n-1}, \notag 
  \\ & \phantom{===========,====} (r-1)q_{k-1}-1+j(r-2)r^k+(r-2)r^{n-1}] \notag 
  \\& \supseteq [(r-2)r^{n-1}+1,(r-2)r^{n-1}+q_{n-1}-1]  =[(r-1)q_{n-1},q_{n}-1].
\end{align}
Note that~$j(r-2)r^k+(r-2)r^{n-1}=(j+r^{n-k-1})(r-2)r^k$.  Moreover, for~$j \in [0,(q_{n-1-k}-1)/(r-2)]$ one has
$$
j+r^{n-k-1} \leq (q_{n-1-k}-1)/(r-2) +r^{n-k-1}=(q_{n-k}-1)/(r-2).
$$
So~$\eqref{tussenstap}$ becomes (by replacing~$j+r^{n-k-1}$ with~$j$) 
\begin{align}\label{tussenstap2} 
    \bigcup_{k \in [1,n-1]} \bigcup_{j \in [r^{n-k-1},(q_{n-k}-1)/(r-2)]} [q_{k-1}+j(r-2)r^k,(r-1)q_{k-1}-1+j(r-2)r^k] \notag \\ \supseteq [(r-1)q_{n-1},q_{n}-1],
\end{align}
so $[(r-1)q_{n-1},q_n -1]$ is contained in the left hand side of~$(\ref{tebewijzen2})$. 

Also, $[q_{n-1},(r-1)q_{n-1}-1]$ is contained in the left hand side of~$(\ref{tebewijzen2})$ (for~$k=n$ and~$j=0$). Hence
$$
[1,q_{n-1}-1] \cup [q_{n-1},(r-1)q_{n-1}-1] \cup [(r-1)q_{n-1},q_n-1] = [1,q_{n}-1]
$$
is contained in the left hand side of~$(\ref{tebewijzen2})$, which concludes the proof of~\eqref{tebewijzen2}, hence of~\eqref{nieuwalign}, hence of~$\eqref{tqi}$, hence of Theorem~\ref{explindth}.
\end{proof} 
Theorem~\ref{explindth} implies the lower bound~$\Theta(C_{q_{n-1},q_n}) \geq \sqrt[n]{q_n}$.

Theorem~\ref{shmaxth} is now proved, since it follows from Theorem~\ref{explindth} together with Proposition~\ref{shanub}. 

\subsection{Further remarks}

The structure of the explicit independent set in Theorem~\ref{explindth} gives rise to the following natural question.\symlistsort{dq,n,r}{$d(q,n,r)$}{number defined in~(\ref{question})}
\begin{align} \label{question}
  \text{Given~$q,n,r \in \N$, what is~$d(q,n,r):= d_{\text{min}}^{L_{\infty}}( \{ t \cdot (1,r,\ldots, r^{n-1}) \,\, | \,\, t \in \Z_{q} \})$?}
 \end{align} 
 Here   $\{ t \cdot (1,r,\ldots, r^{n-1}) \,\, | \,\, t \in \Z_{q} \}$ is considered as subset of $  \Z_q^n$. So
 \begin{align}
 \alpha(C_{d(q,n,r),q}^{\boxtimes n}) \geq q.
 \end{align}

\begin{figure}[ht] \vspace{-.9em}
\centering
\begin{tikzpicture}[scale=1.39]
    \pgfmathsetlengthmacro\MajorTickLength{
      \pgfkeysvalueof{/pgfplots/major tick length} * 0.65
    }
    \begin{axis}[enlargelimits=false,axis on top,xlabel ={\footnotesize\color{red}$q/d$}, ylabel = {\footnotesize\color{red}$\vartheta(C_{d,q})$}, 
                 xtick={2.4,2.5,2.6,2.7,2.8,2.9,3,3.1},ytick={2.1,2.2,2.3,2.4,2.5,2.6,2.7,2.8,2.9,3,3.1},
  major tick length=\MajorTickLength,
        x label style={
        at={(axis description cs:0.5,-0.05)},
        anchor=north,
      },
      y label style={
        at={(axis description cs:-0.0775,.5)}, 
        anchor=south,
      }, 
                ]
                                
       \addplot graphics
       [xmin=2.4,xmax=3.1,ymin=2.1,ymax=3.1,
      includegraphics={trim=5cmm 10.5cm 4.5cm 10.105cm,clip}]{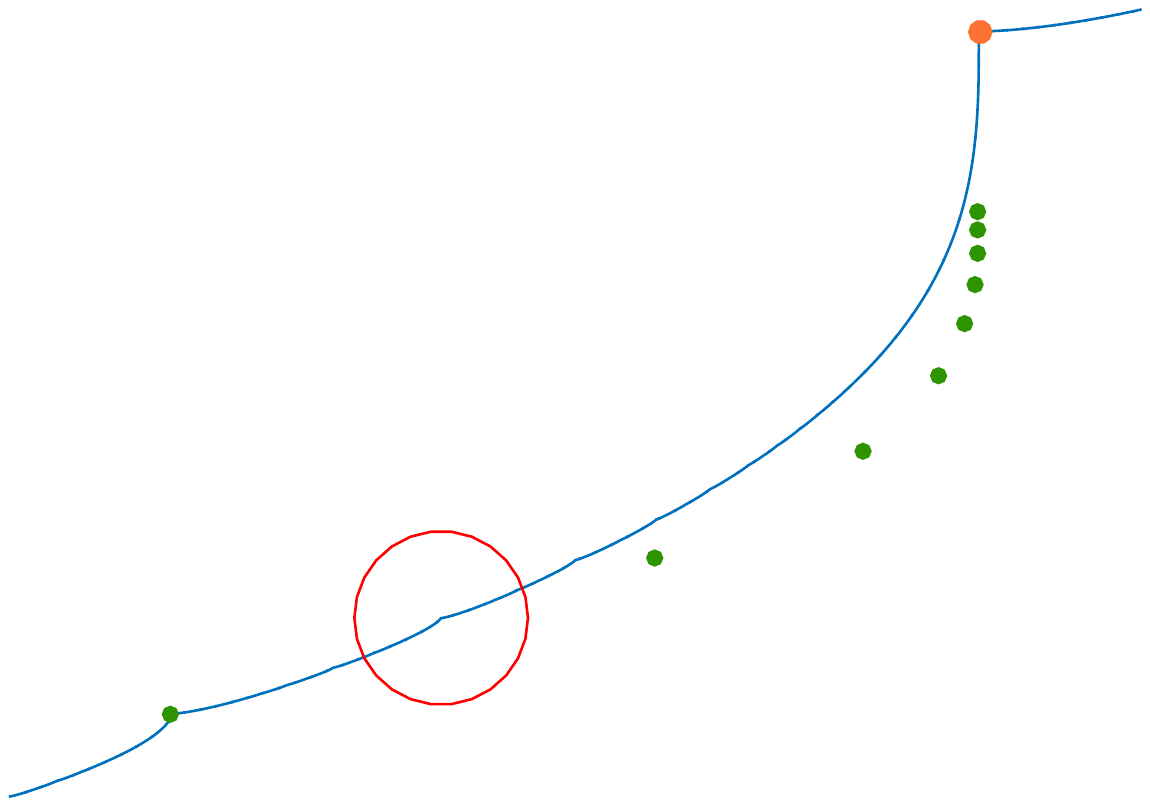};
    \end{axis}
\end{tikzpicture}
\caption{\small A graph of the function~$q/d \mapsto \vartheta(C_{d,q})$ for~$2.4 \leq q/d \leq 3.1$. The function seems nondifferentiable at~$q/d=5/2$ and~$q/d=3$, and also at~$q/d=8/3$ (marked with a red circle). The green points are points in the sequence~$(q_n/q_{n-1}, \sqrt[n]{q_n})$ where~$q_n$ is defined as in~\eqref{qndef} with~$r=3$.  Note that~$\sqrt[n]{q_n}$ is our lower bound on~$\Theta(C_{q_{n-1},q_n})$. The sequence of green points converges to the orange point (3,3) as~$n \to \infty$.\label{thetafigchapzoomin}}
\end{figure}

Question~\eqref{question} seems not easy to solve in general, although for concrete cases of~$q,n,r$ the number~$d(q,n,r)$ is easy to compute by computer. Note that Theorem~\ref{explindth} is equivalent to the statement that $d(q_n,n,r) \geq q_{n-1}$ for~$n \in \N$ and~$r \geq 3$ the fixed integer used to define~$q_n$. In Section~\ref{shannonc7} we will see, motivated by Question~\eqref{question}, that the set  $\{ t \cdot (1,7,\ldots, 7^4) \,\, | \,\, t \in \Z_{382} \} \subseteq \Z_{382}^5$ is independent in~$C_{108,382}^{\boxtimes 5}$. So~$d(382,5,7) \geq 108$. We  will modify this independent set to obtain a new lower bound on~$\Theta(C_7)$.

 Consider the graph of the function~$q/d \mapsto \vartheta(C_{d,q})$ in Figure~\ref{thetafigchap}. It is continuous at integers~$\geq 3$, as~$q/d \mapsto \ev(C_{d,q})$ is continuous at these points. From the figure, the function $q/d \mapsto \vartheta(C_{d,q})$ seems continuous at \emph{all} rational numbers~$\geq 2$, but this is not known. However, it seems not differentiable at some points, for example at points of the form~$q/2$ with~$q$ integer. If we zoom in further (as in~Figure~\ref{thetafigchapzoomin}), we see new points appear where the function seems nondifferentiable (see the point~$(8/3, \vartheta(C_{8,3}))$ marked with a red circle). In fact, we suspect (and numerical experiments  seem to support this) that the function~$q/d \mapsto \vartheta(C_{d,q})$ has left derivative~$+\infty$ and right derivative~$0$ at \emph{every} point~$q/d\in \Q$ with~$q/d >2$.

\section{Uniqueness of~\texorpdfstring{$\alpha(C_{5,14}^{\boxtimes 3})=14$}{alpha(C5,14 cubed)=14}}

Fix an integer~$r \geq 3$ and let the integer~$q_n$ be defined as in~$\eqref{qndef}$, for each~$n \in \N$. By Proposition~\ref{shanub} and Theorem \ref{explindth} we know that~$\alpha(C_{q_{n-1},q_n}^{\boxtimes n}) = q_n$. The question arises whether there is a \emph{unique} independent set  of cardinality $ q_n$ in~$C_{q_{n-1},q_n}^{\boxtimes n}$ up to Lee equivalence.  We do not have a general answer, but in this section we present some partial results. First we note the following.
\begin{proposition}\label{exactlyonceprop}
If~$B \subseteq \Z_{q_n}^n$ is independent in~$C_{q_{n-1},q_n}^{\boxtimes n}$ with~$|B|=q_n$, then each symbol in~$\Z_{q_n}$ appears exactly once in each column of~$B$ (interpreted as a~$q_n \times n$ matrix with the words as rows). 
\end{proposition}
\begin{proof}
Fix a column, and let~$c_j$ denote the number of times symbol~$j \in \Z_{q_n}$ occurs in this column.  Then~$\sum_{j=0}^{q_n-1}c_j=q_n$.
Note that~$r > q_n /q_{n-1}$, which implies with Proposition~\ref{shanub} that~$\alpha(C_{q_{n-1},q_n}^{\boxtimes (n-1)})\leq q_{n-1}$.   This implies~$\sum_{j=t+1}^{t+q_{n-1}} c_j \leq q_{n-1}$ for each~$t \in \Z_{q_n}$, where we take indices of~$c$ mod~$q_n$. As~$\sum_{t=0}^{q_n-1}\sum_{j=t+1}^{t+q_{n-1}} c_j = q_n \cdot q_{n-1}$, we  have~$\sum_{j=t+1}^{t+q_{n-1}} c_j = q_{n-1}$ for each~$t \in \Z_{q_n}$. This implies, using the equations~$\sum_{j=t+1}^{t+q_{n-1}} c_j = q_{n-1}$ for~$t$ and~$t+1$ that~$c_t=c_{t+q_{n-1}}$ for any~$t \in \Z_{q_n}$. As~$q_{n-1}$ and~$q_n$ are coprime, this means that~$c_0=c_1=\ldots=c_{q_n-1}$. Together with~$\sum_{j=0}^{q_n-1}c_j=q_n$, this proves the proposition.   
\end{proof}

Let us now fix~$r:=3$. Then~$q_n = (3^n+1)/2$, cf.~$\eqref{qndef}$. For~$n=1,2$ it is an easy exercise to prove that there exists a \emph{unique}  independent set of size~$q_n$ in~$C_{q_{n-1},q_n}^{\boxtimes n}$ up to Lee equivalence. We will prove that also the case~$n=3$ holds true.

\begin{proposition}
Up to Lee equivalence there is a unique independent set of size~$14$ in~$C_{5,14}^{\boxtimes 3}$.
\end{proposition}
\proof
By Theorem~\ref{explindth}, an independent set of size~$14$ in~$C_{5,14}^{\boxtimes 3}$ is given by
\begin{align}
    A:=\{t \cdot (1,3,9)\,\, | \,\, t \in \Z_{14}  \} \subseteq \Z_{14}^3.
\end{align}
 Let~$B$ be any independent set of size~$14$ in~$C_{5,14}^{\boxtimes 3}$. We prove that~$B$ is Lee equivalent to~$A$. Without loss of generality, we  assume that~$B$ contains~$\bm{0}$, the zero word. Proposition~\ref{exactlyonceprop} implies the following.
\begin{align} \label{exactlyonce}
    \text{ Each symbol in~$\Z_{14}$ appears in each column of~$B$ exactly once.}
\end{align}
 For~$i \in \Z_{14}$, and~$k \in \{1,2,3\}$ let~$f_k(i)$ be the symbol in column~$k$ in the row containing~$i$ in column~$1$, where by assumption~$f_k(0)=0$. (So~$f_1$ is the identity function~$\Z_{14}\to \Z_{14}$ and~$f_2,f_3$ are bijective functions~$Z_{14}\to \Z_{14}$ by~$\eqref{exactlyonce}$.) Define the graphs
$$
G_k := ( \Z_{14}, \,\{\{i,j\} \,\, | \,\, |f_k(i)-f_k(j)|_{14} \geq 5 \}), \,\,\,\,\,\, (k \in \{1,2,3\}). 
$$
Note that~$G_1=\overline{C_{5,14}}$ and that~$G_2$ and~$G_3$ are isomorphic to~$\overline{C_{5,14}}$, so each~$G_k$ contains no triangles. The edge~$\{0,1\}$ is not present in~$G_1$, so it is present in at least one of~$G_2$ and~$G_3$. Assume it is in~$G_2$.  Then, for~$u \neq v \in \Z_{14}$,
\begin{align} \label{adjacencies}
    &\text{$u-v \in \{ \pm 1, \pm 4\} \pmod{14}$   \,\,\,\, $\Longrightarrow$ \,\,\,\,  $\{u, v\} \in E(G_2)$   and }\notag \\
    &\text{$u-v \in \{ \pm 2, \pm 3\} \pmod{14}$   \,\,\,\, $\Longrightarrow$ \,\,\,\,  $\{u, v\} \in E(G_3)$.}
    \end{align}
We now prove~$\eqref{adjacencies}$. Define, for~$t \in \Z_{14}$,~$V_t:=\{t,t+1,\ldots,t+4\} \subseteq \Z_{14}$ and set~$E_{k,t}:=E(G_k[V_t])$, the edge set of the subgraph of~$G_k$ induced by~$V_t$ (for~$k \in \{2,3\}$).\footnote{If~$G=(V,E)$ is a graph and~$U \subseteq V$, we write~$G[U]:=(U, \{e \in E \,\, | \,\, e \subseteq U \})$ for the \emph{subgraph of~$G$ induced by $U$}.}\symlistsort{GU}{$G[U]$}{subgraph of graph~$G$ induced by~$U \subseteq V(G)$}\indexadd{graph!induced subgraph}\indexadd{induced subgraph} Note that for each~$t$, the graph~$(V_t,E_{2,t}\cup E_{3,t})$ is the complete graph on~$V_t$. As both~$G_2[V_t]$ and~$G_3[V_t]$ contain no triangles, each of them cannot be bipartite (otherwise the complement, hence the other from the two would contain a triangle). So each~$G_k[V_t]$ is a $5$-cycle (any non-bipartite graph on $5$ vertices without triangles is a $5$-cycle), and as~$(V_t,E_{2,t}\cup E_{3,t})$ is the complete graph,~$G_2[V_t]$ and~$G_3[V_t]$ are complementary $5$-cycles. 

Moreover, if~$\{t,t+1\} \in E(G_2)$, then also~$\{t+1,t+5\}\in E(G_2)$ and~$\{t+5,t+6\}\in E(G_2)$. To see this, assume that~$\{t,t+1\}\in E(G_2)$ and note that~$G_2[V_t]$ and~$G_2[V_{t+1}]$ are five-cycles. So~$t+5$ has the same neighbors as~$t$ in~$G_2$, hence~$\{t+1,t+5\}\in E(G_2)$. Also,~$G_2[V_{t+1}]$  and~$G_2[V_{t+2}]$ are five-cycles. Hence~$t+6$ has the same neighbors as~$t+1$ in~$G_2$, so~$\{t+5,t+6\} \in E(G_2)$. As~$\{0,1\}\in E(G_2)$ we find by applying this observation repeatedly, using that~$5$ and~$14$ are coprime, that~$G_2$ contains edges~$\{t,t+1\}$ and~$\{t,t+4\}$ for all~$t \in \Z_{14}$.    This proves the first part of~$(\ref{adjacencies})$. The second part follows from the observation that~$G_2[V_t]$ and~$G_3[V_t]$ are complementary $5$-cycles for each~$t \in \Z_{14}$. As vertex~$t$ has neighbors~$t+1$ and~$t+4$ in~$G_2[V_t]$, it has neighbors~$t+2$ and~$t+3$ in~$G_3[V_t]$.  So~$G_3$ contains edges~$\{t,t+2\}$ and~$\{t,t+3\}$ for each~$t \in \Z_{14}$. This proves~$\eqref{adjacencies}$.

Next, we prove that for~$u \neq v \in \Z_{14}$,
\begin{align} \label{adjacencies2}
    &\text{$\{u, v\} \in E(G_2)$  \,\,\,\, $\Longleftrightarrow$ \,\,\,\,~$u-v \in \{ \pm 1, \pm 4, 7\} \pmod{14}$ and }\notag \\
    &\text{$\{u,v\} \in E(G_3)$  \,\,\,\, $\Longleftrightarrow$ \,\,\,\,~$u-v \in \{ \pm 2, \pm 3, 7\} \pmod{14}$. } 
    \end{align}Consider  a vertex~$u$ in~$G_2$. Then~$u$ has neighbors~$u \pm 1$ and~$u \pm 4$. As~$G_2 \cong \overline{C_{5,14}}$, a 5-regular graph, vertex~$u$ must have a fifth neighbor~$v$. This neighbor cannot be~$u +2$, for otherwise~$(u,u+1,u+2,u)$ would be a triangle. This neighbor cannot be~$u+3$, for otherwise~$(u, u+3,u+4,u)$ would be a triangle. This neighbor cannot be~$u+5$, for otherwise~$(u,u+5,u+1,u)$ would be a triangle. This neighbor cannot be~$u+6$, for otherwise~$(u,u+6,u+10,u)$ would be a triangle. By symmetry, this neighbor also not be~$u-3$,~$u-4$, $u-5$ or~$u-6$. So the fifth neighbor must be~$u+7$. This proves~$(\ref{adjacencies2})$ for~$G_2$.
    
  Similarly, consider  a vertex~$u$ in~$G_3$. Then~$u$ has neighbors~$u \pm 2$ and~$u \pm 3$. As~$G_3 \cong \overline{C_{5,14}}$, a 5-regular graph, vertex~$u$ must have a fifth neighbor~$v$. This neighbor cannot be~$u +1$, for otherwise~$(u,u+1,u-2,u)$ would be a triangle. This neighbor cannot be~$u+4$, for otherwise~$(u, u+2,u+4,u)$ would be a triangle. This neighbor cannot be~$u+5$, for otherwise~$(u,u+5,u+3,u)$ would be a triangle. This neighbor cannot be~$u+6$, for otherwise~$(u,u+6,u+3,u)$ would be a triangle. By symmetry, this neighbor also not be~$u-1$,~$u-4$, $u-5$ or~$u-6$. So the fifth neighbor must be~$u+7$. This proves~$(\ref{adjacencies2})$ for~$G_3$.

\begin{figure}[ht]
   \begin{subfigure}{0.32\linewidth}\scalebox{0.935}{
  \begin{tikzpicture} 
    \begin{scope} [vertex style/.style={draw,
                                       circle,
                                       minimum size=2mm,
                                       inner sep=0pt,
                                       outer sep=0pt, fill}] 
      \path \foreach \i in {0,...,13}{%
       (25.714*\i:2.5) coordinate[vertex style] (a\i)}
       ; 
    \end{scope}

     \begin{scope} [edge style/.style={draw=black}]
       \foreach \i  in {0,...,13}{%
       \pgfmathtruncatemacro{\nexta}{mod(\i+5,14)} 
       \pgfmathtruncatemacro{\nextab}{mod(\i+6,14)}   
       \pgfmathtruncatemacro{\nextabc}{mod(\i+7,14)}     
       \draw[edge style,thick,donkergroen] (a\i)--(a\nextab);
       \draw[edge style,thick,blue] (a\i)--(a\nexta);
       \draw[edge style,thick,firebrick] (a\i)--(a\nextabc);       
       }  
     \end{scope}
  \end{tikzpicture} }
   \caption*{$G_1$}
\end{subfigure}\hspace{0.01\textwidth}
   \begin{subfigure}{0.32\linewidth}\scalebox{0.935}{
     \begin{tikzpicture} 
    \begin{scope} [vertex style/.style={draw,
                                       circle,
                                       minimum size=2mm,
                                       inner sep=0pt,
                                       outer sep=0pt, fill}] 
      \path \foreach \i in {0,...,13}{%
       (25.714*\i:2.5) coordinate[vertex style] (a\i)}
       ; 

    \end{scope}

     \begin{scope} [edge style/.style={draw=black}]
       \foreach \i  in {0,...,13}{%
       \pgfmathtruncatemacro{\nexta}{mod(\i+1,14)} 
       \pgfmathtruncatemacro{\nextab}{mod(\i+4,14)}   
       \pgfmathtruncatemacro{\nextabc}{mod(\i+7,14)}     
       \draw[edge style,thick, donkergroen] (a\i)--(a\nextab);
       \draw[edge style,thick, blue] (a\i)--(a\nexta);
       \draw[edge style,thick, firebrick] (a\i)--(a\nextabc);       
       }  
     \end{scope}

  \end{tikzpicture} }
      \caption*{$G_2$}
\end{subfigure}   \hspace{0.01\textwidth}
   \begin{subfigure}{0.32\linewidth}\scalebox{0.935}{
     \begin{tikzpicture} 
    \begin{scope} [vertex style/.style={draw,
                                       circle,
                                       minimum size=2mm,
                                       inner sep=0pt,
                                       outer sep=0pt, fill}] 
      \path \foreach \i in {0,...,13}{%
       (25.714*\i:2.5) coordinate[vertex style] (a\i)}
       ; 
    \end{scope}

     \begin{scope} [edge style/.style={draw=black}]
       \foreach \i  in {0,...,13}{%
       \pgfmathtruncatemacro{\nexta}{mod(\i+2,14)} 
       \pgfmathtruncatemacro{\nextab}{mod(\i+3,14)}   
       \pgfmathtruncatemacro{\nextabc}{mod(\i+7,14)}     
       \draw[edge style,thick, donkergroen] (a\i)--(a\nexta);
       \draw[edge style,thick, blue] (a\i)--(a\nextab);
       \draw[edge style,thick, firebrick] (a\i)--(a\nextabc);       
       }  
     \end{scope}

  \end{tikzpicture} }
        \caption*{$G_3$} 
   \end{subfigure}
   \caption{\small The graphs~$G_1,G_2,G_3$. They all have vertex set~$\Z_{14}$, depicted in cyclic clockwise order in this figure.}
   \end{figure}
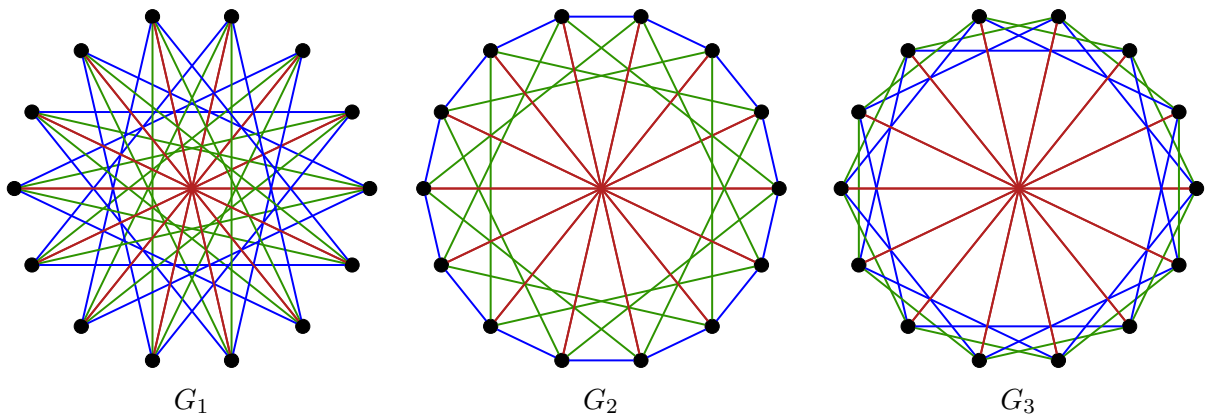

  Note that 
\begin{align} \label{7cycles}
    &\text{$(0,4,8,12,2,6,10,0)$ is an induced~$7$-cycle in~$G_2$ and}\notag \\
    &\text{$(0,2,4,6,8,10,12,0)$ is an induced~$7$-cycle in~$G_3$}. 
    \end{align}
This follows from~$(\ref{adjacencies2})$, as~$\{0,2,4,8,10,12,14\}$ contains no pairs~$\{u,v\}$ with~$u-v \in \{ \pm 1, \pm 3, 7 \} \pmod{14}$, as all its elements are~$0 \pmod{2}$. 
\begin{align} \label{only7cycle}
    \text{The only induced~$7$-cycle containing~$0$ in~$\overline{C_{5,14}}$ is~$(0,6,12,4,10,2,8,0)$}. 
\end{align}
This is easily verified by computer, by checking~$\binom{13}{6}$ possibilities. (It is also not hard to verify by hand. Vertex~$0$ has neighbors~$N:=\{5,6,7,8,9\}$ in~$\overline{C_{5,14}}$. Set~$R:=\Z_{14} \setminus (N \cup \{0\})$. An induced~$7$-cycle containing~$0$ in~$\overline{C_{5,14}}$, starting at~$0$, first traverses~$N$, then four vertices in~$R$, then traverses~$N$ again and then returns to~$0$. The four vertices in~$R$ on the induced $7$-cycle form an induced path of length~$3$ (in edges) in the induced subgraph~$\overline{C_{5,14}}[R]$, the subgraph of $\overline{C_{5,14}}$ induced by~$R$. This graph is bipartite with~$8$ vertices and~$10$ edges. By hand the reader may verify that there are $15$ induced paths of length~$3$ in this graph, but of these paths only the path~$(2,10,4,12)$ can be extended to an induced~$7$-cycle containing~$0$ in~$\overline{C_{5,14}}$, and this extension is unique.)

It remains to determine the numbers~$f_k(i)$, for~$k \in \{2,3\}$ and~$i \in \Z_{14}$. Note that~$f_k$ is a graph isomorphism~$G_k \to \overline{C_{5,14}}$. So induced~$7$-cycles must go to induced~$7$-cycles. Combining~$\eqref{7cycles}$ and~$\eqref{only7cycle}$ we find (using that a 7-cycle can be traversed in two ways) that~$f_2$ either satisfies~$f_2(4t)=6t$ for all~$t \in \Z_{14}$, or~$f_2(4t)=8t$  for all~$t \in \Z_{14}$. Substituting~$t$ by~$4t$ we find mod~$14$ (as~$4 \cdot 4t=2t$) that either~$f_2(2t)=10t=-4t$ for all~$t \in \Z_{14}$, or~$f_2(2t)=4t$  for all~$t \in \Z_{14}$. If we replace~$f_2$ by~$-f_2$ and change the symbols in the second column of~$B$ accordingly, we obtain an independent set which is Lee equivalent to~$B$. So we may assume that~$f_2(2t)=4t$ (otherwise we replace~$f_2$ by~$-f_2$).

\begin{table}[ht]
\centering
\begin{tabular}{l|lllllll}
$t$ & 0 & 2 & 4 & 6 & 8  &10  &12  \\\hline
$f_2(t)$      &0  &4   & 8 & 12 &2  & 6   &10   
\end{tabular}
\caption{\small The values of~$f_2$.}
\label{valf2}
\end{table}
 
  Note that~$\{s,s+1\} \in E(G_2)$ for each~$s \in \Z_{14}$, so~$\{f_2(s),f_2(s+1)\} \in E(\overline{C_{5,14}})$. So~$f_2(2t)=4t$ and~$f_2(2t+2)=4t+4$ have~$f_2(2t+1)$ as common neighbor in~$\overline{C_{5,14}}$ ($t \in \Z_{14})$. As the only common neighbor of~$4t$ and~$4t+4$ in~$\overline{C_{5,14}}$ is~$4t+9$, we have~$f_2(2t+1)=4t+9$ for each~$t$. Together with~$f_2(2t)=4t$ this gives~$f_2(t)=9t$ for all~$t \in \Z_{14}$.

 Now we determine~$f_3$. Combining~$\eqref{7cycles}$ and~$\eqref{only7cycle}$ we find that~$f_3$ either satisfies~$f_3(2t)=6t$ for all~$t \in \Z_{14}$, or~$f_2(2t)=8t=-6t$  for all~$t \in \Z_{14}$. Without loss of generality we may assume that~$f_3(2t)=6t$ for all~$t \in \Z_{14}$ (otherwise we replace~$f_3$ by~$-f_3$). Note that~$\{s,s+3\} \in E(G_3)$ for each~$s \in \Z_{14}$, so~$\{f_3(s),f_3(s+3)\} \in E(\overline{C_{5,14}})$. So~$f_3(2t)=6t$ and~$f_3(2t+6)=6t+4$ have~$f_3(2t+3)$ as common neighbor in~$\overline{C_{5,14}}$ ($t \in \Z_{14})$. As the only common neighbor of~$6t$ and~$6t+4$ in~$\overline{C_{5,14}}$ is~$6t+9$, we have~$f_3(2t+3)=6t+9$, hence~$f_3(2t+1)=6t+3$. Together with~$f_3(2t)=6t$ this gives~$f_3(t)=3t$ for all~$t \in \Z_{14}$.  Hence we have determined~$f_3$. 
 
 So~$f_1(t)=t$, $f_2(t)=9t$ and~$f_3(t)=3t$ for all~$t \in \Z_{14}$. We conclude that~$B$ is Lee equivalent to~$A$. 
 \endproof

\section{Semidefinite programming bounds on \texorpdfstring{$\alpha(C_{d,q}^{\boxtimes n})$}{alpha(Cd,q strongproductpower n}\label{circbounds}}

 We describe how the semidefinite programming bound~$B_3^L(q,n,d)$ from~\eqref{Bleend} can be adapted to an upper bound $B_3^{L_{\infty}}(q,n,d)$ on~$\alpha(C_{d,q}^{\boxtimes n})$, which either improves or is equal to the bound obtained from  Lov\'asz's $\vartheta$-function.   Since~$\vartheta \in \Delta$ (see Section~\ref{shannonprem}), it satisfies for all~$n \in \N$,
\begin{align}
\alpha(C_{d,q}^{\boxtimes n})     \leq \vartheta(C_{d,q}^{\boxtimes n})=\vartheta(C_{d,q})^n.
\end{align}
This implies that~$\Theta(C_{d,q}) \leq \vartheta(C_{d,q})$. 

 However, the new bound~$B_3^L(q,n,d)$ is not multiplicative over the strong product, so it does not need to give an upper bound on~$\Theta(C_{d,q})$.   Define, for~$k \geq 2$,\symlistsort{BLinftyk(q,n,d)}{$B^{L_{\infty}}_k(q,n,d)$}{upper bound on~$\alpha(C_{d,q}^{\boxtimes n})=A^{L_{\infty}}_q(n,d)$}
\begin{align} \label{Bleemaxnd}
B^{L_{\infty}}_k(q,n,d):=    \max \{ \mbox{$\sum_{v \in \Z_q^n} x(\{v\})$}\,\, |\,\,&x:\mathcal{C}_k \to \R, \,\, x(\emptyset )=1,\,\, x(S)=0 \text{ if~$d_{\text{min}}^{L_{\infty}}(S)<d$}, \notag \\
 \,\, &M_{k,D}(x) \succeq 0 \text{ for each~$D$ in~$\mathcal{C}_k$}\}. 
\end{align}
Here~$\mathcal{C}_k$ again is the collection of codes~$C \subseteq \Z_q^n$ with~$|C|\leq k$, and the matrix~$M_{k,D}(x)$ is defined in Section~\ref{semprogboundleesect}, for $x:\mathcal{C}_k \to \R$ and~$D \in \mathcal{C}_k$.

So~$B_k^{L_{\infty}}(q,n,d)$ is obtained from the bound~$B_k^{L}(q,n,d)$ in~$(\ref{Bleend})$ by replacing in the definition $d_{\text{min}}^{L}(S)$ by~$d_{\text{min}}^{L_{\infty}}(S)$. It is not hard to see that~$\alpha(C_{d,q}^{\boxtimes n}) \leq B_k^{L_{\infty}}(q,n,d)$, by a proof analogous to that of Proposition~\ref{trivleeprop}. For comparison,~$\vartheta(C_{d,q}^{\boxtimes n})$ is equal to the bound obtained from~$B_2^{L_{\infty}}(q,n,d)$  by removing the constraints that~$M_{2,D}(x)$ is positive semidefinite for subsets~$D \in \mathcal{C}_2$ with~$D \neq \emptyset$. Moreover,~$B_2^{L_{\infty}}(q,n,d)$ is equal to the Delsarte bound, which is equal to the bound~$\vartheta'(C_{d,q}^{\boxtimes n})$, with~$\vartheta'$ as in~$\cite{thetaprime}$.

\begin{table}[ht]  
\centering
\begin{tabular}[t]{|l||rrrrr|}\hline
$n$ & 1 & 2 & 3 & 4 & 5 \\  \hline
$B_2^{L_{\infty}}(5,n,2)$ &  $2.236 $ & $ 5.000 $ & 11.180  & $25.000$ & $55.902$ \\
$B_3^{L_{\infty}}(5,n,2)$ & 2.000 & $5.000$  & $10.915$ & $25.000$ & $55.902$ \\ \hline
$B_2^{L_{\infty}}(7,n,2)$ & $3.318$ & $ 11.007 $ & 36.517  & $121.152$ & $401.943$ \\
$B_3^{L_{\infty}}(7,n,2)$ & 3.000 & $10.260$  & $35.128$ & $119.537$ & $401.908$\\ \hline
$B_2^{L_{\infty}}(7,n,3)$ & $2.110$ & $ 4.452 $ & $9.393 $  & $19.818$ &  $ 41.814$    \\
$B_3^{L_{\infty}}(7,n,3)$ & 2.000 & $4.139$  & $8.957$ & $19.494$ & $41.782 $\\ \hline
$\#$Vars in $B_3^{L_{\infty}}(5,n,2)$ & 2 & 9  & 48 & 214 & 799 \\
$\#$Vars in $B_3^{L_{\infty}}(7,n,2)$ & 3 & 43  & 423 & 3161 & 19023 \\
$\#$Vars in $B_3^{L_{\infty}}(7,n,3)$ & 2 & 12  & 137 & 1316 & 9745 \\ \hline
\end{tabular}
\caption{\small Bounds on~$\alpha(C_5^{\boxtimes n})$,~$\alpha(C_7^{\boxtimes n})$ and~$\alpha(C_{3,7}^{\boxtimes n})$, rounded to three decimal places. It holds that $B_2^{L_{\infty}}(5,n,2) = \sqrt{5}^n$.\label{b3maxtable} }
\end{table}

To compute~$B_3^{L_{\infty}}(q,n,d)$, the reductions from Section~\ref{reduct} can be used --- see also Section~\ref{pseudocodeleesect} for an overview of the program. The new bound $B_3^{L_{\infty}}(q,n,d)$ does not seem to improve significantly over the bound obtained from Lov\'asz's $\vartheta$-function, except for very small~$n$. See Table~\ref{b3maxtable} for some results for~$q \in \{5,7\}$ and~$1 \leq n \leq 5$. For these cases,~$B_3^{L_{\infty}}(q,n,d)$ does not give new upper bounds on~$\alpha(C_{d,q}^{\boxtimes n})$, as the values~$\alpha(C_5^{\boxtimes 3})=10$, $\alpha(C_7^{\boxtimes 2})=10$,~$\alpha(C_7^{\boxtimes 3})=33$ (cf.~\cite{baumert}),~$\alpha(C_{3,7}^{\boxtimes 3})=8$ (cf.~\cite{circular}) are already known and~$\alpha(C_7^{\boxtimes 4})\leq \floor{(7/2)\alpha(C_7^{\boxtimes 3})}=115$. For the cases in Table~\ref{b3maxtable}, the bound $B_2^{L_{\infty}}(q,n,d)$ is equal to~$\vartheta(C_{d,q})^n$.

The number of variables ``$\#$Vars'' in~$B_3^{L_{\infty}}(q,n,d)$, which  is the number of~$D_q^n \rtimes S_n$-orbits of nonempty codes of size~$\leq 3$ and minimum Lee\textsubscript{$\infty$} distance at least~$d$,  is also given in Table~\ref{b3maxtable} for the cases considered.

\section{New lower bound on the Shannon capacity of \texorpdfstring{$C_7$}{C7}}\label{shannonc7}

As noted in Section~\ref{shanchapintro}, the Shannon capacity of~$C_7$ is still unknown and its determination is a notorious open problem in extremal combinatorics~\cite{bohman, godsil}. Many lower bounds have been given by explicit independent sets in some fixed strong product power of~$C_7$ \cite{baumert, matos, veszer}, while the best known upper bound is~$\Theta(C_7)\leq \vartheta(C_7) < 3.3177$. Here we give an independent set of size~$367$ in~$C_7^{\boxtimes 5}$, which yields~$\Theta(C_7)\geq 367^{1/5} > 3.2578$. The best previously known lower bound on~$\Theta(C_7)$ is~$\Theta(C_7) \geq 350^{1/5} > 3.2271$, found by Mathew and \"Osterg{\aa}rd~$\cite{matos}$. They proved that~$\alpha(C_7^{\boxtimes 5}) \geq 350$ using stochastic search methods that utilize the symmetry of the problem. In~$\cite{baumert}$, a construction is given of an independent set of size~$7^3=343$ in~$C_7^{\boxtimes 5}$. The best known lower bound on~$\alpha(C_7^{\boxtimes 4})$ is~$108$, by Vesel and \v{Z}erovnik~\cite{veszer}. See Table~\ref{knownindep} for the currently best known bounds on~$\alpha(C_7^{\boxtimes n})$ for small~$n$.

\begin{table}[ht]
\centering
\begin{tabular}[t]{l|ccccc}
$d$ & 1 & 2 & 3 & 4 & 5 \\ \hline 
$\alpha(C_7^{\boxtimes n})$ & 3 & $10^a$ & $33^d$ & $108^e$--$115^b$ & $367^f$--$401^c$
\end{tabular} 
\caption{Bounds on~$\alpha(C_7^{\boxtimes n})$. {\small Key:
\\ $^a$ $\alpha(C_q^{\boxtimes 2})= \floor{(q^2-q)/4}$ \cite[Theorem 2]{baumert} 
\\$^b$ $\alpha(C_q^{\boxtimes n}) \leq \alpha(C_q^{\boxtimes (n-1)})q/2$ \cite[Lemma 2]{baumert} 
\\$^c$ $\alpha(G^{\boxtimes n}) \leq \vartheta(G)^n$ by Lov\'asz $\cite{lovasz}$
\\$^d$ Baumert et al.~\cite{baumert}
\\$^e$ Vesel and \v{Z}erovnik~\cite{veszer}
\\$^f$ this chapter, see the Appendix for the explicit independent set.}} \label{knownindep}
\end{table}
For comparison,~$\alpha(C_7^{\boxtimes 3})^{1/3} =33^{1/3} \approx 3.2075$, $\alpha(C_7^{\boxtimes 4})^{1/4} \geq 108^{1/4} \approx 3.2237$ and the previously best known lower bound on~$\alpha(C_7^{\boxtimes 5})^{1/5}$  is~$350^{1/5} \approx 3.2271$. Now we know that~$\alpha(C_7^{\boxtimes 5})\geq 367 > 3.2578^5$.

An independent set in~$C_{d,q}^{\boxtimes n}$ gives an independent set in~$C_{\ceil{2q/d}}^{\boxtimes n}$, since $\overline{C_{d,q}} \to \overline{C_{2,\ceil{2q/d}}}$. Explicitly, consider the elements of~$\Z_q$ as integers between~$0$ and~$q-1$ and replace each element~$i$ by~$\floor{2i/d}$, and consider the outcome as an element of~$\Z_{\ceil{2q/d}}$. This gives indeed a homomorphism $\overline{C_{d,q}} \to \overline{C_{2,\ceil{2q/d}}}$ as the image of any two elements with distance at least~$d$ has distance at least~$2$.

In Section~\ref{circular} we give an explicit description of an independent set~$S$ of size~$382$ in the graph~$C_{108,382}^{\boxtimes 5}$. It is of the form~$\{ t \cdot (1,r,\ldots, r^{n-1}) \,\, | \,\, t \in \Z_{q} \}$, as in~\eqref{question}.  As~$382/108 > 7/2$ this does not directly give an independent set in~$C_{2,7}^{\boxtimes 5}$. However, in Section~$\ref{descr}$ we describe how one can obtain an independent set of size~$367$ in~$C_7^{\boxtimes 5}$ from~$S$, by adapting~$S$, removing vertices and adding new ones. This independent set  is given explicitly in the Appendix.

\subsection{Independent set of size 382 in \texorpdfstring{$C_{108,382}^{\boxtimes 5}$}{C108,382 strongproductpower 5}\label{circular}}


\begin{proposition}\label{is}
The set~$S:=\{t \cdot (1,7,7^2,7^3,7^4) \,\, | \,\, t \in \Z_{382}\} \subseteq \Z_{382}^5$ is independent in $C_{108,382}^{\boxtimes 5}$. 
\end{proposition}
\proof
If~$x,y \in S$ then also~$x-y \in S$. So it suffices to check that for all nonzero~$x \in S$:\footnote{Recall that we write~$[a,b]:=\{a,a+1,\ldots,b\}$ for any two integers~$a,b$.}
\begin{align} \label{coor108}
\text{$\exists \,\, i \in [1,5]$ such that~$x_i \in [108, 274]$.}
\end{align}
Let~$x=t \cdot (1,7,7^2,7^3,7^4) \in S$ be arbitrary, with~$0 \neq t \in \Z_{382}$.  For~$t \in [108,274]$ clearly~$\eqref{coor108}$ holds with~$i=1$ (as then~$x_i =t \in [108,274]$).  Also we have  $[275,381]=-[1,107]$, so it suffices to verify~\eqref{coor108} for~$t \in [1,107]$. Note that for~$t \in [16,39]$ one has~$108\leq 7t\leq274 $, so~\eqref{coor108} is satisfied with~$i=2$. Also note that~$69\cdot 7 \equiv 101 \pmod{382}$. So for~$t \in [70,93]$ one has~$7t \equiv 101+7(t-69) \pmod{382} \in [108, 274]$, i.e.,~\eqref{coor108} is satisfied with~$i=2$. For the remaining~$t \in [1,107]$, please take a glance at Table~\ref{verif}. In each row, in each of the three subtables, there is at least one entry in~$[108,274]$. This completes the proof. 
\endproof 

\begin{table}[ht]
\centering
\tiny
\begin{tabular}[t]{rrrrr}
1 & 7 & 49 & 343 & \cellcolor{donkergroen!25} 109 \\ 
2 & 14 & 98 & 304 & \cellcolor{donkergroen!25} 218 \\ 
3 & 21 & \cellcolor{donkergroen!25}147 & \cellcolor{donkergroen!25} 265 & 327 \\ 
4 & 28 & \cellcolor{donkergroen!25}196 & \cellcolor{donkergroen!25}226 & 54 \\ 
5 & 35 & \cellcolor{donkergroen!25}245 & \cellcolor{donkergroen!25}187 &\cellcolor{donkergroen!25} 163 \\ 
6 & 42 & 294 &\cellcolor{donkergroen!25} 148 &\cellcolor{donkergroen!25} 272 \\ 
7 & 49 & 343 &\cellcolor{donkergroen!25} 109 & 381 \\ 
8 & 56 & 10 & 70 &\cellcolor{donkergroen!25} 108 \\ 
9 & 63 & 59 & 31 &\cellcolor{donkergroen!25} 217 \\ 
10 & 70 &\cellcolor{donkergroen!25} 108 & 374 & 326 \\ 
11 & 77 &\cellcolor{donkergroen!25} 157 & 335 & 53 \\ 
12 & 84 &\cellcolor{donkergroen!25} 206 & 296 &\cellcolor{donkergroen!25} 162 \\ 
13 & 91 &\cellcolor{donkergroen!25} 255 &\cellcolor{donkergroen!25} 257 &\cellcolor{donkergroen!25} 271 \\ 
14 & 98 & 304 &\cellcolor{donkergroen!25} 218 & 380 \\ 
15 & 105 & 353 &\cellcolor{donkergroen!25} 179 & 107 \\ 
&&&& \\ 
40 & 280 & 50 & 350 &\cellcolor{donkergroen!25} 158 \\ 
41 & 287 & 99 & 311 &\cellcolor{donkergroen!25} 267 \\ 
42 & 294 &\cellcolor{donkergroen!25} 148 &\cellcolor{donkergroen!25} 272 & 376 \\ 
43 & 301 & \cellcolor{donkergroen!25}197 & \cellcolor{donkergroen!25}233 & 103 \\ 
44 & 308 &\cellcolor{donkergroen!25} 246 & \cellcolor{donkergroen!25}194 & \cellcolor{donkergroen!25}212 \\ 
\end{tabular}\,\,\,\,\,\,\,\,\,\,\,\,\,\,\,
\begin{tabular}[t]{rrrrr}
45 & 315 & 295 &\cellcolor{donkergroen!25} 155 & 321 \\ 
46 & 322 & 344 & \cellcolor{donkergroen!25}116 & 48 \\ 
47 & 329 & 11 & 77 & \cellcolor{donkergroen!25}157 \\ 
48 & 336 & 60 & 38 &\cellcolor{donkergroen!25} 266 \\ 
49 & 343 &\cellcolor{donkergroen!25} 109 & 381 & 375 \\ 
50 & 350 & \cellcolor{donkergroen!25}158 & 342 & 102 \\ 
51 & 357 & \cellcolor{donkergroen!25}207 & 303 &\cellcolor{donkergroen!25} 211 \\ 
52 & 364 &\cellcolor{donkergroen!25} 256 & \cellcolor{donkergroen!25}264 & 320 \\ 
53 & 371 & 305 &\cellcolor{donkergroen!25} 225 & 47 \\ 
54 & 378 & 354 &\cellcolor{donkergroen!25} 186 &\cellcolor{donkergroen!25} 156 \\ 
55 & 3 & 21 &\cellcolor{donkergroen!25} 147 &\cellcolor{donkergroen!25} 265 \\ 
56 & 10 & 70 & \cellcolor{donkergroen!25}108 & 374 \\ 
57 & 17 &\cellcolor{donkergroen!25} 119 & 69 & 101 \\ 
58 & 24 & \cellcolor{donkergroen!25}168 & 30 &\cellcolor{donkergroen!25} 210 \\ 
59 & 31 &\cellcolor{donkergroen!25} 217 & 373 & 319 \\ 
60 & 38 &\cellcolor{donkergroen!25} 266 & 334 & 46 \\ 
61 & 45 & 315 & 295 &\cellcolor{donkergroen!25} 155 \\ 
62 & 52 & 364 & \cellcolor{donkergroen!25}256 &\cellcolor{donkergroen!25} 264 \\ 
63 & 59 & 31 & \cellcolor{donkergroen!25}217 & 373 \\ 
64 & 66 & 80 &\cellcolor{donkergroen!25} 178 & 100 \\ 
\end{tabular} \,\,\,\,\,\,\,\,\,\,\,\,\,\,\,
\begin{tabular}[t]{rrrrr}
65 & 73 &\cellcolor{donkergroen!25} 129 &\cellcolor{donkergroen!25} 139 &\cellcolor{donkergroen!25} 209 \\ 
66 & 80 &\cellcolor{donkergroen!25} 178 & 100 & 318 \\ 
67 & 87 &\cellcolor{donkergroen!25} 227 & 61 & 45 \\ 
68 & 94 & 276 & 22 &\cellcolor{donkergroen!25} 154 \\ 
69 & 101 & 325 & 365 &\cellcolor{donkergroen!25} 263 \\ 
&&&& \\ 
94 & 276 & 22 & \cellcolor{donkergroen!25}154 & 314 \\ 
95 & 283 & 71 & \cellcolor{donkergroen!25}115 & 41 \\ 
96 & 290 &\cellcolor{donkergroen!25} 120 & 76 & \cellcolor{donkergroen!25}150 \\ 
97 & 297 & \cellcolor{donkergroen!25}169 & 37 &\cellcolor{donkergroen!25} 259 \\ 
98 & 304 &\cellcolor{donkergroen!25} 218 & 380 & 368 \\ 
99 & 311 & \cellcolor{donkergroen!25}267 & 341 & 95 \\ 
100 & 318 & 316 & 302 & \cellcolor{donkergroen!25}204 \\ 
101 & 325 & 365 &\cellcolor{donkergroen!25} 263 & 313 \\ 
102 & 332 & 32 &\cellcolor{donkergroen!25} 224 & 40 \\ 
103 & 339 & 81 & \cellcolor{donkergroen!25}185 & \cellcolor{donkergroen!25}149 \\ 
104 & 346 & \cellcolor{donkergroen!25}130 & \cellcolor{donkergroen!25}146 &\cellcolor{donkergroen!25} 258 \\ 
105 & 353 & \cellcolor{donkergroen!25}179 & 107 & 367 \\ 
106 & 360 & \cellcolor{donkergroen!25}228 & 68 & 94 \\ 
107 & 367 & 277 & 29 & \cellcolor{donkergroen!25}203 \\ 
\end{tabular}
\caption{\small Part of the verification that~$S$ is independent in~$C_{108,382}^{\boxtimes 5}$.} \label{verif}
\end{table}
\noindent  We found the above independent set by computer, when looking for answers to~\eqref{question}  (with~$q\geq 350$ and~$n=5$ such that~$q/d(q,n,r)$ is close to~$7/2$). 
 
\subsection{Description of the method\label{descr}}
Here we describe how to use the independent set from Proposition~\ref{is} to find an independent set of size~$367$ in~$C_7^{\boxtimes 5}$. The procedure is as follows.
\begin{speciaalenumerate}
    \item Start with the independent set~$S$ in~$C_{108,382}^{\boxtimes 5}$ from Proposition~\ref{is}.
    \item Add the word~$(40,123,40,123,40)$ mod~$382$ to each word in~$S$.
    \item Replace each letter~$i$, which we now consider to be an integer between~$0$ and~$381$ and not anymore an element in~$\Z_{382}$, in each word from~$S$ by~$\floor{i/54.5}$. Now we have a set of words~$S'$ with only symbols in~$[0,6]$ in it, which we consider as elements of~$\Z_7$.
    \item Remove each word~$u\in S'$ for which there is a~$v \in S'$ such that~$uv \in E(C_7^{\boxtimes 5})$ from~$S'$, i.e., we remove~$u$ if there is a~$v \in S'$ with~$v \neq u$ such that~$u_i-v_i\in \{0,1,6\}$ for all~$i \in [1,5]$. We denote the set of words which are not removed from~$S'$ by this procedure by~$M$. The computer finds~$|M| = 327$. Note that~$M$ is independent in~$C_7^{\boxtimes 5}$.
    \item \label{v} Find the best possible extension of~$M$ to a larger independent set in~$C_7^{\boxtimes 5}$. To do this, consider the subgraph~$G$ of~$C_7^{\boxtimes 5}$ induced by the words~$x$ in~$\Z_7^5$ with the property that~$M\cup \{x\}$ is independent in~$C_7^{\boxtimes 5}$. This graph is not large, in this case it has~$71$ vertices and~$85$ edges, so a computer finds a maximum size independent set~$I$ in~$G$ quickly. The computer finds $|I|=\alpha(G)=40$, so we can add~$40$ words to~$M$. Write~$R:=M \cup I$. Then~$|R|=327+40=367$ and~$R$ is independent in~$C_7^{\boxtimes 5}$.
\end{speciaalenumerate}
The maximum size independent set~$I$ in the graph~$G$ in (\ref{v}) was found using Gurobi~$\cite{gurobi}$. In steps (ii) and (iii), many possibilities for adding a constant word and for the division factor were tried, but no independent set of size~$368$ or larger was found. Also, the independent set~$R$ of size~$367$ did not seem to be easily extendable. A local search was performed, showing that there exists no triple of words from~$R$ such that if one removes these three words from~$R$, four words can be added to obtain an independent set of size~$368$ in~$C_7^{\boxtimes 5}$. 

 \begin{remark}
One other new bound on~$\alpha(C_{q}^{\boxtimes n})$  was obtained (for~$q \leq 15$ and~$n \leq 5$) using independent sets of the form from~\eqref{question}. With $q= 4009$, $n=5$ and $r=27$, we found~$d(q,n,r)= 729$. As~$q/(d(q,n,r))=4009/729<11/2$, this directly yields the new lower bound~$\alpha(C_{11}^{\boxtimes 5}) \geq 4009$. The previously best known lower bound is~$\alpha(C_{11}^{\boxtimes 5})\geq 3996$ from~\cite{circular, matos}. However, the new lower bound on~$\alpha(C_{11}^{\boxtimes 5})$ does not imply a new lower bound on~$\Theta(C_{11})$. It is known that $\Theta(C_{11}) \geq \alpha(C_{11}^{\boxtimes 3})^{1/3} = 148^{1/3}>5.2895$ (cf.~\cite{baumert}), which is larger than~$4009^{1/5}$. 
\end{remark}

\section{Appendix: Explicit code}
The following~$367$ words form an independent set in~$C_7^{\boxtimes 5}$, which proves the new bound $\Theta(C_7) \geq 367^{1/5} > 3.2578$. It is the set~$R$ from Section~$\ref{descr}$. 

$\phantom{,}$

{\begin{spacing}{0.1}\spaceskip=0.14em\tiny\noindent \texttt{02020, 02112, 02204, 02306, 02461, 02553, 03645, 03040, 03032, 03124, 03226, 03311, 03403, 14144, 14231, 14323, 14415, 14510, 15602, 15064, 15166, 15251, 15343, 15430, 15522, 16614, 16016, 16101, 16263, 16355, 16450, 16542, 10636, 10021, 10113, 10205, 10300, 10462, 10554, 11656, 11041, 11033, 11125, 11220, 11312, 11404, 11506, 12661, 12053, 12145, 12240, 12232, 12324, 12426, 12511, 13603, 13065, 13160, 13252, 13344, 13446, 13431, 24010, 24102, 24264, 24366, 24451, 24543, 25630, 25022, 25114, 25216, 25301, 25463, 25555, 26650, 26042, 26034, 26136, 26221, 26313, 26405, 26500, 20662, 20054, 20156, 20241, 20233, 20325, 20420, 20512, 21604, 21006, 21161, 21253, 21345, 21440, 21432, 22626, 22011, 22103, 22265, 22360, 22452, 22544, 23631, 23023, 23115, 23210, 23302, 23464, 23566, 34130, 34222, 34314, 34416, 34501, 35663, 35055, 35150, 35242, 35234, 35336, 35421, 35513, 36605, 36000, 36162, 36254, 36356, 36441, 36433, 30620, 30012, 30104, 30206, 30361, 30453, 30545, 31632, 31024, 31126, 31211, 31303, 31465, 31560, 32652, 32044, 32131, 32223, 32315, 32410, 32502, 33664, 33066, 33151, 33243, 33235, 33330, 33422, 44616, 44001, 44163, 44255, 44350, 44442, 44434, 44536, 45621, 45013, 45105, 45200, 45362, 45454, 45556, 46633, 46025, 46120, 46212, 46304, 46406, 46561, 40653, 40045, 40132, 40224, 40326, 40411, 40503, 41665, 41060, 41152, 41244, 41331, 41423, 41515, 42610, 42002, 42164, 42266, 42351, 42443, 42435, 43622, 43014, 43116, 43201, 43363, 43455, 43550, 54634, 54036, 54121, 54213, 54305, 54400, 54562, 55654, 55056, 55141, 55133, 55225, 55320, 55412, 55504, 56606, 56061, 56153, 56245, 56332, 56424, 56526, 50611, 50003, 50165, 50260, 50352, 50444, 51623, 51015, 51110, 51202, 51364, 51551, 52643, 52635, 52030, 52122, 52214, 53655, 53134, 64332, 64424, 64526, 65611, 65003, 65260, 65352, 65444, 65546, 66623, 66110, 66202, 66364, 66466, 66551, 60643, 60645, 60030, 60122, 60214, 60316, 60401, 60563, 61050, 61142, 61134, 61236, 61321, 61413, 62600, 62062, 62154, 62256, 62341, 62333, 62520, 63612, 63004, 63106, 63261, 63353, 63445, 63540, 64532, 04026, 04111, 04203, 04460, 04552, 05644, 05031, 05123, 05310, 05402, 05564, 06666, 06051, 06143, 06230, 06322, 06414, 06516, 00601, 00063, 00155, 00250, 00342, 00334, 00436, 01613, 01100, 01262, 01354, 01456, 01541, 02625, 00521, 01005, 02533, 03565, 04052, 04365, 04624, 04660, 05046, 05225, 10534, 14246, 15435, 22524, 24615, 24651, 32046, 34035, 34043, 36525, 40040, 41246, 42530, 43514, 45641, 50531, 51456, 52400, 52563, 53050, 53142, 53320, 53412, 56340, 61505, 62425, 64154, 64340, 65105, 66025}. \end{spacing}}%

\backmatter

\setcounter{chapter}{19}
\setcounter{equation}{0}
\setcounter{figure}{0}
\renewcommand{\thechapter}{\Alph{chapter}}%

\Chapter{Summary}{New methods in coding theory: error-correcting codes and the Shannon capacity}
\markboth{Summary}{Summary}
 Error-correcting codes have been studied since 1948, when Claude Shannon published his influential paper \emph{A Mathematical Theory of Communication} \cite{shannonseminal}. 
Fix three positive integers~$q,n,d$ and let~$Q:=\{0,\ldots,q-1\}$ be our \emph{alphabet}. We identify elements of~$Q^n$ with \emph{words} of length~$n$ consisting of letters (elements) from~$Q$. A \emph{code of length $n$} is any subset of~$Q^n$.  For two words~$u,v$, their \emph{Hamming distance} is the number of~$i$ with~$u_i \neq v_i$. The \emph{minimum distance} of a code~$C$ is the minimum Hamming distance between any two distinct elements of~$C$.
The central question in coding theory is the following. 
\begin{align}\label{sumalign}
\text{What is the maximum size of a code $C\subseteq Q^n$ with minimum distance at least~$d$?}
\end{align}
The maximum size in~\eqref{sumalign} is denoted by~$A_q(n,d)$. A code~$C$ with minimum distance~$d:=2e+1$ (for some integer~$e$) is called~$e$-\emph{error-correcting}. If a codeword from~$C$ is distorted in at most~$e$ positions, we can recover the original codeword by taking the codeword that is closest to the distorted codeword in Hamming distance. This principle is used in communication systems for the correction of transmission errors.

 The numbers~$A_q(n,d)$ are hard to compute in general. For many~$q,n,d$, only upper and lower bounds are known. Explicit codes yield lower bounds on~$A_q(n,d)$ and they can be used for error-correction as explained. A classical upper bound on~$A_q(n,d)$ is Delsarte's linear programming upper bound~\cite{delsarte}, which bound can be interpreted as a semidefinite programming (SDP) bound based on pairs of codewords.
 
 In this thesis we try to improve upper bounds on~$A_q(n,d)$. We give an SDP bound based on quadruples of codewords, which a priori has size exponential in~$n$.  The optimization problem is highly symmetric: it can be assumed that the optimal solution is invariant under the group of distance preserving permutations of~$Q^n$. This symmetry group is the wreath product~$S_q^n \rtimes S_n$. By the symmetry of the problem, the SDP can be reduced to a size bounded by a polynomial in~$n$.
 
 In Chapter~\ref{orbitgroupmon} we give a general method for symmetry reduction, based on representation theory. If~$G$ is a finite group acting on a finite set~$Z$ and~$n \in \N$, we give a reduction of~$Z^n \times Z^n$-matrices which are invariant under the simultaneous action of the group~$G^n \rtimes S_n$ on their rows and columns. In the reduction, we assume that a reduction is known of~$Z \times Z$-matrices which are invariant under the simultaneous action of~$G$ on their rows and columns. 
 
 In Chapter~\ref{onsartchap} we apply this general method to reduce the mentioned SDP based on quadruples of codewords for computing upper bounds on~$A_q(n,d)$. With the method, we sharpen known upper bounds for five triples~$(q,n,d)$. 
 
 In Chapter~\ref{divchap}, we  explore other methods of finding upper bounds on~$A_q(n,d)$, based on combinatorial divisibility arguments.  The methods yield new upper bounds for four triples~$(q,n,d)$.
Our most prominent result in this direction is the following bound, which gives in certain cases a strengthening of a bound implied by the Plotkin bound~\cite{plotkinoriginal}. 
\begin{theoremnn}
Suppose that~$q,n,d,m$ are positive integers with $q\geq 2$, such that~$d=m(qd-(q-1)(n-1))$, and such that~$n-d$ does not divide~$m(n-1)$. If~$r \in \{1,\ldots,q-1\}$ satisfies~
\begin{align*}
n(n-1-d)(r-1)r <  (q-r+1)(qm(q+r-2)-2r),
\end{align*}
 then~$A_q(n,d) < q^2m -r$. 
\end{theoremnn}

In Chapter~\ref{cw4chap} we consider (binary) \emph{constant weight} codes. Here the alphabet is~$\{0,1\}$.  The \emph{weight} of a word is the number of $1$'s it contains. For~$n,d,w \in \N$, the number~$A(n,d,w)$ denotes the maximum size of a code~$C \subseteq \{0,1\}^n$ with minimum distance at least~$d$ and in which every codeword has weight~$w$. (Such a code is called a `constant weight code' with weight~$w$.)  With SDP based on quadruples of codewords and a symmetry reduction with the method of Chapter~\ref{orbitgroupmon}, we find several new upper bounds on~$A(n,d,w)$. Two upper bounds matching the known lower bounds are obtained, so that we know the value of~$A(n,d,w)$ exactly: $A(22,8,10)=616$ and~$A(22,8,11)=672$.

In Chapter~\ref{cu17chap}, we prove with the SDP-output, using `complementary slackness', that the optimal constant weight codes achieving $A(23,8,11)=1288$, $A(22,8,10)=616$ and $A(22,8,11)=672$ are unique up to coordinate permutations. The mentioned unique constant weight codes can be obtained from the binary Golay code ---a famous code with good error-correcting properties--- by taking subcodes and deleting coordinates (`shortening').  

\begin{figure}[ht]
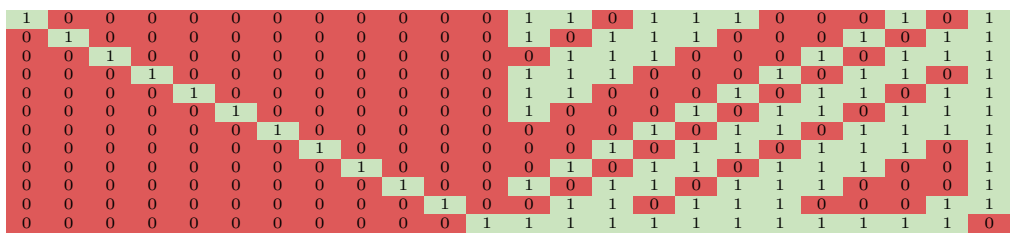

\centering
\tiny
\begin{tabular}{rrrrrrrrrrrrrrrrrrrrrrrr}
  \cellcolor{donkergroen!25}1 &  \cellcolor{donkerrood!65}0 &  \cellcolor{donkerrood!65}0 &  \cellcolor{donkerrood!65}0 &  \cellcolor{donkerrood!65}0 &  \cellcolor{donkerrood!65}0 &  \cellcolor{donkerrood!65}0 &  \cellcolor{donkerrood!65}0 &  \cellcolor{donkerrood!65}0 &  \cellcolor{donkerrood!65}0 &  \cellcolor{donkerrood!65}0 &  \cellcolor{donkerrood!65}0 &
   \cellcolor{donkergroen!25}1 &  \cellcolor{donkergroen!25}1 &  \cellcolor{donkerrood!65}0 &  \cellcolor{donkergroen!25}1 &  \cellcolor{donkergroen!25}1 &  \cellcolor{donkergroen!25}1 &  \cellcolor{donkerrood!65}0 &  \cellcolor{donkerrood!65}0 &  \cellcolor{donkerrood!65}0 &  \cellcolor{donkergroen!25}1 &  \cellcolor{donkerrood!65}0 &  \cellcolor{donkergroen!25}1 \\
  \cellcolor{donkerrood!65}0 &  \cellcolor{donkergroen!25}1 &  \cellcolor{donkerrood!65}0 &  \cellcolor{donkerrood!65}0 &  \cellcolor{donkerrood!65}0 &  \cellcolor{donkerrood!65}0 &  \cellcolor{donkerrood!65}0 &  \cellcolor{donkerrood!65}0 &  \cellcolor{donkerrood!65}0 &  \cellcolor{donkerrood!65}0 &  \cellcolor{donkerrood!65}0 &  \cellcolor{donkerrood!65}0 &  \cellcolor{donkergroen!25}1 &  \cellcolor{donkerrood!65}0 &  \cellcolor{donkergroen!25}1 &  \cellcolor{donkergroen!25}1 &  \cellcolor{donkergroen!25}1 &  \cellcolor{donkerrood!65}0 &  \cellcolor{donkerrood!65}0 &  \cellcolor{donkerrood!65}0 &  \cellcolor{donkergroen!25}1 &  \cellcolor{donkerrood!65}0 &  \cellcolor{donkergroen!25}1 &  \cellcolor{donkergroen!25}1 \\
  \cellcolor{donkerrood!65}0 &  \cellcolor{donkerrood!65}0 &  \cellcolor{donkergroen!25}1 &  \cellcolor{donkerrood!65}0 &  \cellcolor{donkerrood!65}0 &  \cellcolor{donkerrood!65}0 &  \cellcolor{donkerrood!65}0 &  \cellcolor{donkerrood!65}0 &  \cellcolor{donkerrood!65}0 &  \cellcolor{donkerrood!65}0 &  \cellcolor{donkerrood!65}0 &  \cellcolor{donkerrood!65}0 &  \cellcolor{donkerrood!65}0 &  \cellcolor{donkergroen!25}1 &  \cellcolor{donkergroen!25}1 &  \cellcolor{donkergroen!25}1 &  \cellcolor{donkerrood!65}0 &  \cellcolor{donkerrood!65}0 &  \cellcolor{donkerrood!65}0 &  \cellcolor{donkergroen!25}1 &  \cellcolor{donkerrood!65}0 &  \cellcolor{donkergroen!25}1 &  \cellcolor{donkergroen!25}1 &  \cellcolor{donkergroen!25}1 \\
  \cellcolor{donkerrood!65}0 &  \cellcolor{donkerrood!65}0 &  \cellcolor{donkerrood!65}0 &  \cellcolor{donkergroen!25}1 &  \cellcolor{donkerrood!65}0 &  \cellcolor{donkerrood!65}0 &  \cellcolor{donkerrood!65}0 &  \cellcolor{donkerrood!65}0 &  \cellcolor{donkerrood!65}0 &  \cellcolor{donkerrood!65}0 &  \cellcolor{donkerrood!65}0 &  \cellcolor{donkerrood!65}0 &  \cellcolor{donkergroen!25}1 &  \cellcolor{donkergroen!25}1 &  \cellcolor{donkergroen!25}1 &  \cellcolor{donkerrood!65}0 &  \cellcolor{donkerrood!65}0 &  \cellcolor{donkerrood!65}0 &  \cellcolor{donkergroen!25}1 &  \cellcolor{donkerrood!65}0 &  \cellcolor{donkergroen!25}1 &  \cellcolor{donkergroen!25}1 &  \cellcolor{donkerrood!65}0 &  \cellcolor{donkergroen!25}1 \\
  \cellcolor{donkerrood!65}0 &  \cellcolor{donkerrood!65}0 &  \cellcolor{donkerrood!65}0 &  \cellcolor{donkerrood!65}0 &  \cellcolor{donkergroen!25}1 &  \cellcolor{donkerrood!65}0 &  \cellcolor{donkerrood!65}0 &  \cellcolor{donkerrood!65}0 &  \cellcolor{donkerrood!65}0 &  \cellcolor{donkerrood!65}0 &  \cellcolor{donkerrood!65}0 &  \cellcolor{donkerrood!65}0 &  \cellcolor{donkergroen!25}1 &  \cellcolor{donkergroen!25}1 &  \cellcolor{donkerrood!65}0 &  \cellcolor{donkerrood!65}0 &  \cellcolor{donkerrood!65}0 &  \cellcolor{donkergroen!25}1 &  \cellcolor{donkerrood!65}0 &  \cellcolor{donkergroen!25}1 &  \cellcolor{donkergroen!25}1 &  \cellcolor{donkerrood!65}0 &  \cellcolor{donkergroen!25}1 &  \cellcolor{donkergroen!25}1 \\
  \cellcolor{donkerrood!65}0 &  \cellcolor{donkerrood!65}0 &  \cellcolor{donkerrood!65}0 &  \cellcolor{donkerrood!65}0 &  \cellcolor{donkerrood!65}0 &  \cellcolor{donkergroen!25}1 &  \cellcolor{donkerrood!65}0 &  \cellcolor{donkerrood!65}0 &  \cellcolor{donkerrood!65}0 &  \cellcolor{donkerrood!65}0 &  \cellcolor{donkerrood!65}0 &  \cellcolor{donkerrood!65}0 &  \cellcolor{donkergroen!25}1 &  \cellcolor{donkerrood!65}0 &  \cellcolor{donkerrood!65}0 &  \cellcolor{donkerrood!65}0 &  \cellcolor{donkergroen!25}1 &  \cellcolor{donkerrood!65}0 &  \cellcolor{donkergroen!25}1 &  \cellcolor{donkergroen!25}1 &  \cellcolor{donkerrood!65}0 &  \cellcolor{donkergroen!25}1 &  \cellcolor{donkergroen!25}1 &  \cellcolor{donkergroen!25}1 \\
  \cellcolor{donkerrood!65}0 &  \cellcolor{donkerrood!65}0 &  \cellcolor{donkerrood!65}0 &  \cellcolor{donkerrood!65}0 &  \cellcolor{donkerrood!65}0 &  \cellcolor{donkerrood!65}0 &  \cellcolor{donkergroen!25}1 &  \cellcolor{donkerrood!65}0 &  \cellcolor{donkerrood!65}0 &  \cellcolor{donkerrood!65}0 &  \cellcolor{donkerrood!65}0 &  \cellcolor{donkerrood!65}0 &  \cellcolor{donkerrood!65}0 &  \cellcolor{donkerrood!65}0 &  \cellcolor{donkerrood!65}0 &  \cellcolor{donkergroen!25}1 &  \cellcolor{donkerrood!65}0 &  \cellcolor{donkergroen!25}1 &  \cellcolor{donkergroen!25}1 &  \cellcolor{donkerrood!65}0 &  \cellcolor{donkergroen!25}1 &  \cellcolor{donkergroen!25}1 &  \cellcolor{donkergroen!25}1 &  \cellcolor{donkergroen!25}1 \\
  \cellcolor{donkerrood!65}0 &  \cellcolor{donkerrood!65}0 &  \cellcolor{donkerrood!65}0 &  \cellcolor{donkerrood!65}0 &  \cellcolor{donkerrood!65}0 &  \cellcolor{donkerrood!65}0 &  \cellcolor{donkerrood!65}0 &  \cellcolor{donkergroen!25}1 &  \cellcolor{donkerrood!65}0 &  \cellcolor{donkerrood!65}0 &  \cellcolor{donkerrood!65}0 &  \cellcolor{donkerrood!65}0 &  \cellcolor{donkerrood!65}0 &  \cellcolor{donkerrood!65}0 &  \cellcolor{donkergroen!25}1 &  \cellcolor{donkerrood!65}0 &  \cellcolor{donkergroen!25}1 &  \cellcolor{donkergroen!25}1 &  \cellcolor{donkerrood!65}0 &  \cellcolor{donkergroen!25}1 &  \cellcolor{donkergroen!25}1 &  \cellcolor{donkergroen!25}1 &  \cellcolor{donkerrood!65}0 &  \cellcolor{donkergroen!25}1 \\
  \cellcolor{donkerrood!65}0 &  \cellcolor{donkerrood!65}0 &  \cellcolor{donkerrood!65}0 &  \cellcolor{donkerrood!65}0 &  \cellcolor{donkerrood!65}0 &  \cellcolor{donkerrood!65}0 &  \cellcolor{donkerrood!65}0 &  \cellcolor{donkerrood!65}0 &  \cellcolor{donkergroen!25}1 &  \cellcolor{donkerrood!65}0 &  \cellcolor{donkerrood!65}0 &  \cellcolor{donkerrood!65}0 &  \cellcolor{donkerrood!65}0 &  \cellcolor{donkergroen!25}1 &  \cellcolor{donkerrood!65}0 &  \cellcolor{donkergroen!25}1 &  \cellcolor{donkergroen!25}1 &  \cellcolor{donkerrood!65}0 &  \cellcolor{donkergroen!25}1 &  \cellcolor{donkergroen!25}1 &  \cellcolor{donkergroen!25}1 &  \cellcolor{donkerrood!65}0 &  \cellcolor{donkerrood!65}0 &  \cellcolor{donkergroen!25}1 \\
  \cellcolor{donkerrood!65}0 &  \cellcolor{donkerrood!65}0 &  \cellcolor{donkerrood!65}0 &  \cellcolor{donkerrood!65}0 &  \cellcolor{donkerrood!65}0 &  \cellcolor{donkerrood!65}0 &  \cellcolor{donkerrood!65}0 &  \cellcolor{donkerrood!65}0 &  \cellcolor{donkerrood!65}0 &  \cellcolor{donkergroen!25}1 &  \cellcolor{donkerrood!65}0 &  \cellcolor{donkerrood!65}0 &  \cellcolor{donkergroen!25}1 &  \cellcolor{donkerrood!65}0 &  \cellcolor{donkergroen!25}1 &  \cellcolor{donkergroen!25}1 &  \cellcolor{donkerrood!65}0 &  \cellcolor{donkergroen!25}1 &  \cellcolor{donkergroen!25}1 &  \cellcolor{donkergroen!25}1 &  \cellcolor{donkerrood!65}0 &  \cellcolor{donkerrood!65}0 &  \cellcolor{donkerrood!65}0 &  \cellcolor{donkergroen!25}1 \\
  \cellcolor{donkerrood!65}0 &  \cellcolor{donkerrood!65}0 &  \cellcolor{donkerrood!65}0 &  \cellcolor{donkerrood!65}0 &  \cellcolor{donkerrood!65}0 &  \cellcolor{donkerrood!65}0 &  \cellcolor{donkerrood!65}0 &  \cellcolor{donkerrood!65}0 &  \cellcolor{donkerrood!65}0 &  \cellcolor{donkerrood!65}0 &  \cellcolor{donkergroen!25}1 &  \cellcolor{donkerrood!65}0 &  \cellcolor{donkerrood!65}0 &  \cellcolor{donkergroen!25}1 &  \cellcolor{donkergroen!25}1 &  \cellcolor{donkerrood!65}0 &  \cellcolor{donkergroen!25}1 &  \cellcolor{donkergroen!25}1 &  \cellcolor{donkergroen!25}1 &  \cellcolor{donkerrood!65}0 &  \cellcolor{donkerrood!65}0 &  \cellcolor{donkerrood!65}0 &  \cellcolor{donkergroen!25}1 &  \cellcolor{donkergroen!25}1 \\
  \cellcolor{donkerrood!65}0 &  \cellcolor{donkerrood!65}0 &  \cellcolor{donkerrood!65}0 &  \cellcolor{donkerrood!65}0 &  \cellcolor{donkerrood!65}0 &  \cellcolor{donkerrood!65}0 &  \cellcolor{donkerrood!65}0 &  \cellcolor{donkerrood!65}0 &  \cellcolor{donkerrood!65}0 &  \cellcolor{donkerrood!65}0 &  \cellcolor{donkerrood!65}0 &  \cellcolor{donkergroen!25}1 &  \cellcolor{donkergroen!25}1 &  \cellcolor{donkergroen!25}1 &  \cellcolor{donkergroen!25}1 &  \cellcolor{donkergroen!25}1 &  \cellcolor{donkergroen!25}1 &  \cellcolor{donkergroen!25}1 &  \cellcolor{donkergroen!25}1 &  \cellcolor{donkergroen!25}1 &  \cellcolor{donkergroen!25}1 &  \cellcolor{donkergroen!25}1 &  \cellcolor{donkergroen!25}1 &  \cellcolor{donkerrood!65}0 \\
\end{tabular}
\caption{\small A generator matrix of the extended binary Golay code (i.e., the $2^{12}$ codewords are sums mod~$2$ of the rows of this matrix). The extended binary Golay code was used for error-correction in the \emph{Voyager} missions to Jupiter and Saturn~\cite{wicker}.} \label{golaygenmatsam}
\end{figure}

For unrestricted (non-constant weight) binary codes, the bound~$A_2(20,8)\leq256$ was obtained by Gijswijt, Mittelmann and Schrijver~\cite{semidef}, implying that the quadruply shortened extended binary Golay code of size~$256$ is optimal. Two unrestricted codes~$C,D \subseteq  \{0,1\}^n$ are equivalent if there is a~$g \in S_2^n \rtimes S_n$ such that~$g \cdot C=D$.  Up to equivalence the optimal binary codes attaining $A_2(24-i,8)=2^{12-i}$ for $i=0,1,2,3$ are unique, namely they are the~$i$ times shortened extended binary Golay codes \cite{brouwer2}. We show that there exist several nonequivalent optimal codes achieving~$A_2(20,8)=256$.  We classify such codes under the additional condition that all distances are divisible by~$4$, and find~$15$ such codes. We also show that there exist such codes with not all distances divisible by~$4$.

In Chapter~\ref{leechap} we consider \emph{Lee codes}. Fix three integers~$q,n,d \in \N$ and define~$Q:=\Z_q$ (the cyclic group of order~$q$). For two words~$u,v \in Q^n $, their \emph{Lee distance} is $\sum_{i=1}^n \min\{|u_i-v_i|, q-|u_i-v_i| \}$. The \emph{minimum Lee distance} of a code~$C \subseteq Q^n$ is the minimum Lee distance between any two distinct elements of~$C$. Let~$A^L_q(n,d)$ denote the maximum size of a code~$C \subseteq Q^n$ with minimum Lee distance at least~$d$. We give an SDP upper bound based on triples of codewords and show that it can be computed efficiently, using the symmetry reduction method of Chapter~\ref{orbitgroupmon}. This finally yields new upper bounds on~$A_q^L(n,d)$ for several triples~$(q,n,d)$. 

Chapter~\ref{shannonchap} is about the Shannon capacity of circular graphs.  For any graph~$G=(V,E)$ and~$n \in \N$, the \emph{$n$-th strong product power} $G^{\boxtimes n}$ is the  graph with vertex set~$V^n$ in which two distinct vertices~$(u_1,\ldots,u_n)$ and~$(v_1,\ldots,v_n)$ of~$G^{\boxtimes n}$ are adjacent if and only if for each~$i \in \{1,\ldots,n\}$ one has either~$u_i  = v_i$ or~$u_i v_i \in E$. The \emph{Shannon capacity} of~$G$ is defined as
$$
\Theta(G):=\sup_{n \in \N} \sqrt[n]{\alpha(G^{\boxtimes n})},
$$
where for any graph~$G$, the maximum cardinality of an independent set in~$G$ (a set of vertices, no two of which are adjacent) is denoted by~$\alpha(G)$.  The circular graph $C_{d,q}$ is the graph with vertex set~$\Z_q$, in which two distinct vertices are adjacent if and only if their distance (mod~$q$) is strictly less than~$d$. The value of~$\alpha(C_{d,q}^{\boxtimes n})$ (for fixed~$n$) and~$\Theta(C_{d,q})$ can be seen to only depend on the quotient~$q/d$.  We show that the function~$q/d \mapsto \Theta(C_{d,q})$ is continuous at \emph{integer} points~$q/d \geq 3$. It implies that also the function~$q/d \mapsto \vartheta(C_{d,q})$, Lov\'asz's upper bound on~$\Theta(C_{d,q})$~\cite{lovasz}, is continuous at these points --- see Figure~\ref{thetafigchap2}.

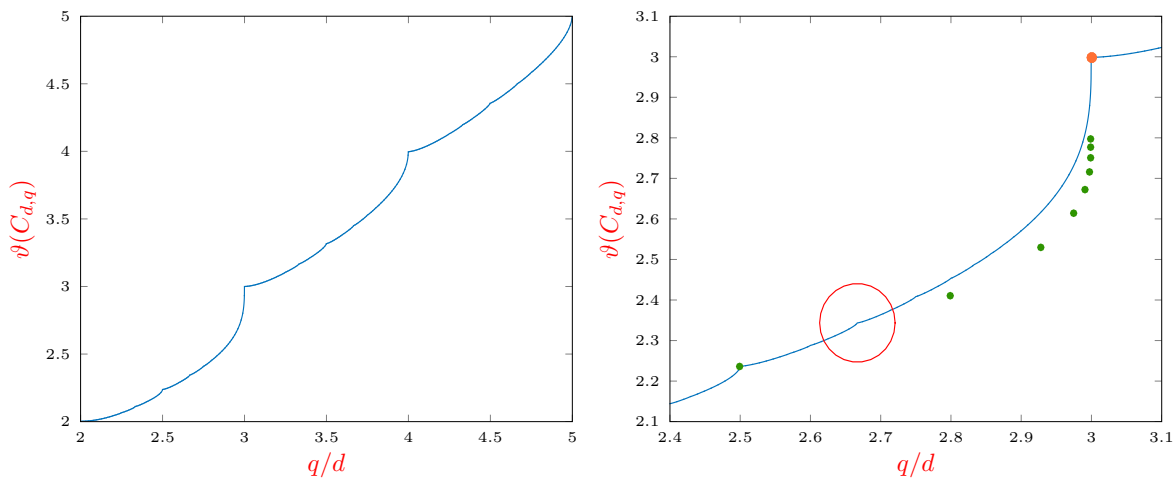
\begin{figure}[ht]
   \begin{subfigure}{.495\linewidth}
\centering\scalebox{.8315}{
\begin{tikzpicture}[scale=1.135441]
    \pgfmathsetlengthmacro\MajorTickLength{
      \pgfkeysvalueof{/pgfplots/major tick length} * 0.65
    }
    \begin{axis}[enlargelimits=false,axis on top,xlabel ={\footnotesize\color{red}$q/d$}, ylabel = {\footnotesize\color{red}$\vartheta(C_{d,q})$}, 
                 xtick={2,2.5,3,3.5,4,4.5,5},ytick={2,2.5,3,3.5,4,4.5,5},
  major tick length=\MajorTickLength,
        x label style={
        at={(axis description cs:0.5,-0.05)},
        anchor=north,
      },
      y label style={
        at={(axis description cs:-0.07,.5)}, 
        anchor=south,
      }, 
                ]
                                
       \addplot graphics
       [xmin=2,xmax=5,ymin=2,ymax=5,
      includegraphics={trim=5cmm 10.5cm 4.5cm 10.105cm,clip}]{lexplot2-5_step5000box_noaxis.pdf};
    \end{axis}
\end{tikzpicture}}
\end{subfigure} \hspace{-.025\linewidth}
   \begin{subfigure}{0.495\linewidth}
\centering\scalebox{.8315}{
\begin{tikzpicture}[scale=1.135441]
    \pgfmathsetlengthmacro\MajorTickLength{
      \pgfkeysvalueof{/pgfplots/major tick length} * 0.65
    }
    \begin{axis}[enlargelimits=false,axis on top,xlabel ={\footnotesize\color{red}$q/d$}, ylabel = {\footnotesize\color{red}$\vartheta(C_{d,q})$}, 
                 xtick={2.4,2.5,2.6,2.7,2.8,2.9,3,3.1},ytick={2.1,2.2,2.3,2.4,2.5,2.6,2.7,2.8,2.9,3,3.1},
  major tick length=\MajorTickLength,
        x label style={
        at={(axis description cs:0.5,-0.05)},
        anchor=north,
      },
      y label style={
        at={(axis description cs:-0.07,.5)}, 
        anchor=south,
      }, 
                ]
                                
       \addplot graphics
       [xmin=2.4,xmax=3.1,ymin=2.1,ymax=3.1,
      includegraphics={trim=5cmm 10.5cm 4.5cm 10.105cm,clip}]{plot24-31markbox_noaxis.pdf};
    \end{axis}
\end{tikzpicture}}
\end{subfigure}     
\caption{\small Two graphs of the function~$q/d \mapsto \vartheta(C_{d,q})$. The green points (converging to the orange point $(3,3)$) are some of our lower bounds on~$\Theta(C_{d,q})$, which are used to prove left-continuity of~$q/d \mapsto \Theta(C_{d,q})$ at integers~$\geq 3$. \label{thetafigchap2}}
   \end{figure}

Left-continuity we derive from the following  result (proved using an explicit construction).
\begin{theoremnn}
For each~$r,n\in \N$ with~$r \geq 3$, we have 
$$
\max_{\frac{q}{d} < r } \alpha(C_{d,q}^{\boxtimes n}) = \frac{1+r^n(r-2)}{r-1}.
$$ 
\end{theoremnn}

We also prove that the independent set achieving~$\alpha(C_{5,14}^{\boxtimes 3})=14$, one of the independent sets used in our proof, is unique up to Lee equivalence.  Here two sets~$C,D \subseteq \Z_q^n$ are \emph{Lee equivalent} if there is a~$g \in D_q^n \rtimes S_n$ with~$g \cdot C = D$, where~$D_q$ is the dihedral group of order~$2q$. We adapt our SDP upper bound for Lee codes to compute upper bounds on~$\alpha(C_{d,q}^{\boxtimes n})$. Finally, we give a new lower bound of~$367^{1/5}>3.2578$ on the Shannon capacity of the~$7$-cycle. 
 
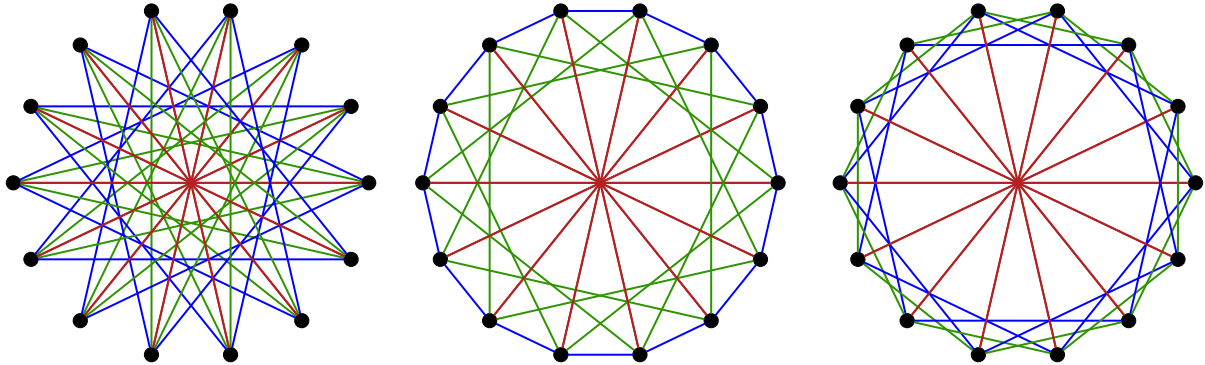
\begin{figure}[ht]
   \begin{subfigure}{0.32\linewidth}\scalebox{0.935}{
  \begin{tikzpicture} 
    \begin{scope} [vertex style/.style={draw,
                                       circle,
                                       minimum size=2mm,
                                       inner sep=0pt,
                                       outer sep=0pt, fill}] 
      \path \foreach \i in {0,...,13}{%
       (25.714*\i:2.5) coordinate[vertex style] (a\i)}
       ; 
    \end{scope}

     \begin{scope} [edge style/.style={draw=black}]
       \foreach \i  in {0,...,13}{%
       \pgfmathtruncatemacro{\nexta}{mod(\i+5,14)} 
       \pgfmathtruncatemacro{\nextab}{mod(\i+6,14)}   
       \pgfmathtruncatemacro{\nextabc}{mod(\i+7,14)}     
       \draw[edge style,thick,donkergroen] (a\i)--(a\nextab);
       \draw[edge style,thick,blue] (a\i)--(a\nexta);
       \draw[edge style,thick,firebrick] (a\i)--(a\nextabc);       
       }  
     \end{scope}
  \end{tikzpicture} }
\end{subfigure}\hspace{0.01\textwidth}
   \begin{subfigure}{0.32\linewidth}\scalebox{0.935}{
     \begin{tikzpicture} 
    \begin{scope} [vertex style/.style={draw,
                                       circle,
                                       minimum size=2mm,
                                       inner sep=0pt,
                                       outer sep=0pt, fill}] 
      \path \foreach \i in {0,...,13}{%
       (25.714*\i:2.5) coordinate[vertex style] (a\i)}
       ; 

    \end{scope}

     \begin{scope} [edge style/.style={draw=black}]
       \foreach \i  in {0,...,13}{%
       \pgfmathtruncatemacro{\nexta}{mod(\i+1,14)} 
       \pgfmathtruncatemacro{\nextab}{mod(\i+4,14)}   
       \pgfmathtruncatemacro{\nextabc}{mod(\i+7,14)}     
       \draw[edge style,thick, donkergroen] (a\i)--(a\nextab);
       \draw[edge style,thick, blue] (a\i)--(a\nexta);
       \draw[edge style,thick, firebrick] (a\i)--(a\nextabc);       
       }  
     \end{scope}

  \end{tikzpicture} }
\end{subfigure}   \hspace{0.01\textwidth}
   \begin{subfigure}{0.32\linewidth}\scalebox{0.935}{
     \begin{tikzpicture} 
    \begin{scope} [vertex style/.style={draw,
                                       circle,
                                       minimum size=2mm,
                                       inner sep=0pt,
                                       outer sep=0pt, fill}] 
      \path \foreach \i in {0,...,13}{%
       (25.714*\i:2.5) coordinate[vertex style] (a\i)}
       ; 
    \end{scope}

     \begin{scope} [edge style/.style={draw=black}]
       \foreach \i  in {0,...,13}{%
       \pgfmathtruncatemacro{\nexta}{mod(\i+2,14)} 
       \pgfmathtruncatemacro{\nextab}{mod(\i+3,14)}   
       \pgfmathtruncatemacro{\nextabc}{mod(\i+7,14)}     
       \draw[edge style,thick, donkergroen] (a\i)--(a\nexta);
       \draw[edge style,thick, blue] (a\i)--(a\nextab);
       \draw[edge style,thick, firebrick] (a\i)--(a\nextabc);       
       }  
     \end{scope}

  \end{tikzpicture} }
   \end{subfigure}
   \caption{\small Three graphs used in the proof to show that the independent set achieving~$\alpha(C_{5,14}^{\boxtimes 3})=14$ is unique up to Lee equivalence.}
   \end{figure} 
\vspace{-12pt}

\subsubsection*{An overarching theme: independent sets in graph products.}

  For any graph~$G=(V,E)$,   define the number
\begin{align}
\alpha_d(G) := \max \{|U| \,\, | \,\, U \subseteq V,\,\, d_G(u,v) \geq d \text{ for all distinct } u,v\in U  \}. 
\end{align} 
Here~$d_G(u,v)$ denotes the smallest length (in edges) of a path between~$u$ and~$v$ in~$G$. So~$\alpha_2(G)=\alpha(G)$. Let~$K_q$ denote the complete graph on~$q$ vertices, and let~$C_q$ be the circuit on~$q$ vertices. Then
\begin{align*}
A_q(n,d)&= \alpha_d(K_q^{\square n}), \phantom{veelletters}
\\A_q^L(n,d)& = \alpha_d(C_q^{\square n}),
\\ A_q^{L_{\infty}}(n,d) := \alpha(C_{d,q}^{\boxtimes n}) &= \alpha_d(C_q^{\boxtimes n}).
\end{align*}
Here~$G^{\square n}$ denotes the \emph{$n$-th Cartesian product power of~$G$}: the graph with vertex set~$V(G)^n$ in which two distinct vertices~$(u_1,\ldots,u_n)$ and~$(v_1,\ldots,v_n)$  are adjacent if and only if there is an~$ i \in \{1,\ldots,n\}$ such that~$u_i v_i \in E$, and~$u_j=v_j$ for all~$j \neq i$.

So the main objects studied in this thesis are of the form~$\alpha_d(G^n)$, where~$G \in \{C_q, K_q\}$, and where~$G^n$ denotes either~$G^{\boxtimes n}$ or~$G^{\square  n}$.  Moreover, we have~$A(n,d,w)=\alpha_d(H)$, where~$H$ is the subgraph of~$K_2^{\square n}$ induced by the vertices~$(u_1,\ldots,u_n)$ with~$u_i=1$ for exactly~$w$ indices~$i \in \{1,\ldots,n\}$  (where the vertices of~$K_2$ are labeled with~$0$ and~$1$).

\begin{otherlanguage}{dutch}
\Chapter{Samenvatting}{Nieuwe methoden in coderingstheorie: foutencorrigerende codes en de Shannoncapaciteit}
\markboth{Samenvatting}{Samenvatting}
 Foutencorrigerende codes worden al bestudeerd sinds~$1948$, toen Claude Shannon zijn invloedrijke artikel \emph{A Mathematical Theory of Communication} publiceerde \cite{shannonseminal}.
 Laat~$q,n,d$ positieve gehele getallen zijn. De verzameling~$Q:=\{0,\ldots,q-1\}$ is ons \emph{alfabet}. We identificeren elementen van~$Q^n$ met \emph{woorden} van lengte~$n$ die bestaan uit letters uit~$Q$. Een \emph{code van lengte~$n$} is een deelverzameling van~$Q^n$. De \emph{Hammingafstand} tussen twee woorden~$u,v$ is het aantal~$i$ met~$u_i \neq v_i$. De \emph{minimumafstand} van een code~$C$ is de kleinste Hammingafstand die voorkomt tussen twee verschillende elementen van~$C$. We formuleren nu de centrale vraag uit  de coderingstheorie.
\begin{align}\label{sumalignnl}
\text{Wat is de maximale grootte van een code~$C$ met minimumafstand ten minste~$ d$?}
\end{align}
De  maximale grootte in~\eqref{sumalignnl} wordt aangegeven met~$A_q(n,d)$. Een code~$C$ met minimumafstand~$d:=2e+1$ (voor een geheel getal~$e$) heet~$e$-\emph{foutencorrigerend}. Als ---bijvoorbeeld door storing--- een codewoord uit~$C$ veranderd wordt in ten hoogste~$e$ posities, kunnen we het originele codewoord terugvinden door het codewoord te nemen dat het dichtst bij het veranderde woord  in Hammingafstand zit. Dit principe wordt gebruikt voor het corrigeren van transmissiefouten  in communicatiesystemen.

De getallen~$A_q(n,d)$ zijn over het algemeen moeilijk te berekenen. Voor veel~$q,n,d$ zijn er alleen boven- en ondergrenzen op~$A_q(n,d)$ bekend. Expliciete codes geven ondergrenzen op~$A_q(n,d)$ en kunnen gebruikt worden voor foutencorrectie. Een klassieke bovengrens op~$A_q(n,d)$ is Delsartes lineairprogrammeergrens~\cite{delsarte}. Deze grens kan beschreven worden als een semidefiniet programma (SDP) gebaseerd op paren van codewoorden.

In dit proefschrift proberen we bovengrenzen op~$A_q(n,d)$ te verbeteren. We geven een SDP gebaseerd op viertallen van codewoorden. Het optimalisatieprobleem is zeer symmetrisch: we kunnen aannemen dat de optimale oplossing invariant is onder de groepswerking van de groep van afstandsbewarende permutaties van~$Q^n$:  het kransproduct~$S_q^n \rtimes S_n$. Vanwege de symmetrie van het probleem, kan het SDP gereduceerd worden tot een grootte polynomiaal begrensd door~$n$.

 In Hoofdstuk~\ref{orbitgroupmon} geven we een algemene methode voor symmetriereductie, gebaseerd op representatietheorie. Stel dat~$G$ een eindige groep is die werkt op een eindige verzameling~$Z$ en laat~$n \in \N$. We geven een reductie van~$Z^n \times Z^n$ matrices die invariant zijn onder de simultane werking van de groep~$G^n \rtimes S_n$ op de rijen en kolommen. In de reductie gaan we uit van een reductie van~$Z \times Z$ matrices die invariant zijn onder de simultane werking van~$G$ op de rijen en kolommen. 

 In Hoofdstuk~\ref{onsartchap} passen we deze algemene methode toe om de grootte te reduceren van het genoemde SDP  gebaseerd op viertallen van codewoorden voor het berekenen van bovengrenzen op~$A_q(n,d)$.  Met deze methode verbeteren we bekende bovengrenzen voor vijf drietallen~$(q,n,d)$. 
 
 In Hoofdstuk~\ref{divchap} verkennen we andere methoden om bovengrenzen op~$A_q(n,d)$ te vinden, gebaseerd op combinatorische  deelbaarheidsargumenten.  De methoden geven nieuwe bovengrenzen voor vier drietallen~$(q,n,d)$.
Ons meest prominente resultaat in deze richting is de volgende grens, die in sommige gevalen een verscherping oplevert van een grens die ge\"impliceerd wordt door de Plotkingrens~\cite{plotkinoriginal}. 
\begin{stellingnn}
Stel dat~$q,n,d,m$ positieve gehele getallen zijn met $q\geq 2$, dat~$d=m(qd-(q-1)(n-1))$, en dat~$n-d$ geen deler is van~$m(n-1)$. Als~$r \in \{1,\ldots,q-1\}$ voldoet aan
\begin{align*}
n(n-1-d)(r-1)r <  (q-r+1)(qm(q+r-2)-2r),
\end{align*}
 dan geldt~$A_q(n,d) < q^2m -r$. 
\end{stellingnn}

In Hoofdstuk~\ref{cw4chap} bekijken we (binaire) \emph{constant-gewicht-codes}. Hier is het alfabet $\{0,1\}$.  Het \emph{gewicht} van een woord is het aantal $1$'en dat het bevat. Voor~$n,d,w \in \N$, noteren we met~$A(n,d,w)$ de maximale grootte van een code~$C \subseteq \{0,1\}^n$  met minimumafstand ten minste~$d$ en waarin ieder codewoord gewicht~$w$ heeft. (Een dergelijke code heet een `constant-gewicht-code' met gewicht~$w$.)  Met een SDP gebaseerd op viertallen van codewoorden en een symmetriereductie met de methode van Hoofdstuk~\ref{orbitgroupmon}, vinden we een groot aantal nieuwe bovengrenzen op~$A(n,d,w)$. Twee van onze bovengrenzen zijn gelijk aan de best bekende ondergrenzen zodat we~$A(n,d,w)$ precies kennen: $A(22,8,10)=616$ en~$A(22,8,11)=672$.

In Hoofdstuk~\ref{cu17chap} bewijzen we met hulp van de duale oplossingen van de SDP-uitvoer dat de optimale constant-gewicht-codes die de waarden~$A(23,8,11) = 1288$, $A(22,8,10)=616$ en~$A(22,8,11)=672$ aantonen uniek zijn, op co\"ordinaatpermutaties na. De genoemde unieke constant-gewicht-codes kunnen uit de binaire Golay-code ---een beroemde code met goede foutencorrigerende eigenschappen--- gehaald worden, door deelcodes te nemen en co\"ordinaten weg te gooien.

\begin{figure}[ht]
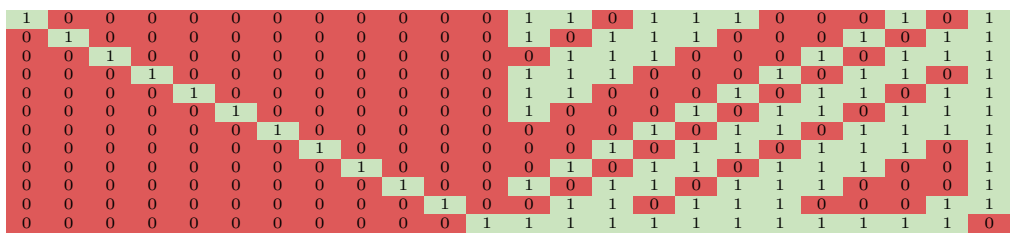

\centering
\tiny
\begin{tabular}{rrrrrrrrrrrrrrrrrrrrrrrr}
  \cellcolor{donkergroen!25}1 &  \cellcolor{donkerrood!65}0 &  \cellcolor{donkerrood!65}0 &  \cellcolor{donkerrood!65}0 &  \cellcolor{donkerrood!65}0 &  \cellcolor{donkerrood!65}0 &  \cellcolor{donkerrood!65}0 &  \cellcolor{donkerrood!65}0 &  \cellcolor{donkerrood!65}0 &  \cellcolor{donkerrood!65}0 &  \cellcolor{donkerrood!65}0 &  \cellcolor{donkerrood!65}0 &
   \cellcolor{donkergroen!25}1 &  \cellcolor{donkergroen!25}1 &  \cellcolor{donkerrood!65}0 &  \cellcolor{donkergroen!25}1 &  \cellcolor{donkergroen!25}1 &  \cellcolor{donkergroen!25}1 &  \cellcolor{donkerrood!65}0 &  \cellcolor{donkerrood!65}0 &  \cellcolor{donkerrood!65}0 &  \cellcolor{donkergroen!25}1 &  \cellcolor{donkerrood!65}0 &  \cellcolor{donkergroen!25}1 \\
  \cellcolor{donkerrood!65}0 &  \cellcolor{donkergroen!25}1 &  \cellcolor{donkerrood!65}0 &  \cellcolor{donkerrood!65}0 &  \cellcolor{donkerrood!65}0 &  \cellcolor{donkerrood!65}0 &  \cellcolor{donkerrood!65}0 &  \cellcolor{donkerrood!65}0 &  \cellcolor{donkerrood!65}0 &  \cellcolor{donkerrood!65}0 &  \cellcolor{donkerrood!65}0 &  \cellcolor{donkerrood!65}0 &  \cellcolor{donkergroen!25}1 &  \cellcolor{donkerrood!65}0 &  \cellcolor{donkergroen!25}1 &  \cellcolor{donkergroen!25}1 &  \cellcolor{donkergroen!25}1 &  \cellcolor{donkerrood!65}0 &  \cellcolor{donkerrood!65}0 &  \cellcolor{donkerrood!65}0 &  \cellcolor{donkergroen!25}1 &  \cellcolor{donkerrood!65}0 &  \cellcolor{donkergroen!25}1 &  \cellcolor{donkergroen!25}1 \\
  \cellcolor{donkerrood!65}0 &  \cellcolor{donkerrood!65}0 &  \cellcolor{donkergroen!25}1 &  \cellcolor{donkerrood!65}0 &  \cellcolor{donkerrood!65}0 &  \cellcolor{donkerrood!65}0 &  \cellcolor{donkerrood!65}0 &  \cellcolor{donkerrood!65}0 &  \cellcolor{donkerrood!65}0 &  \cellcolor{donkerrood!65}0 &  \cellcolor{donkerrood!65}0 &  \cellcolor{donkerrood!65}0 &  \cellcolor{donkerrood!65}0 &  \cellcolor{donkergroen!25}1 &  \cellcolor{donkergroen!25}1 &  \cellcolor{donkergroen!25}1 &  \cellcolor{donkerrood!65}0 &  \cellcolor{donkerrood!65}0 &  \cellcolor{donkerrood!65}0 &  \cellcolor{donkergroen!25}1 &  \cellcolor{donkerrood!65}0 &  \cellcolor{donkergroen!25}1 &  \cellcolor{donkergroen!25}1 &  \cellcolor{donkergroen!25}1 \\
  \cellcolor{donkerrood!65}0 &  \cellcolor{donkerrood!65}0 &  \cellcolor{donkerrood!65}0 &  \cellcolor{donkergroen!25}1 &  \cellcolor{donkerrood!65}0 &  \cellcolor{donkerrood!65}0 &  \cellcolor{donkerrood!65}0 &  \cellcolor{donkerrood!65}0 &  \cellcolor{donkerrood!65}0 &  \cellcolor{donkerrood!65}0 &  \cellcolor{donkerrood!65}0 &  \cellcolor{donkerrood!65}0 &  \cellcolor{donkergroen!25}1 &  \cellcolor{donkergroen!25}1 &  \cellcolor{donkergroen!25}1 &  \cellcolor{donkerrood!65}0 &  \cellcolor{donkerrood!65}0 &  \cellcolor{donkerrood!65}0 &  \cellcolor{donkergroen!25}1 &  \cellcolor{donkerrood!65}0 &  \cellcolor{donkergroen!25}1 &  \cellcolor{donkergroen!25}1 &  \cellcolor{donkerrood!65}0 &  \cellcolor{donkergroen!25}1 \\
  \cellcolor{donkerrood!65}0 &  \cellcolor{donkerrood!65}0 &  \cellcolor{donkerrood!65}0 &  \cellcolor{donkerrood!65}0 &  \cellcolor{donkergroen!25}1 &  \cellcolor{donkerrood!65}0 &  \cellcolor{donkerrood!65}0 &  \cellcolor{donkerrood!65}0 &  \cellcolor{donkerrood!65}0 &  \cellcolor{donkerrood!65}0 &  \cellcolor{donkerrood!65}0 &  \cellcolor{donkerrood!65}0 &  \cellcolor{donkergroen!25}1 &  \cellcolor{donkergroen!25}1 &  \cellcolor{donkerrood!65}0 &  \cellcolor{donkerrood!65}0 &  \cellcolor{donkerrood!65}0 &  \cellcolor{donkergroen!25}1 &  \cellcolor{donkerrood!65}0 &  \cellcolor{donkergroen!25}1 &  \cellcolor{donkergroen!25}1 &  \cellcolor{donkerrood!65}0 &  \cellcolor{donkergroen!25}1 &  \cellcolor{donkergroen!25}1 \\
  \cellcolor{donkerrood!65}0 &  \cellcolor{donkerrood!65}0 &  \cellcolor{donkerrood!65}0 &  \cellcolor{donkerrood!65}0 &  \cellcolor{donkerrood!65}0 &  \cellcolor{donkergroen!25}1 &  \cellcolor{donkerrood!65}0 &  \cellcolor{donkerrood!65}0 &  \cellcolor{donkerrood!65}0 &  \cellcolor{donkerrood!65}0 &  \cellcolor{donkerrood!65}0 &  \cellcolor{donkerrood!65}0 &  \cellcolor{donkergroen!25}1 &  \cellcolor{donkerrood!65}0 &  \cellcolor{donkerrood!65}0 &  \cellcolor{donkerrood!65}0 &  \cellcolor{donkergroen!25}1 &  \cellcolor{donkerrood!65}0 &  \cellcolor{donkergroen!25}1 &  \cellcolor{donkergroen!25}1 &  \cellcolor{donkerrood!65}0 &  \cellcolor{donkergroen!25}1 &  \cellcolor{donkergroen!25}1 &  \cellcolor{donkergroen!25}1 \\
  \cellcolor{donkerrood!65}0 &  \cellcolor{donkerrood!65}0 &  \cellcolor{donkerrood!65}0 &  \cellcolor{donkerrood!65}0 &  \cellcolor{donkerrood!65}0 &  \cellcolor{donkerrood!65}0 &  \cellcolor{donkergroen!25}1 &  \cellcolor{donkerrood!65}0 &  \cellcolor{donkerrood!65}0 &  \cellcolor{donkerrood!65}0 &  \cellcolor{donkerrood!65}0 &  \cellcolor{donkerrood!65}0 &  \cellcolor{donkerrood!65}0 &  \cellcolor{donkerrood!65}0 &  \cellcolor{donkerrood!65}0 &  \cellcolor{donkergroen!25}1 &  \cellcolor{donkerrood!65}0 &  \cellcolor{donkergroen!25}1 &  \cellcolor{donkergroen!25}1 &  \cellcolor{donkerrood!65}0 &  \cellcolor{donkergroen!25}1 &  \cellcolor{donkergroen!25}1 &  \cellcolor{donkergroen!25}1 &  \cellcolor{donkergroen!25}1 \\
  \cellcolor{donkerrood!65}0 &  \cellcolor{donkerrood!65}0 &  \cellcolor{donkerrood!65}0 &  \cellcolor{donkerrood!65}0 &  \cellcolor{donkerrood!65}0 &  \cellcolor{donkerrood!65}0 &  \cellcolor{donkerrood!65}0 &  \cellcolor{donkergroen!25}1 &  \cellcolor{donkerrood!65}0 &  \cellcolor{donkerrood!65}0 &  \cellcolor{donkerrood!65}0 &  \cellcolor{donkerrood!65}0 &  \cellcolor{donkerrood!65}0 &  \cellcolor{donkerrood!65}0 &  \cellcolor{donkergroen!25}1 &  \cellcolor{donkerrood!65}0 &  \cellcolor{donkergroen!25}1 &  \cellcolor{donkergroen!25}1 &  \cellcolor{donkerrood!65}0 &  \cellcolor{donkergroen!25}1 &  \cellcolor{donkergroen!25}1 &  \cellcolor{donkergroen!25}1 &  \cellcolor{donkerrood!65}0 &  \cellcolor{donkergroen!25}1 \\
  \cellcolor{donkerrood!65}0 &  \cellcolor{donkerrood!65}0 &  \cellcolor{donkerrood!65}0 &  \cellcolor{donkerrood!65}0 &  \cellcolor{donkerrood!65}0 &  \cellcolor{donkerrood!65}0 &  \cellcolor{donkerrood!65}0 &  \cellcolor{donkerrood!65}0 &  \cellcolor{donkergroen!25}1 &  \cellcolor{donkerrood!65}0 &  \cellcolor{donkerrood!65}0 &  \cellcolor{donkerrood!65}0 &  \cellcolor{donkerrood!65}0 &  \cellcolor{donkergroen!25}1 &  \cellcolor{donkerrood!65}0 &  \cellcolor{donkergroen!25}1 &  \cellcolor{donkergroen!25}1 &  \cellcolor{donkerrood!65}0 &  \cellcolor{donkergroen!25}1 &  \cellcolor{donkergroen!25}1 &  \cellcolor{donkergroen!25}1 &  \cellcolor{donkerrood!65}0 &  \cellcolor{donkerrood!65}0 &  \cellcolor{donkergroen!25}1 \\
  \cellcolor{donkerrood!65}0 &  \cellcolor{donkerrood!65}0 &  \cellcolor{donkerrood!65}0 &  \cellcolor{donkerrood!65}0 &  \cellcolor{donkerrood!65}0 &  \cellcolor{donkerrood!65}0 &  \cellcolor{donkerrood!65}0 &  \cellcolor{donkerrood!65}0 &  \cellcolor{donkerrood!65}0 &  \cellcolor{donkergroen!25}1 &  \cellcolor{donkerrood!65}0 &  \cellcolor{donkerrood!65}0 &  \cellcolor{donkergroen!25}1 &  \cellcolor{donkerrood!65}0 &  \cellcolor{donkergroen!25}1 &  \cellcolor{donkergroen!25}1 &  \cellcolor{donkerrood!65}0 &  \cellcolor{donkergroen!25}1 &  \cellcolor{donkergroen!25}1 &  \cellcolor{donkergroen!25}1 &  \cellcolor{donkerrood!65}0 &  \cellcolor{donkerrood!65}0 &  \cellcolor{donkerrood!65}0 &  \cellcolor{donkergroen!25}1 \\
  \cellcolor{donkerrood!65}0 &  \cellcolor{donkerrood!65}0 &  \cellcolor{donkerrood!65}0 &  \cellcolor{donkerrood!65}0 &  \cellcolor{donkerrood!65}0 &  \cellcolor{donkerrood!65}0 &  \cellcolor{donkerrood!65}0 &  \cellcolor{donkerrood!65}0 &  \cellcolor{donkerrood!65}0 &  \cellcolor{donkerrood!65}0 &  \cellcolor{donkergroen!25}1 &  \cellcolor{donkerrood!65}0 &  \cellcolor{donkerrood!65}0 &  \cellcolor{donkergroen!25}1 &  \cellcolor{donkergroen!25}1 &  \cellcolor{donkerrood!65}0 &  \cellcolor{donkergroen!25}1 &  \cellcolor{donkergroen!25}1 &  \cellcolor{donkergroen!25}1 &  \cellcolor{donkerrood!65}0 &  \cellcolor{donkerrood!65}0 &  \cellcolor{donkerrood!65}0 &  \cellcolor{donkergroen!25}1 &  \cellcolor{donkergroen!25}1 \\
  \cellcolor{donkerrood!65}0 &  \cellcolor{donkerrood!65}0 &  \cellcolor{donkerrood!65}0 &  \cellcolor{donkerrood!65}0 &  \cellcolor{donkerrood!65}0 &  \cellcolor{donkerrood!65}0 &  \cellcolor{donkerrood!65}0 &  \cellcolor{donkerrood!65}0 &  \cellcolor{donkerrood!65}0 &  \cellcolor{donkerrood!65}0 &  \cellcolor{donkerrood!65}0 &  \cellcolor{donkergroen!25}1 &  \cellcolor{donkergroen!25}1 &  \cellcolor{donkergroen!25}1 &  \cellcolor{donkergroen!25}1 &  \cellcolor{donkergroen!25}1 &  \cellcolor{donkergroen!25}1 &  \cellcolor{donkergroen!25}1 &  \cellcolor{donkergroen!25}1 &  \cellcolor{donkergroen!25}1 &  \cellcolor{donkergroen!25}1 &  \cellcolor{donkergroen!25}1 &  \cellcolor{donkergroen!25}1 &  \cellcolor{donkerrood!65}0 \\
\end{tabular}
\caption{\small Een voortbrengersmatrix van de uitgebreide binaire Golay-code (dit betekent dat de $2^{12}$ codewoorden sommen mod~$2$ van de rijen van deze matrix zijn). Deze code werd gebruikt om transmissiefouten te corrigeren in de \emph{Voyager}-missies naar Jupiter and Saturnus~\cite{wicker}.} \label{golaygenmatsamnl}
\end{figure}

Voor `gewone' (niet-constant-gewicht-) codes hebben Gijswijt, Mittelmann en Schrijver met SDP aangetoond dat~$A_2(20,8)\leq 256$~\cite{semidef}. Dit impliceert dat de vier keer verkorte uitgebreide binaire Golay-code van grootte 256 optimaal is. Twee niet-constant-gewicht-codes~$C, D \subseteq  \{0,1\}$  zijn \emph{equivalent} als er een~$g \in S_2^n \rtimes S_n$ is met $g \cdot C = D$. Op equivalentie na zijn de optimale codes die~$A_2(24-i,8)=2^{12-i}$ bereiken, voor~$i=0,1,2,3$, uniek~\cite{brouwer2}. Dit zijn de $i$-keer verkorte uitgebreide binaire Golay-codes. Wij bewijzen dat er verschillene niet-equivalente codes zijn die~$A_2(20,8)=256$ bereiken. We classificeren zulke codes met de extra eis dat alle afstanden deelbaar zijn door~$4$: we vinden $15$ verschillende codes. We bewijzen ook dat er zulke codes bestaan met niet alle afstanden deelbaar door~$4$.

In Hoofdstuk~\ref{leechap} bekijken we \emph{Lee-codes}. Laat~$q,n,d \in \N$ en definieer~$Q:=\{0,\ldots,q-1\}$. De \emph{Lee-afstand} tussen twee woorden~$u,v$ is $\sum_{i=1}^n \min\{|u_i-v_i|, q-|u_i-v_i| \}$. De \emph{minimum-Lee-afstand} van een code~$C$ is de kleinste Lee-afstand die voorkomt tussen twee verschillende elementen van~$C$. We schrijven~$A^L_q(n,d)$ voor de maximale grootte van een code~$C \subseteq Q^n$ met minimum-Lee-afstand ten minste~$d$. We geven een SDP-bovengrens gebaseerd op drietallen van codewoorden en laten zien dat deze effici\"ent berekend kan worden, door middel van symmetriereducties met de methode van Hoofdstuk~\ref{orbitgroupmon}. Dit geeft nieuwe bovengrenzen op~$A_q^L(n,d)$ voor veel drietallen~$(q,n,d)$.

Hoofdstuk~\ref{shannonchap} gaat over de  Shannoncapaciteit van circulaire grafen. Zij~$G=(V,E)$ een graaf en~$n \in \N$. De \emph{$n$-de sterkproductmacht} $G^{\boxtimes n}$ is de graaf met puntenverzameling~$V^n$, en twee verschillende punten~$(u_1,\ldots,u_n)$ en~$(v_1,\ldots,v_n)$ van~$G^{\boxtimes n}$ zijn verbonden dan en slechts dan als voor iedere~$i \in \{1,\ldots,n\}$ geldt dat ofwel~$u_i  = v_i$ ofwel~$u_i v_i \in E$. De \emph{Shannoncapaciteit} van~$G$ is gedefinieerd als
$$
\Theta(G):=\sup_{n \in \N} \sqrt[n]{\alpha(G^{\boxtimes n})},
$$
waar voor iedere graaf~$G$ de maximale grootte van een onafhankelijke verzameling in~$G$ (een verzameling punten waarvan geen twee punten verbonden zijn door een lijn) genoteerd wordt met~$\alpha(G)$. De \emph{circulaire graaf} $C_{d,q}$ is de graaf met puntenverzameling~$\Z_q$ waarin twee verschillende punten verbonden zijn dan en slechts dan als hun afstand (mod~$q$) strikt kleiner is dan~$d$. Het is in te zien dat de waarde van~$\alpha(C_{d,q}^n)$ (voor vaste~$n$) en~$\Theta(C_{d,q})$ alleen van de breuk~$q/d$ afhangt.  
We bewijzen dat de functie~$q/d \mapsto \Theta(C_{d,q})$ continu is in \emph{gehele getallen}~$q/d \geq 3$. Dit impliceert dat ook de functie~$q/d \mapsto \vartheta(C_{d,q})$, Lov\'asz' bovengrens op $\Theta(C_{d,q})$ \cite{lovasz}, continu is in deze punten --- zie Figuur~\ref{thetafigchap2nl}. 
\begin{figure}[H]
   \begin{subfigure}{.495\linewidth}
\centering\scalebox{.8315}{
\begin{tikzpicture}[scale=1.135441]
    \pgfmathsetlengthmacro\MajorTickLength{
      \pgfkeysvalueof{/pgfplots/major tick length} * 0.65
    }
    \begin{axis}[enlargelimits=false,axis on top,xlabel ={\footnotesize\color{red}$q/d$}, ylabel = {\footnotesize\color{red}$\vartheta(C_{d,q})$}, 
                 xtick={2,2.5,3,3.5,4,4.5,5},ytick={2,2.5,3,3.5,4,4.5,5},
  major tick length=\MajorTickLength,
        x label style={
        at={(axis description cs:0.5,-0.05)},
        anchor=north,
      },
      y label style={
        at={(axis description cs:-0.07,.5)}, 
        anchor=south,
      }, 
                ]
                                
       \addplot graphics
       [xmin=2,xmax=5,ymin=2,ymax=5,
      includegraphics={trim=5cmm 10.5cm 4.5cm 10.105cm,clip}]{lexplot2-5_step5000box_noaxis.pdf};
    \end{axis}
\end{tikzpicture}}
\end{subfigure} \hspace{-.025\linewidth}
   \begin{subfigure}{0.495\linewidth}
\centering\scalebox{.8315}{
\begin{tikzpicture}[scale=1.135441]
    \pgfmathsetlengthmacro\MajorTickLength{
      \pgfkeysvalueof{/pgfplots/major tick length} * .65
    }
    \begin{axis}[enlargelimits=false,axis on top,xlabel ={\footnotesize\color{red}$q/d$}, ylabel = {\footnotesize\color{red}$\vartheta(C_{d,q})$}, 
                 xtick={2.4,2.5,2.6,2.7,2.8,2.9,3,3.1},ytick={2.1,2.2,2.3,2.4,2.5,2.6,2.7,2.8,2.9,3,3.1},
  major tick length=\MajorTickLength,
        x label style={
        at={(axis description cs:0.5,-0.05)},
        anchor=north,
      },
      y label style={
        at={(axis description cs:-0.07,.5)}, 
        anchor=south,
      }, 
                ]
                                
       \addplot graphics
       [xmin=2.4,xmax=3.1,ymin=2.1,ymax=3.1,
      includegraphics={trim=5cmm 10.5cm 4.5cm 10.105cm,clip}]{plot24-31markbox_noaxis.pdf};
    \end{axis}
\end{tikzpicture}}
\end{subfigure}        
\caption{\small Twee grafieken van de functie~$q/d \mapsto \vartheta(C_{d,q})$. De groene punten (die naar het oranje punt $(3,3)$ toebewegen) zijn een paar van onze ondergrenzen op~$\Theta(C_{d,q})$ uit het bewijs van de linkscontinu\"iteit van~$q/d \mapsto \Theta(C_{d,q})$.\label{thetafigchap2nl}}
   \end{figure}
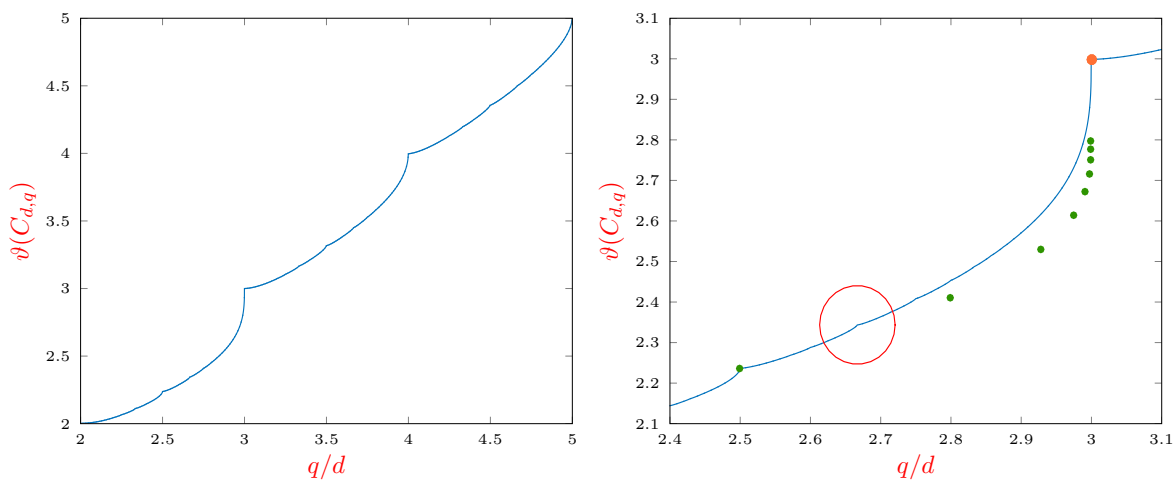

Linkscontinu\"iteit van~$q/d \mapsto \Theta(C_{d,q})$ leiden we af uit het volgende (bewezen met een expliciete constructie).
\begin{stellingnn}
Voor iedere~$r,n\in \N$ met~$r \geq 3$ geldt dat 
$$ 
\max_{\frac{q}{d} < r } \alpha(C_{d,q}^{\boxtimes n}) = \frac{1+r^n(r-2)}{r-1}.
$$
\end{stellingnn}

We bewijzen ook dat de onafhankelijke verzameling die~$\alpha(C_{5,14}^{\boxtimes 3})=14$ aantoont, een van de onafhankelijke verzamelingen gebruikt in ons bewijs, uniek is op Lee-equivalentie na. Hier zijn twee verzamelingen~$C,D \subseteq \Z_q^n$ \emph{Lee-equivalent} als er een~$g \in D_q^n \rtimes S_n$ is met~$g \cdot C = D$, waar~$D_q$ de dihedrale groep van orde~$2q$ is. We passen onze SDP-bovengrens voor Lee-codes aan om bovengrenzen op~$\alpha(C_{d,q}^{\boxtimes n})$ te berekenen. Tot slot geven we een nieuwe ondergrens van~$367^{1/5}>3.2578$ op de Shannoncapaciteit van de $7$-cykel.

\begin{figure}[ht]
    \begin{subfigure}{0.32\linewidth}\scalebox{0.935}{
  \begin{tikzpicture} 
    \begin{scope} [vertex style/.style={draw,
                                       circle,
                                       minimum size=2mm,
                                       inner sep=0pt,
                                       outer sep=0pt, fill}] 
      \path \foreach \i in {0,...,13}{%
       (25.714*\i:2.5) coordinate[vertex style] (a\i)}
       ; 
    \end{scope}

     \begin{scope} [edge style/.style={draw=black}]
       \foreach \i  in {0,...,13}{%
       \pgfmathtruncatemacro{\nexta}{mod(\i+5,14)} 
       \pgfmathtruncatemacro{\nextab}{mod(\i+6,14)}   
       \pgfmathtruncatemacro{\nextabc}{mod(\i+7,14)}     
       \draw[edge style,thick,donkergroen] (a\i)--(a\nextab);
       \draw[edge style,thick,blue] (a\i)--(a\nexta);
       \draw[edge style,thick,firebrick] (a\i)--(a\nextabc);       
       }  
     \end{scope}
  \end{tikzpicture} }
\end{subfigure}\hspace{0.01\textwidth}
   \begin{subfigure}{0.32\linewidth}\scalebox{0.935}{
     \begin{tikzpicture} 
    \begin{scope} [vertex style/.style={draw,
                                       circle,
                                       minimum size=2mm,
                                       inner sep=0pt,
                                       outer sep=0pt, fill}] 
      \path \foreach \i in {0,...,13}{%
       (25.714*\i:2.5) coordinate[vertex style] (a\i)}
       ; 

    \end{scope}

     \begin{scope} [edge style/.style={draw=black}]
       \foreach \i  in {0,...,13}{%
       \pgfmathtruncatemacro{\nexta}{mod(\i+1,14)} 
       \pgfmathtruncatemacro{\nextab}{mod(\i+4,14)}   
       \pgfmathtruncatemacro{\nextabc}{mod(\i+7,14)}     
       \draw[edge style,thick, donkergroen] (a\i)--(a\nextab);
       \draw[edge style,thick, blue] (a\i)--(a\nexta);
       \draw[edge style,thick, firebrick] (a\i)--(a\nextabc);       
       }  
     \end{scope}

  \end{tikzpicture} }
\end{subfigure}   \hspace{0.01\textwidth}
   \begin{subfigure}{0.32\linewidth}\scalebox{0.935}{
     \begin{tikzpicture} 
    \begin{scope} [vertex style/.style={draw,
                                       circle,
                                       minimum size=2mm,
                                       inner sep=0pt,
                                       outer sep=0pt, fill}] 
      \path \foreach \i in {0,...,13}{%
       (25.714*\i:2.5) coordinate[vertex style] (a\i)}
       ; 
    \end{scope}

     \begin{scope} [edge style/.style={draw=black}]
       \foreach \i  in {0,...,13}{%
       \pgfmathtruncatemacro{\nexta}{mod(\i+2,14)} 
       \pgfmathtruncatemacro{\nextab}{mod(\i+3,14)}   
       \pgfmathtruncatemacro{\nextabc}{mod(\i+7,14)}     
       \draw[edge style,thick, donkergroen] (a\i)--(a\nexta);
       \draw[edge style,thick, blue] (a\i)--(a\nextab);
       \draw[edge style,thick, firebrick] (a\i)--(a\nextabc);       
       }  
     \end{scope}

  \end{tikzpicture} }
   \end{subfigure}
   \caption{\small Drie grafen die we gebruiken om te bewijzen dat de onafhankelijke verzameling die~$\alpha(C_{5,14}^{\boxtimes 3})=14$ aantoont, uniek is (op Lee-equivalentie na).}
   \end{figure}
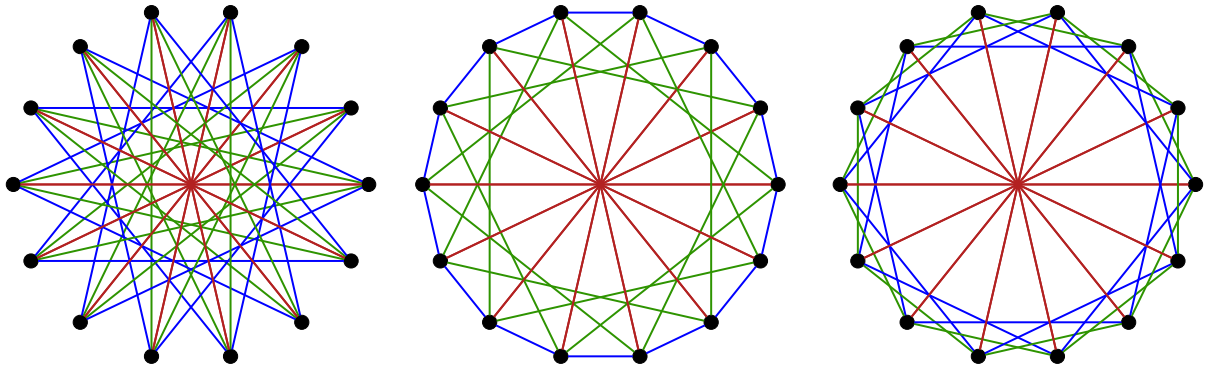 
\vspace{-12pt}
\subsubsection*{Een overkoepelend thema: onafhankelijke verzamelingen in graafproducten.}

Voor iedere graaf~$G=(V,E)$ en~$d\in \N$, definieren we het getal
  \begin{align}
\alpha_d(G) := \max \{|U| \,\, | \,\, U \subseteq V,\,\, d_G(u,v) \geq d \text{ voor alle } u\neq v\in U  \}. 
\end{align} 
Hier schrijven we~$d_G(u,v)$ voor de kleinste lengte (in lijnen) van een pad tussen~$u$ en~$v$ in~$G$. Dus~$\alpha_2(G) = \alpha(G)$. We schrijven~$K_q$ respectievelijk~$C_q$ voor de complete graaf respectievelijk het circuit op~$q$ punten. Dan geldt
\begin{align*}
A_q(n,d)&= \alpha_d(K_q^{\square n}), \phantom{veelletters}
\\A_q^L(n,d)& = \alpha_d(C_q^{\square n}),
\\ A_q^{L_{\infty}}(n,d) := \alpha(C_{d,q}^{\boxtimes n}) &= \alpha_d(C_q^{\boxtimes n}).
\end{align*}
Hierin is~$G^{\square n}$ de \emph{$n$-de Cartesisch-product-macht van~$G$}:  dit is de graaf met puntenverzameling~$V(G)^n$, waarin twee verschillende punten~$(u_1,\ldots,u_n)$ en~$(v_1,\ldots,v_n)$ verbonden zijn dan en slechts dan als er een $ i \in \{1,\ldots,n\}$ is zo dat~$u_i v_i \in E$, en~$u_j=v_j$ voor alle~$j \neq i$.

De objecten die we in dit proefschrift bestuderen hebben dus de vorm~$\alpha_d(G^n)$, waar~$G \in \{C_q, K_q\}$, en waar~$G^n$ ofwel~$G^{\boxtimes n}$ of~$G^{\square  n}$ betekent.  Verder is~$A(n,d,w)=\alpha_d(H)$, met~$H$ de deelgraaf van~$K_2^{\square n}$ ge\"induceerd door de punten~$(u_1,\ldots,u_n)$ met~$u_i=1$ voor precies~$w$ indices~$i \in \{1,\ldots,n\}$  (hier zijn de punten van~$K_2$ genummerd met~$0$ and~$1$).

\end{otherlanguage}

\newpage \thispagestyle{empty}
\phantomsection
\addcontentsline{toc}{chapter}{Index}
\markboth{Index}{Index}
\printindex[gen][Index]


\begingroup 
\makeatletter 
\extendtheindex
{\let\twocolumn\@firstoptofone 
\let\onecolumn\@firstoptofone 
\let\clearpage\relax 
}
{}
{}
{}
\makeatother 

\newpage 
\phantomsection
\addcontentsline{toc}{chapter}{List of Symbols} 
\markboth{List of Symbols}{List of Symbols}
\printindex[sym][List of symbols] 

\endgroup 

\newpage
\thispagestyle{empty} 
\mbox{}
\end{document}